\documentclass[a4paper,11pt,twoside]{article}
\usepackage[english]{babel}

\usepackage[margin=2cm]{geometry}

\usepackage{amsmath,amsthm,amssymb}
\usepackage{times}
\usepackage{enumerate}

\makeatletter
\def\author#1{\gdef\autrun{\def\and{\unskip, }#1}\gdef\@author{#1}}
 \makeatother

\usepackage{amssymb,latexsym}
\usepackage{mathrsfs}
\usepackage{amssymb}
\usepackage{amsmath,latexsym}
\usepackage{amscd,latexsym} %for communtative diagram
\usepackage[all]{xy}
\newcommand{\N}{{\mathbb N}}

\newcommand{\R}{{\mathbb R}}

\newtheorem{theorem}{Theorem}[section]

\newtheorem{corollary}[theorem]{Corollary}

\newtheorem{example}[theorem]{Example}
\newtheorem{remark}[theorem]{Remark}
\newtheorem{hypothesis}[theorem]{Hypothesis}
\newtheorem{lemma}[theorem]{Lemma}

\newtheorem{proposition}[theorem]{Proposition}
\newtheorem{claim}[theorem]{Claim}

\numberwithin{equation}{section}

\begin{document}

\title{Parameterized splitting theorems and bifurcations\\ for potential operators\thanks
{Partially supported by the NNSF  11271044 of China.  }}
\date{January 11, 2020}
\author{Guangcun Lu
\thanks{Corresponding author
%\endgraf \hspace{2mm} Partially supported
%by the NNSF  11271044 of China.
\endgraf\hspace{2mm} 2010 {\it Mathematics Subject Classification.}
 58E05, 49J52 (primary), 49J45 (secondary).
 \endgraf\hspace{2mm} {\it Key words and phrases.}
 Bifurcation, potential operator, splitting theorem,
quasi-linear elliptic Euler equations.
}}

\date{March 25, 2020}
%\subjclass[2010]{53C25, 53D35, 57R17}
%\footnotetext{{\it Key words and phrases.}  Bifurcation, potential operator, splitting theorem,
%quasi-linear elliptic Euler equations. }
%\footnotetext{{{\it Mathematics Subject Classification.}} 58E05, 49J52 (primary), 49J45 (secondary).}

 \maketitle \vspace{-0.3in}

\abstract{
We show that parameterized versions of splitting theorems in Morse theory
can be effectively used to generalize  some famous bifurcation theorems
for potential operators which may have weaker differentiability. %or be even discontinuous.
 %In particular, such generalizations based on the author's recent splitting theorems  and that of
% [Nonlinear Analysis, \textbf{18} (1992), 595-604] by N. A. Bobylev and Yu. M. Burman  are given
% though some involved potential operators  have weaker differentiability, even discontinuous.
As applications, many new bifurcation results for quasi-linear elliptic Euler equations and systems
of higher order are obtained.} \vspace{-0.1in}
\medskip\vspace{12mm}

%\maketitle \vspace{-0.5in}

%\noindent {\bf MSC-classification (2010):} 58E05, 49J52, 49J45.\\
%
%\noindent {\bf Keywords and phrases:} bifurcation, potential operator, splitting theorem,
%quasi-linear elliptic Euler equations\vspace{2mm}

%%%%%%%%%%%%%%%%%%%%%%%%%%%%%%%%%%%%%%%%%%%%%%%%%%%%%%%%%%%%%%%%%%%%%%%%%%%%%%%%%%%%%%%%%%%%%%%%%%%%%%%%%%%%%%%%%%%%%%%%%%
%%\noindent{\it Keywords}: Ekeland-Hofer symplectic capacity; Hofer-Zehnder symplectic capacity; Brunn-Minkowski type inequality; Minkowski billiard \vspace{2mm}
%%%%%%%%%%%%%%%%%%%%%%%%%%%%%%%%%%%%%%%%%%%%%%%%%%%%%%%%%%%%%%%%%%%%%%%%%%%%%%%%%%%%%%%%%%%%%%%%%%%%%%%%%%%%%%%%%%%%%%%%%%
%Critical groups;

\tableofcontents

\section{Introduction}\label{sec:Intro}
\setcounter{equation}{0}

Let $H$ be a real Hilbert space, $U$ an open neighborhood of $0$ in $H$,
 and $I$  an open interval in $\mathbb{R}$. Suppose that
$\mathcal{F}:I\times U\to\mathbb{R}$  is G\^ateaux differentiable in the second variable
and that
\begin{equation}\label{e:Intro.1}
\mathcal{F}'_u(\lambda,u)=0
\end{equation}
possesses the trivial solution $u=0$ for each $\lambda\in I$.
A point $(\mu,0)\in I\times H$ is called  a {\it bifurcation point} of (\ref{e:Intro.1})
if every neighborhood of it in $I\times H$ contains a solution $(\lambda,u)$ of (\ref{e:Intro.1})
with $u\ne 0$.
 Bifurcation theory is concerned with the structure of the solutions of (\ref{e:Intro.1}) near
 a bifurcation point.
%Such a bifurcation point for potential operators has a long study history.
 H. Poincare  \cite{Poi} initiated the mathematical study of this subject in 1885.
 %In a fundamental paper [1] published in 1885
 A. Liapunov \cite{Lia} and  E. Schmidt \cite{Sch08} reduced bifurcation theory to a finite-dimensional problem.
 M. A. Krasnosel'skii \cite{Kra} and J. Cronin \cite{Cro}
 studied bifurcation problems with  the method of topological degrees.
 The former also first used variational methods to study such problems,
which leaded to tremendous  progress  since the late 1960s, %had been made because of use of variational methods,
 see \cite{Pro, Ber72, Boh, Mar, Rab, FaRa1, FaRa2, Ber, Ch2, Ben, Wa, Kie, MaWi, Liu, BaCl, Ba1, IoSch, ChWa, Can}
  etc and references therein. Nevertheless, the key first step in all of these studies is to reduce the problem
  to a finite dimensional situation via either the Lyapunov-Schmidt  reduction or the center
 manifold theorems. It was because of this point that the functional $\mathcal{F}$ was often assumed to be $C^2$, except
  \cite{McTu, Kie, Liu, IoSch, Can},
  where functionals of class $C^{1,1}$ or even $C^1$ are considered.

In this paper we  also use variational methods, precisely Morse theory, to study bifurcation theory
of (\ref{e:Intro.1}) for the functional $\mathcal{F}$ with lower smoothness.
The reference \cite{Pro} by G. Prodi  seems to be the first paper studying bifurcation problems with Morse theory. See also
  M. Berger \cite{Ber72}, B\"ohme \cite{Boh} and  Marino \cite{Mar}.
  L. Nirenberg \cite{Ni} used the Morse lemma and the splitting theorem to study
  such problems from a different view.
 K. C. Chang \cite{Ch2} gave a Morse theory proof of Rabinowitz bifurcation theorem
\cite{Rab}. The basic idea is to study changes of
critical groups of $\mathcal{F}(\lambda,\cdot)$ at $0$ as $\lambda$ varying.
As we know, main tools  to compute critical groups are the splitting theorem and its corollary (the shifting theorem), which
are stated for $C^2$ functionals on Hilbert spaces (\cite[Theorem~I.5.1]{Ch} and \cite[Theorem~8.3]{MaWi}.
 In fact, the first step in the proof of the splitting theorem
(cf. \cite{Ch, MaWi}) is also the Lyapunov-Schmidt finite-dimensional reduction, which requires
 $C^2$-smoothness of functionals. Thus the Rabinowitz bifurcation theorem
    are inapplicable for studying bifurcations
  of quasilinear elliptic equations of potential type
since the corresponding  variational functionals on natural chosen Sobolev spaces
 cannot be of class $C^2$  in general.
Recently, the author developed  Morse theory methods for a class of
quasilinear elliptic equations  by proving some  splitting theorems for some classes of non-$C^2$ functionals \cite{Lu1}--\cite{Lu7}.
In their proofs our finite-dimensional reductions either required weaker differentiability for potential operators or
 were completed on a smaller space.
The ideas of this paper come from these studies.

%inspect changes of
%critical groups of $\mathcal{F}(\lambda,\cdot)$ at $0$ as $\lambda$ varies near $\mu$.
% The splitting theorem and its corollary (the shifting theorem), which
%are stated for $C^2$ functionals on Hilbert spaces (\cite[Theorem~I.5.1]{Ch}, \cite[Theorem~8.3]{MaWi}),
%are used to inspect changes of
%critical groups of $\mathcal{F}(\lambda,\cdot)$ at $0$ as $\lambda$ varies near $\mu$.
%In fact,

 ({\bf Notation}. Let $X, Y$ be Banach spaces, and $H$ a Hilbert space. Throughout the paper,
 we denote by $\mathscr{L}(X,Y)$ the space of all bounded linear  operators from $X$ to $Y$,
 by $\mathscr{L}(X)=\mathscr{L}(X,X)$, and by $\mathscr{L}_s(H)$ the space of all bounded linear self-adjoint
 operators on $H$, by  $B_X(x,r)$ (resp. $\bar{B}_X(x,r)$) the open (resp. closed) ball in $X$ with
 center $x\in X$ and radius $r>0$. For a map $f$ from $X$ to $Y$,
 $Df(x)$ (resp. $df(x)$ or $f'(x)$) denotes the G\^ateaux (resp. Fr\'echet) derivative of $f$ at $x\in X$,
 which is an element in $\mathscr{L}(X,Y)$. Of course, we also use $f'(x)$ to denote $Df(x)$
 without occurring of confusions. When $Y=\mathbb{R}$, $f'(x)\in \mathscr{L}(X,\mathbb{R})=X^\ast$,
 and if $X=H$ we call the Riesz representation of $f'(x)$ in $H$
 gradient of $f$ at $x$, denoted by $\nabla f(x)$. The the Fr\'echet (or G\^ateaux) derivative of $\nabla f$ at $x\in H$
 is denoted by $f''(x)$, which is an element in $\mathscr{L}_s(H)$.
($f''(x)$ can be seen as a symmetric bilinear form on $H$ without confusion occuring.)
 Moreover, we always use $0$ to
 denote the origin in different vector spaces without special statements.
 As usual $\mathbb{N}$ denotes the set of all positive integers, and $\mathbb{N}_0=\mathbb{N}\cup\{0\}$.
 The unit sphere in a Banach space $X$ is denoted by $SX$. If $G$ is a compact Lie group, and
 $M$ is a $G$-space or $G$-module, we use $M^G$ to denote the fixed point set of the $G$-action,
 i.e., $M^G=\{x\in M\,|\, gx=x\;\forall g\in G\}$. )

%In order to outline our methods, we starts with the following hypothesis under which
% a new finite dimension reduction was given in \cite{Lu6,Lu7}.
%

Our main aim is to study generalizations of some previous famous bifurcation theorems
such as those by Krasnosel'skii \cite{Kra}, by Rabinowitz  \cite{Rab}, and by Fadelll and Rabinowitz  \cite{FaRa1, FaRa2}
with parameterized versions of two new splitting theorems established by the author \cite{Lu6,Lu7} and \cite{Lu1,Lu2, Lu4}
 (see Appendix~\ref{app:A}) and the parameterized version of a splitting  lemma  by Bobylev and Burman \cite{BoBu} (see
Appendix~\ref{app:B}). The following are our basic assumptions related to them.

%Here are our fundamental hypothesises.

\begin{hypothesis}\label{hyp:1.1}
{\rm Let $H$ be a Hilbert space with inner product $(\cdot,\cdot)_H$
and the induced norm $\|\cdot\|$, and let $X$ be a dense linear subspace in $H$.
Let  $U$ be an open neighborhood of $0$ in $H$,
and let $\mathcal{L}\in C^1(U,\mathbb{R})$ satisfy $\mathcal{L}'(0)=0$.
Assume that the gradient $\nabla\mathcal{L}$ has a G\^ateaux derivative $B(u)\in \mathscr{L}_s(H)$ at every point
$u\in U\cap X$, and that the map $B: U\cap X\to
\mathscr{L}_s(H)$  has a decomposition
$B=P+Q$, where for each $x\in U\cap X$,  $P(x)\in\mathscr{L}_s(H)$ is  positive definitive and
$Q(x)\in\mathscr{L}_s(H)$ is compact. $B$, $P$ and $Q$ are also assumed to satisfy the following
properties:
\begin{description}
\item[(D1)]  $\{u\in H\,|\, B(0)u=\mu u,\;\mu\le 0\}\subset X$.
\item[(D2)] For any sequence $(x_k)\subset
U\cap X$ with $\|x_k\|\to 0$, $\|P(x_k)u-P(0)u\|\to 0$ for any $u\in H$.
\item[(D3)] The  map $Q: U\cap X\to \mathscr{L}_s(H)$ is continuous at $0$ with respect to the topology
on $H$.
\item[(D4)] For any sequence $(x_k)\subset U\cap X$ with $\|x_k\|\to 0$, there exist
 constants $C_0>0$ and $k_0\in\N$ such that
$(P(x_k)u, u)_H\ge C_0\|u\|^2$ for all $u\in H$ and for all $k\ge k_0$.
%(This condition is equivalent to (D4*) in \cite{Lu2} by \cite[Lemma~2.7]{Lu7}.)
\end{description}}
\end{hypothesis}

By \cite[Lemma~2.7]{Lu7} the condition ({\rm D4}) is equivalent to the following
\begin{description}
\item[\bf (D4*)] There exist positive constants $\eta_0>0$ and  $C'_0>0$ such that $\bar{B}_H(0,\eta_0)\subset U$ and
$$
(P(x)u, u)\ge C'_0\|u\|^2\quad\forall u\in H,\;\forall x\in
\bar{B}_H(0,\eta_0)\cap X.
$$
\end{description}

\begin{hypothesis}\label{hyp:1.2}
{\rm Let $U\subset H$ be as in Hypothesis~\ref{hyp:1.1},  $\mathcal{L}\in C^1(U,\mathbb{R})$ satisfy
$\mathcal{L}'(0)=0$  and the gradient $\nabla\mathcal{L}$ have the G\^ateaux derivative
$\mathcal{L}''(u)\in\mathscr{L}_s(H)$ at any $u\in U$, which is a compact operator
 and approaches to $\mathcal{L}''(0)$ in $\mathscr{L}_s(H)$ as $u\to 0$ in $H$.
}
\end{hypothesis}

\begin{hypothesis}\label{hyp:1.3}
{\rm Let $H$ be a Hilbert space with inner product $(\cdot,\cdot)_H$
and the induced norm $\|\cdot\|$, and let $X$ be a Banach space with
norm $\|\cdot\|_X$, such that $X\subset H$ is dense in $H$ and $\|x\|\le\|x\|_X\;\forall x\in X$. %{\bf (S)} of Appendix~\ref{app:B} is satisfied.
For an open neighborhood $U$ of $0$ in $H$, $U\cap X$
is also an open neighborhood of $0$ in $X$, denoted by $U^X$.
 Let $\mathcal{L}:U\to\mathbb{R}$ be a continuous functional  satisfying  the following
conditions:
\begin{description}
\item[(F1)] $\mathcal{L}$ is continuously directional
differentiable and $D\mathcal{L}(0)=0$.
\item[(F2)] There exists a continuous and continuously directional differentiable
 map $A: U^X\to X$, which is also  strictly Fr\'{e}chet differentiable
 at $0$,   such that
$D\mathcal{ L}(x)[u]=(A(x), u)_H$ for all $x\in U\cap X$ and $u\in X$.
\item[(F3)] There exists a map $B: U\cap X\to \mathscr{L}_s(H)$ such that
$(DA(x)[u], v)_H=(B(x)u, v)_H$ for all $x\in U\cap X$  and
$u, v\in X$. (So $B(x)$ induces an element in $\mathscr{L}(X)$, denoted by $B(x)|_X$,
and $B(x)|_X=DA(x)\in\mathscr{L}(X),\;\forall x\in U\cap X$.)
%\item[(C1)] The origin $0\in X$ is a critical point of $\mathcal{L}|_{U^X}$ (and thus $\mathcal{L}$).
\item[(C)]   $\{u\in H\,|\, B(0)(u)\in X\}\subset X$,
in particular  ${\rm Ker}(B(0))\subset X$.
\item[(D)] $B$ satisfies the same conditions as in Hypothesis~\ref{hyp:1.1}.
%For each $x\in U^X$,  $B(x)$  has a decomposition $B(x)=P(x)+ Q(x)$,
%where $P(x)\in\mathscr{L}_s(H)$ is  positive definitive  and
%$Q(x)\in\mathscr{L}_s(H)$ is a compact linear  operator with the properties (D1)--(D4) in Hypothesis~\ref{hyp:1.1}.
\end{description}}
\end{hypothesis}

({\it Note}: (F2) in \cite[\S8]{Lu2} does not have the requirement that $A: U^X\to X$ is continuous since
we consider that it is contained in the assumption of the strict Fr\'{e}chet differentiability of $A$ at $0$.
Actually, the latter may only imply that $A$ is continuous near $0\in X$, which is not sufficient
for us using the mean value theorem in the third line of \cite[p.2958]{Lu2}.
Actually, in order to assure effectiveness of such an argument we only need the assumption  that
$A: U^X\to X$ is continuous and G\^ateaux differentiable, which is  weaker since
the continuously directional differentiability of $A$ implies that $A$ is G\^ateaux differentiable.)
%Here ``continuous and continuously directional differentiable" may be replaced by ``continuous and
%G\^ateaux differentiable" by the first line of \cite[p.2958]{Lu2}.)

\begin{hypothesis}\label{hyp:1.4}
{\rm Let $H, X$ and $U\subset H$ be as in Hypothesis~\ref{hyp:1.3},
 $\mathcal{L}\in C^1(U,\mathbb{R})$ has the critical point $0\in U$, and there exist
maps $A\in C^1(U^X, X)$ and $B: U\cap X\to \mathscr{L}_s(H)$ satisfying (F2)-(F3) in
Hypothesis~\ref{hyp:1.3} and
\begin{description}
\item[(D*)] For each $x\in U\cap X$,  $B(x)\in\mathscr{L}_s(H)$ is a compact linear  operator
and the  map $B: U\cap X\to \mathscr{L}_s(H)$ is continuous at $0$ with respect to the topology on $H$.
\end{description}}
\end{hypothesis}

Clearly, Hypothesis~\ref{hyp:1.2} and Hypothesis~\ref{hyp:1.4} are stronger than
Hypothesis~\ref{hyp:1.1} and Hypothesis~\ref{hyp:1.3}, respectively.
In applications, bifurcation theorems obtained under Hypothesises~\ref{hyp:1.1},\ref{hyp:1.2}
are mainly applied to quasi-linear elliptic systems and Lagrange systems;
those obtained under Hypothesises~\ref{hyp:1.3}, \ref{hyp:1.4} or the following are
more powerful for us studying bifurcations in Lagrange systems and geodesics on Finsler manifolds.

An important case of the bifurcation problem (\ref{e:Intro.1})
is that of eigenvalues of nonlinear problems. The following are two related hypothesises.

\begin{hypothesis}\label{hyp:Bif.2.2.0}
{\rm Let tuples $(H,X,U,\mathcal{L}, A, B)$ and
 $(H,X,U, \widehat{\mathcal{L}}, \widehat{A}, \widehat{B})$ satisfy  (F1)-(F3)
  in Hypothesis~\ref{hyp:1.3}, $\mathcal{L},\widehat{\mathcal{L}}\in C^1(U,\mathbb{R})$ and $A, \widehat{A}\in C^1(U^X,X)$.
   Suppose that $\lambda^\ast$ is an eigenvalue  of
\begin{equation}\label{e:Spl.2.1}
B(0)v-\lambda\widehat{B}(0)v=0,\quad v\in H,
\end{equation}
and that  for each $\lambda$ near $\lambda^\ast$ the operator $\mathfrak{B}_{\lambda}:=B(0)-\lambda \widehat{B}(0)$
satisfies (C) in Hypothesis~\ref{hyp:1.3} and (D1) in Hypothesis~\ref{hyp:1.1}, i.e.,
$\{u\in H\,|\,\mathfrak{B}_{\lambda}u\in X\}\cup\{u\in H\,|\, \mathfrak{B}_{\lambda}u=\mu u,\;\mu\le 0\}\subset X$.
 Assume also that one of the following two conditions holds:
 \begin{description}
 \item[(I)] $\lambda^\ast\ne 0$,
 and for each $x\in U\cap X$,  $B(x)$  has a decomposition
$B(x)=P(x)+ Q(x)$,
where $P(x)\in\mathscr{L}_s(H)$ is  positive definitive,
$Q(x)\in\mathscr{L}_s(H)$ is  compact, and both satisfy the properties (D2)--(D4) in Hypothesis~\ref{hyp:1.1};
 $\widehat{B}$ satisfies (D*) in Hypothesis~\ref{hyp:1.4}, i.e., for each $x\in U\cap X$,
  $\widehat{B}(x)\in\mathscr{L}_s(H)$ is compact and  $\widehat{B}(x)\to\widehat{B}(0)$ in $\mathscr{L}_s(H)$
  as $x\to 0$ along $X\cap U$ in $H$.
    \item[(II)]  $\lambda^\ast=0$, $B$ is as in (I), and   $\widehat{B}(x)\to\widehat{B}(0)$ in $\mathscr{L}_s(H)$
  as $x\to 0$ along $X\cap U$ in $H$.
\end{description}}
\end{hypothesis}

Under Hypothesis~\ref{hyp:Bif.2.2.0}, for each $\lambda$ near $\lambda^\ast$ it is easily checked that  the functional
${\mathcal{L}}_\lambda:={\mathcal{L}}-\lambda\widehat{\mathcal{L}}$ satisfies Hypothesis~\ref{hyp:1.3}.
In fact, this also holds under the following weaker one.

\begin{hypothesis}\label{hyp:Bif.2.2.0+}
{\rm ``$\mathcal{L},\widehat{\mathcal{L}}\in C^1(U,\mathbb{R})$ and $A, \widehat{A}\in C^1(U^X,X)$"
in Hypothesis~\ref{hyp:Bif.2.2.0} is replaced by
``$\widehat{\mathcal{L}}\in C^1(U,\mathbb{R})$ and $\widehat{A}\in C^1(U^X,X)$".
}
\end{hypothesis}

Let us outline  our main ideas to study bifurcation problem of (\ref{e:Intro.1})
when $\mathcal{F}(\lambda,u)=\mathcal{L}(u)-\lambda\widehat{\mathcal{L}}(u)$, where $\mathcal{L}\in C^1(U,\mathbb{R})$ is as in Hypothesis~\ref{hyp:1.1} with $X=H$ and  $\widehat{\mathcal{L}}\in C^1(U,\mathbb{R})$ satisfies Hypothesis~\ref{hyp:1.2}.
Note that for each $\lambda\in\mathbb{R}$, $\mathcal{L}_{\lambda}:=\mathcal{L}-\lambda\widehat{\mathcal{L}}$
also satisfies Hypothesis~\ref{hyp:1.1} with $X=H$. Let $\lambda^\ast\in\mathbb{R}$
be an eigenvalue of $\mathcal{L}''(0)u=\lambda\widehat{\mathcal{L}}''(0)u$.
New finite-dimensional reduction in the proof of our parameterized splitting theorem
Theorem~\ref{th:A.2} (see \cite[Theorem~2.16]{Lu7})
shows that solving %the bifurcation problem
$$
\mathcal{L}'(u)-\lambda\widehat{\mathcal{L}}'(u)=0
$$
for $(\lambda,u)$ near $(\lambda^\ast, 0)$ in $\mathbb{R}\times H$ %can be reduced to
is equivalent to solving % a finite dimension one near $(\lambda^\ast, 0)\in\mathbb{R}\times H^0$,
 $$
 d\mathcal{L}^\circ_\lambda(z)=0 %,\quad (\lambda, z)\in (\lambda^\ast-\delta, \lambda^\ast+\delta)\times B_{H^0}(0,\epsilon),
 $$
 for $(\lambda,z)$  near $(\lambda^\ast, 0)$ in $\mathbb{R}\times H^0$,
 where $H^0={\rm Ker}(B(0)-\lambda^\ast\widehat{B}(0))$ and  $(\lambda^\ast-\delta, \lambda^\ast+\delta)\ni\lambda\mapsto \mathcal{L}^\circ_\lambda\in
 C^1(\bar{B}_{H^0}(0,\epsilon))$ is continuous.
  %(by \cite[(2.47)]{Lu7} or \cite[(2.50)]{Lu6}).
    Hence under suitable additional conditions we may carry out other
 arguments along \cite{Rab, FaRa1, FaRa2, Ch2, Wa, BaCl, Ba1} etc, and obtain many bifurcation theorems.
For example, suppose that the eigenvalue $\lambda^\ast$ is also isolated. Then for each $\lambda\ne\lambda^\ast$
 close to $\lambda^\ast$ the origin $0\in H$ is a nondegenerate critical point
 of $\mathcal{L}_{\lambda}$ in the sense of \cite{Lu6,Lu7}, in particular
  an isolated critical point of $\mathcal{L}_{\lambda}$ and thus $0\in H^\circ$
 is such a critical point of $\mathcal{L}^\circ_{\lambda}$  as well. Under some additional conditions we
 can use the parameterized shifting theorem in  \cite{Lu6,Lu7}
  to compute critical groups of $\mathcal{L}^\circ_{\lambda}$ at $0\in H^0$
 and conclude that  $\mathcal{L}^\circ_{\lambda}$ takes a local maximum (resp.
 minimum) at $0\in H^0$ as $\lambda$ varies in one (resp. other) side of $\lambda^\ast$.
 Thus  a generalization of Rabinowitz bifurcation theorem \cite{Rab} can follow from these and
 Canino's finite dimension version \cite[Theorem~5.1]{Can} %(see Theorem~\ref{th:Bi.2.1})
  for an extension of the Rabinowitz's theorem by  Ioffe and Schwartzman \cite{IoSch}.
If $H$ is equipped with an orthogonal action of  a compact Lie group $G$
 for which  $U$, $\mathcal{L}$ and $\widehat{\mathcal{L}}$ are $G$-invariant,
we may apply methods in Fadelll and Rabinowitz  \cite{FaRa1, FaRa2} to
$G$-invariant $\mathcal{L}^\circ_\lambda$ to generalize their results as $G=\mathbb{Z}_2$ or $S^1$.
Sometime, (for example, when $\mathcal{L}$ and $\widehat{\mathcal{L}}$ satisfy  Hypothesis~\ref{hyp:1.3} and
 Hypothesis~\ref{hyp:1.4}, respectively) the reduced functional  $\mathcal{L}^\circ_\lambda$ is of class $C^2$,
 we may apply a result by  Bartsch and Clapp \cite[\S4]{BaCl} to
 $G$-invariant $\mathcal{L}^\circ_\lambda$ to obtain corresponding generalizations for any compact Lie group
 because we may explicitly compute the number $d$ in \cite[\S4]{BaCl} in our situation.

Once parameterized versions of other splitting theorems on Banach spaces are established,
our above methods are applicable. For example, we  may also write the parameterized versions of the splitting
lemmas at infinity in \cite{Lu3}, and then use them to derive some theorems of bifurcations at infinity.
These will be given otherwise.

 The bifurcation theorems  by Krasnosel'skii \cite{Kra}, by Rabinowitz  \cite{Rab}, and by Fadelll and Rabinowitz  \cite{FaRa1, FaRa2},
and their previous some generalizations, often require that  functionals $\mathcal{F}(\lambda,x)$
depend on $(\lambda,x)$ in $C^2$ way and that corresponding potential operators
$\nabla_x\mathcal{F}(\lambda,x)$ have forms either $A_\lambda x+ o(\|x\|)$ or
$(A-\lambda) x+ o(\|x\|)$ for $x\to 0$. Compared these with the above six hypothesises
it is easily seen that our functionals require lower smoothness and corresponding potential operators
 might have higher nonlinearity.

\textsf{ Outline of the paper.} This paper includes three parts.

Part I, abstract theory, is the core of this paper, consisting of Sections 2,3,4,5,6.
Section~\ref{sec:CGS} contains some notions, simple recall of necessary results about stability of critical groups,
and  a proposition about the uniform (PS) condition for a family of functionals.
In Section~\ref{sec:B.2.1}, for $\mathcal{F}\in C^0(I\times U,\mathbb{R})$ with  nonlinear dependence on parameters,
 we first discuss some necessary conditions for a point $(\mu,0)\in I\times H$ to be a bifurcation point of (\ref{e:Intro.1}),
 and then  as applications of splitting theorems in Appendix~\ref{app:A}
we prove two sufficient conditions for bifurcations occuring, % results with  nonlinear dependence on parameters in Section~\ref{sec:B.1},
which are part generalizations of results by Chow and Lauterbach \cite{ChowLa} and by Kielh\"ofer \cite{Kie}.
%
%Section~\ref{sec:B.2.1} studies generalizations of Krasnoselsi potential bifurcation theorem.
In Sections~\ref{sec:B.2}, \ref{sec:B.3} we shall use the parameterized versions of  splitting theorems
\cite{Lu6,Lu7, Lu1,Lu2, Lu4} to generalize some older bifurcation theorems.
Section~\ref{sec:B.2}  is dedicated to some generalizations of  Rabinowitz bifurcation theorem  \cite{Rab}.
  % in particular  two multiparameter bifurcation results, Theorems~\ref{th:Bi.2.2} and \ref{th:Bi.2.3},
%are  also established.
Section~\ref{sec:B.3} deals with the equivariant case;
  some previous bifurcation theorems, such as those by Fadelll and Rabinowitz  \cite{FaRa1, FaRa2}
  by Bartsch and Clapp  \cite{BaCl} are generalized so that they can be used to study
variational bifurcation  for the integral functionals as in \cite[(1.3)]{Lu7}.
In order to compare our methods with previous those we state corresponding results
for $C^2$ functionals which can be obtained with our methods
at the end of each of Sections~\ref{sec:B.2.1},\ref{sec:B.2},\ref{sec:B.3}.
%%%%%%%%%%%%%%%%%%%%%%%%%%%%%%%%%%%%%%%%%%%%%%%%%%%%%%%%%%%%%%%%%%%%%%%%%%%%%%%%%%%%%%%%%%%%%%
%%Our methods are mainly based our Morse lemma, \cite[Theorem~2.1]{Lu7} (or \cite[Theorem~2.1]{Lu6}),
%%and the parameterized splitting and shifting theorems, \cite[Theorems~2.16,2.18]{Lu7}
%%(or \cite[Theorems~2.19,2.20]{Lu6}). The latter suggest that multiparameter bifurcations can be
%%studied similarly; we here give two, Theorems~\ref{th:Bi.2.2},~\ref{th:Bi.2.3}.
%%%%%%%%%%%%%%%%%%%%%%%%%%%%%%%%%%%%%%%%%%%%%%%%%%%%%%%%%%%%%%%%%%%%%%%%%%%%%%%%%%%%%%%%%%%%%%%%%%%%
In Section~\ref{sec:BBH},  for potential operators of Banach-Hilbert regular functionals
(cf. Appendix~\ref{app:B}) we prove parallel results to some  bifurcation theorems in last two sections.
%and corresponding conclusions to  \cite{ Ba1, BaCl} by Bartsch and Clapp.
 %Then, we devote the entire Part~\ref{sec:BifE}  to  give some bifurcation results
%for quasi-linear elliptic equations or systems as applications
%of abstract theories in previous four sections.
%In particular, when applying the theory in Section~\ref{sec:BBH} to
%quasi-linear elliptic systems we do not need any growth conditions,
%suitable smoothness assumptions are sufficient, and insure that bifurcation solutions
%are also classical ones.

Part II give applications of the abstract theory developed in Part I to  quasi-linear elliptic systems.
%Section~\ref{sec:BifE},
Section~\ref{sec:BifE.1} contains some fundamental hypotheses and preliminaries
on quasi-linear elliptic systems considered in next sections.
Section~\ref{sec:BifE.2} proves some bifurcations results for quasi-linear elliptic systems with some growth restrictions
 with theorems proved in Sections~\ref{sec:B.2.1},\ref{sec:B.2} and \ref{sec:B.3}.
In Section~\ref{sec:BifE.3} we use abstract theorems in Section~\ref{sec:BBH}
to study bifurcations for quasi-linear elliptic systems
without growth restrictions unless higher smoothness requirements for the related functions $F$ and domains $\Omega$.
In Section~\ref{sec:BifE.4} we  study bifurcations  for quasi-linear elliptic Dirichlet problems
 from deformations of domains, and generalize previous results
   for semilinear elliptic Dirichlet problems on a ball.

Part III consists of two appendixes. Appendix~\ref{app:A} contains
parameterized versions of splitting  lemmas by author \cite{Lu7} and \cite{Lu1, Lu2}.
  Appendix~\ref{app:B} gives parameterized versions for Bobylev-Burman splitting
   lemmas in \cite{BoBu} and related results.

\section*{Acknowledgments}
The most part of this revised version was completed during my visit in  Rutgers University.
The author is deeply grateful to Professor Xiaochun Rong for his invitation and helps.

%The author is deeply grateful to the anonymous referee for very helpful comments and suggestions to improve the exposition.

%some interesting questions,  numerous comments and

\newpage

%In Section~\ref{sec:CR}, as concluding remarks we state corresponding bifurcation theorems
%to those in the last sections for potential operators for non-$C^1$ functionals in the framework
%of \cite{Lu2} without proofs (since they are almost same).

%
%As in Theorem~\ref{th:Bif.1.1}, many bifurcation theorems in last three sections
%can be given in the setting of \cite{Lu1,Lu2}, which is more suitable for
%variational problems in Finsler geometry (\cite{Lu5}).

%Corresponding to Theorem~\ref{th:A.5} and bifurcation theorems above
%we may write the parameterized versions of the splitting lemmas at infinity in \cite{Lu3},
%and then use them to derive some theorems of bifurcations at infinity.
% Of course, it is also possible to give corresponding bifurcation results at infinity with those of
%Sections~\ref{sec:B.1},~\ref{sec:B.2},~\ref{sec:B.3}. These will be explored later.

%\section{Generalizations of a bifurcation theorem by Chow and Lauterbach}\label{sec:B.1}
%\setcounter{equation}{0}\cite{ChowLa}

%\noindent{\bf Concluding remarks}.\quad
%\section{Preliminary results}\label{PrelimSec}
%
%\subsection{Notation and conventions}

\part{Abstract bifurcation theory}\label{par:CGS}
%\section{Preliminaries}\label{sec:Pre}

%\section{Stability of critical groups and bifurcation}\label{sec:CGS}
%\setcounter{equation}{0}

%\section{Preliminary results}\label{sec:CGS}

\section{Preliminary materials}\label{sec:CGS}
\setcounter{equation}{0}

%\subsection{Palais-Smale condition and stability of critical groups}\label{sec:PS}

In this section, for the reader¡¯s reference we list necessary notions and
results on stability of critical groups, % recall some notions
 and then give related  propositions.

%In Subsection 2.1, we define measure solutions to (1.3), present some equivalent formalisms, and recall the regularity theory.
%In Subsection 2.2, we establish a priori estimates for measure solutions to (1.3).
%2.1. Measure solutions and regularity

Recall that a $C^1$ functional $f$ on a Banach space $X$ is said to satisfy
the {\bf Palais-Smale condition} ((PS) {\bf condition}, for short)  in a closed subset $S\subset X$ if
every sequence $(u_n)\subset S$ with $f'(u_n)\to 0$ and $(f(u_n))$  bounded
 has a subsequence converging to $u\in X$. For some $c\in\mathbb{R}$, if
 ``$(f(u_n))$  bounded" is replaced by ``$f(u_n)\to c$", we say that $f$ satisfies
 the {\bf Palais-Smale condition at level $c$} ($(PS)_c$ {\bf condition}, for short) in $S$.
 A family of $C^1$ functionals on $X$  parameterized by
  a metric space $\Lambda$,  $\{f_\lambda\}_{\lambda\in\Lambda}$, is called to satisfy
 the {\bf uniform Palais-Smale condition} ({\bf uniform} (PS) {\bf condition}, for short)  in a closed subset $S\subset X$ if
every sequence $(\lambda_n, u_n)\subset \Lambda\times S$ such that $f'_{\lambda_n}(u_n)\to 0$ and $(f_{\lambda_n}(u_n))$  bounded
 has a subsequence converging to $(\lambda,u)\in \Lambda\times S$
 with $f'_{\lambda}(u)=0$  (cf. \cite[Definition~2.4]{CorH}).
 (If $\Lambda$ is compact we can assume $\lambda_n\to\lambda\in\Lambda$ in this definition.)

\begin{lemma}[\hbox{\cite[Remark 2.1(b)]{CorH}}]\label{lem:pre.1}
For a $C^1$ functional $f$ on a Banach space (or $C^1$ Finsler manifold) $X$ with norm $\|\cdot\|$,
 the {\bf weak slope} $|df|(u)$  of $f$ (as a continuous functional on the metric space $X$)  at any $u\in X$ (cf. \cite[Definition (2.1)]{DegMa}) is equal to $\|f'(u)\|$.
\end{lemma}

%As a direct consequence of this and \cite[Remark 2.3(b)]{CorH} we have:
%
%\begin{lemma}\label{lem:pre.2}
%Let $S$ be a closed subset in  a Banach space $X$, and let $\{f_\lambda\}_{\lambda\in\Lambda}$
%be a family of $C^1$ functionals on $X$  parameterized by  a  metric space $\Lambda$.
%Suppose that the following two conditions are satisfied:
%\begin{description}
%\item[(i)] every $f_\lambda$ is bounded on $S$ and satisfies the (PS) condition on $S$;
%
%\item[(ii)] for every $\lambda_0\in\Lambda$ and every closed subset $R$ of $S$, there holds
%$$
%\lim_{\lambda\to\lambda_0}\inf\left(\inf_{u\in R}\|f'_\lambda(u)\|\right)\ge\inf_{u\in R}\|f'_{\lambda_0}(u)\|.
%$$
%\end{description}
%Then $\{f_\lambda\}_{\lambda\in\Lambda}$ satisfies  the  uniform (PS) condition on $S$.
% \end{lemma}
%

%By \cite[Proposition~2.2]{CorH} we immediately obtain:
%
%  \begin{claim}\label{cl:pre.2}
%Let $S$ be a closed subset in  a Banach space $X$, and let $\{f_\lambda\}_{\lambda\in\Lambda}$
%be a family of $C^1$ functionals on $X$  parameterized by  a compact metric space $\Lambda$.
%Suppose that the following three conditions are satisfied:
%\begin{description}
%\item[(i)] $\Lambda\times S\to\mathbb{R},\;  (\lambda,u)\mapsto f_\lambda(u)$ is continuous and bounded;
%\item[(ii)] $\{f_\lambda\}_{\lambda\in\Lambda}$ satisfies  the  uniform (PS) condition on $S$;
%\item[(iii)] for every $\lambda\in\Lambda$, $f_\lambda$ has a unique  critical point $u_\lambda$ in $S$.
%\end{description}
%Then $\lambda\mapsto u_\lambda$  is continuous.
% \end{claim}

This lemma and \cite[Th.5.1]{CorH} directly lead to:

\begin{theorem}\label{th:stablity1}
For a family of functionals on a Banach space $X$, $\{f_\sigma\in C^1(X,\mathbb{R})\,|\, \sigma\in [0,1]\}$,
suppose that there exists an open set $U$ such that
\begin{description}
\item[(i)]  $z_\sigma\in U$ is a unique critical point of $f_\sigma$ in $\overline{U}$, $\forall\sigma\in [0,1]$;
\item[(ii)] $\sigma\to f_\sigma$ is continuous in $C^1(\overline{U})$ topology;
\item[(iii)]  $f_\sigma$ satisfies the (PS) condition in $\overline{U}$,  $\forall\sigma\in [0,1]$.
\end{description}
Then $C_\ast(f_\sigma,z_\sigma;{\bf K})$ is independent of $\sigma$ for any Abel group ${\bf K}$.
\end{theorem}

Actually, \cite[Th.5.1]{CorH} was stated for ${\bf K}=\mathbb{R}$. However, the proof therein is still effective for any
any Abel group ${\bf K}$ because the arguments involving in ${\bf K}$ is to use \cite[Prop.5.2]{CorH} and the latter
also clearly holds true if $\mathbb{R}$ is replaced by ${\bf K}$ by its proof.
Moreover, by \cite[Prop.3.7]{Cor} the definition of critical groups used in \cite{CorH} is equivalent to the usual one
as in \cite{Ch,Ch1,MaWi, McTu}, i.e.,
$$
C_q(f,z;{\bf K})=H_q(\{f\le f(z)\}\cap U, \{f\le f(z)\}\cap(U\setminus\{z\});{\bf K})
$$
if $z\in X$ is an isolated lower critical point of $f$, where $U$ is a neighborhood of $z$.

When $X$ is a Hilbert space, Theorem~\ref{th:stablity1} was proved in  \cite[Theorem~8.8]{MaWi} (if $f_\sigma\in C^{2-0}(X,\mathbb{R})$),
in Theorem~5.6 of \cite[Chap.I]{Ch} (if $f_\sigma\in C^2(X,\mathbb{R})$), and in \cite[Corollary~5.1.25]{Ch1} (if $f_\sigma\in C^1(X,\mathbb{R})$).

The conditions (ii)-(iii) in Theorem~\ref{th:stablity1} can also be replaced by others.

\begin{theorem}[\hbox{\cite[Th.3.6]{CiDe}}]\label{th:stablity2}
For a family of functionals on a Banach space $X$, $\{f_\sigma\in C^1(X,\mathbb{R})\,|\, \sigma\in [0,1]\}$,
suppose that there exists an open set $U$ such that
\begin{description}
\item[(i)] for each $\sigma\in [0,1]$, $z_\sigma\in U$ is a unique critical point of $f_\sigma$ in $\overline{U}$, and
            $[0,1]\ni \sigma\to z_\sigma\in U$ is also continuous;
\item[(ii)] $\sigma\to f_\sigma$ is continuous in $C^0(\overline{U})$ topology;
\item[(iii)] $\{f_\sigma\}_{\sigma\in[0,1]}$ satisfies the uniform (PS) condition on $\overline{U}$, that is,
for every sequence $\sigma_k\to \sigma$ in $[0,1]$ and $(u_k)$ in $\overline{U}$ with $f'_{\sigma_k}(u_k)\to 0$ and
$(f_{\sigma_k}(u_k))$ bounded, there exists a subsequence $(u_{k_j})$ convergent to some $u$ with $f'_\sigma(u)=0$.
\end{description}
Then $C_\ast(f_\sigma,z_\sigma;{\bf K})$ is independent of $\sigma$ for any Abel group ${\bf K}$.
\end{theorem}

This result has also a corresponding version at the level of continuous functionals, \cite[Theorem~3.1]{CiDe}.

%
%\begin{theorem}[ \hbox{\cite[Theorem~5.1.21]{Ch1}}]\label{th:chang}
%Let $M$ be a $C^1$ Banach-Finsler manifold.  Suppose that $\{f_\sigma\in C^1(M,\mathbb{R})\,|\, \sigma\in [0,1]\}$ is a family of
%functions satisfying the (PS) condition. Assume that $a(\sigma)$ and $b(\sigma)$ are two
%continuous functions defined on $[0, 1]$ with $a(\sigma)<b(\sigma)$, and that both $a(\sigma)$
%and $b(\sigma)$ are regular values of $f_\sigma$, $\forall\sigma\in [0,1]$. Assume that $\sigma\mapsto f_\sigma$ is continuous
%in $C^1(M)$ topology. Then the homology group $H_\ast((f_\sigma)_{b(\sigma)}, (f_\sigma)_{a(\sigma)};G)$ is
%independent of $\sigma$.
%\end{theorem}

 Let $X$ be a Banach space with dual space $X^\ast$.
 A map $T$ from a subset $D$ of $X$ to $X^\ast$ is said to be of {\bf class $(S)_+$ }
 if  for any sequence $(u_j)\subset D$  with  converging weakly to $u$ in $X$
  for which $\varlimsup_{j\to\infty}\langle T(u_j), u_j-u\rangle\le 0$, it
follows that $(u_j)$ converges strongly to $u$ in $X$.
 The definition was introduced by Browder \cite{Bro} and  Skrypnik \cite{Skr} (the condition of belonging to class $(S)_+$ is called condition $\alpha$ in the latter paper).

\begin{proposition}\label{prop:stablity}
Under Hypothesis~\ref{hyp:1.1} or Hypothesis~\ref{hyp:1.3} the restriction of the gradient $\nabla\mathcal{L}$
to a small neighborhood of $0\in H$ is  of the class $(S)_+$.
Thus $\mathcal{L}$ satisfies the (PS) condition in a closed neighborhood of  $0\in H$.
\end{proposition}

The case that Hypothesis~\ref{hyp:1.3} is satisfied  was proved in the proof of Theorem~2.12 of \cite[p.2966-2967]{Lu2}.
 If Hypothesis~\ref{hyp:1.1} holds the proposition may be proved
by almost same arguments.  It may also be contained in the proof of the following more general claim.

\begin{proposition}\label{prop:stablity1}
Let $H$, $X$ and $U$ be as in Hypothesis~\ref{hyp:1.1},  $I$  an  interval  in $\mathbb{R}$.
  Suppose that $\{\mathcal{F}_\lambda\in C^1(U,\mathbb{R})\,|\,\lambda\in I\}$  satisfies:
$\mathcal{F}'_\lambda(0)=0$,  the gradient $\nabla\mathcal{F}_\lambda$ has a G\^ateaux derivative $B_\lambda(u)\in \mathscr{L}_s(H)$ at every point
$u\in U\cap X$, and that the map $B_\lambda: U\cap X\to
\mathscr{L}_s(H)$  has a decomposition
$B_\lambda=P_\lambda+Q_\lambda$, where for each $x\in U\cap X$,  %$P(x)\in\mathscr{L}_s(H)$ is  positive definitive and
$Q_\lambda(x)\in\mathscr{L}_s(H)$ is compact. For
some small $\delta>0$ with $\bar{B}_H(0, \delta)\subset U$ operators $P_\lambda$, $Q_\lambda$ and $\nabla\mathcal{F}_\lambda$ are also assumed to satisfy the following
properties:
 \begin{description}
%\item[(i)]   $\lambda\mapsto \mathcal{F}_\lambda$
%    is continuous at $\lambda=0$ in $C^0(\bar{B}_H(0, \delta))$ topology;
     %and     $(\lambda,u)\mapsto\nabla\mathcal{F}_{\lambda}(u)$ is continuous on  $I\times\bar{B}_H(0, \delta)$;
 \item[(i)]   There exist positive constants $c_0>0$ such that
$$
(P_\lambda(x)u, u)\ge c_0\|u\|^2\quad\forall u\in H,\;\forall x\in
\bar{B}_H(0,\delta)\cap X,\quad\forall\lambda\in I;
$$
 \item[(ii)]  As $x\in U\cap X$ approaches $0$ in $H$,
$Q_\lambda(x)\to Q_\lambda(0)$ in $\mathscr{L}_s(H)$ uniformly  with respect to $\lambda\in I$;
 %$Q_\lambda:  U\cap X\to \mathscr{L}_s(H)$ is uniformly continuous at $0$  with respect to $\lambda\in I$;
 \item[(iii)]  If $(\lambda_n)\subset I$ converges to $\lambda\in I$ then
  $$
  \|Q_{\lambda_n}(0)-Q_\lambda(0)\|\to 0\quad\hbox{and}\quad \|\nabla\mathcal{F}_{\lambda_n}(x)-\nabla\mathcal{F}_\lambda(x)\|\to 0\quad\forall x\in \bar{B}_H(0, \delta)\cap X.
  $$
\end{description}
 Then for any sequences $\lambda_n\to\lambda_0$ in $I$ and  $(u_n)\subset\bar{B}_H(0, \delta)$ such that $\mathcal{F}'_{\lambda_n}(u_n)\to 0$ and $(\mathcal{F}_{\lambda_n}(u_n))$ is bounded, there exists  a subsequence $u_{n_k}\to u_0\in \bar{B}_H(0, \delta)$
with $\mathcal{F}'_{\lambda_0}(u_0)=0$.

Moreover, if $X$ itself is a normed line space with norm $\|\cdot\|$ such that the inclusion $X\hookrightarrow H$ is
continuous (and so $U\cap X$ is an open neighborhood of $0$ in $X$, denoted by $U^X$ as before),
then the sentence ``\textsf{the gradient $\nabla\mathcal{F}_\lambda$ has a G\^ateaux derivative $B_\lambda(u)\in \mathscr{L}_s(H)$ at every point
$u\in U\cap X$}" in the above statement  can be replaced by
``\textsf{there exists a map $B_\lambda: U^X\to \mathscr{L}_s(H)$ and a continuous and continuously directional differentiable
 map $A_\lambda: U^X\to X$ such that
 \begin{eqnarray}\label{e:pre.1}
D\mathcal{F}_\lambda(x)[u]=(A_\lambda(x), u)_H\quad\hbox{and}\quad
(DA_\lambda(x)[u], v)_H=(B_\lambda(x)u, v)_H
\end{eqnarray}
 for all $x\in U^X$  and $u, v\in X$}".
\end{proposition}
\begin{proof}
By (ii) we may shrink $\delta>0$ so that
$$
\|Q_\lambda(x)-Q_\lambda(0)\|<\frac{c_0}{2},\quad\forall x\in \bar{B}_H(0, \delta)\cap X,\quad\forall\lambda\in I.
$$
It follows from this  and (i) that for all $x\in \bar{B}_H(0, \delta)\cap X$ and $u\in H$,
\begin{eqnarray}\label{e:pre.2}
\bigl(B_\lambda(x)u,u\bigr)_H&=&\bigl(P_\lambda(x)u,u\bigr)_H+
\bigl([Q_\lambda(x)-Q_\lambda(0)]u,u\bigr)_H+ \bigl(Q_\lambda(0)u, u\bigr)_H\nonumber\\
&\ge& \frac{c_0}{2}\|u\|^2+ \bigl(Q_\lambda(0)u, u\bigr)_H.
\end{eqnarray}
Clearly, $(u_n)$ has a subsequence $(u_{n_k})$ weakly converging to $u_0\in \bar{B}_H(0, \delta)$.
Since $\bar{B}_H(0, \delta)\cap X$ is dense in $\bar{B}_H(0, \delta)$, and $\mathcal{F}_\lambda\in C^1(U)$, we have sequences
$(v_{n_k})\subset\bar{B}_H(0, \delta)\cap X$ and $(u_{0m})\subset\bar{B}_H(0, \delta)\cap X$ such that
for all $k,m=1,2,\cdots$,
\begin{eqnarray}\label{e:pre.3}
\|v_{n_k}-u_{n_k}\|<\frac{1}{k},\quad\|\nabla\mathcal{F}_{\lambda_{n_k}}(v_{n_k})-\nabla\mathcal{F}_{\lambda_{n_k}}(u_{n_k})\|<\frac{1}{k}
\quad\hbox{and}\quad\|u_{0m}-u_0\|<\frac{1}{m}.
\end{eqnarray}
Since $\nabla\mathcal{F}_\lambda$ is continuous and has a G\^ateaux derivative $B_\lambda(u)\in \mathscr{L}_s(H)$ at every point
$u\in U\cap X$, for each fixed $k$ using the mean value theorem we have $\tau\in (0, 1)$ such that
\begin{eqnarray*}
&&(\nabla\mathcal{F}_{\lambda_{n_k}}(v_{n_k}), v_{n_k}-u_{0m})_H\\
&=&(\nabla\mathcal{F}_{\lambda_{n_k}}(v_{n_k})-\nabla\mathcal{F}_{\lambda_{n_k}}(u_{0m}),
v_{n_k}-u_{0m})_H-(\nabla\mathcal{F}_{\lambda_{n_k}}(u_{0m}), v_{n_k}-u_{0m})_H\\
%&=&(A_{\lambda_{n_k}}(v_{n_k})-A_{\lambda_{n_k}}(u_{0m}), v_{n_k}-u_{0m})_H-(\nabla\mathcal{F}_{\lambda_{n_k}}(u_{0m}), v_{n_k}-u_{0m})_H\\
&=&\bigl(D(\nabla\mathcal{F}_{\lambda_{n_k}})(\tau v_{n_k}+ (1-\tau)u_{0m})[v_{n_k}-u_{0m}], v_{n_k}-u_{0m}\bigr)_H-(\nabla\mathcal{F}_{\lambda_{n_k}}(u_{0m}), v_{n_k}-u_{0m})_H\\
&=&\bigl(B_{\lambda_{n_k}}(\tau v_{n_k}+ (1-\tau)u_{0m})(v_{n_k}-u_{0m}), v_{n_k}-u_{0m}\bigr)_H-(\nabla\mathcal{F}_{\lambda_{n_k}}(u_{0m}), v_{n_k}-u_{0m})_H\\
&\ge& \frac{c_0}{2}\|v_{n_k}-u_{0m}\|^2-(\nabla\mathcal{F}_{\lambda_{n_k}}(u_{0m}), v_{n_k}-u_{0m})_H+
(Q_{\lambda_{n_k}}(0)(v_{n_k}-u_{0m}), v_{n_k}-u_{0m})_H,
\end{eqnarray*}
where the final inequality is because of (\ref{e:pre.2}).
Since $\mathcal{F}_\lambda\in C^1(U)$ letting $m\to\infty$ we get
\begin{eqnarray*}
(\nabla\mathcal{F}_{\lambda_{n_k}}(v_{n_k}), v_{n_k}-u_{0})_H&\ge& \frac{c_0}{2}\|v_{n_k}-u_{0}\|^2-(\nabla\mathcal{F}_{\lambda_{n_k}}(u_{0}), v_{n_k}-u_{0})_H\\
&&+(Q_{\lambda_{n_k}}(0)(v_{n_k}-u_{0}), v_{n_k}-u_{0})_H,\quad\forall k=1,2,\cdots.
\end{eqnarray*}
Note that (\ref{e:pre.3}) implies $\lim_{k\to\infty}\nabla\mathcal{F}_{\lambda_{n_k}}(v_{n_k})=\lim_{k\to\infty}\nabla\mathcal{F}_{\lambda_{n_k}}(u_{n_k})=0$ and $v_{n_k}\rightharpoonup u_0$.
It easily follows these  and (iii) that
$$
\frac{c_0}{2}\lim_{k\to\infty}\|v_{n_k}-u_0\|^2\le
\lim_{k\to\infty}(\nabla\mathcal{F}_{\lambda_{n_k}}(v_{n_k}), v_{n_k}-u_0)_H=0
$$
and thus $u_{n_k}\to u_0$ in $H$. The desired claim is proved.

In order to prove the second part, we only need to change the five lines below (\ref{e:pre.3}) into:\\
``Since $A_{\lambda}: U^X\to X$  is continuous, and continuously directional differentiable,
 and satisfies (\ref{e:pre.1}),  for each fixed $k$ using the mean value theorem we have $\tau\in (0, 1)$ such that
\begin{eqnarray*}
&&(\nabla\mathcal{F}_{\lambda_{n_k}}(v_{n_k}), v_{n_k}-u_{0m})_H\\
&=&(\nabla\mathcal{F}_{\lambda_{n_k}}(v_{n_k})-\nabla\mathcal{F}_{\lambda_{n_k}}(u_{0m}),
v_{n_k}-u_{0m})_H-(\nabla\mathcal{F}_{\lambda_{n_k}}(u_{0m}), v_{n_k}-u_{0m})_H\\
&=&(A_{\lambda_{n_k}}(v_{n_k})-A_{\lambda_{n_k}}(u_{0m}), v_{n_k}-u_{0m})_H-(\nabla\mathcal{F}_{\lambda_{n_k}}(u_{0m}), v_{n_k}-u_{0m})_H\\
&=&\bigl(DA_{\lambda_{n_k}}(\tau v_{n_k}+ (1-\tau)u_{0m})[v_{n_k}-u_{0m}], v_{n_k}-u_{0m}\bigr)_H-(\nabla\mathcal{F}_{\lambda_{n_k}}(u_{0m}), v_{n_k}-u_{0m})_H
%&=&\bigl(B_{\lambda_{n_k}}(\tau v_{n_k}+ (1-\tau)u_{0m})(v_{n_k}-u_{0m}), v_{n_k}-u_{0m}\bigr)_H-(\nabla\mathcal{F}_{\lambda_{n_k}}(u_{0m}), v_{n_k}-u_{0m})_H\\
%&\ge& \frac{c_0}{2}\|v_{n_k}-u_{0m}\|^2-(\nabla\mathcal{F}_{\lambda_{n_k}}(u_{0m}), v_{n_k}-u_{0m})_H+
%(Q_{\lambda_{n_k}}(0)(v_{n_k}-u_{0m}), v_{n_k}-u_{0m})_H,
\end{eqnarray*}
"
the same conclusion is obtained.
 \end{proof}

\begin{corollary}\label{cor:stablity2}
%Let $\mathcal{L}\in C^1(U,\mathbb{R})$ and $\widehat{\mathcal{L}}\in C^1(U,\mathbb{R})$ satisfy
Suppose that one of the following two conditions holds:
 \begin{description}
\item[(I)]   $\mathcal{L}\in C^1(U,\mathbb{R})$ and $\widehat{\mathcal{L}}\in C^1(U,\mathbb{R})$ satisfy
 Hypothesis~\ref{hyp:1.1} and  Hypothesis~\ref{hyp:1.2}, respectively.
 \item[(II)]  $\mathcal{L}\in C^1(U,\mathbb{R})$ and $\widehat{\mathcal{L}}\in C^1(U,\mathbb{R})$ satisfy
 Hypothesis~\ref{hyp:1.3} and  Hypothesis~\ref{hyp:1.4}, respectively.
 %
% $\mathcal{L}\in C^1(U,\mathbb{R})$ satisfy
% Hypothesis~\ref{hyp:1.3} and $\widehat{\mathcal{L}}\in C^1(U,\mathbb{R})$ satisfy  Hypothesis~\ref{hyp:1.4}
% with corresponding operators $\widehat{A}$ and $\widehat{B}=\widehat{P}+\widehat{Q}$.
 \end{description}
  Then for a bounded interval $I$ in $\mathbb{R}$ and for
   some small $\delta>0$ with $\bar{B}_H(0, \delta)\subset U$
  the conclusion of Proposition~\ref{prop:stablity1} holds true  on $\bar{B}_H(0, \delta)$
for the family $\{\mathcal{F}_\lambda:=\mathcal{L}-\lambda\widehat{\mathcal{L}}\,|\, \lambda\in I\}$,
 that is, for any sequences $\lambda_n\to\lambda_0$ in $I$ and  $(u_n)\subset\bar{B}_H(0, \delta)$ such that
 $\mathcal{F}'_{\lambda_n}(u_n)\to 0$ and $(\mathcal{F}_{\lambda_n}(u_n))$ is bounded, there exists
   a subsequence $u_{n_k}\to u_0\in \bar{B}_H(0, \delta)$
with $\mathcal{F}'_{\lambda_0}(u_0)=0$.
\end{corollary}

%for any sequences $\lambda_n\to\lambda_0$ in $I$ and  $(u_n)\subset\bar{B}_H(0, \delta)$ such that $\mathcal{F}'_{\lambda_n}(u_n)\to 0$ and $(\mathcal{F}_{\lambda_n}(u_n))$ is bounded, there exists  a subsequence $u_{n_k}\to u_0\in \bar{B}_H(0, \delta)$
%with $\mathcal{F}'_{\lambda_0}(u_0)=0$.

\begin{proof}
{\bf Case (I)}. %If $\mathcal{L}\in C^1(U,\mathbb{R})$ satisfies  Hypothesis~\ref{hyp:1.1},
The gradient $\nabla\mathcal{F}_\lambda=\nabla\mathcal{L}-\lambda\nabla\mathcal{G}$ has a G\^ateaux derivative $B_\lambda(u)=B(u)-\lambda\widehat{\mathcal{L}}''(u)\in \mathscr{L}_s(H)$ at every point $u\in U\cap X$, and that $B_\lambda(x)=P_\lambda(x)+Q_\lambda(x)=P(x)+Q_\lambda(x)$ for each $x\in U\cap X$, where
$Q_\lambda(x)=Q(x)-\lambda\widehat{\mathcal{L}}''(x)\in \mathscr{L}_s(H)$ is compact.
By (D4) or (D4*) in Hypothesis~\ref{hyp:1.1}, $P_\lambda\equiv P$ satisfies Proposition~\ref{prop:stablity1}(i) for some small $\delta>0$ with $\bar{B}_H(0, \delta)\subset U$.
Since the internal $I\subset\mathbb{R}$ is bounded, it is easily seen that
the conditions (ii) and (iii) in Proposition~\ref{prop:stablity1} are satisfied.

%Suppose that $\mathcal{L}\in C^1(U,\mathbb{R})$ satisfies  Hypothesis~\ref{hyp:1.3}.
{\bf Case (II)}. For each $x\in U\cap X$, let $A_\lambda(x)=A(x)-\lambda\widehat{\mathcal{L}}'(x)$, $B_\lambda(x)=B(x)-\lambda\widehat{\mathcal{L}}''(x)$, $P_\lambda\equiv P$  and $Q_\lambda(x)=Q(x)-\lambda\widehat{\mathcal{L}}''(x)$.
Similar arguments show that the conditions of Proposition~\ref{prop:stablity1} may be satisfied.
\end{proof}

%\label{sec:B.2.1S}

\section{New necessary conditions and sufficient criteria for  bifurcations of gradient mappings }\label{sec:B.2.1}

%\section{Generalizations of Krasnoselsi potential bifurcation theorem}\label{sec:B.2.1}

%For the bifurcation problem of (\ref{e:Intro.1})
%there exist a series of studies initiated by
%Krasnosel'ski in the 50s .

For a completely continuous operator $A:U\to H$,
%i.e., it is continuous and compact (the latter means that it maps each bounded subset of $U$
%into a relatively compact subset of $H$,
 if  $A$ is the gradient of a weakly continuous and uniformly
differentiable functional $f:U\to\mathbb{R}$ and has the Fr\'echet derivative $A'(0)$,
(which must be compact and self-adjoint,) Krasnosel'ski \cite{Kra}
%used minimax arguments to
proved that  each nonzero eigenvalue $\mu$ of $A'(0)$ gives a bifurcation point
$(\mu,0)$ of the equation $A(u)=\lambda u$.
More precisely, for any sufficiently small $r>0$ there exists $\lambda_r\in\mathbb{R}$, $u_r\in SH$
such that $A(u_r)=\lambda_r u_r$ and $\lambda_r\to\mu$ as $r\to 0$.
Since then several extensions and improvements of this  work have been made.
For example, if $f\in C^2(U, \mathbb{R})$ satisfies $f'(0)=0$
and $\mu$ is an isolated eigenvalue of $f^{\prime\prime}(0)$ of finite multiplicity,
Rabinowitz \cite{Rab74} proved that  $(\mu,0)$
is a bifurcation point of the equation $\nabla f(u)=\lambda u$.
See \cite{Rab74,To} and references therein for further details.
%Recently, for $f\in C^1(U,\mathbb{R})$ whose gradient $A=\nabla f:U\to H$ has
%the Fr\'echet derivative $T=A'(0)$ at $0\in H$, if $A(0)=0$ and $T$ is self-adjoint and compact,
%Elliot Tonkes \cite{To} showed that the largest eigenvalue $\mu_0$ of $T$ gave
%a bifurcation point $(1/\mu_0, 0)$ of $A(x)=\lambda x$.
In this section we give  necessary conditions and
sufficient criteria for a point $(\mu,0)\in I\times H$ to be a bifurcation point of (\ref{e:Intro.1})
under some new assumptions of $\mathcal{F}$. More sufficient criteria
will be given in next three sections.

\subsection{Necessary conditions}\label{sec:B.2.1N}

%In this subsection we give a few generalizations of
%the classical Krasnoselsi potential bifurcation theorem \cite{Kra}.
%% the  necessity part of Theorem~12 in \cite[Chapter~4, \S4.3]{Skr}
%%(including the classical Krasnoselsi potential bifurcation theorem \cite{Kra}).

%\begin{theorem}\label{th:Bi.1.1}
\begin{theorem}\label{th:Ka1}
Let $H$, $X$ and $U$ be as in Hypothesis~\ref{hyp:1.1},
% be a real Hilbert space, $U$ an open neighborhood of $0$ in $H$,
 and $\Lambda$ a topological space.
  For each $\lambda\in\Lambda$, let $\{\mathcal{F}_\lambda\in C^1(U, \mathbb{R})\,|\,\lambda\in\Lambda\}$
    satisfy $\mathcal{F}'_\lambda(0)=0$, and let the gradient $\nabla\mathcal{F}_\lambda$ have
 a G\^ateaux derivative $B_\lambda(u)\in \mathscr{L}_s(H)$ at every point
$u\in U\cap X$. Suppose that the map $B_\lambda: U\cap X\to
\mathscr{L}_s(H)$  has a decomposition $B_\lambda=P_\lambda+Q_\lambda$, where for each $x\in U\cap X$,
 $P_\lambda(x)\in\mathscr{L}_s(H)$ is  positive definitive and
$Q_\lambda(x)\in\mathscr{L}_s(H)$ is compact, and
   that $P_\lambda$ and $Q_\lambda$ satisfy the following conditions:
    \begin{description}
\item[(i)]  If $(x_k)\subset U\cap X$ approaches to $0$ in $H$ and $(\lambda_n)\subset \Lambda$ converges to $\lambda^\ast$ then
$\|P_{\lambda_k}(x_k)h-P_{\lambda^\ast}(0)h\|\to 0$ for each $h\in H$.
 \item[(ii)]  For some small $\delta>0$, there exists a positive constant $c_0>0$ such that
$$
(P_\lambda(x)u, u)\ge c_0\|u\|^2\quad\forall u\in H,\;\forall x\in
\bar{B}_H(0,\delta)\cap X,\quad\forall\lambda\in \Lambda.
$$
 \item[(iii)]  $Q_\lambda: U\cap X\to \mathscr{L}_s(H)$ is uniformly continuous at $0$  with respect to $\lambda\in \Lambda$.
  \item[(iv)]  If $(\lambda_n)\subset \Lambda$ converges to $\lambda^\ast$ then
  $\|Q_{\lambda_n}(0)-Q_\lambda(0)\|\to 0$.
  \end{description}
Then $0\in H$ is a degenerate critical point of   $\mathcal{F}_{\lambda^\ast}$ if
  $({\lambda}^\ast,0)\in \Lambda\times U$ is a bifurcation point of
    \begin{equation}\label{e:Ka0}
 \mathcal{F}'_{{\lambda}}(u)=0,\quad (\lambda,u)\in\Lambda\times U.
\end{equation}
%Under the above assumptions,
%if $({\lambda}^\ast,0)\in \Lambda\times U$ is a bifurcation point of
%    \begin{equation}\label{e:Ka0}
% \mathcal{F}'_{{\lambda}}(u)=0,\quad (\lambda,u)\in\Lambda\times U,
%\end{equation}
%  then  ${\rm Ker}(B_{\lambda^\ast}(0))\ne\{0\}$, i.e., $0\in H$ is a degenerate critical point of
%  $\mathcal{F}_{\lambda^\ast}$.
  \end{theorem}

\begin{proof}
%[Proof of Theorem~\ref{th:Bi.2.2}]
 Since $({\lambda}^\ast, 0)\in\Lambda\times U$ is a bifurcation point  of the equation  (\ref{e:Ka0}),
 there exists a sequence $({\lambda}_k, \bar{u}_k)\in\Lambda\times(U\setminus\{0\})$
such that ${\lambda}_k\to{\lambda}^\ast$, $\bar{u}_k\to 0$ and
 \begin{equation}\label{e:KBi.2.2+}
 \mathcal{F}'_{{\lambda}_k}(\bar{u}_k)=0,\quad\forall k\in\mathbb{N}.
\end{equation}
  Passing to a subsequence, if necessary, we can assume $\bar{v}_k=\bar{u}_k/\|\bar{u}_k\|\rightharpoonup v^\ast$.
By (ii)  there exist positive constants $\eta_0>0$ and  $c_0>0$ such that $B_H(0,\eta_0)\subset U$ and
 \begin{equation}\label{e:KBi.2.3}
(P_\lambda(u)h, h)_H\ge c_0\|h\|^2,\quad\forall h\in H,\;\forall u\in
B_H(0,\eta_0)\cap X,\;\forall\lambda\in I.
\end{equation}
 Clearly, we can assume that $(\bar{u}_k)\subset B_H(0,\eta_0)\setminus\{0\}$.
Since $B_H(0,\eta_0)\cap X$ is dense in $B_H(0,\eta_0)$ we may choose
 $({u}_k)\subset (B_H(0,\eta_0)\cap X)\setminus\{0\}$ such that
% $$
% \|u_k-\bar{u}_k\|\to 0\quad\hbox{and so}\quad u_k\to 0,\;v_k:={u}_k/\|{u}_k\|\rightharpoonup v^\ast,
% $$
%and that for each $k\in\mathbb{N}$,
  \begin{eqnarray}\label{e:KBi.2.3.0}
&&  \|u_k-\bar{u}_k\|\to 0\quad\hbox{and so}\quad u_k\to 0,\;v_k:={u}_k/\|{u}_k\|\rightharpoonup v^\ast,\\
&&\left\|\frac{1}{\|u_k\|}\nabla\mathcal{F}_{{\lambda}_k}(u_k)-
  \frac{1}{\|\bar{u}_k\|}\nabla\mathcal{F}_{{\lambda}_k}(\bar{u}_k)\right\|<\frac{1}{2^k},\;\forall k,\label{e:KBi.2.3.1}\\
&&\biggl|\frac{1}{\|u_k\|^2}(\nabla\mathcal{F}_{{\lambda}_k}(u_k), u_k)_H
-\frac{1}{\|\bar{u}_k\|^2}(\nabla\mathcal{F}_{{\lambda}_k}(\bar{u}_k), \bar{u}_k)_H\biggr|<\frac{1}{2^k},\;\forall k.\label{e:KBi.2.3.1+}
\end{eqnarray}
  For each fixed $k$, since $\nabla\mathcal{F}_{\lambda_k}(0)=0$
  using the Mean Value Theorem  we get $t_k\in (0, 1)$ such that
\begin{eqnarray}\label{e:KBi.2.4.1}
\frac{1}{\|u_k\|^2}(\nabla\mathcal{F}_{{\lambda}_k}(u_k), u_k)_H&=&
(D(\nabla\mathcal{F}_{{\lambda}_k})(t_ku_k)v_k, v_k)_H\nonumber\\
&=&(P_{{\lambda}_k}(t_ku_k)v_k, v_k)_H +(Q_{{\lambda}_k}(t_ku_k)v_k, v_k)_H\nonumber\\
&\ge& c_0+(Q_{{\lambda}_k}(t_ku_k)v_k, v_k)_H\nonumber\\
&=& c_0+([Q_{{\lambda}_k}(t_ku_k)-Q_{{\lambda}_k}(0)]v_k, v_k)_H+ (Q_{{\lambda}_k}(0)v_k, v_k)_H.\nonumber\\
\end{eqnarray}
Since $t_ku_k\to 0$ by (\ref{e:KBi.2.3.0}),  using (iii) and (iv), respectively,   we deduce that for sufficiently $k$,
$$
|([Q_{{\lambda}_k}(t_ku_k)-Q_{{\lambda}_k}(0)]v_k, v_k)_H|<\frac{c_0}{4}\quad\hbox{and}\quad
 |(Q_{{\lambda}_k}(0)v_k, v_k)_H-(Q_{{\lambda}^\ast}(0)v^\ast, v^\ast)_H|< \frac{c_0}{2}
$$
and thus (\ref{e:KBi.2.4.1}) leads to
\begin{eqnarray*}\label{e:KBi.2.4.1+}
\frac{1}{\|u_k\|^2}(\nabla\mathcal{F}_{{\lambda}_k}(u_k), u_k)_H>
\frac{c_0}{2}+ (Q_{{\lambda}^\ast}(0)v^\ast, v^\ast)_H.
\end{eqnarray*}
 By (\ref{e:KBi.2.2+}) and (\ref{e:KBi.2.3.1+}) the left side approaches to zero.
Hence $v^\ast\ne 0$.

 %In order to prove that $v^\ast$ and $\vec{\lambda}^\ast$ also satisfy (\ref{e:Bi.2.2}),
Obverse that (\ref{e:KBi.2.2+}) and (\ref{e:KBi.2.3.1}) imply
  \begin{equation}\label{e:KBi.2.6}
\Big|\frac{1}{\|u_k\|}(\nabla\mathcal{F}_{{\lambda}_k}(u_k), h)_H\Bigr|\le\frac{1}{2^k}\|h\|,\quad
\forall h\in H,\quad\forall k\in\mathbb{N}.
 \end{equation}
  Fixing  $h\ne 0$, as in (\ref{e:KBi.2.4.1}), for some $\tau_k\in (0, 1)$, depending on $u_k$ and $h$,
 \begin{eqnarray}\label{e:KBi.2.7+}
\frac{1}{\|u_k\|}(\nabla\mathcal{F}_{{\lambda}_k}(u_k), h)_H&=&(D(\nabla\mathcal{F}_{{\lambda}_k})(\tau_ku_k)v_k, h)_H\nonumber\\
&=&(P_{{\lambda}_k}(\tau_ku_k)v_k, h)_H +(Q_{{\lambda}_k}(\tau_ku_k)v_k, h)_H\nonumber\\
&=&(v_k, P_{{\lambda}_k}(\tau_ku_k)h)_H +(v_k, Q_{{\lambda}_k}(\tau_ku_k)h)_H.
\end{eqnarray}
As above, by (iii)-(iv) we deduce that $(v_k, Q_{{\lambda}_k}(\tau_ku_k)h)_H\to (v^\ast, Q_{{\lambda}^\ast}(0)h)_H$.
Moreover, (i) implies that $P_{{\lambda}_k}(\tau_ku_k)h\to P_{{\lambda}^\ast}(0)h$. It follows from these
and (\ref{e:KBi.2.3.0}), (\ref{e:KBi.2.6}) and (\ref{e:KBi.2.7+}) that
$$
(v^\ast, P_{{\lambda}^\ast}(0)h)_H +(v^\ast, Q_{{\lambda}^\ast}(0)h)_H=0,\quad\forall h\in H,
$$
and thus $D(\nabla\mathcal{F}_{{\lambda}^\ast})(0)v^\ast=0$.
%That is, $\vec{\lambda}^\ast$ is an eigenvalue  of (\ref{e:Bi.2.2}).
\end{proof}

Similarly, we have
\begin{theorem}\label{th:Ka2}
In Theorem~\ref{th:Ka1}, if we replace ``Hypothesis~\ref{hyp:1.1}"  by
``Hypothesis~\ref{hyp:1.3}", and  ``the gradient $\nabla\mathcal{F}_\lambda$ has
 a G\^ateaux derivative $B_\lambda(u)\in \mathscr{L}_s(H)$ at every point
$u\in U\cap X$" by ``there exists a map $B_\lambda: U\cap X\to \mathscr{L}_s(H)$ and a continuous and continuously directional differentiable
 map $A_\lambda: U^X\to X$ such that $D\mathcal{ L}(x)[u]=(A_\lambda(x), u)_H$ and
$(DA_\lambda(x)[u], v)_H=(B_\lambda(x)u, v)_H$ for all $x\in U\cap X$  and
$u, v\in X$", then the conclusion of Theorem~\ref{th:Ka1} also is true.
\end{theorem}

Indeed, in the proof of Theorem~\ref{th:Ka1} we only need  to replace
$D(\nabla\mathcal{F}_{{\lambda}_k})(t_ku_k)$ and $D(\nabla\mathcal{F}_{{\lambda}_k})(\tau_ku_k)$
by $A(t_ku_k)$ and $A(\tau_ku_k)$, respectively, and complete the proof of Theorem~\ref{th:Ka2}.

%\noindent{\bf 3.1}.\quad {\bf A generalization of Krasnoselsi potential bifurcation theorem}.

\begin{corollary}\label{cor:Bi.2.2}
 Let $\mathcal{L}\in C^1(U,\mathbb{R})$ satisfy
 Hypothesis~\ref{hyp:1.1} with $X=H$, and let
  $\widehat{\mathcal{L}}_j\in C^1(U,\mathbb{R})$, $j=1,\cdots,n$, satisfy
 Hypothesis~\ref{hyp:1.2}.
Suppose that $(\vec{\lambda}^\ast, 0)\in\mathbb{R}^n\times U$ is a (multiparameter) bifurcation point  for the equation
\begin{equation}\label{e:Bi.2.1}
\mathcal{L}'(u)=\sum^n_{j=1}\lambda_j\widehat{\mathcal{L}}'_j(u),\quad u\in U.
\end{equation}
Then $\vec{\lambda}^\ast=(\lambda^\ast_1,\cdots,\lambda^\ast_n)$ is an  eigenvalue  of
\begin{equation}\label{e:Bi.2.2}
\mathcal{L}''(0)v-\sum^n_{j=1}\lambda_j\widehat{\mathcal{L}}''_j(0)v=0,\quad v\in H,
\end{equation}
 that is, $0$ is
a degenerate critical point of the functional $\mathcal{L}-\sum^n_{j=1}\lambda^\ast_j\widehat{\mathcal{L}}_j$
in the sense stated above Theorem~\ref{th:A.1}.
Moreover, if $\vec{\lambda}^\ast=0$, we only need  that
each $\widehat{\mathcal{L}}_j\in C^1(U,\mathbb{R})$ has properties:
$\widehat{\mathcal{L}}_j'(0)=0$  and the gradient $\nabla\widehat{\mathcal{L}}_j$ has the G\^ateaux derivative
$\mathcal{L}''(u)\in\mathscr{L}_s(H)$ at any $u\in U$, which approaches to
 $\mathcal{L}''(0)$ in $\mathscr{L}_s(H)$ as $u\to 0$ in $H$.
\end{corollary}

This result generalizes  the  necessity part of Theorem~12 in \cite[Chapter~4, \S4.3]{Skr}
(including the classical Krasnoselsi potential bifurcation theorem \cite{Kra}).
The  sufficiency  part of Theorem~12 in \cite[Chapter~4, \S4.3]{Skr} is contained in the case that the
condition (a) in Corollary~\ref{cor:Bi.2.4.2} holds.

 Denoted by  $H(\vec{\lambda})$ the solution space  of (\ref{e:Bi.2.2}).
 It is of finite dimension as the kernel of a linear Fredholm operator.

When $n=1$, comparing with Theorem~12 in \cite[Chapter~4, \S4.3]{Skr1}, the latter also required:
\begin{description}
\item[(a)]  $\widehat{\mathcal{L}}$ is
weakly continuous and uniformly differentiable in $U$;
\item[(b)] $\mathcal{L}'$ has uniformly positive definite Frech\`et derivatives and satisfies the condition
$\alpha)$ in \cite[Chapter~3, \S2.2]{Skr1}.
\end{description}

%If $\nabla\mathcal{G}$ is completely continuous (that is, mapping a weakly convergent sequence
%into a convergent one in norm) and has Frech\'et  derivative
%$\mathcal{G}''(u)$ at $u\in U$, then $\mathcal{G}''(u)\in\mathscr{L}(H)$ is a compact linear operator
%(cf. \cite[Remark~2.4.6]{Ber}).

\begin{proof}[Proof of Corollary~\ref{cor:Bi.2.2}]
First, we consider the case that $\widehat{\mathcal{L}}_j\in C^1(U,\mathbb{R})$, $j=1,\cdots,n$, satisfy
 Hypothesis~\ref{hyp:1.2}.
Let $B=P+Q$ and $\widehat{B}_j=\widehat{P}_j+\widehat{Q}_j$ be the
 corresponding operators with $\mathcal{L}$ and
  $\widehat{\mathcal{L}}_j$, $j=1,\cdots,n$, respectively.
 Then $\widehat{P}_j=0$ and $\widehat{Q}_j=\widehat{\mathcal{L}}''_j$ for $j=1,\cdots,n$.
 Let $\mathcal{F}_{\vec{\lambda}}(u)=\mathcal{L}(u)-\sum^n_{j=1}\lambda_j\widehat{\mathcal{L}}'_j(u)$.
 Denote by $B_{\vec{\lambda}}=P_{\vec{\lambda}}+Q_{\vec{\lambda}}$
   the  corresponding operators.
   Then $P_{\vec{\lambda}}=P$ and $Q_{\vec{\lambda}}=Q-\sum^n_{j=1}\lambda_j\widehat{\mathcal{L}}''_j(u)$.
  Take $\Lambda$ to be any compact neighborhood of $\vec{\lambda}^\ast$ in $\mathbb{R}^n$.
  It is easily checked that  $\{\mathcal{F}_{\vec{\lambda}}\,|\,\vec{\lambda}\in\Lambda\}$
   satisfies the conditions of Theorem~\ref{th:Ka1}.

    Next, we prove the part of ``Moreover".
 We can assume that there exist positive constants $\eta_0>0$ and  $c_0>0$ such that $B_H(0,\eta_0)\subset U$ and
 \begin{equation}\label{e:KBi.2.31}
(P(u)h, h)_H\ge 2c_0\|h\|^2,\quad\forall h\in H,\;\forall u\in
B_H(0,\eta_0)\cap X.
\end{equation}
Since each $\widehat{\mathcal{L}}''_j(u)$ is continuous at $0\in H$, we can shrink $\eta_0>0$ and choose a
small compact neighborhood  $\Lambda$ of $0\in\mathbb{R}^n$ such that
\begin{eqnarray}\label{e:KBi.2.32}
&&\|\widehat{\mathcal{L}}''_j(u)-\widehat{\mathcal{L}}''_j(0)\|<1,\quad\forall u\in
B_H(0,\eta_0),\;j=1,\cdots,n,\\
&&\sum^n_{j=1}|\lambda_j(\widehat{\mathcal{L}}''_j(u)h,h)_H|<c_0\|h\|^2, \quad\forall u\in
B_H(0,\eta_0),\;\forall \vec{\lambda}\in\Lambda,\;\forall h\in H.\label{e:KBi.2.33}
\end{eqnarray}
Let us write $Q_{\vec{\lambda}}\equiv Q$ and $P_{\vec{\lambda}}(u)=P(u)-\sum^n_{j=1}\lambda_j\widehat{\mathcal{L}}''_j(u)$.
Then (\ref{e:KBi.2.33}) and (\ref{e:KBi.2.31}) lead to
 \begin{equation}\label{e:KBi.2.34}
(P_{\vec{\lambda}}(u)h, h)_H\ge c_0\|h\|^2,\quad\forall h\in H,\;\forall u\in
B_H(0,\eta_0)\cap X.
\end{equation}
Moreover, if $(x_k)\subset B_H(0,\eta_0)\cap X$ approaches to $0$ in $H$ and $(\vec{\lambda}_k)\subset \Lambda$ converges to $0$ then
for any fixed $h\in H$ we deduce from (\ref{e:KBi.2.33}) and the definition of $P_{\vec{\lambda}}$ that
\begin{eqnarray*}
\|P_{\vec{\lambda}_k}(x_k)h-P_{0}(0)h\|\le \sum^n_{j=1}|\lambda_{k,j}|\|\widehat{\mathcal{L}}''_j(x_k)h\|\le
\sum^n_{j=1}|\lambda_{k,j}|(\|\widehat{\mathcal{L}}''_j(0)\|+1)\|h\|\to 0.
\end{eqnarray*}
 These show that $Q_{\vec{\lambda}}$ and $P_{\vec{\lambda}}$ for $\vec{\lambda}\in\Lambda$
satisfy the assumptions of Theorem~\ref{th:Ka1}.
\end{proof}

Using Theorem~\ref{th:Ka2} and similar reasoning we can obtain

\begin{corollary}\label{cor:Bi.2.2*}
 Let $\mathcal{L}\in C^1(U,\mathbb{R})$ satisfy
 Hypothesis~\ref{hyp:1.3}, and let
  $\widehat{\mathcal{L}}_j\in C^1(U,\mathbb{R})$, $j=1,\cdots,n$, satisfy
 Hypothesis~\ref{hyp:1.4}.
Suppose that $(\vec{\lambda}^\ast, 0)\in\mathbb{R}^n\times U^X$ is a (multiparameter) bifurcation point  for the equation
\begin{equation}\label{e:Bi.2.1*}
A(u)=\sum^n_{j=1}\lambda_j\widehat{A}_j(u),\quad u\in U^X.
\end{equation}
Then $\vec{\lambda}^\ast=(\lambda^\ast_1,\cdots,\lambda^\ast_n)$ is an  eigenvalue  of
\begin{equation}\label{e:Bi.2.2*}
B(0)v-\sum^n_{j=1}\lambda_j\widehat{B}_j(0)v=0,\quad v\in H,
\end{equation}
 that is, $0$ is
a degenerate critical point of the functional $\mathcal{L}-\sum^n_{j=1}\lambda^\ast_j\widehat{\mathcal{L}}_j$
in the sense stated above Theorem~\ref{th:A.1}.
 Moreover, if $\vec{\lambda}^\ast=0$, we only need
that each $\widehat{\mathcal{L}}_j\in C^1(U,\mathbb{R})$ satisfies
 Hypothesis~\ref{hyp:1.4} without requirement that $B(x)\in\mathscr{L}_s(H)$ is compact.
\end{corollary}

Conversely, if $\vec{\lambda}^\ast$ is an isolated eigenvalue  of (\ref{e:Bi.2.2*}), under some additional conditions
we shall show in Theorem~\ref{th:Bi.2.3} that $(\vec{\lambda}^\ast, 0)\in\mathbb{R}^n\times U^X$ is a  bifurcation point
 of (\ref{e:Bi.2.1*}).

%In this subsection we shall due to Krasnoselsi    \cite{Kra}

\subsection{Sufficient criteria}\label{sec:B.2.1S}

%\subsection{Bifurcations with  nonlinear dependence on parameters}\label{sec:B.1}
%\setcounter{equation}{0}

%In this section we present some bifurcation results as direct applications of stability of
%critical groups and splitting theorems in \cite{Lu6,Lu7}.

Changes of Morse type numbers imply existence of bifurcation instants \cite{Boh, Ber72}.
Different generalizations are given in \cite{ChowLa,Kie, MaWi, SmoWa, PRS,Ryb}.
For example,  \cite[Theorem~8.8]{MaWi} showed that changes of critical groups lead to bifurcations.
However, it is difficult to compute critical groups for non-twice continuously differentiable functionals.
In this section, with helps of splitting theorems in \cite{Lu2,Lu3, Lu6,Lu7}
we give some general bifurcation results for potential operator families of a class of non-twice continuously differentiable functionals.
% which partly generalizations  two extensions of \cite{ChowLa}
% rather than their parameterized versions

Let $H$ be a real Hilbert space,  $I$  a bounded open interval containing $0$ in $\mathbb{R}$,
and $\{B_\lambda\}_{\lambda\in I}\subset \mathscr{L}_s(H)$
such that $\|B_\lambda-B_0\|\to 0$ as $\lambda\to 0$.
Suppose that  $0$ is an isolated point of the spectrum $\sigma(B_0)$
with $n=\dim{\rm Ker}(B_0)\in (0, \infty)$, and that
${\rm Ker}(B_\lambda)=\{0\}\;\forall\pm\lambda\in (0,\varepsilon_0)$ for some
positive number $\varepsilon_0\ll 1$.
%Then $B_0$ and hence each $B_\lambda$ is
%a Fredholm operator of index zero if $\varepsilon_0$ is small enough.
%A result in \cite[Section IV.5.4]{Ka} implies that the generalized eigenspace
%$E_0$ of the eigenvalue $0$ is finite-dimensional too.
%
% the generalized eigenspace $E_0$ perturbs to $E_\lambda$ of the same finite dimension, and the
%eigenprojection $P_\lambda$ depends continuously on $\lambda$
% The eigenvalue $0$ of $B_0$ perturbs to eigenvalues of $B_\lambda$ near $0$
%(the so-called $0$-group in Kato's terminology in \cite{Ka}), which are the eigenvalues of $B_\lambda\in L(E_\lambda)$ (which
%is a finite-dimensional operator).
By \cite[Remark I.21.1]{Kie1}  the generalized eigenspace $E_0$ of $B_0$ with eigenvalue $0$
is equal to $N(B_0)={\rm Ker}(B_0)$ (i.e., the algebraic and geometric multiplicities of $0$ are same),
and $H=E_0\oplus R(B_0)$, where $R(B_0)={\rm Im}(B_0)$.
It was shown in  Sections II.5.1 and III.6.4 of \cite{Ka} that %the number $n$ is an invariant in the following sense:
the generalized eigenspace $E_0$ %of the eigenvalue $0$ of $B_0$ having dimension $n$
is perturbed to an invariant space $E_\lambda$  of $B_\lambda$ of dimension $n$,
and all perturbed eigenvalues near $0$ (the so-called $0$-group  in Kato's terminology in \cite{Ka}),
 %\cite[page 107]{Ka})
denoted by ${\rm eig}_0(B_\lambda)$,
are eigenvalues of the finite-dimensional operator $B_\lambda$ restricted to the $n$-dimensional
invariant space $E_\lambda$ (cf. \cite[Remark II.4]{Kie1}). In other words,
${\rm eig}_0(B_\lambda)$ is the set of eigenvalues of $B_\lambda$ which approach $0$ as $\lambda\to 0$.
%
%The eigenvalues in that $0$-group depend continuously on $\lambda$.
Hence shrinking $\varepsilon_0$ if necessary,  for each $\lambda\in (-\varepsilon_0,\varepsilon_0)\setminus\{0\}$, $B_\lambda$ has exactly
$n$ eigenvalues (depending  continuously on $\lambda$) near zero, and none of them is zero.
%By the arguments on the pages 107 and 203 in \cite{Ka},
%In Kato's terminology in \cite[page 107]{Ka},
%we have the so-called $0$-group ${\rm eig}_0(B_\lambda)$ consisting
% of eigenvalues of $B_\lambda$ which approach $0$ as $\lambda\to 0$.
 Let $r(B_\lambda)$ be the number of elements in ${\rm eig}_0(B_\lambda)\cap\mathbb{R}^-$ and
 \begin{equation}\label{e:Bi.1.1}
 r^+_{B_\lambda}=\lim_{\lambda\to 0+}r(B_\lambda),\qquad
 r^-_{B_\lambda}=\lim_{\lambda\to 0-}r(B_\lambda).
 \end{equation}
Then if $|\lambda|>0$ is small enough we have $\mu_\lambda-\mu_0=r^+_{B_\lambda}$ for $\lambda>0$,
and $\mu_\lambda-\mu_0=r^-_{B_\lambda}$ for $\lambda<0$, where $\mu_\lambda$
is the dimension of negative definite space of $B_\lambda$.

In some sense the following may be viewed as a converse of Theorem~\ref{th:Ka1}.

\begin{theorem}\label{th:Bi.1.1}
Let $H$ and $I$ be as above,  $U$ an open neighborhood of $0$ in $H$,
 and let $\mathcal{F}:I\times U\to\mathbb{R}$ be such that  each
    $\mathcal{F}_\lambda:=\mathcal{F}(\lambda,\cdot)$ satisfies  Hypothesis~\ref{hyp:1.1} on $U$
    with  corresponding operators  $B_\lambda$, $P_\lambda$ and $Q_\lambda$.
Suppose that the following eight conditions are satisfied:
    \begin{description}
\item[(i)]  For some small $\delta>0$, $\lambda\mapsto \mathcal{F}_\lambda$
    is continuous at $\lambda=0$ in $C^0(\bar{B}_H(0, \delta))$ topology.
 \item[(ii)]  For some small $\delta>0$, there exist positive constants $c_0>0$ such that
$$
(P_\lambda(x)u, u)\ge c_0\|u\|^2\quad\forall u\in H,\;\forall x\in
\bar{B}_H(0,\delta)\cap X,\quad\forall\lambda\in I.
$$
 \item[(iii)]  $Q_\lambda: U\cap X\to \mathscr{L}_s(H)$ is uniformly continuous at $0$  with respect to $\lambda\in I$.
% \item[(iii)]  $Q_\lambda: U\to \mathscr{L}_s(H)$ is uniformly continuous at $0$ with respect to $\lambda\in I$;
  \item[(iv)]  If $(\lambda_n)\subset I$ converges to $\lambda\in I$ then
  $$
  \|Q_{\lambda_n}(0)-Q_\lambda(0)\|\to 0\quad\hbox{and}\quad \|\nabla\mathcal{F}_{\lambda_n}(x)-\nabla\mathcal{F}_\lambda(x)\|\to 0\quad\forall x\in U\cap X.
  $$
%\item[(v)] for some small $\delta>0$, $\lambda\mapsto \mathcal{F}_\lambda$
%    is continuous at $\lambda=0$ in $C^0(\bar{B}_H(0, \delta))$ topology; and
%    for any sequences $\lambda_n\to\lambda_0$ in $I$ and  $(u_n)\subset\bar{B}_H(0, \delta)$ such that $\mathcal{F}'_{\lambda_n}(u_n)\to 0$ and
%    $(\mathcal{F}_{\lambda_n}(u_n))$ is bounded, there exists  a subsequence $u_{n_k}\to u_0\in \bar{B}_H(0, \delta)$
%with $\mathcal{F}'_{\lambda_0}(u_0)=0$.
\item[(v)]  ${\rm Ker}(B_\lambda(0))=\{0\}$
 for small $|\lambda|\ne 0$.
 \item[(vi)] $B_\lambda(0)\to B_0(0)$  as $\lambda\to 0$;
\item[(vii)] $0\in\sigma(B_0(0))$.
\item[(viii)]  $r^+_{B_\lambda(0)}\ne
r^-_{B_\lambda(0)}$.
 \end{description}
   Then $(0,0)\in I\times U$ is  a bifurcation point of the equation (\ref{e:Intro.1}).
 Moreover, the same conclusion still holds if the above conditions (i)-(iv) are replaced by the following two
\begin{description}
\item[(a)]  for some small $\delta>0$, $\lambda\mapsto \mathcal{F}_\lambda$
    is continuous at $\lambda=0$ in $C^1(\bar{B}_H(0, \delta))$ topology;
 \item[(b)]  for some small $\delta>0$, each $\mathcal{F}_\lambda$ satisfies the (PS) condition in $\bar{B}_H(0, \delta)$.
\end{description}
\end{theorem}

\begin{proof}
By a contradiction, suppose that $(0,0)\in I\times U$ is not a bifurcation point of the equation
(\ref{e:Intro.1}). Then by shrinking $\delta>0$  we can find $0<\varepsilon_0\ll 1$  such that for each
$\lambda\in [-\varepsilon_0,\varepsilon_0]$ the functional $\mathcal{F}_\lambda$ has a unique
critical point $0$ sitting in $\bar{B}_H(0, \delta)$.

By the first part of Proposition~\ref{prop:stablity1} we see that $\{\mathcal{F}_\lambda\}_{|\lambda|\le\varepsilon_0}$ satisfies  the  uniform (PS) condition on
 $\bar{B}_H(0, \delta)$. It follows from this, (i) and  Theorem~\ref{th:stablity2} that
 \begin{eqnarray}\label{e:Bi.1.2}
 C_\ast(\mathcal{F}_\lambda, 0;{\bf K})=C_\ast(\mathcal{F}_0, 0;{\bf K}),\quad\forall \lambda\in [-\varepsilon_0,\varepsilon_0].
 \end{eqnarray}

It remains to prove  that the assumptions (v)-(viii) insure that (\ref{e:Bi.1.2}) cannot occur.

By (v), we can assume that $0$ is a nondegenerate critical point of $\mathcal{F}_\lambda$
for $0<|\lambda|\le\varepsilon_0$ by shrinking $\varepsilon_0>0$ if necessary.
It follows from this, (\ref{e:Bi.1.2}) and
Theorem~\ref{th:A.1} with $\lambda=0$ (or \cite[Theorem~2.1]{Lu7})
that all $\mathcal{F}_\lambda$, $0<|\lambda|\le\varepsilon_0$,
have the same Morse index $\mu_\lambda$ at $0\in H$, that is,
\begin{eqnarray}\label{e:Bi.1.3}
 [-\varepsilon_0, \varepsilon_0]\setminus\{0\}\ni\lambda\mapsto \mu_\lambda\quad\hbox{is constant.}
 \end{eqnarray}

On the other hand, by \cite[Proposition~B.2]{Lu2}, each $\varrho\in\sigma(B_0(0))\cap\{t\in\mathbb{R}^-\,|\,
t\le 0\}$ is an isolated point in $\sigma(B_0(0))$, which is also an
eigenvalue of finite multiplicity. (This can also be derived from \cite[Lemma~2.2]{BoBu}.)
Since $0\in\sigma(B_0(0))$ by (vii),
$0$ is an isolated point of the spectrum $\sigma(B_0(0))$ and an eigenvalue of $B_0(0)$ of the
 finite multiplicity $s_0$ by \cite[Lemma~2.2]{BoBu}.
Thus we can assume
$$
\sigma(B_0(0))\cap\{t\in\mathbb{R}^-\,|\,
t\le 0\}=\{0,\varrho_1,\cdots,\varrho_k\},
$$
 where $\varrho_i<0$ and has multiplicity $s_i$ for each $i=1,\cdots,k$.
As above (\ref{e:Bi.1.1}), we may use this and (vi) to prove: %and the arguments on the pages 107 and 203 in \cite{Ka},
if $0<|\lambda|$ is small enough, $B_\lambda(0)$
 has exactly $s_i$ (possible same) eigenvalues near $\varrho_i$, but total dimension
 of corresponding eigensubspaces is equal to that of eigensubspace of $\varrho_i$.
Hence if $\lambda\in (0, \varepsilon_0]$ (resp. $-\lambda\in (0,\varepsilon_0]$) is small enough
we obtain
$\mu_\lambda=\mu_0+ r^+_{B_\lambda(0)}$
(resp. $\mu_{-\lambda}=\mu_0+ r^-_{B_\lambda(0)}$).
These and (viii) imply
$$
\mu_\lambda-\mu_{-\lambda}=r^+_{B_\lambda(0)}-
r^-_{B_\lambda(0)}\ne 0\quad\hbox{for small  $\lambda\in (0, \varepsilon_0]$},
$$
which contradicts the claim in (\ref{e:Bi.1.3}).

In order to prove the final part, note that under the assumptions of the first paragraph
we may use (a)-(b) and Theorem~\ref{th:stablity1} to derive (\ref{e:Bi.1.2}).
The remained arguments are same.
\end{proof}

Correspondingly, we have the following converse of Theorem~\ref{th:Ka2}.

\begin{theorem}\label{th:Bif.1.1}
 Let $H$, $X$ and $U$ be as in Hypothesis~\ref{hyp:1.3},
  and $I\subset\mathbb{R}$  a bounded  open interval containing $0$.
  Let $\mathcal{F}:I\times U\to\mathbb{R}$ be such that  each
    $\mathcal{F}_\lambda:=\mathcal{F}(\lambda,\cdot)$ satisfies  Hypothesis~\ref{hyp:1.3}
    with  corresponding operators $A_\lambda$, $B_\lambda$, $P_\lambda$ and $Q_\lambda$.
    Suppose  that either (i)-(viii) in Theorem~\ref{th:Bi.1.1} or
(a)-(b) and (v)-(viii) in Theorem~\ref{th:Bi.1.1}
 are satisfied.  Then $(0,0)\in I\times U^X$ is  a bifurcation point of the equation
%(\ref{e:Intro.1}).
\begin{eqnarray}\label{e:Bi.1.3.1}
 A_\lambda(x)=0.
 \end{eqnarray}
%   \begin{description}
%\item[(i)]  for some small $\delta>0$, $\lambda\mapsto \mathcal{F}_\lambda$
%    is continuous at $\lambda=0$ in $C^0(\bar{B}_H(0, \delta))$ topology;
%      \item[(ii)]  for some small $\delta>0$, there exist positive constants $c_0>0$ such that
%$$
%(P_\lambda(x)u, u)\ge c_0\|u\|^2\quad\forall u\in H,\;\forall x\in
%\bar{B}_H(0,\delta)\cap X, \;\forall\lambda\in I;
%$$
% \item[(iii)]  $Q_\lambda: U\cap X\to \mathscr{L}_s(H)$ is uniformly continuous at $0$ in $I$ with respect to the topology on $U$;
% \item[(iv)]  if $(\lambda_n)\subset I$ converges to $\lambda\in I$ then
%  $$
%  \|Q_{\lambda_n}(0)-Q_\lambda(0)\|\to 0\quad\hbox{and}\quad \|\nabla\mathcal{F}_{\lambda_n}(x)-\nabla\mathcal{F}_\lambda(x)\|\to 0\quad\forall x\in
%\bar{B}_H(0,\delta)\cap X;
%  $$
%\item[(v)]  ${\rm Ker}(B_\lambda(0))=\{0\}$
% for small $|\lambda|\ne 0$;
% \item[(vi)] $B_\lambda(0)\to B_0(0)$  as $\lambda\to 0$;
%\item[(vii)] $0\in\sigma(B_0(0))$;
%\item[(viii)]  $r^+_{B_\lambda(0)}\ne
%r^-_{B_\lambda(0)}$.
% \end{description}
%  Moreover, the same conclusion still holds if the above conditions (i)-(iv) are replaced by the conditions (a)-(b) in
%Theorem~\ref{th:Bi.1.1}.
\end{theorem}

%\begin{proof}
The proof is almost repeating that of Theorem~\ref{th:Bi.1.1}. In fact, it suffices to replace
 ``By Proposition~\ref{prop:stablity1}"  and
``Theorem~\ref{th:A.1} with $\lambda=0$ (or \cite[Theorem~2.1]{Lu7})"
with   ``By the latter part of Proposition~\ref{prop:stablity1}" and
``Theorem~\ref{th:A.4} with $\lambda=0$ (or \cite[(2.7)]{Lu2})", respectively.
%  we immediately obtain the following theorem.
%\end{proof}

\begin{corollary}\label{cor:Bi.3}
Let $\lambda^\ast\in\mathbb{R}$.   Suppose that $\mathcal{L}\in C^1(U,\mathbb{R})$ satisfy
 Hypothesis~\ref{hyp:1.1} without (D1),  $\widehat{\mathcal{L}}\in C^1(U,\mathbb{R})$ satisfy  Hypothesis~\ref{hyp:1.2},
 and that for each real $\lambda$ near $\lambda^\ast$ there holds:
 \begin{description}
  \item[(I.a)] $\{u\in H\,|\, {\mathcal{L}}''(0)u-\lambda\widehat{\mathcal{L}}''(0)u=\mu u,\,\mu\le 0\}\subset X$ for each $\lambda$ near $\lambda^\ast$;
  \item[(I.b)] ${\rm Ker}({\mathcal{L}}''(0)-\lambda^\ast\widehat{\mathcal{L}}''(0))\ne\{0\}$;
 \item[(I.c)] ${\rm Ker}({\mathcal{L}}''(0)-\lambda\widehat{\mathcal{L}}''(0))=\{0\}$ for each $\lambda\ne\lambda^\ast$ near $\lambda^\ast$;
 \item[(I.d)] $\lim_{\lambda\to 0+}r({\mathcal{L}}''(0)-(\lambda+\lambda^\ast)\widehat{\mathcal{L}}''(0))\ne
 \lim_{\lambda\to 0-}r({\mathcal{L}}''(0)-(\lambda+\lambda^\ast)\widehat{\mathcal{L}}''(0))$.
  \end{description}
  Then  $(\lambda^\ast,0)\in \mathbb{R}\times U$ is  a bifurcation point of the equation
 \begin{eqnarray}\label{e:Bi.1.4}
  \mathcal{L}'(u)-\lambda\widehat{\mathcal{L}}'(u)=0.
  \end{eqnarray}
   Moreover, the condition (I.d)  can be replaced by
  \begin{description}
  \item[(I.d')] ${\mathcal{L}}''(0)$ is invertible, ${\mathcal{L}}''(0)\widehat{\mathcal{L}}''(0)=\widehat{\mathcal{L}}''(0){\mathcal{L}}''(0)$
  and the positive and negative indexes of inertia of the restriction of ${\mathcal{L}}''(0)$ to $H^0_{\lambda^\ast}:={\rm Ker}({\mathcal{L}}''(0)-\lambda^\ast\widehat{\mathcal{L}}''(0))$
  are different.
  \end{description}
  \end{corollary}

\begin{corollary}\label{cor:Bi.3.1}
Let $\lambda^\ast\in\mathbb{R}$.  Suppose that $\mathcal{L}\in C^1(U,\mathbb{R})$ satisfy
 Hypothesis~\ref{hyp:1.3} without (C) and (D1), $\widehat{\mathcal{L}}\in C^1(U,\mathbb{R})$ satisfy  Hypothesis~\ref{hyp:1.4}
 with corresponding operators $\widehat{A}$ and $\widehat{B}=\widehat{P}+\widehat{Q}$,
 and that for each real $\lambda$ near $\lambda^\ast$
 there holds:
 \begin{description}
  \item[(II.a)] $\{u\in H\,|\, B(0)u-\lambda\widehat{B}(0)u=\mu u,\,\mu\le 0\}\subset X$ for each $\lambda$ near $\lambda^\ast$;
  \item[(II.b)] ${\rm Ker}(B(0)-\lambda^\ast\widehat{B}(0))\ne\{0\}$;
 \item[(II.c)] ${\rm Ker}(B(0)-\lambda\widehat{B}(0))=\{0\}$ for each $\lambda\ne\lambda^\ast$ near $\lambda^\ast$;
 \item[(II.d)] $\lim_{\lambda\to 0+}r(B(0)-(\lambda+\lambda^\ast)\widehat{B}(0))\ne
 \lim_{\lambda\to 0-}r(B(0)-(\lambda+\lambda^\ast)\widehat{B}(0))$.
  \end{description}
  Then  $(\lambda^\ast,0)\in \mathbb{R}\times U$ is  a bifurcation point of the equation
 \begin{eqnarray}\label{e:Bi.1.4.1}
  A(u)-\lambda\widehat{A}(u)=0.
  \end{eqnarray}
   Moreover, the condition (II.d)) can be replaced by
  \begin{description}
   \item[(II.d')] ${\mathcal{L}}''(0)$ and $\widehat{\mathcal{L}}''(0)$ in (I.d') are replaced by $B(0)$ and $\widehat{B}(0)$, respectively.
 \end{description}
 \end{corollary}

Consider $\mathcal{F}_\lambda:=\widehat{\mathcal{L}}-(\lambda+\lambda^\ast)\widehat{\mathcal{L}}$ for each real $\lambda$ near $0$.
Then Corollary~\ref{cor:Bi.3.1} (resp. Corollary~\ref{cor:Bi.3.1}) can follows from  Corollary~\ref{cor:stablity2} and
Theorem~\ref{th:Bi.1.1} (resp. Corollary~\ref{cor:stablity2} and Theorem~\ref{th:Bif.1.1})
if (I) (resp. (II)) holds. As to the ``Moreover" parts in these two corollaries, by (\ref{e:Bi.2.19.1}) in the proof of Corollary~\ref{cor:Bi.2.4.2},
we see that $\mathcal{F}_\lambda$ has different Morse indexes at $0\in H$ as $\lambda$ varies in both sides of $\lambda^\ast$.

Based on the arguments in \cite{Lu1}, we may use
Theorem~\ref{th:Bi.1.1} to get a generalization of \cite[Theorem~5.4.1]{EKBB} immediately. See
Theorem~\ref{th:BifE.9} for a high dimensional analogue (corresponding to
\cite[Theorems~5.4.2 and 5.7.4]{EKBB}).

In Corollary~\ref{cor:Bi.3}, if the condition ``${\mathcal{L}}''(0)$ is invertible" in (I.d')
is strengthened as ``${\mathcal{L}}''(0)$ is positive definite", other conditions are unnecessary.
That is, we have

\begin{theorem}\label{th:Bi.3}
Let $\mathcal{L}\in C^1(U,\mathbb{R})$ (resp.  $\widehat{\mathcal{L}}\in C^1(U,\mathbb{R})$) satisfy
 Hypothesis~\ref{hyp:1.1} with $X=H$ (resp.  Hypothesis~\ref{hyp:1.2}).
Suppose that ${\mathcal{L}}''(0)$ is positive definite.
Then $(\lambda^\ast,0)\in \mathbb{R}\times U$ is  a bifurcation point of the equation
(\ref{e:Bi.1.4}) if and only if  $\lambda^\ast\in\mathbb{R}$ is an  eigenvalue of
  \begin{eqnarray}\label{e:Bi.1.5}
  \mathcal{L}''(0)u-\lambda\widehat{\mathcal{L}}''(0)u=0,\; u\in H.
  \end{eqnarray}
\end{theorem}

\begin{proof}
Necessity is contained in Corollary~\ref{cor:Bi.2.2}. It remains to prove sufficiency.
Let $J$ denote the inverse of $(\mathcal{L}''(0))^{1/2}$, which is in $\mathscr{L}_s(H)$
and commutes with any $S\in\mathscr{L}_s(H)$ that commutes with $\mathcal{L}''(0)$. Define functionals
$$
\mathfrak{L}(u):= \mathcal{L}(Ju),\quad    \widehat{\mathfrak{L}}(u):= \widehat{\mathcal{L}}(Ju),\quad u\in J^{-1}(U).
$$
Then for any $u\in J^{-1}(U)$ and all $v\in H$ we have $\nabla\mathfrak{L}(u)= J\nabla\mathcal{L}(Ju)$,
$\nabla\widehat{\mathfrak{L}}(u)= J\nabla\widehat{\mathcal{L}}(Ju)$ and
$$
D(\nabla\mathfrak{L})(u)[v]= JD(\nabla\mathcal{L})(Ju)[Jv]\quad\hbox{and}\quad
D(\nabla\widehat{\mathfrak{L}})(u)[v]= JD(\nabla\widehat{\mathcal{L}})(Ju)[Jv].
$$
It follows that $\mathfrak{L}\in C^1(J^{-1}(U),\mathbb{R})$ (resp.  $\widehat{\mathfrak{L}}\in C^1(J^{-1}(U),\mathbb{R})$) satisfy
 Hypothesis~\ref{hyp:1.1} with $X=H$ (resp.  Hypothesis~\ref{hyp:1.2}).
Moreover, the bifurcation problem (\ref{e:Bi.1.4}) is equivalent to the following
\begin{eqnarray}\label{e:Bi.1.6}
  \mathfrak{L}'(u)-\lambda\widehat{\mathfrak{L}}'(u)=0,\quad u\in J^{-1}(U)
  \end{eqnarray}
and that $\lambda^\ast\in\mathbb{R}$ is an  eigenvalue of (\ref{e:Bi.1.5}) if and only if it is that of
  \begin{eqnarray}\label{e:Bi.1.7}
  J\mathcal{L}''(0)Ju-\lambda J\widehat{\mathcal{L}}''(0)Ju=0\quad\Longleftrightarrow\quad
  u-\lambda Lu=0,
  \end{eqnarray}
where $L:=J\widehat{\mathcal{L}}''(0)J$. Since $L\in\mathscr{L}_s(H)$ is compact,
The eigenvalues of (\ref{e:Bi.1.7}) consists of a sequence of nonzero  reals
$\{\lambda_k\}_{k=1}^\infty$ diverging  to infinity, and each of them is
 of  finite multiplicity. Hence $\lambda^\ast=\lambda_{k_0}$ for some $k_0\in\mathbb{N}$.
 Let $H_k$ be the eigensubspace corresponding to $\lambda_k$ for $k\in\mathbb{N}$, i.e.,
 \begin{equation}\label{e:Bi.2.8sec3}
H_k={\rm Ker}(I-\lambda_kL)={\rm Ker}(\mathfrak{L}''(0)-\lambda_k \widehat{\mathfrak{L}}''(0)),\quad
\forall k\in\mathbb{N}.
\end{equation}
 Then $H=\oplus^\infty_{k=0}H_k$, where $H_0={\rm Ker}(L)={\rm Ker}(\widehat{\mathcal{L}}''(0))$.

Now the eigenvalue $\lambda^\ast=\lambda_{k_0}$  is isolated and
has finite multiplicity. Let us choose $\delta>0$ such that $(\lambda^\ast-\delta,
\lambda^\ast+ \delta)\setminus\{\lambda^\ast\}$ has no intersection with $\{\lambda_k\}^\infty_{k=0}$,
where $\lambda_0=0$. Since $\mathfrak{L}_\lambda:=\mathfrak{L}-\lambda\widehat{\mathfrak{L}}$
 satisfies Hypothesis~\ref{hyp:1.1} with $X=H$, it has finite Morse index $\mu_\lambda$ and nullity $\nu_\lambda$ at $0$.
Moreover, since each $H_k$ is an invariant subspace of $\mathfrak{L}''(0)-\lambda\widehat{\mathfrak{L}}''(0)=I-\lambda L$, and
$H=\oplus^\infty_{k=0}H_k$ is an orthogonal decomposition, we have
$\mu_\lambda=\sum^\infty_{k=0}\mu_{\lambda,k}$,
where $\mu_{\lambda,k}$ is the dimension of
the maximal negative definite space of the quadratic functional
$H_k\ni u\mapsto f_{\lambda,k}(u):=(\mathfrak{L}''(0)u-\lambda\widehat{\mathfrak{L}}''(0)u,u)_H$.
Note that $f_{\lambda,k}(h)=(1-\lambda/\lambda_k)(h,h)_H,\;\forall h\in H_k$.
We deduce that $\mu_{\lambda,k}$ is equal to
$\dim H_k$ (resp. $0$) if $\lambda>\lambda_k$ (resp. $\lambda<\lambda_k$).
%$H_k$ has an orthogonal decomposition $H_k^+\oplus H_k^-$, where
%$H_k^+$ (resp. $H_k^-$) is the positive (resp. negative) definite subspace
%of $\mathcal{L}''(0)|_{H_k}$, %It is possible that $H_k^+=\{0\}$ or $H_k^-=\{0\}$.
%and
%As in (\ref{e:Bi.2.11}), if $H^+_k\ne\{0\}$ (resp. $H^-_k\ne\{0\}$) and $\lambda>\lambda_k$ (resp. $\lambda<\lambda_k$) we have
%
Hence  the Morse index of $\mathfrak{L}_\lambda$ at $0$,
 \begin{eqnarray}\label{e:Bi.2.13sec3}
\mu_\lambda=\sum_{\lambda_k<\lambda}\dim H_k.
\end{eqnarray}
%(Since $\mu_\lambda$ is finite, the right side must be finite sum.)
%Since $\lambda^\ast=\lambda_{k_0}$, as in (\ref{e:Bi.2.13})
This implies  that
$$
\mu_\lambda=\left\{\begin{array}{ll}
\mu_{\lambda^\ast},&\quad\forall \lambda\in (\lambda^\ast-\delta, \lambda^\ast),\\
\mu_{\lambda^\ast}+ \nu_{\lambda^\ast}, &\quad
\forall\lambda\in (\lambda^\ast, \lambda^\ast+\delta),
\end{array}
\right.
$$
Theorem~\ref{th:A.1} with $\lambda=0$ (or \cite[Theorem~2.1]{Lu7}) we obtain that
$$
C_q(\mathcal{L}_\lambda, 0;{\bf K})=C_q(\mathfrak{L}_\lambda, 0;{\bf K})=\left\{\begin{array}{ll}
\delta_{q\mu_{\lambda^\ast}}{\bf K},&\quad\forall \lambda\in (\lambda^\ast-\delta, \lambda^\ast),\\
\delta_{q(\mu_{\lambda^\ast}+ \nu_{\lambda^\ast})}, &\quad \forall\lambda\in (\lambda^\ast, \lambda^\ast+\delta).
\end{array}
\right.
$$
On the other hand, as in the proof of Theorem~\ref{th:Bi.1.1},
suppose that $(\lambda^\ast,0)\in \mathbb{R}\times U$ is not a bifurcation point of the equation
(\ref{e:Bi.1.4}). Then $C_q(\mathcal{L}_\lambda, 0;{\bf K})$ is independent of
$\lambda\in (\lambda^\ast-\delta,\lambda^\ast+\delta)$ by shrinking $\delta>0$ (if necessary),
which leads to a contradiction.
\end{proof}

 It is easily seen that the functional $\mathcal{L}$ in
\cite[\S4.3, Theorem~4.3]{Skr} or  in \cite[Chap.1, Theorem~3.4]{Skr2}.
satisfies the conditions of Theorem~\ref{th:Bi.3}. Hence the latter is a generalization of
\cite[Chap.1, Theorem~3.4]{Skr2}. A stronger generalization will be given in Corollary~\ref{cor:Bi.2.4.2}.

Since  $\mathcal{L}''(0)(X)\subset X$  implies $J(X)\subset X$,
slightly modifying the proof of Theorem~\ref{th:Bi.3} we can ontain:

\begin{theorem}\label{th:Bi.4}
$\mathcal{L}\in C^1(U,\mathbb{R})$ satisfy
 Hypothesis~\ref{hyp:1.3} without (C) and (D1), $\widehat{\mathcal{L}}\in C^1(U,\mathbb{R})$ satisfy  Hypothesis~\ref{hyp:1.4}
 with corresponding operators $\widehat{A}$ and $\widehat{B}=\widehat{P}+\widehat{Q}$.
Suppose that $B(0)$ is positive definite  and that
$\{u\in H\,|\, B(0)u-\lambda\widehat{B}(0)u=\mu u,\,\mu\le 0\}\subset X$ for each $\lambda$.
% near $\lambda^\ast$;
Then $(\lambda^\ast,0)\in \mathbb{R}\times U^X$ is  a bifurcation point of the equation
(\ref{e:Bi.1.4.1}) if and only if  $\lambda^\ast\in\mathbb{R}$ is an  eigenvalue of
  \begin{eqnarray}\label{e:Bi.1.8}
  B(0)u-\lambda\widehat{B}(0)u=0,\; u\in H.
  \end{eqnarray}
\end{theorem}

Note that the proofs of Theorems~\ref{th:Bi.1.1},~\ref{th:Bif.1.1} and \ref{th:Bi.3}
do not need the parameterized versions of the splitting lemmas and Morse-Palais lemmas.
In next three sections we shall use them to get stronger results.

In order to compare our method with previous those  let us consider  the case that
 the classical splitting lemma and Morse-Palais lemma for $C^2$ functionals can be used.
Indeed, it is easily seen that our above proof methods of Theorem~\ref{th:Bi.1.1}
 can lead to:

\begin{theorem}\label{th:Bi.6}
Let $H$, $I$ and $U$ be as in Theorem~\ref{th:Bi.1.1},
 and let $\mathcal{F}:I\times U\to\mathbb{R}$ satisfy the following conditions:
 \begin{description}
 \item[(1)] Each $\mathcal{F}_\lambda\in C^2(U,\mathbb{R})$, $\mathcal{F}'_\lambda(0)=0$ and $\mathcal{F}_0^{\prime\prime}(0)$ is a Fredholm operator.
 \item[(2)] The assumptions (v)-(viii) in Theorem~\ref{th:Bi.1.1} hold with $B_\lambda(0)=\mathcal{F}_\lambda^{\prime\prime}(0)$.
 \item[(3)] Either the assumptions (a)-(b) in Theorem~\ref{th:Bi.1.1} hold, or (i) in Theorem~\ref{th:Bi.1.1} holds and
 $\{\mathcal{F}_\lambda\}_{|\lambda|<\epsilon}$ satisfies the uniform (PS) condition on some closed neighborhood of $0\in H$ for
 small $\epsilon>0$.
 \end{description}
  Then $(0,0)\in I\times U$ is  a bifurcation point of the equation (\ref{e:Intro.1}).
\end{theorem}

By (vi) in Theorem~\ref{th:Bi.1.1} with $B_\lambda(0)=\mathcal{F}_\lambda^{\prime\prime}(0)$,
 $\mathcal{F}_\lambda^{\prime\prime}(0)$ is  Fredholm  for each $\lambda$ near $0\in\mathbb{R}$.

\begin{corollary}\label{cor:Bi.7}
Let $U$ be an open neighborhood of $0$ in a real Hilbert space $H$,
and let $\mathcal{L}, \widehat{\mathcal{L}}\in C^2(U,\mathbb{R})$ satisfy
$\mathcal{L}'(0)=0$ and  $\widehat{\mathcal{L}}'(0)=0$. Suppose that
$\lambda^\ast\in\mathbb{R}$ is an isolated eigenvalue of (\ref{e:Bi.1.5})
 of finite multiplicity, that the following  conditions are satisfied:
 \begin{description}
 \item[(1)] For each $\lambda$ near $\lambda^\ast$,
 $\mathcal{L}-\lambda\widehat{\mathcal{L}}$ satisfies the (PS) condition in a fixed closed neighborhood of $0$.
 %$\bar{B}_H(0, \delta)$.
%$\mathcal{L}''(0)-\lambda\widehat{\mathcal{L}}''(0)$ is a Fredholm operator.
% \item[(2)] For some small $\delta>0$, each $\mathcal{L}-\lambda\widehat{\mathcal{L}}$ satisfies the (PS) condition in $\bar{B}_H(0, \delta)$.
 \item[(2)] $\lim_{\lambda\to 0+}r(\mathcal{L}''(0)-(\lambda+\lambda^\ast)\widehat{\mathcal{L}}''(0))\ne
 \lim_{\lambda\to 0-}r(\mathcal{L}''(0)-(\lambda+\lambda^\ast)\widehat{\mathcal{L}}''(0))$.
 \end{description}
  Then $(0,0)\in I\times U$ is  a bifurcation point of the equation (\ref{e:Bi.1.4}).
   Moreover, the condition (2)  can be replaced by one of form (I.d') in Corollary~\ref{cor:Bi.3}.
 \end{corollary}

Since $\mathcal{L}''(0)-\lambda^\ast\widehat{\mathcal{L}}''(0)$ is a Fredholm operator, so is
$\mathcal{L}''(0)-\lambda\widehat{\mathcal{L}}''(0)$ for each real $\lambda$ near $\lambda^\ast$.
Compare Corollary~\ref{cor:Bi.7} with Corollaries~\ref{cor:Bi.2.6},\ref{cor:Bi.2.7}.

By the proof of Theorem~\ref{th:Bi.3} we easily get

\begin{theorem}\label{th:Bi.7.1}
Let $U$ be an open neighborhood of $0$ in a real Hilbert space $H$,
and let $\mathcal{L}, \widehat{\mathcal{L}}\in C^2(U,\mathbb{R})$ satisfy
$\mathcal{L}'(0)=0$ and  $\widehat{\mathcal{L}}'(0)=0$.
Suppose that  $\mathcal{L}''(0)$ is positive definite and  $\widehat{\mathcal{L}}''(0)$ is compact.
Then $(\lambda^\ast,0)\in \mathbb{R}\times U$ is  a bifurcation point of the equation
(\ref{e:Bi.1.4}) if and only if  $\lambda^\ast\in\mathbb{R}$ is an  eigenvalue of
(\ref{e:Bi.1.5}).
 \end{theorem}

Since $\mathcal{L}''(0)$ is positive definite we have $\mu_0=0$ and thus each $\mu_\lambda$
(the Morse index of $\mathcal{L}-\lambda\widehat{\mathcal{L}}$  at $0\in H$) is finite by
(\ref{e:Bi.2.13sec3}).
% For each $\lambda\in\mathbb{R}$,

\begin{remark}\label{rem:Bi.8}
{\rm Based on the center manifold theory  Chow and Lauterbach  \cite{ChowLa} proved: if $\mathcal{F}\in C^2(I\times U,\mathbb{R})$
 satisfies the following four conditions:
\begin{description}
\item[1)] $\mathcal{F}'_\lambda(0)=0\;\forall\lambda\in I$,
\item[2)] $0<\dim{\rm Ker}(\mathcal{F}''_0(0))<\infty$,
\item[3)]  $0$ is isolated in $\sigma(\mathcal{F}''_0(0))$,
\item[4)]  $r^+_{\mathcal{F}''_\lambda(0)}\ne
r^-_{\mathcal{F}''_\lambda(0)}$,
\end{description}
then $(0,0)\in I\times U$ must be a bifurcation point of (\ref{e:Intro.1}).

This result was generalized by Kielh\"ofer \cite{Kie}
with the Lyapunov-Schmidt reduction and Conley's theorem on bifurcation of invariant sets.
The main result of \cite{Kie} (see also \cite[Theorem~II.7.3]{Kie1})
can be stated as follows (in our notations):
\textsl{Let $X\subset H$ be a continuously embedded subspace having norm $\|\cdot\|_X$,
$G:I\times X\to H$  continuous, $G(\lambda,0)=0\;\forall\lambda\in I$. Suppose
\begin{description}
\item[(A)] $G$ has a continuous
Fr\'echet derivative with respect to $u$ in a neighborhood of $(0,0)\in I\times X$, $G'_u(\lambda,0)=D_\lambda$,
and $D_0:X\to H$ is a Fredholm operator of index zero having an isolated eigenvalue $0$;
\item[(B)] there is a differentiable potential $\mathcal{F}:I\times X\to\mathbb{R}$
 such that $\mathcal{F}'_u(\lambda,u)[h]= (G(\lambda,u),  h)$  for
all $h\in X$ and for all $(\lambda,u)$ in a neighborhood of $(0,0)\in I\times D$;
\item[(C)]  the crossing number of $D_\lambda$ at $\lambda=0$
(defined in \cite[Definition II.7.1]{Kie1}) is nonzero, (which is equivalent to the condition that
$r^+_{B_\lambda}\ne r^-_{B_\lambda}$ if $D_\lambda=B_\lambda$ as in (\ref{e:Bi.1.1})).
\end{description}
Then  $(0,0)\in I\times X$ is a bifurcation point for the equation $G(\lambda,u)=0$.}

 The major difference between these two results and our Theorems~\ref{th:Bi.1.1},\ref{th:Bif.1.1}
 is that they require the higher smoothness for $G$ or $\mathcal{F}$, while we
require some kinds of (PS) conditions.
% as in Theorem~\ref{th:Bi.1.1+}(b) or Claim 1 in the proof of Theorem~\ref{th:Bi.1.1}.
%It is easily seen that assumptions of Theorem~\ref{th:Bif.1.1} and Kielh\"ofer's those in \cite{Kie} cannot be contained each other too.   Theorems~\ref{th:Bi.1.1} and \ref{th:Bif.1.1} are applicable for some examples in \cite{Kie1}.
In particular, comparing Theorem~\ref{th:Bi.6} with the result by Chow and Lauterbach  \cite{ChowLa}
stated above  it is easily seen that my methods and theirs have both advantages and disadvantages. %different strengths and weaknesses
In applications to quasi-linear elliptic systems it is impossible to require higher smoothness for potential functionals. }
\end{remark}

%Recall the basic assumption for the setting of
%In \cite{Lu1,Lu2, Lu4} we also established
%some splitting theorems for a class of continuously directional differentiable functionals
%on a Hilbert space under the following basic assumption, which are more suitable for
%variational problems in Finsler geometry (\cite{Lu5}).

%\section{Generalizations of  Rabinowitz's  bifurcation theorem}\label{sec:B.2}
%\setcounter{equation}{0}

\section{Some new bifurcation results of  Rabinowitz type}\label{sec:B.2}
\setcounter{equation}{0}

Let $H$ be a real Hilbert space, $U$ an open neighborhood of $0$ in $H$,
and let $f\in C^2(U, \mathbb{R})$ satisfy  $f'(0)=0$.
A classical bifurcation theorem by Rabinowitz \cite{Rab} claimed:
If $\lambda^\ast\in\mathbb{R}$ is an isolated eigenvalue of $f^{\prime\prime}(0)$ of finite multiplicity,
then $(\mu,0)$ is a bifurcation point of the equation $\nabla f(u)=\lambda u$ and
at least one of the following alternatives occurs:
\begin{description}
\item[(i)] $(\lambda^\ast, 0)$ is not an isolated solution in $\{\lambda^\ast\}\times U$ of the equation $\nabla f(u)=\lambda u$;
\item[(ii)]  for every $\lambda\in\mathbb{R}$ near $\lambda^\ast$ there is a nontrivial solution
$u_\lambda$ of the equation $\nabla f(u)=\lambda u$ converging to $0$ as $\lambda\to\lambda^\ast$;

\item[(iii)] there is an one-sided  neighborhood $\Lambda$ of $\lambda^\ast$ such that
for any $\lambda\in\Lambda\setminus\{\lambda^\ast\}$, the equation $\nabla f(u)=\lambda u$ has at least
two nontrivial solutions converging to zero as $\lambda\to\lambda^\ast$.
\end{description}
Under the same hypotheses B\"ohme \cite{Boh} and Marino \cite{Mar} independently proved that
for each small $r>0$, the equation $\nabla f(u)=\lambda u$ possesses  at least
two distinct one parameter families of solutions $(\lambda(r), u(r))$ having $\|u(r)\|=r$ and $\lambda(r)\to\lambda^\ast$
as $r\to 0$.  McLeod and Turner \cite{McTu} extended the latter result  to functions
 of class $C^{1,1}$ depending linearly on two parameters and such that the Lipschitz constant of the gradient going to
zero as the parameters and $u$ go to zero. Ioffe and Schwartzman \cite{IoSch}
generalized  the above Rabinowitz's bifurcation theorem  to equations $Lu+ \nabla\varphi_\lambda(u)=\lambda u$
with $\varphi$ continuous jointly in $(\lambda,u)$ and $\varphi_\lambda$  of class $C^{1,1}$, where $L\in\mathscr{L}_s(H)$
the Lipschitz constant of $\nabla\varphi_\lambda$ satisfies some additional conditions.

In this section we shall give a few new bifurcation results as the above Rabinowitz bifurcation theorem.
As said in Introduction, we can reduce some of them  to the following
finite-dimensional result of the above Rabinowitz's type,  which may be obtained as a corollary of \cite[Theorem~2]{IoSch}
(or \cite[Theorem~4.2]{CorH}).

\begin{theorem}[\hbox{\cite[Theorem~5.1]{Can}}]\label{th:Bi.2.1}
 Let $X$ be a finite dimensional normed space, let $\delta>0$, $\epsilon>0$, $\lambda^\ast\in\mathbb{R}$ and
for every $\lambda\in [\lambda^\ast-\delta, \lambda^\ast+\delta]$, let
$f_\lambda:B_X(0,\epsilon)\to\mathbb{R}$ be a function of class $C^1$.
Assume that
\begin{description}
\item[a)] the functions $\{(\lambda,u)\to f_\lambda(u)\}$ and
$\{(\lambda,u)\to f'_\lambda(u)\}$  are continuous on
$[\lambda^\ast-\delta, \lambda^\ast+\delta]\times B_X(0,\epsilon)$;
\item[b)] $u=0$ is a critical point of $f_{\lambda^\ast}$;
\item[c)] $f_\lambda$ has an isolated local minimum (maximum) at zero for every
$\lambda\in (\lambda^\ast,\lambda^\ast+\delta]$ and an isolated local maximum (minimum) at
zero for every $\lambda\in [\lambda^\ast-\delta, \lambda^\ast)$.
\end{description}
Then one at least of the following assertions holds:
\begin{description}
\item[i)] $u=0$ is not an isolated critical point of $f_{\lambda^\ast}$;
\item[ii)] for every $\lambda\ne\lambda^\ast$ in a neighborhood of $\lambda^\ast$ there is a nontrivial critical point of
$f_\lambda$ converging to zero as $\lambda\to\lambda^\ast$;
\item[iii)] there is an one-sided (right or left) neighborhood of $\lambda^\ast$ such that for every
$\lambda\ne\lambda^\ast$ in the neighborhood there are two distinct nontrivial critical points of $f_\lambda$
converging to zero as $\lambda\to\lambda^\ast$.
\end{description}
In particular,  $(\lambda^\ast, 0)\in [\lambda^\ast-\delta, \lambda^\ast+\delta]\times B_X(0,\epsilon)$
is a bifurcation point of $f'_\lambda(u)=0$.
\end{theorem}

Note that the local minimum (maximum) in the assumption c) must not be strict!

%\subsection{Generalizations of  Rabinowitz bifurcation theorem \cite{Rab}}
%\label{sec:B.2.2}

%\noindent{\bf 3.2}.\quad {\bf Generalizations of  Rabinowitz bifurcation theorem \cite{Rab}}.
%The following results partially generalize  Rabinowitz bifurcation theorem in \cite{Rab}.

%Consider the case $n=1$ and write
%$\widehat{\mathcal{L}}=\widehat{\mathcal{L}}_1$. Then (\ref{e:Bi.2.1}) and (\ref{e:Bi.2.2}) become, respectively,
%\begin{eqnarray}
%&&\mathcal{L}'(u)=\lambda\widehat{\mathcal{L}'}(u),\quad u\in U,\label{e:Bi.2.7.3}\\
%&&\mathcal{L}''(0)v-\lambda\widehat{\mathcal{L}}''(0)v=0,\quad v\in H.\label{e:Bi.2.7.4}
%\end{eqnarray}
%Without the conditions (D) and (E) as in Theorem~\ref{th:Bi.2.3} we have

The following is the main result in this section.

\begin{theorem}\label{th:Bi.2.4}
Let $\mathcal{L}\in C^1(U,\mathbb{R})$ (resp.  $\widehat{\mathcal{L}}\in C^1(U,\mathbb{R})$) satisfy
 Hypothesis~\ref{hyp:1.1} with $X=H$ (resp.  Hypothesis~\ref{hyp:1.2}),
and let $\lambda^\ast\in\mathbb{R}$ be an isolated eigenvalue of
\begin{eqnarray}\label{e:Bi.2.7.4}
\mathcal{L}''(0)v-\lambda\widehat{\mathcal{L}}''(0)v=0,\quad v\in H.
\end{eqnarray}
(If $\lambda^\ast=0$, it is enough that $\widehat{\mathcal{L}}\in C^1(U,\mathbb{R})$ satisfies Hypothesis~\ref{hyp:1.2}
without requirement that each $\widehat{\mathcal{L}}''(u)\in\mathscr{L}_s(H)$ is compact.)
Suppose that the Morse indexes of $\mathcal{L}_\lambda:=\mathcal{L}-\lambda\widehat{\mathcal{L}}$
at $0\in H$ take values $\mu_{\lambda^\ast}$ and $\mu_{\lambda^\ast}+\nu_{\lambda^\ast}$
 as $\lambda\in\mathbb{R}$ varies in both sides of $\lambda^\ast$ and is close to $\lambda^\ast$,
where $\mu_{\lambda}$ and $\nu_{\lambda}$ are the Morse index and the nullity of  $\mathcal{L}_{\lambda}$
at $0$, respectively.      Then  one of the following alternatives occurs:
\begin{description}
\item[(i)] $(\lambda^\ast, 0)$ is not an isolated solution in $\{\lambda^\ast\}\times U$ of
\begin{eqnarray}\label{e:Bi.2.7.3}
\mathcal{L}'(u)=\lambda\widehat{\mathcal{L}'}(u);
\end{eqnarray}
 %in $\{\lambda^\ast\}\times U$;
\item[(ii)]  for every $\lambda\in\mathbb{R}$ near $\lambda^\ast$ there is a nontrivial solution
$u_\lambda$ of (\ref{e:Bi.2.7.3}) converging to $0$ as $\lambda\to\lambda^\ast$;

\item[(iii)] there is an one-sided  neighborhood $\Lambda$ of $\lambda^\ast$ such that
for any $\lambda\in\Lambda\setminus\{\lambda^\ast\}$,
(\ref{e:Bi.2.7.3}) has at least two nontrivial solutions converging to
zero as $\lambda\to\lambda^\ast$.
\end{description}
In particular,  $(\lambda^\ast, 0)\in\mathbb{R}\times U$ is a bifurcation point  for the equation
(\ref{e:Bi.2.7.3}).
%Moreover, if ${\lambda}^\ast=0$, it suffices to require that $\widehat{\mathcal{L}}\in C^1(U,\mathbb{R})$ satisfies
% Hypothesis~\ref{hyp:1.1} with $X=H$.
\end{theorem}

When $X=H$ let us compare the assumptions of Theorem~\ref{th:Bi.2.4} with those of Corollary~\ref{cor:Bi.3}.
The assumption on the Morse indexes of $\mathcal{L}_\lambda$ at $0\in H$ in the former
is stronger than the condition (I.d) in the latter. It means that one of
$\lim_{\lambda\to 0+}r(B(0)-(\lambda+\lambda^\ast)\widehat{\mathcal{L}}''(0))$ and
 $\lim_{\lambda\to 0-}r(B(0)-(\lambda+\lambda^\ast)\widehat{\mathcal{L}}''(0))$
is equal to $0$, and another is equal to $\nu_{\lambda^\ast}$.

Clearly, our assumptions and those by Ioffe and Schwartzman \cite{IoSch} cannot be contained each other.
In Remark~\ref{rm:Bi.2.4.3} we shall compare it with the above Rabinowitz bifurcation theorem.

%Corollary~\ref{cor:Bi.2.4.2}
%
%Rabinowitz \cite{Rab} consider the case that $\widehat{\mathcal{L}}=\frac{1}{2}\|u\|^2$ and
%$\mathcal{L}\in C^2(U,\mathbb{R})$ has the form $\frac{1}{2}(\mathfrak{B}u,u)_H+ \varphi(u)$,
%where $\mathfrak{B}\in\mathscr{L}_s(H)$ and $\nabla\varphi(u)=o(\|u\|)$ at $u=0$.
%When $\lambda^\ast$ is an isolated eigenvalue of $\mathfrak{B}$ of finite multiplicity,
%he first proved the conclusions as in Theorem~\ref{th:Bi.2.4}.
%
%Ioffe and Schwartzman \cite{IoSch} gave an extension of this result.
%He studied bifurcation equation $\mathfrak{B}u+\nabla\varphi_\lambda(u)=\lambda u$ near $u=0$,
%where $\varphi_\lambda\in C^{1,1}(U,\mathbb{R})$ is such that both $\varphi_\lambda(u)$ and
%$\nabla\varphi_\lambda(u)$ depend on $(\lambda,u)$ continuously.
%Under some additional assumptions on the Lipschitz constant of $\nabla\varphi_\lambda$
%they proved the same claims.

\begin{proof}[\it {Proof of Theorem~\ref{th:Bi.2.4}}]
%\begin{proof}\quad
Take $\delta>0$ such that $[\lambda^\ast-\delta,\lambda^\ast+\delta]\setminus\{\lambda^\ast\}$
contains no eigenvalues of (\ref{e:Bi.2.7.4}). Then
 $0\in H$ is a nondegenerate (and thus isolated) critical point of $\mathcal{L}_\lambda$ for each $\lambda\in [\lambda^\ast-\delta,\lambda^\ast+\delta]\setminus\{\lambda^\ast\}$.

\textsf{Firstly, we assume that} the Morse index $\mu_\lambda$ of $\mathcal{L}_\lambda$  satisfies
  \begin{eqnarray}\label{e:Bi.2.7.5}
 \mu_\lambda=\left\{\begin{array}{ll}
 \mu_{\lambda^\ast},&\quad\forall \lambda\in [\lambda^\ast-\delta, \lambda^\ast),\\ \mu_{\lambda^\ast}+\nu_{\lambda^\ast}, &\quad
\forall\lambda\in (\lambda^\ast, \lambda^\ast+\delta].
\end{array}
\right.
 \end{eqnarray}
  It follows from  Theorem~\ref{th:A.1} that  for any $q\in\mathbb{N}_0$,
\begin{eqnarray}\label{e:Bi.2.7.6}
C_q(\mathcal{L}_{\lambda},0;{\bf K})=
\left\{\begin{array}{ll}
\delta_{q\mu_{\lambda^\ast}}{\bf K},&\quad
\forall \lambda\in [\lambda^\ast-\delta, \lambda^\ast),\\
\delta_{q(\mu_{\lambda^\ast}+\nu_{\lambda^\ast})}{\bf K},&\quad
\forall \lambda\in (\lambda^\ast, \lambda^\ast+\delta].
\end{array}\right.
\end{eqnarray}
 Let $H^0_{\lambda^\ast}$ be the eigenspace of (\ref{e:Bi.2.7.4}) associated with $\lambda^\ast$ and $(H^0_{\lambda^\ast})^\bot$
the orthogonal complementary of $H^0_{\lambda^\ast}$ in $H$.
Applying Theorem~\ref{th:A.2}
to $\mathcal{L}_{\lambda}=\mathcal{L}_{\lambda^\ast}+(\lambda-\lambda^\ast)\widehat{\mathcal{L}}$
with $\lambda\in [\lambda^\ast-\delta, \lambda^\ast+\delta]$,
 we have  $\epsilon>0$ and a unique continuous map
 $$
 \psi:[\lambda^\ast-\delta, \lambda^\ast+\delta]\times (B_H(0,\epsilon)\cap H^0_{\lambda^\ast})\to (H^0_{\lambda^\ast})^\bot
 $$
   such that for each $\lambda\in [\lambda^\ast-\delta, \lambda^\ast+\delta]$,
 $\psi(\lambda, 0)=0$ and
\begin{eqnarray*}
 P^\bot_{\lambda^\ast}\nabla\mathcal{L}(z+ \psi(\lambda, z))-
 \lambda P^\bot_{\lambda^\ast}\nabla\widehat{\mathcal{L}}(z+ \psi(\lambda, z))=0\quad\forall z\in B_{H}(0,\epsilon)\cap H^0_{\lambda^\ast},
 \end{eqnarray*}
 where  $P^\bot_{\lambda^\ast}$ is the orthogonal projection onto $(H^0_{\lambda^\ast})^\bot$; moreover
  the functional
 \begin{equation}\label{e:Bi.2.15}
 \mathcal{L}^\circ_\lambda: B_H(0, \epsilon)\cap H^0_{\lambda^\ast}\to\mathbb{R},\;z\mapsto
 \mathcal{L}(z+\psi(\lambda, z))- \lambda\widehat{\mathcal{L}}(z+\psi(\lambda, z))
  \end{equation}
   is of class $C^1$, and has differential at $z\in B_H(0, \epsilon)\cap H^0_{\lambda^\ast}$ given by
\begin{eqnarray}\label{e:Bi.2.15+}
D\mathcal{L}^\circ_\lambda(z)[h]=D\mathcal{L}(z+\psi(\lambda,z))[h]-
\lambda D\widehat{\mathcal{L}}(z+\psi(\lambda,z))[h],\quad\forall h\in H^0_{\lambda^\ast}.
\end{eqnarray}
Note that for each $\lambda\in[\lambda^\ast-\delta, \lambda^\ast+\delta]$,
the map $z\mapsto z+ \psi({\lambda}, z))$ induces an one-to-one correspondence
 between the critical points of  $\mathcal{L}_{\lambda}^\circ$ near $0\in H^0_{\lambda^\ast}$
and those of $\mathcal{L}_{\lambda}$ near $0\in H$. Hence the problem is reduced to finding
the critical points of $\mathcal{L}^\circ_\lambda$
near $0\in H^0_{\lambda^\ast}$ for $\lambda$ near $\lambda^\ast$.

Now by  Theorems~\ref{th:A.1},~\ref{th:A.3}, $0\in H^0_{\lambda^\ast}$ is also an
isolated critical point of $\mathcal{L}^\circ_\lambda$, and
 from (\ref{e:Bi.2.7.5}) and (\ref{e:Bi.2.7.6}) we may derive that for any $j\in\mathbb{N}_0$,
\begin{eqnarray}\label{e:Bi.2.16}
C_{j}(\mathcal{L}^\circ_{\lambda},0;{\bf K})=\left\{\begin{array}{ll}
\delta_{j0}{\bf K},&\quad \forall \lambda\in [\lambda^\ast-\delta, \lambda^\ast),\\
\delta_{j\nu_{\lambda^\ast}}{\bf K},&\quad
\forall \lambda\in (\lambda^\ast, \lambda^\ast+\delta].
\end{array}\right.
\end{eqnarray}

For a $C^1$ function $\varphi$ on a neighborhood  of the origin $0\in\mathbb{R}^N$
we may always find  $\tilde\varphi\in C^1(\mathbb{R}^N,
\mathbb{R})$ such that it agrees with $\varphi$ near $0\in\mathbb{R}^N$ and is also
coercive (so satisfies the (PS)-condition).
Suppose that $0$ is an isolated critical point of $\varphi$. By Corollary~6.96, Proposition~6.97 and Example~6.45 in \cite{MoMoPa},
%\cite[p.193]{MaWi}
we obtain
\begin{equation}\label{e:Bi.2.17}
\left.
\begin{array}{ll}
&C_k(\varphi,0;{\bf K})=\delta_{k0}\;\Longleftrightarrow\;\hbox{ $0$ is a strict local minimizer of $\varphi$},\\
&C_k(\varphi,0;{\bf K})=\delta_{kN}\;\Longleftrightarrow\;\hbox{ $0$ is a strict local maximizer of $\varphi$}
\end{array}\right\}
\end{equation}
and thus $C_0(\varphi, 0;{\bf K})=0=C_N(\varphi, 0;{\bf K})$
if $0\in  \mathbb{R}^N$ is neither a local  maximizer nor a local
minimizer of $\varphi$.%(cf. \cite[Example~6.45]{MoMoPa}).

Because of these,  (\ref{e:Bi.2.16}) and (\ref{e:Bi.2.17}) lead to
\begin{eqnarray}\label{e:Bi.2.18}
0\in H^0_{\lambda^\ast}\;\hbox{is a strict local}\left\{
\begin{array}{ll}
\hbox{minimizer of}\;\mathcal{L}^\circ_{\lambda},&\quad
\forall \lambda\in [\lambda^\ast-\delta, \lambda^\ast),\\
\hbox{maximizer of}\;\mathcal{L}^\circ_{\lambda},&\quad
\forall \lambda\in (\lambda^\ast, \lambda^\ast+\delta].
\end{array}\right.
\end{eqnarray}
By this and Theorem~\ref{th:Bi.2.1}, one of the following possibilities occurs:
\begin{description}
\item[(1)]  $0\in H^0_{\lambda^\ast}$ is not an isolated critical point of $\mathcal{L}^\circ_{\lambda^\ast}$;
\item[(2)]  for every $\lambda\in\mathbb{R}$ near $\lambda^\ast$,
$\mathcal{L}^\circ_{\lambda}$ has a nontrivial critical point  converging to
$0\in H^0_{\lambda^\ast}$ as $\lambda\to\lambda^\ast$;

\item[(3)] there is an one-sided  neighborhood $\Lambda$ of $\lambda^\ast$ such that
for any $\lambda\in\Lambda\setminus\{\lambda^\ast\}$,
$\mathcal{L}^\circ_{\lambda}$ has two nontrivial critical points
 converging to zero as $\lambda\to\lambda^\ast$.
\end{description}
Obviously, they lead to (i), (ii) and (iii), respectively.

\textsf{Next,  assume}
\begin{eqnarray}\label{e:Bi.2.19-}
\mu_\lambda=\left\{\begin{array}{ll}
 \mu_{\lambda^\ast}+\nu_{\lambda^\ast},&\quad\forall \lambda\in [\lambda^\ast-\delta, \lambda^\ast),\\ \mu_{\lambda^\ast}, &\quad \forall\lambda\in (\lambda^\ast, \lambda^\ast+\delta].
\end{array}
\right.
\end{eqnarray}
Then we can obtain
\begin{eqnarray}\label{e:Bi.2.19}
0\in H^0_{\lambda^\ast}\;\hbox{is a strict local}\left\{
\begin{array}{ll}
\hbox{maximizer of}\;\mathcal{L}^\circ_{\lambda},&\quad
\forall \lambda\in [\lambda^\ast-\delta, \lambda^\ast),\\
\hbox{minimizer of}\;\mathcal{L}^\circ_{\lambda},&\quad
\forall \lambda\in (\lambda^\ast, \lambda^\ast+\delta].
\end{array}\right.
\end{eqnarray}
This also leads to either (i) or (ii) or (iii).\end{proof}
%
%Finally, if $\lambda^\ast=0$, the conclusion may be obtained  as in Step 3 of the proof of Theorem~\ref{th:Bi.2.3}.
%\end{proof}

From the proof of Theorem~\ref{th:Bi.2.4} it is easily seen that  Theorem~\ref{th:Bi.2.4}(i)
may be replaced by ``$0\in H^0_{\lambda^\ast}$ is not an isolated critical point of
$\mathcal{L}^\circ_{\lambda^\ast}$". In fact, they are equivalent in the present case.
%However, the later Theorem~\ref{th:Bif.2.2} will become different.

%The assumption on the Morse indexes of $\mathcal{L}_\lambda$ at $0\in H$ in Theorem~\ref{th:Bi.2.4}
%can be satisfied in some situations as showed in the following two corollaries.

The following two corollaries strengthen Corollary~\ref{cor:Bi.3}.

\begin{corollary}\label{cor:Bi.2.4.1}
Let $\mathcal{L}\in C^1(U,\mathbb{R})$ (resp.  $\widehat{\mathcal{L}}\in C^1(U,\mathbb{R})$) satisfy
 Hypothesis~\ref{hyp:1.1} with $X=H$ (resp.  Hypothesis~\ref{hyp:1.2}),
and let $\lambda^\ast\in\mathbb{R}$ be an isolated eigenvalue of (\ref{e:Bi.2.7.4}).
(If $\lambda^\ast=0$, it is enough that $\widehat{\mathcal{L}}\in C^1(U,\mathbb{R})$ satisfies Hypothesis~\ref{hyp:1.2}
without requirement that each $\widehat{\mathcal{L}}''(u)\in\mathscr{L}_s(H)$ is compact.)
Suppose that $\widehat{\mathcal{L}}''(0)$ is either semi-positive or semi-negative.
Then the conclusions of Theorem~\ref{th:Bi.2.4} hold true.
%Moreover, if ${\lambda}^\ast=0$,
% it suffices  that $\widehat{\mathcal{L}}\in C^1(U,\mathbb{R})$ satisfies
% Hypothesis~\ref{hyp:1.1} with $X=H$.
\end{corollary}

\begin{proof}
It suffices to prove that the assumption about
the Morse indexes of $\mathcal{L}_\lambda=\mathcal{L}-\lambda\widehat{\mathcal{L}}$
at $0\in H$ in Theorem~\ref{th:Bi.2.4} can be satisfied.
Since $\lambda^\ast\in\mathbb{R}$ is an isolated eigenvalue of (\ref{e:Bi.2.7.4}),
by the proof of Theorem~\ref{th:Bi.2.4} we have $\delta>0$ such that
 $0\in H$ is a nondegenerate critical point of $\mathcal{L}_\lambda$ for each $\lambda\in [\lambda^\ast-\delta,\lambda^\ast+\delta]\setminus\{\lambda^\ast\}$. This is equivalent to
 the fact that  $0\in H$ is a nondegenerate critical point of the $C^\infty$ functional
 $H\ni u\mapsto\mathfrak{L}_\lambda(u):=([\mathcal{L}''(0)-\lambda\widehat{\mathcal{L}}''(0)]u,u)_H$
 for each $\lambda\in [\lambda^\ast-\delta,\lambda^\ast+\delta]\setminus\{\lambda^\ast\}$. Hence
 $$
 \mathcal{L}''(0)u-\lambda\widehat{\mathcal{L}}''(0)u\ne 0,\quad\forall u\in H\setminus\{0\},\quad\forall \lambda\in [\lambda^\ast-\delta,\lambda^\ast+\delta]\setminus\{\lambda^\ast\}.
 $$
 This shows that $0\in H$ is a unique critical point of each $\mathfrak{L}_\lambda$.
 Since $ \nabla\mathfrak{L}_\lambda(u)=\mathcal{L}''(0)u-\lambda\widehat{\mathcal{L}}''(0)u=
 P(0)u+ Q(0)u-\lambda\widehat{\mathcal{L}}''(0)u$ is a bounded linear Fredholm operator,
 it is easily seen that $\mathfrak{L}_\lambda$ satisfies the (PS) condition on any closed ball.
  Hence for any small $0<\epsilon<\delta$,
 %   suppose that sequences $(\lambda_n)\subset [\lambda^\ast-\delta,\lambda^\ast+\delta]\setminus\{\lambda^\ast\}$
% and $(u_n)\subset\bar{B}_H(0,\epsilon)$ are such that $\nabla\mathfrak{L}_{\lambda_n}(u_n)\to 0$ and $\lambda_n\to\lambda_0$.
%  Since $P(0):H\to H$ is positive definite, and $Q(0), \widehat{\mathcal{L}}''(0):H\to H$ are compact, it is easy to prove that
%  $(u_n)$ has a convergence subsequence in $\bar{B}_H(0,\epsilon)$.
%  These show that  on a small closed ball $\bar{B}_H(0,\epsilon)$
  the families
 $\{\mathfrak{L}_\lambda\,|\,\lambda\in [\lambda^\ast-\delta,\lambda^\ast-\epsilon]\}$ and
 $\{\mathfrak{L}_\lambda\,|\,\lambda\in [\lambda^\ast+\epsilon,\lambda^\ast+\delta]\}$
 satisfy the conditions of Theorem~\ref{th:stablity1} on any closed ball $\bar{B}_H(0, r)$.
  %for continuity of the critical groups (\cite[Theorem~8.9]{MaWi}).
 Thus the critical groups $C_\ast(\mathfrak{L}_\lambda, 0;{\bf K})$
 are all isomorphic as $\lambda$ varies in $\lambda\in [\lambda^\ast-\delta,\lambda^\ast)$ (resp.
      $\lambda\in (\lambda^\ast,\lambda^\ast+\delta]$) because of arbitrariness of $\epsilon>0$.
 %%%%%%%%%%%%%%%%%%%%%%%%%%%%%%%%%%%%%%%%%%%%%%%%%%%%%%%%%%%%%%%%%%%%%%%%%%%%%%%%%%%%%%%%%%%%
 %%%%%%%%%%%%%%%%%%%%%%%%%%%%%%%%%%%%%%%%%%%%%%%%%%%%%%%%%%%%%%%%%%%%%%%%%%%%%%%%%%%%%%%
 %% The first equality in (\ref{e:Bi.2.19.1})  follows immediately.
 %%The second is similar.)
 %%As in the proof of (\ref{e:Bi.1.3}), $C_\ast(\mathcal{L}_{\lambda}, \theta;{\bf K})
 %%=C_\ast(\mathcal{L}_{\lambda^\ast-\delta}, \theta;{\bf K})$ for all
 %%$\lambda\in [\lambda^\ast-\delta,\lambda^\ast)$, and
 %%$C_\ast(\mathcal{L}_{\lambda'}, \theta;{\bf K})=C_\ast(\mathcal{L}_{\lambda^\ast+\delta}, \theta;{\bf K})$
 %% for all $\lambda'\in (\lambda^\ast,\lambda^\ast+\delta]$.
 %%%%%%%%%%%%%%%%%%%%%%%%%%%%%%%%%%%%%%%%%%%%%%%%%%%%%%%%%%%%%%%%%%%%%%%%%%%%%%%%%%%%%%
 %%%%%%%%%%%%%%%%%%%%%%%%%%%%%%%%%%%%%%%%%%%%%%%%%%%%%%%%%%%%%%%%%%%%%%%%%%%%%%%%%%%%%%
Note that
$C_q(\mathfrak{L}_\lambda, 0;{\bf K})=\delta_{q\mu_{\lambda}}{\bf K}$
for each $\lambda\in [\lambda^\ast-\delta,\lambda^\ast+\delta]\setminus\{\lambda^\ast\}$.
 It follows that
   \begin{equation}\label{e:Bi.2.19.1}
   \mu_\lambda=\left\{\begin{array}{ll}
 \mu_{\lambda^\ast-\delta},&\quad\forall \lambda\in [\lambda^\ast-\delta, \lambda^\ast),\\
 \mu_{\lambda^\ast+\delta}, &\quad \forall\lambda\in (\lambda^\ast, \lambda^\ast+\delta].
\end{array}
\right.
\end{equation}
 %by \cite[Theorem~2.1]{Lu7} (or \cite[Theorem~2.1]{Lu6}).
%(This can also be proved as in the proof of Theorem~\ref{th:BBH.2}
%by directly using the shifting theorem and the stability for critical groups for families
%$\{\mathcal{L}^\circ_{\lambda}\,|\,\lambda\in [\lambda^\ast-\delta,\lambda^\ast)\}$ and
%$\{\mathcal{L}^\circ_{\lambda}\,|\,\lambda\in (\lambda^\ast,\lambda^\ast+\delta]\}$.)

 We assume now that $\widehat{\mathcal{L}}''(0)$ is semi-positive. Then
 $$
 \mathcal{L}''(0)-\lambda_1\widehat{\mathcal{L}}''(0)\ge \mathcal{L}''(0)-\lambda_2\widehat{\mathcal{L}}''(0),
\quad\forall\; \lambda^\ast-\delta\le\lambda_1\le\lambda_2\le\lambda^\ast+\delta.
$$
Since all $\mu_{\lambda}$ and $\nu_{\lambda}$ are finite, 
from  this, \cite[Proposition~2.3.3]{Ab} and (\ref{e:Bi.2.19.1}) we derive
 \begin{eqnarray*}
&&\mu_{\lambda^\ast-\delta}\le\mu_{\lambda^\ast}\le \mu_{\lambda^\ast+\delta},\\
&&\mu_{\lambda^\ast-\delta}\le\mu_{\lambda^\ast}+\nu_{\lambda^\ast}\le \mu_{\lambda^\ast+\delta},\\
&&\mu_{\lambda^\ast}\le\lim_{\lambda\to\lambda^\ast}\inf\mu_\lambda=\mu_{\lambda^\ast-\delta},\\
&&\mu_{\lambda^\ast}+\nu_{\lambda^\ast}\ge\lim_{\lambda\to\lambda^\ast}\sup\mu_\lambda
=\mu_{\lambda^\ast+\delta}.
\end{eqnarray*}
These imply that (\ref{e:Bi.2.7.5}) is satisfied.

 Similarly, if
$\widehat{\mathcal{L}}''(0)$ is semi-negative, then (\ref{e:Bi.2.19-}) holds true.
 The desired conclusions are obtained.
\end{proof}

 Note that (\ref{e:Bi.2.2}) has no isolated eigenvalues if
$\cap^n_{j=1}{\rm Ker}(\widehat{\mathcal{L}}''_j(0))\cap{\rm Ker}(\mathcal{L}''(0))\ne\{0\}$.
It is natural to ask when  $\vec{\lambda}^\ast$ is an isolated  eigenvalue  of (\ref{e:Bi.2.2}).
We also consider the case $n=1$ merely.
Suppose that  $\mathcal{L}''(0)$ is  invertible.
Then $\mathcal{L}''(0)$ cannot be negative definite because $\mathcal{L}\in C^1(U,\mathbb{R})$ satisfies Hypothesis~\ref{hyp:1.1} with $X=H$.
It follows that $0$ is not an eigenvalue  of (\ref{e:Bi.2.7.4}), and that $\lambda\in\mathbb{R}\setminus\{0\}$
is an eigenvalue  of (\ref{e:Bi.2.7.4}) if and only if $1/\lambda$
is an eigenvalue  of the compact linear self-adjoint operator $L:=[\mathcal{L}''(0) ]^{-1}\widehat{\mathcal{L}}''(0)\in\mathscr{L}_s(H)$.
By Riesz-Schauder theory, the spectrum of $L$, $\sigma(L)$,
consists of zero and a sequence of real eigenvalues of  finite multiplicity
converging to zero, $\{1/\lambda_k\}_{k=1}^\infty$. (Hence $|\lambda_k|\to\infty$ as $k\to\infty$.)
%           contains a unique accumulation point $0$, and
%$\sigma(L)\setminus\{0\}$ is a real countable set of eigenvalues of finite multiplicity, denoted by
Let $H_k$ be the eigensubspace corresponding to $1/\lambda_k$ for $k\in\mathbb{N}$.
Then $H=\oplus^\infty_{k=0}H_k$, where $H_0={\rm Ker}(L)={\rm Ker}(\widehat{\mathcal{L}}''(0))$ and
\begin{equation}\label{e:Bi.2.8}
H_k={\rm Ker}(I/\lambda_k-L)={\rm Ker}(\mathcal{L}''(0)-\lambda_k \widehat{\mathcal{L}}''(0)),\quad
\forall k\in\mathbb{N}.
\end{equation}

%%%%%%%%%%%%%%%%%%%%%%%%%%%%%%%%%%%%%%%%%%%%%%%%%%%%%%%%%%%%%%%%%%%%%%%%%%%%%%%%%%%%%%%%%%%%%
%%As another generalization of  Rabinowitz bifurcation theorem \cite{Rab} we have the following
%%improvement of sufficiency of Theorem~12 in \cite[Chap.4, \S4.3]{Skr}.
%%%%%%%%%%%%%%%%%%%%%%%%%%%%%%%%%%%%%%%%%%%%%%%%%%%%%%%%%%%%%%%%%%%%%%%%%%%%%%%%%

\begin{corollary}\label{cor:Bi.2.4.2}
Let $\mathcal{L}\in C^1(U,\mathbb{R})$ (resp.  $\widehat{\mathcal{L}}\in C^1(U,\mathbb{R})$) satisfy
 Hypothesis~\ref{hyp:1.1} with $X=H$ (resp.  Hypothesis~\ref{hyp:1.2}). Suppose that the following two conditions
  satisfied:
 \begin{description}
\item[(a)] $\mathcal{L}''(0)$ is invertible
and $\lambda^\ast=\lambda_{k_0}$ is an eigenvalue of (\ref{e:Bi.2.7.4}).
 \item[(b)] $\mathcal{L}''(0)\widehat{\mathcal{L}}''(0)=\widehat{\mathcal{L}}''(0)\mathcal{L}''(0)$ (so
 each $H_k$  is an invariant subspace of $\mathcal{L}''(0)$),   and $\mathcal{L}''(0)$
 is either positive  or negative on $H_{k_0}$.
 \end{description}
 Then the conclusions of Theorem~\ref{th:Bi.2.4} hold true.
 Moreover, if ${\mathcal{L}}''(0)$ is positive definite, the condition (b) is unnecessary.
\end{corollary}

Clearly, the ``Moreover" part strengthens Theorem~\ref{th:Bi.3} and so
%(a) and (b) hold if $\mathcal{L}''(0)$ is positive and commutes with $\widehat{\mathcal{L}}''(0)$.
%\begin{description}
%\item[(a)] $\mathcal{L}''(0)$ is positive, and $\mathcal{L}''(0)\widehat{\mathcal{L}}''(0)=\widehat{\mathcal{L}}''(0)\mathcal{L}''(0)$.
% \item[(b)] Each $H_k$ in (\ref{e:Bi.2.8}) with $L=[\mathcal{L}''(0)]^{-1}\widehat{\mathcal{L}}''(0)$
% is an invariant subspace of $\mathcal{L}''(0)$  (e.g. these are true if $\mathcal{L}''(0)$
% commutes with $\widehat{\mathcal{L}}''(0)$), and $\mathcal{L}''(0)$
% is either positive definite or negative one on $H_{k_0}$.
% \end{description}
%Compare this case with %It is easily seen that the functional $\mathcal{L}$ in
\cite[\S4.3, Theorem~4.3]{Skr} or  in \cite[Chap.1, Theorem~3.4]{Skr2}.
%satisfies the conditions of Corollary~\ref{cor:Bi.2.4.2} in the case (a).
%Theorem~\ref{th:Bi.2.3} with $n=1$ and Corollary~\ref{cor:Bi.2.4.2} cannot be contained
% each other.

\begin{proof}[\it Proof of Corollary~\ref{cor:Bi.2.4.2}]
We only need to prove that the conditions of Theorem~\ref{th:Bi.2.4} are satisfied.
By the assumptions the eigenvalue  $\lambda^\ast\ne 0$ is isolated and
has finite multiplicity. Let us choose $\delta>0$ such that $(\lambda^\ast-\delta,
\lambda^\ast+ \delta)\setminus\{\lambda^\ast\}$ has no intersection with $\{\lambda_k\}^\infty_{k=0}$,
where $\lambda_0=0$. Since each $H_k$ is an invariant subspace of $\mathcal{L}''(0)-\lambda\widehat{\mathcal{L}}''(0)$ by (b), and
$H=\oplus^\infty_{k=0}H_k$ is an orthogonal decomposition, we have
$\mu_\lambda=\sum^\infty_{k=0}\mu_{\lambda,k}$,
where $\mu_{\lambda,k}$ is the dimension of
the maximal negative definite space of the quadratic functional
$H_k\ni u\mapsto f_{\lambda,k}(u):=(\mathcal{L}''(0)u-\lambda\widehat{\mathcal{L}}''(0)u,u)_H$.
Note that $H_k$ has an orthogonal decomposition $H_k^+\oplus H_k^-$, where
$H_k^+$ (resp. $H_k^-$) is the positive (resp. negative) definite subspace
of $\mathcal{L}''(0)|_{H_k}$, %It is possible that $H_k^+=\{0\}$ or $H_k^-=\{0\}$.
and $f_{\lambda,k}(h)=(1-\lambda/\lambda_k)(\mathcal{L}''(0)h,h)_H,\;\forall h\in H_k$, we deduce
that $\mu_{\lambda,k}$ is equal to
$\dim H^+_k$ (resp. $\dim H^-_k$) if $\lambda>\lambda_k$ (resp. $\lambda<\lambda_k$).
%As in (\ref{e:Bi.2.11}), if $H^+_k\ne\{0\}$ (resp. $H^-_k\ne\{0\}$) and $\lambda>\lambda_k$ (resp. $\lambda<\lambda_k$) we have
%
Hence  the Morse index of $\mathcal{L}_\lambda$ at $0$,
 \begin{eqnarray}\label{e:Bi.2.13}
\mu_\lambda=\sum_{\lambda_k<\lambda}\dim H_k^+ + \sum_{\lambda_n>\lambda}\dim H_n^-.
\end{eqnarray}
%(Since $\mu_\lambda$ is finite, the right side must be finite sum.)
%Since $\lambda^\ast=\lambda_{k_0}$, as in (\ref{e:Bi.2.13})
(Here we do not use the finiteness of $\mu_\lambda$.) Let $\nu^+_{\lambda^\ast}=\dim H^+_{k_0}$ (resp. $\nu^-_{\lambda^\ast}=\dim H^-_{k_0}$).
Using (\ref{e:Bi.2.13}) it is easy to verify that
$$
\mu_\lambda=\left\{\begin{array}{ll}
\mu_{\lambda^\ast}+ \nu^-_{\lambda^\ast},&\quad\forall \lambda\in (\lambda^\ast-\delta, \lambda^\ast),\\
\mu_{\lambda^\ast}+ \nu^+_{\lambda^\ast}, &\quad
\forall\lambda\in (\lambda^\ast, \lambda^\ast+\delta),
\end{array}
\right.
$$
 Since $\mathcal{L}''(0)$
 is either positive  or negative  on $H^0_{\lambda^\ast}$, we get
 either (\ref{e:Bi.2.7.5}) or (\ref{e:Bi.2.19-}).
Note that $\mathcal{L}_\lambda:=\mathcal{L}-\lambda\widehat{\mathcal{L}}$
 satisfies Hypothesis~\ref{hyp:1.1} with $X=H$. It has finite Morse index $\mu_\lambda$ and nullity $\nu_\lambda$ at $0$.
Thus (\ref{e:Bi.2.7.6}) and (\ref{e:Bi.2.18}) can be obtained.

If ${\mathcal{L}}''(0)$ is positive definite, by the proof of Theorem~\ref{th:Bi.3}
we have (\ref{e:Bi.2.7.5}) and hence (\ref{e:Bi.2.7.6}).
\end{proof}

\begin{remark}\label{rm:Bi.2.4.3}
{\rm It is possible to prove a generalization of Theorem~\ref{th:A.2} similar to  Theorem~\ref{th:A.5-},
 which leads to a more general version of  Theorem~\ref{th:Bi.2.4}  as the following Theorem~\ref{th:Bif.2.2.0}.
These will be explored otherwise.}
%If the parameterized splitting theorem  has
% Indeed, under the assumptions of Theorem~\ref{th:Bi.1.1}
%a parameterized splitting theorems corresponding to Theorem~\ref{th:A.2} might be proved.
%Suppose that the Morse indexes of $\mathcal{F}_\lambda$
%at $0\in H$ take values $\mu_{0}$ and $\mu_{0}+\nu_{0}$
% as $\lambda\in\mathbb{R}$ varies in both sides of $0\in I$ and is close to $0$,
%where $\mu_0$ and $\nu_{0}$ are the Morse index and the nullity of  $\mathcal{F}_{0}$
%at $0$, respectively.  Then the conclusions of Theorem~\ref{th:Bi.2.4} holds. }
\end{remark}

Now we study corresponding analogue  of the above results in the setting of \cite{Lu1,Lu2}.
Here is a more general form of the analogue of Theorem~\ref{th:Bi.2.4}.

\begin{theorem}\label{th:Bif.2.2.0}
Let $H$, $X$ and $U$ be as in Hypothesis~\ref{hyp:1.3},
and let $\{\mathcal{L}_\lambda\in C^1(U, \mathbb{R})\,|\,\lambda\in\Lambda\}$ be a continuous family of functionals
parameterized by an open interval $\Lambda\subset\mathbb{R}$ containing  $\lambda^\ast$.
 For each $\lambda\in\Lambda$, assume $\mathcal{L}'_\lambda(0)=0$, and that
 there exist maps $A_\lambda\in C^1(U^X, X)$ and $B_\lambda: U\cap X\to \mathscr{L}_s(H)$,
which  depend on $\lambda$ continuously,  such that for all $x\in U\cap X$  and $u, v\in X$,
 $$
 D\mathcal{L}_\lambda(x)[u]=(A_\lambda(x), u)_H\quad\hbox{and}\quad
(DA_\lambda(x)[u], v)_H=(B_\lambda(x)u, v)_H.
$$
 Suppose also that the following conditions hold.
 \begin{description}
 \item[(a)]   $B_\lambda$ has a decomposition
$B_\lambda=P_\lambda+Q_\lambda$, where for each $x\in U\cap X$,
 $P_\lambda(x)\in\mathscr{L}_s(H)$ is  positive definitive and
$Q_\lambda(x)\in\mathscr{L}_s(H)$ is compact, so that
$(\mathcal{L}_{\lambda}, H, X, U, A_{\lambda}, B_{\lambda}=P_{\lambda}+ Q_{\lambda})$ satisfies Hypothesis~\ref{hyp:1.3}.

\item[(b)]  For each $h\in H$, it holds that $\|P_{\lambda}(x)h-P_{\lambda^\ast}(0)h\|\to 0$
as $x\in U\cap X$ approaches to $0$ in $H$ and $\lambda\in\Lambda$ converges to $\lambda^\ast$.

 \item[(c)]  For some small $\delta>0$, there exist positive constants $c_0>0$ such that
$$
(P_\lambda(x)u, u)\ge c_0\|u\|^2\quad\forall u\in H,\;\forall x\in
\bar{B}_H(0,\delta)\cap X,\quad\forall\lambda\in \Lambda.
$$
 \item[(d)]  $Q_\lambda: U\cap X\to \mathscr{L}_s(H)$ is uniformly continuous at $0$  with respect to $\lambda\in \Lambda$.
  \item[(e)]  If $\lambda\in \Lambda$ converges to $\lambda^\ast$ then
  $\|Q_{\lambda}(0)-Q_{\lambda^\ast}(0)\|\to 0$.
   \item[(f)]  ${\rm Ker}(B_{\lambda^\ast}(0))\ne\{0\}$, ${\rm Ker}(B_{\lambda}(0))=\{0\}$ for any $\lambda\in\Lambda\setminus\{\lambda^\ast\}$,
   and the Morse indexes of $\mathcal{L}_\lambda$
at $0\in H$ take values $\mu_{\lambda^\ast}$ and $\mu_{\lambda^\ast}+\nu_{\lambda^\ast}$
 as $\lambda\in\mathbb{R}$ varies in both sides of $\lambda^\ast$ and is close to $\lambda^\ast$,
where $\mu_{\lambda}$ and $\nu_{\lambda}$ are the Morse index and the nullity of  $\mathcal{L}_{\lambda}$
at $0$, respectively.
    \end{description}
 Then  one of the following alternatives occurs:
\begin{description}
\item[(i)] $(\lambda^\ast,0)$ is not an isolated solution  in  $\{\lambda^\ast\}\times U^X$ of the equation
\begin{equation}\label{e:Bif.2.2.1*}
A_\lambda(u)=0,\quad (\lambda, u)\in\Lambda\times U^X.
\end{equation}

\item[(ii)]  For every $\lambda\in\Lambda$ near $\lambda^\ast$ there is a nontrivial solution $u_\lambda$ of (\ref{e:Bif.2.2.1*}) in $U^X$, which  converges to $0$ in $X$ as $\lambda\to\lambda^\ast$.

\item[(iii)] There is an one-sided  neighborhood $\Lambda^\ast$ of $\lambda^\ast$ such that
for any $\lambda\in\Lambda^\ast\setminus\{\lambda^\ast\}$,
(\ref{e:Bif.2.2.1*}) has at least two nontrivial solutions in $U^X$,
which  converge to $0$ in $X$  as $\lambda\to\lambda^\ast$.
\end{description}
Therefore $(\lambda^\ast,0)\in\Lambda\times U^X$ is a bifurcation point  for the equation
(\ref{e:Bif.2.2.1*}); in particular $(\lambda^\ast,0)\in\Lambda\times U$ is a bifurcation point  for the equation
\begin{equation}\label{e:Bif.2.2*}
D\mathcal{L}_\lambda(u)=0,
\quad (\lambda,u)\in \Lambda\times U.
\end{equation}
 \end{theorem}

 \begin{proof}
Though the proof is almost repeat of that of Theorem~\ref{th:Bi.2.4}
we give it for completeness.
Let $H^+_\lambda$, $H^-_\lambda$ and $H^0_\lambda$ be the positive definite, negative definite and zero spaces of
${B}_\lambda(0)$.  Denote by $P^0_\lambda$ and $P^\pm_\lambda$ the orthogonal projections onto
$H^0_\lambda$ and $H^\pm_\lambda=H^+_\lambda\oplus H^-_\lambda$,
and by $X^\star_\lambda=X\cap H^\star_\lambda$ for $\star=+,-$, and by  $X^\pm_\lambda=P^\pm_\lambda(X)$.
By Theorem~\ref{th:A.5-}   there exists $\delta>0$ with $[\lambda^\ast-\delta,\lambda^\ast+\delta]\subset\Lambda$,
$\epsilon>0$ and a (unique) $C^1$ map
$$
\psi:[\lambda^\ast-\delta,\lambda^\ast+\delta]\times B_{H^0_{\lambda^\ast}}(0,\epsilon)\to X^\pm_{\lambda^\ast}
$$
which is $C^1$ in the second variable and
satisfies $\psi(\lambda, 0)=0\;\forall\lambda\in [\lambda^\ast-\delta,\lambda^\ast+\delta]$ and
$$
 P^\pm_{\lambda^\ast}A_{\lambda}(z+ \psi(\lambda,z))=0\quad\forall (\lambda,z)\in [\lambda^\ast-\delta,\lambda^\ast+\delta]\times B_{H^0_{\lambda^\ast}}(0,\epsilon),
 $$
an open neighborhood $W$ of $0$ in $H$ and an origin-preserving homeomorphism
\begin{eqnarray*}
&&[\lambda^\ast-\delta, \lambda^\ast+\delta]\times B_{H^0_{\lambda^\ast}}(0,\epsilon)\times
\left(B_{H^+_{\lambda^\ast}}(0, \epsilon) + B_{H^-_{\lambda^\ast}}(0, \epsilon)\right)\to [\lambda^\ast-\delta, \lambda^\ast+\delta]\times W,\nonumber\\
&&\hspace{20mm}({\lambda}, z, u^++u^-)\mapsto ({\lambda},\Phi_{{\lambda}}(z, u^++u^-)),
\end{eqnarray*}
such that for each $\lambda\in [\lambda^\ast-\delta, \lambda^\ast+\delta]$,
\begin{eqnarray}\label{e:NSpl.2.2}
&&\mathcal{L}_{\lambda}\circ\Phi_{\lambda}(z, u^++ u^-)=\|u^+\|^2-\|u^-\|^2+ \mathcal{
L}_{{\lambda}}(z+ \psi({\lambda}, z))\\
&& \quad\quad \forall (z, u^+ + u^-)\in  B_{H^0_{\lambda^\ast}}(0,\epsilon)\times
\left(B_{H^+_{\lambda^\ast}}(0, \epsilon) + B_{H^-_{\lambda^\ast}}(0, \epsilon)\right).\nonumber
\end{eqnarray}
(Clearly, each $\Phi_\lambda$ is  an origin-preserving homeomorphism.)
 Moreover, the functional
\begin{equation}\label{e:NSpl.2.3}
\mathcal{L}_{\lambda}^\circ: B_{H^0_{\lambda^\ast}}(0,\epsilon)\to \mathbb{R},\;
z\mapsto\mathcal{L}_{\lambda}(z+ \psi({\lambda}, z))
\end{equation}
 is of class $C^{2}$,
% its first-order and second-order differentials  at $z_0\in
%B_{H^0}(0, \epsilon)$ are given by
% \begin{eqnarray}\label{e:Spl.2.4}
%&& d\mathcal{L}^\circ_\lambda(z_0)[z]=\bigl(A(z_0+ \psi(\lambda, z_0))-\lambda \widehat{A}(z_0+ \psi(\lambda, z_0)), z\bigr)_H\quad\forall z\in H^0,\\
%  &&d^2\mathcal{L}^\circ_\lambda(0)[z,z']=\left(P^0_{\lambda^\ast}\bigr[\mathfrak{B}_\lambda-
% \mathfrak{B}_\lambda(0)(P^\pm_{\lambda^\ast}\mathfrak{B}_{\lambda^\ast}|_{X^\pm_{\lambda^\ast}})^{-1}
% (P^\pm_{\lambda^\ast}\mathfrak{B}_\lambda)\bigr]z, z'\right)_H,\nonumber\\
%&& \hspace{40mm} \forall z,z'\in H^0;\label{e:Spl.2.5}
% \end{eqnarray}
and for each $\lambda\in[\lambda^\ast-\delta, \lambda^\ast+\delta]$
the map $z\mapsto z+ \psi({\lambda}, z))$ induces an one-to-one correspondence
 between the critical points of  $\mathcal{L}_{\lambda}^\circ$ near $0\in H^0_{\lambda^\ast}$
and solutions of $A_\lambda(u)=0$ near $0\in U^X$.
 Hence the problem is reduced to finding
the critical points of $\mathcal{L}^\circ_\lambda$
near $0\in H^0_{\lambda^\ast}$ for $\lambda$ near $\lambda^\ast$.

By (f), we can shrink $\delta>0$ so that for each $\lambda\in [\lambda^\ast-\delta, \lambda^\ast+\delta]\setminus\{\lambda^\ast\}$,
$0\in H$  is a nondegenerate critical point  $\mathcal{L}_\lambda$ and the Morse indexes of $\mathcal{L}_\lambda$
at $0\in H$ take constant values in $[\lambda^\ast-\delta, \lambda^\ast)$ and $(\lambda^\ast, \lambda^\ast+\delta]$, respectively.
Assume that (\ref{e:Bi.2.7.5}) is satisfied.
Then Theorem~\ref{th:A.4} with $\lambda=0$ leads to
(\ref{e:Bi.2.7.6}). From these and Corollary~\ref{cor:A.6} we deduce
(\ref{e:Bi.2.16}) and therefore the expected conclusions as in the proof of Theorem~\ref{th:Bi.2.4}.
Another case can be proved similarity.
\end{proof}

 As showed below Theorem~\ref{th:Bi.2.4},
Theorem~\ref{th:Bif.2.2.0} is a strengthened version of  Theorem~\ref{th:Bif.1.1}
under the stronger assumptions on the Morse indexes of $\mathcal{L}_\lambda$ at $0\in H$.
Consequently,  for bifurcations of eigenvalues of nonlinear problems
we obtain the following two corollaries, which strengthen Corollary~\ref{cor:Bi.3.1}
and  correspond to Corollaries~\ref{cor:Bi.2.4.1},~\ref{cor:Bi.2.4.2}, respectively.

\begin{corollary}\label{cor:Bif.2.2.1}
Under Hypothesis~\ref{hyp:Bif.2.2.0+},
%Under the assumptions of Theorem~\ref{th:A.5},
suppose that the eigenvalue $\lambda^\ast\in\mathbb{R}$  is isolated and
 that $B(0)$ is either semi-positive or semi-negative.
  Then the conclusions of Theorem~\ref{th:Bif.2.2.0} hold if
  (\ref{e:Bif.2.2.1*}) and (\ref{e:Bif.2.2*}) are, respectively, replaced by the following two equations
  \begin{eqnarray}\label{e:Bif.2.2.1}
&&A(u)=\lambda\widehat{A}(u),\quad (\lambda, u)\in\mathbb{R}\times U^X,\\
&&D\mathcal{L}(u)=\lambda D\widehat{\mathcal{L}}(u),\label{e:Bif.2.2}
\quad (\lambda,u)\in \mathbb{R}\times U.
\end{eqnarray}
   \end{corollary}

  Indeed, by the proof of Corollary~\ref{cor:Bi.2.2} it is easily seen  that
  for sufficiently small $\delta>0$   the family
  $\{\mathcal{L}_\lambda:=\mathcal{L}-\lambda \widehat{\mathcal{L}}\,|\,\lambda\in (\lambda^\ast-\delta, \lambda^\ast+\delta)\}$
  satisfies the assumptions of Theorem~\ref{th:Bif.2.2.0}.
  (Of course, as in the proof Corollary~\ref{cor:Bi.2.4.1}, we need to replace
$\mathcal{L}''(0)$ and $\widehat{\mathcal{L}}''(0)$ by $B(0)$ and $\widehat{B}(0)$, respectively.)

Similarly,  we may modify the proof of Corollary~\ref{cor:Bi.2.4.2} to get:

\begin{corollary}\label{cor:Bif.2.2.2}
Under Hypothesis~\ref{hyp:Bif.2.2.0+},
%Under the assumptions of Theorem~\ref{th:A.5},
suppose that the following two conditions are satisfied:
 \begin{description}
\item[(a)] $B(0)$ is invertible
and $\lambda^\ast$ is an eigenvalue of $B(0)u-\lambda\widehat{B}(0)u=0$.
 \item[(b)] $B(0)\widehat{B}(0)=\widehat{B}(0)B(0)$,   and $B(0)$
 is either positive  or negative on $H^0_{\lambda^\ast}={\rm Ker}(B(0)-\lambda^\ast\widehat{B}(0))$.
 \end{description}
 Then the conclusions of Corollary~\ref{cor:Bif.2.2.1} are true.
 Moreover, if $B(0)$ is positive definite, the condition (b) is unnecessary.
\end{corollary}

Since $B(0)$ is invertible, the case (II) of Hypothesis~\ref{hyp:Bif.2.2.0} cannot occur.
The ``Moreover" part strengthens Theorem~\ref{th:Bi.4}.

In applications to Lagrange systems using
Theorem~\ref{th:Bif.2.2.0} and Corollaries~\ref{cor:Bif.2.2.1},~\ref{cor:Bif.2.2.2}
produce stronger results than using Theorem~\ref{th:Bi.2.4} and Corollaries~\ref{cor:Bi.2.4.1},~\ref{cor:Bi.2.4.2},
see \cite{Lu8}. Moreover, we may give parameterized versions of splitting lemmas for the Finsler energy functional
in \cite{Lu5} as Theorem~\ref{th:A.5-} and then obtain similar results to
Theorem~\ref{th:Bif.2.2.0} and Corollaries~\ref{cor:Bif.2.2.1},~\ref{cor:Bif.2.2.2} for geodesics on Finsler manifolds.

%\subsection{Multiparameter bifurcations}\label{sec:B.2.2.3}

%Finally, we conclude this section with
Now we give a multiparameter bifurcation result,
which is not only a converse of Corollary~\ref{cor:Bi.2.2*} in stronger conditions, but also
a generalizations of Corollaries~\ref{cor:Bif.2.2.1},~\ref{cor:Bif.2.2.2} in some sense.

\begin{theorem}\label{th:Bi.2.3}
Let $\mathcal{L}\in C^1(U,\mathbb{R})$ satisfy
Hypothesis~\ref{hyp:1.3}, except (C) and (D1),   and let
  $\widehat{\mathcal{L}}_j\in C^1(U,\mathbb{R})$, $j=1,\cdots,n$, satisfy
 Hypothesis~\ref{hyp:1.4}.
Suppose also that the following conditions hold:
\begin{description}
\item[(A)]  $\vec{\lambda}^\ast$ is an isolated  eigenvalue  of (\ref{e:Bi.2.2}), (writting $H(\vec{\lambda}^\ast)$
the solution space  of (\ref{e:Bi.2.2}) with $\vec{\lambda}=\vec{\lambda}^\ast$),
 and for each point $\vec{\lambda}$ near $0\in\mathbb{R}^n$, it holds that
\begin{equation}\label{e:Bi.2.2.1}
\{u\in H\,|\,\mathfrak{B}_{\vec{\lambda}}u\in X\}\cup\{u\in H\,|\, \mathfrak{B}_{\vec{\lambda}}u
=\mu u,\;\mu\le 0\}\subset X
\end{equation}
where $\mathfrak{B}_{\vec{\lambda}}:=B(0)-\sum^n_{j=1}\lambda_j\widehat{B}_j(0)$.
 \item[(B)]  Either $H(\vec{\lambda}^\ast)$ has odd dimension, or $\mathfrak{B}_{\vec{\lambda}^\ast}\widehat{B}_j(0)=\widehat{B}_j(0)\mathfrak{B}_{\vec{\lambda}^\ast}$
 for $j=1,\cdots,n$, and there exists  $\vec{\lambda}\in\mathbb{R}^n\setminus\{{0}\}$ such that  the  symmetric bilinear form
\begin{equation}\label{e:Bi.2.2.3}
H(\vec{\lambda}^\ast)\times H(\vec{\lambda}^\ast)\ni (z_1,z_2)\mapsto \mathscr{Q}_{\vec{\lambda}}(z_1,z_2):=
 \sum^n_{j=1}(\lambda_j-\lambda^\ast_j)(\widehat{B}_j(0)z_1,z_2)_H
\end{equation}
has different Morse indexes and coindexes.
%\item[(C)]  the corresponding finite dimension reduction  $\mathcal{L}^\circ_{\vec{\lambda}}$
%with the functional
%$$
%\mathcal{L}_{\vec{\lambda}}:=\mathcal{L}-\sum^n_{j=1}(\lambda_j+\lambda^\ast_j)\widehat{\mathcal{L}}_j=
%\mathcal{L}_{\vec{0}}-\sum^n_{j=1}\lambda_j\widehat{\mathcal{L}}_j
%$$
%as in Theorem~\ref{th:A.2}  is of class $C^2$;
\end{description}
(If $\vec{\lambda}^\ast=0$, we only need  that
each $\widehat{\mathcal{L}}_j\in C^1(U,\mathbb{R})$ is as in ``Moreover" part of Corollary~\ref{cor:Bi.2.2}.)
Then $(\vec{\lambda}^\ast, 0)\in\mathbb{R}^n\times U^X$ is a  bifurcation point of
(\ref{e:Bi.2.1*}). Moreover, suppose that (B) is replaced by
\begin{description}
\item[(B')] $\mathfrak{B}_{\vec{\lambda}^\ast}\widehat{B}_j(0)=\widehat{B}_j(0)\mathfrak{B}_{\vec{\lambda}^\ast}$
 for $j=1,\cdots,n$, and there exists  $\vec{\mu}\in\mathbb{R}^n\setminus\{{0}\}$ such that
the form $ \mathscr{Q}_{\vec{\mu}}$ on $H(\vec{\lambda}^\ast)$ is either positive definite or negative one.
\end{description}
Then one of the following alternatives occurs:
\begin{description}
\item[(i)] $(\vec{\lambda}^\ast, 0)$ is not an isolated solution of (\ref{e:Bi.2.1*}) in
 $\{\vec{\lambda}^\ast\}\times U^X$.

\item[(ii)]  For every $t$ near $0\in\mathbb{R}$ there is a nontrivial solution $u_t$ of (\ref{e:Bi.2.1*}) with $\vec{\lambda}=t(\vec{\mu}-\vec{\lambda}^\ast)+\vec{\lambda}^\ast$
    converging to $0$ in $X$ as $t\to 0$.

\item[(iii)] There is an one-sided  neighborhood $\mathfrak{T}$ of $0\in\mathbb{R}$ such that
for any $t\in\mathfrak{T}\setminus\{0\}$,
(\ref{e:Bi.2.1*}) with $\vec{\lambda}=t(\vec{\mu}-\vec{\lambda}^\ast)+\vec{\lambda}^\ast$ has at least two nontrivial solutions converging to
$0$ in $X$ as $t\to 0$.
\end{description}
%(Clearly, when the second case in {\bf (B)} occurs, ``$\vec{\mu}\in\mathbb{R}^n\setminus\{\vec{0}\}$ with very small $|\mu|$,
%the form $d^2\mathcal{L}^\circ_{\vec{\mu}}(0)$" may be replaced by
%``$\vec{\mu}\in\mathbb{R}^n\setminus\{\vec{0}\}$, the form $\mathscr{Q}_{\vec{\mu}}$".)

%Finally, if $\vec{\lambda}^\ast=0$, it suffices to assume that
%$\widehat{\mathcal{L}}_j\in C^1(U,\mathbb{R})$, $j=1,\cdots,n$, satisfy
% Hypothesis~\ref{hyp:1.4}  without requirement of compactness of $\widehat{B}_j(u)\in\mathscr{L}_s(H)$.
\end{theorem}

\begin{proof} {\bf Step 1}.\quad {\it Prove the first claim}.
Let $H^+_{\vec{\lambda}}$, $H^-_{\vec{\lambda}}$ and $H^0_{\vec{\lambda}}$ be the positive definite, negative definite and zero spaces of
$\mathfrak{B}_{\vec{\lambda}}=B(0)-\sum^n_{j=1}\lambda_j\widehat{B}_j(0)$.
 Denote by $P^0_{\vec{\lambda}}$ and $P^\pm_{\vec{\lambda}}$ the orthogonal projections onto
  $H^0_{\vec{\lambda}}$ and $H^\pm_{\vec{\lambda}}=H^+_{\vec{\lambda}}\oplus H^-_{\vec{\lambda}}$,
and by $X^\star_{\vec{\lambda}}=X\cap H^\star_{\vec{\lambda}}$ for $\star=+,-$, and by  $X^\pm_{\vec{\lambda}}=P^\pm_{\vec{\lambda}}(X)$.
By the assumption, for each $\vec{\lambda}\in\mathbb{R}^n$ near $\vec{\lambda}^\ast$ spaces
$H^-_{\vec{\lambda}}$ and $H^0_{\vec{\lambda}}$ have finite dimensions and are contained in $X$, so
$H^-_{\vec{\lambda}}=X^-_{\vec{\lambda}}$ and $H^0_{\vec{\lambda}}=X^0_{\vec{\lambda}}$.
Because the corresponding operator $A$ with $\mathcal{L}\in C^1(U,\mathbb{R})$
is not assumed to be $C^1$,  Theorem~\ref{th:A.5-}  cannot be used.
But modifying its proof  (as in the proof of \cite[Theorem~2.1]{Lu2})  there exist  $\delta>0$ and
$\epsilon>0$, a (unique) $C^{1-0}$ map
\begin{eqnarray*}%\label{e:BiSpl.2.1.1}
\psi:\vec{\lambda}^\ast+[-\delta,\delta]^n\times B_{H^0}(0,\epsilon)\to X^\pm_{\vec{\lambda}^\ast}
\end{eqnarray*}
which is strictly Fr\'echet differentiable in the second variable at $0\in H^0$ and satisfies $\psi(\vec{\lambda}, 0)=0\;\forall\vec{\lambda}\in \vec{\lambda}^\ast+[-\delta,\delta]^n$, and
%\begin{equation}\label{e:BiSpl.2.1.2}
% P^\pm_{\lambda^\ast}(A-\lambda \widehat{A})(z+ \psi(\lambda,z))=0\quad\forall (\lambda,z)\in [\lambda^\ast-\delta,\lambda^\ast+\delta]\times B_{H^0_{\lambda^\ast}}(0,\epsilon),
% \end{equation}
an open neighborhood $W$ of $0$ in $H$ and an origin-preserving homeomorphism
\begin{eqnarray*}
&&(\vec{\lambda}^\ast+[-\delta,\delta]^n)\times B_{H^0_{\vec{\lambda}^\ast}}(0,\epsilon)\times
\left(B_{H^+_{\vec{\lambda}^\ast}}(0, \epsilon) + B_{H^-_{\vec{\lambda}^\ast}}(0, \epsilon)\right)\to
(\vec{\lambda}^\ast+[-\delta,\delta]^n)\times W,\nonumber\\
&&\hspace{20mm}({\vec{\lambda}}, z, u^++u^-)\mapsto ({\vec{\lambda}},\Phi_{\vec{\lambda}}(z, u^++u^-)),
\end{eqnarray*}
such that for each $\vec{\lambda}\in \vec{\lambda}^\ast+[-\delta,\delta]^n$,
the functional $\mathcal{L}_{\vec{\lambda}}$ satisfies
\begin{eqnarray}\label{e:BiSpl.2.2}
&&\mathcal{L}_{\vec{\lambda}}\circ\Phi_{\vec{\lambda}}(z, u^++ u^-)=\|u^+\|^2-\|u^-\|^2+ \mathcal{
L}_{\vec{\lambda}}(z+ \psi(\vec{\lambda}, z))\\
&& \quad\quad \forall (z, u^+ + u^-)\in  B_{H^0_{\vec{\lambda}^\ast}}(0,\epsilon)\times
\left(B_{H^+_{\vec{\lambda}^\ast}}(0, \epsilon) + B_{H^-_{\vec{\lambda}^\ast}}(0, \epsilon)\right).\nonumber
\end{eqnarray}
 Moreover,  the functional
$\mathcal{L}_{\vec{\lambda}}^\circ: B_{H^0_{\vec{\lambda}^\ast}}(0,\epsilon)\to \mathbb{R}$ given by
$\mathcal{L}_{\vec{\lambda}}^\circ(z)=\mathcal{L}_{\vec{\lambda}}(z+ \psi({\vec{\lambda}}, z))$
 is of class $C^{2-0}$, has first-order  differential  at $z_0\in
B_{H^0_{\vec{\lambda}^\ast}}(0, \epsilon)$ are given by
 \begin{eqnarray}\label{e:BiSpl.2.4}
d\mathcal{L}^\circ_{\vec{\lambda}}(z_0)[z]=\bigl(A(z_0+ \psi(\vec{\lambda}, z_0))-\sum^n_{j=1}\lambda_j \widehat{A}_j(z_0+ \psi(\vec{\lambda}, z_0)), z\bigr)_H\quad\forall z\in H^0_{\vec{\lambda}^\ast},
\end{eqnarray}
and $d\mathcal{L}^\circ_{\vec{\lambda}}$ has strict Fr\'echet derivative at $0\in H^0_{\vec{\lambda}^\ast}$,
\begin{eqnarray}\label{e:BiSpl.2.5}
  d^2\mathcal{L}^\circ_{\vec{\lambda}}(0)[z,z']=
 \left(P^0_{\vec{\lambda}^\ast}\bigr[\mathfrak{B}_{\vec{\lambda}}-
 \mathfrak{B}_{\vec{\lambda}}(P^\pm_{\vec{\lambda}^\ast}\mathfrak{B}_{\vec{\lambda}}|_{X^\pm_{\vec{\lambda}^\ast}})^{-1}
 (P^\pm_{\vec{\lambda}^\ast}\mathfrak{B}_{\vec{\lambda}})\bigr]z, z'\right)_H,\;
  \forall z,z'\in H^0_{\vec{\lambda}^\ast}.
 \end{eqnarray}
%
%\begin{eqnarray}\label{e:BiSpl.2.5}
% && d^2\mathcal{L}^\circ_{\vec{\lambda}}(0)[z,z']=
% \left(P^0_{\vec{\lambda}^\ast}\bigr[\mathfrak{B}_{\vec{\lambda}}-
% \mathfrak{B}_{\vec{\lambda}}(P^\pm_{\vec{\lambda}^\ast}\mathfrak{B}_{\vec{\lambda}}|_{X^\pm_{\vec{\lambda}^\ast}})^{-1}
% (P^\pm_{\vec{\lambda}^\ast}\mathfrak{B}_{\vec{\lambda}})\bigr]z, z'\right)_H,\\
% &&\hspace{40mm} \forall z,z'\in H^0_{\vec{\lambda}^\ast}.\nonumber
% \end{eqnarray}
 Clearly, (\ref{e:BiSpl.2.5}) implies $d^2\mathcal{L}^\circ_{\vec{\lambda}^\ast}(0)=0$.
 As the arguments above \cite[Claim~2.17]{Lu7} we have
  \begin{equation}\label{e:Bi.2.2.2}
  d^2\mathcal{L}^\circ_{\vec{\lambda}}(0)[z_1,z_2]=\sum^n_{j=1}(\lambda_j-\lambda^\ast_j)(\widehat{B}_j(0)(z_1+D_z\psi(\vec{\lambda}, 0)[z_1]),z_2)_H,\quad\forall z_1, z_2\in H^0_{\vec{\lambda}^\ast}.
  \end{equation}
 (The negative sign in \cite[(2.59)]{Lu7} should be removed.)

 As before, using the above results   the
  bifurcation problem (\ref{e:Bi.2.1*}) is reduced to the following one
 \begin{equation}\label{e:Bif.2.2.5}
 \nabla\mathcal{L}^\circ_{\vec{\lambda}}(u)=0,\;u\in B_{H^0_{\vec{\lambda}^\ast}}(0,\epsilon),\;\vec{\lambda}\in \vec{\lambda}^\ast+[-\delta,\delta]^n.
 \end{equation}
By ({\bf A}) we may shrink $\delta>0$ so that
for each $\vec{\lambda}\in \vec{\lambda}^\ast+([-\delta,\delta]^n\setminus\{0\})$,
$\mathcal{L}\in C^1(U,\mathbb{R})$ satisfy Hypothesis~\ref{hyp:1.3} and has
$0\in H$ as a nondegenerate critical point.
Using Theorem~\ref{th:A.4} with $\widehat{\mathcal{L}}=0$
(or \cite[Remark~2.2]{Lu2})  and a consequence of (\ref{e:BiSpl.2.2})  as in Theorem~\ref{th:A.3} we obtain
that for any Abel group ${\bf K}$ and for each $\vec{\lambda}\in \vec{\lambda}^\ast+([-\delta,\delta]^n\setminus\{0\})$,
\begin{equation}\label{e:Bif.2.2.6}
C_q(\mathcal{L}^\circ_{\vec{\lambda}}, 0;{\bf K})=\delta_{(q+{\mu_{\vec{\lambda}^\ast}})\mu_{\vec{\lambda}}}{\bf K},\quad
\forall q\in\mathbb{N}\cup\{0\},
\end{equation}
where $\mu_{\vec{\lambda}}$ is the Morse index of $\mathcal{L}_{\vec{\lambda}}$  at $0\in H$.

By a contradiction suppose now that $(\vec{\lambda}^\ast, 0)\in\mathbb{R}^n\times U^X$ is not a  bifurcation point of
(\ref{e:Bi.2.1*}). Then $(\vec{\lambda}^\ast, 0)\in\mathbb{R}^n\times  H^0_{\vec{\lambda}^\ast}$ is not a  bifurcation point of
(\ref{e:Bif.2.2.5}). Hence we may shrink $\epsilon>0$ and $\delta>0$ such that
$\mathcal{L}^\circ_{\vec{\lambda}}$ has a unique critical point $0$ in $B_{H^0_{\vec{\lambda}^\ast}}(0,\epsilon)$
for each $\vec{\lambda}\in \vec{\lambda}^\ast+[-\delta,\delta]^n$. As before it follows from the stability of critical groups
(cf.Theorem~\ref{th:stablity1}) that
\begin{equation}\label{e:Bi.2.7.2}
C_\ast(\mathcal{L}^\circ_{\vec{\lambda}}, 0;{\bf K})=C_\ast(\mathcal{L}^\circ_{\vec{\lambda}^\ast}, 0;{\bf K}),\quad
\forall \vec{\lambda}\in \vec{\lambda}^\ast+[-\delta,\delta]^n.
\end{equation}
This and (\ref{e:Bif.2.2.6}) imply  that the Morse index of $\mathcal{L}^\circ_{\vec{\lambda}}$ at $0$
is constant with respect to $\vec{\lambda}\in \vec{\lambda}^\ast+([-\delta,\delta]^n\setminus\{{0}\})$.
On the other hand,  for every $\vec{\lambda}\in \vec{\lambda}^\ast+([-\delta,\delta]^n\setminus\{{0}\})$,
(\ref{e:Bi.2.2.2}) implies
%\begin{equation}\label{e:Bi.2.7.3}
%d^2\mathcal{L}^\circ_{2\vec{\lambda}^\ast-\vec{\lambda}}(0)=-d^2\mathcal{L}^\circ_{\vec{\lambda}}(0)
%\end{equation}
$$
d^2\mathcal{L}^\circ_{2\vec{\lambda}^\ast-\vec{\lambda}}(0)=-d^2\mathcal{L}^\circ_{\vec{\lambda}}(0)
$$
 as the nondegenerate quadratic forms on $H^0_{{\lambda}^\ast}$. Note that $H^0_{\vec{\lambda}^\ast}=H(\vec{\lambda}^\ast)$.
So if $\dim H^0_{\vec{\lambda}^\ast}$ is odd, $d^2\mathcal{L}^\circ_{\vec{\lambda}}(0)$ and $d^2\mathcal{L}^\circ_{2\vec{\lambda}^\ast-\vec{\lambda}}(0)$ must have different Morse indexes. This contradicts (\ref{e:Bi.2.7.2}).

In other situation, since for each $j=1,\cdots,n$, $\mathfrak{B}_{\vec{\lambda}^\ast}\widehat{B}_j(0)=\widehat{B}_j(0)\mathfrak{B}_{\vec{\lambda}^\ast}$
implies that $P^\star_{\vec{\lambda}^\ast}\widehat{B}_j(0)=\widehat{B}_j(0)P^\star_{\vec{\lambda}^\ast}$,
 $\star=0,+,-$,   and $D_z\psi(\vec{\lambda}, 0)[z_1]\in H^\pm_{\vec{\lambda}}$ for all $z_1\in H^0_{\vec{\lambda}}$,
we deduce
\begin{eqnarray*}
(D_z\psi(\vec{\lambda}, 0)[z_1], \widehat{B}_j(0)z_2)_H&=&(D_z\psi(\vec{\lambda}, 0)[z_1], \widehat{B}_j(0)P^0_{\vec{\lambda}^\ast}z_2)_H\\
&=&(D_z\psi(\vec{\lambda}, 0)[z_1], P^0_{\vec{\lambda}^\ast}\widehat{B}_j(0)z_2)_H=0,\quad
\forall z_1, z_2\in H^0_{\vec{\lambda}^\ast}.
\end{eqnarray*}
Thus (\ref{e:Bi.2.2.2}) becomes
$$
d^2\mathcal{L}^\circ_{\vec{\lambda}}(0)[z_1,z_2]=\sum^n_{j=1}(\lambda_j-\lambda^\ast_j)(\widehat{B}_j(0)z_1,z_2)_H=
  \mathscr{Q}_{\vec{\lambda}}(z_1,z_2),\quad\forall z_1, z_2\in H^0_{\vec{\lambda}^\ast}.
$$
% \begin{equation}\label{e:Bi.2.7.4}
%  d^2\mathcal{L}^\circ_{\vec{\lambda}}(0)[z_1,z_2]=\sum^n_{j=1}(\lambda_j-\lambda^\ast_j)(\widehat{B}_j(0)z_1,z_2)_H=
%  \mathscr{Q}_{\vec{\lambda}}(z_1,z_2),\quad\forall z_1, z_2\in H^0_{\vec{\lambda}^\ast}.
%  \end{equation}
Then the same reasoning  leads to a contradiction. The first claim is proved.

\noindent{\bf Step 2}.\quad {\it Prove the second claim}.
By (\ref{e:Bi.2.2.3}), $\mathscr{Q}_{t(\vec{\mu}-\vec{\lambda}^\ast)+\vec{\lambda}^\ast}=t\mathscr{Q}_{\vec{\mu}}$
for any $t\in\mathbb{R}$. By replacing $\vec{\mu}$ by $2\vec{\lambda}^\ast-\vec{\mu}$ (if necessary)  we may assume that
the form $d^2\mathcal{L}^\circ_{\vec{\mu}}(0)=\mathscr{Q}_{\vec{\mu}}$ is positive definite.
Then when $t>0$ with small $|t|$ the forms
$d^2\mathcal{L}^\circ_{t(\vec{\mu}-\vec{\lambda}^\ast)+\vec{\lambda}^\ast}$
are positive definite and so $0\in H^0_{\vec{\lambda}^\ast}$ is a strict local
minimizer of $\mathcal{L}^\circ_{t(\vec{\mu}-\vec{\lambda}^\ast)+\vec{\lambda}^\ast}$;
when $t<0$ with small $|t|$ the forms
$d^2\mathcal{L}^\circ_{t(\vec{\mu}-\vec{\lambda}^\ast)+\vec{\lambda}^\ast}$
are negative definite and so $0\in H^0_{\vec{\lambda}^\ast}$ is a strict local
maximizer of $\mathcal{L}^\circ_{t(\vec{\mu}-\vec{\lambda}^\ast)+\vec{\lambda}^\ast}$.
As before the desired  claims follows from  Theorem~\ref{th:Bi.2.1}.
%\noindent{\bf Step 3}.\quad If $\vec{\lambda}^\ast={0}$, the above arguments are still effective in weaker
%assumptions therein.
\end{proof}

Some results in next Sections~\ref{sec:B.3},\ref{sec:BBH} can also be generalized 
to such more general multiparameter forms.

%For Lagrangian systems considered in \cite{Lu1} we may prove that the condition (D) in Theorem~\ref{th:Bi.2.3}
%can be satisfied.

%\begin{remark}\label{rem:Bi.2.4}
%{\rm

Finally, in order to understand advantages and disadvantages of our above methods
let us state results which can be obtained with the parameterized version of the classical splitting lemma
for $C^2$ functionals stated in Remark~\ref{rm:Spl.2.5} and the classical Morse-Palais lemma  for $C^2$ functionals.
% our above proof methods Theorem~\ref{th:Bif.2.2.0}
%can lead to:

%}
%\end{remark}

\begin{theorem}\label{th:Bi.2.5}
Let $U$ be an open neighborhood of $0$ in a real Hilbert space $H$,
$\Lambda\subset\mathbb{R}$ an open interval  containing  $\lambda^\ast$,
and let $\{\mathcal{L}_\lambda\in C^2(U, \mathbb{R})\,|\,\lambda\in\Lambda\}$ be a  family of functionals
such that $\Lambda\times U\ni (\lambda,u)\mapsto\nabla\mathcal{L}_\lambda(u)\in H$ is  continuous.
Suppose that the following conditions are satisfied:
\begin{description}
\item[(a)] For each $\lambda\in\Lambda$,  $0$ is a critical point of $\mathcal{L}_\lambda$ with
 finite Morse index $\mu_\lambda$ and nullity $\nu_\lambda$.

\item[(b)]  If  $\nu_{\lambda^\ast}\ne 0$, $\nu_{\lambda}=0$ for any $\lambda\in\Lambda\setminus\{\lambda^\ast\}$,
   and $\mu_\lambda$  take values $\mu_{\lambda^\ast}$ and $\mu_{\lambda^\ast}+\nu_{\lambda^\ast}$
 as $\lambda\in\mathbb{R}$ varies in both sides of $\lambda^\ast$ and is close to $\lambda^\ast$.
\end{description}
  Then  one of the following alternatives occurs:
\begin{description}
\item[(i)] $(\lambda^\ast,0)$ is not an isolated solution  in  $\{\lambda^\ast\}\times U$ of the equation
\begin{equation}\label{e:Bi.2.7.3}
D\mathcal{L}_\lambda(u)=0,\quad (\lambda,u)\in \Lambda\times U.
\end{equation}
\item[(ii)]  For every $\lambda\in\Lambda$ near $\lambda^\ast$ there is a nontrivial solution $u_\lambda$ of (\ref{e:Bi.2.7.3}) in $U$
converging to $0$  as $\lambda\to\lambda^\ast$.

\item[(iii)] There is an one-sided  neighborhood $\Lambda^\ast$ of $\lambda^\ast$ such that
for any $\lambda\in\Lambda^\ast\setminus\{\lambda^\ast\}$, (\ref{e:Bi.2.7.3}) has at least two nontrivial solutions in $U$
converging to $0$  as $\lambda\to\lambda^\ast$.
\end{description}
In particular, $(\lambda^\ast,0)\in\Lambda\times U$ is a bifurcation point  for the equation
(\ref{e:Bi.2.7.3}).
\end{theorem}

In fact, we only need to make suitable replacements in the  proof of Theorem~\ref{th:Bif.2.2.0}.
For example, we replace symbols ``$X$" by ``$H$", ``$A_\lambda$" by ``$\nabla \mathcal{L}_\lambda$", and
 ``$B_\lambda$" by ``$\mathcal{L}^{\prime\prime}_{\lambda}$", and phrases
 `` By Theorem~\ref{th:A.5-}" by ``the parameterized version of the classical splitting lemma  for $C^2$ functionals
stated in Remark~\ref{rm:Spl.2.5}",  and ``Theorem~\ref{th:A.4} with $\lambda=0$" by
   ``the classical Morse-Palais lemma  for $C^2$ functionals",
 ``Corollary~\ref{cor:A.6}" by ``a corresponding result with Corollary~\ref{cor:A.6}".

\begin{theorem}\label{th:Bi.2.6}
The conclusions of Theorem~\ref{th:Bi.2.5} hold true if (a) and (b) are replaced by the following conditions:
\begin{description}
\item[(a')] For each $\lambda\in\Lambda$,  $\mathcal{L}'_\lambda(0)=0$, $0<\nu_{\lambda^\ast}:=\dim{\rm Ker}(\mathcal{L}''_{\lambda^\ast}(0))<\infty$.
\item[(b')] For some small $\delta>0$, $\mathcal{L}_\lambda$ has an isolated local minimum (maximum) at zero for every
$\lambda\in (\lambda^\ast,\lambda^\ast+\delta]$ and an isolated local maximum (minimum) at
zero for every $\lambda\in [\lambda^\ast-\delta, \lambda^\ast)$.
\end{description}
 \end{theorem}

\begin{proof}
As in the proof of Theorem~\ref{th:Bif.2.2.0},  the original bifurcation problem  is reduced to that of
 \begin{eqnarray*}
 \nabla\mathcal{L}^\circ_{{\lambda}}(u)=0,\;u\in B_{H^0_{{\lambda}^\ast}}(0,\epsilon),\;
 \lambda\in [\lambda^\ast-\delta,\lambda^\ast+\delta]\setminus\{\lambda^\ast\}.
 \end{eqnarray*}
Here each $\mathcal{L}^\circ_\lambda: B_H(0, \epsilon)\cap H^0_{\lambda^\ast}\to\mathbb{R}$ is
   a $C^2$ functional given by (\ref{e:NSpl.2.3}), and
 has differential at $z\in B_H(0, \epsilon)\cap H^0_{\lambda^\ast}$ given by
 $D\mathcal{L}_{\lambda}^\circ(z)[\zeta]=(\nabla\mathcal{L}_{\lambda}(z+ \psi({\lambda}, z)),\zeta)_H$
 for $\zeta\in H^0_{\lambda^\ast}$. Clearly, the assumption (b') implies that $\mathcal{L}_\lambda^\circ$
 has an isolated local minimum (maximum) at zero for every
$\lambda\in (\lambda^\ast,\lambda^\ast+\delta]$ and an isolated local maximum (minimum) at
zero for every $\lambda\in [\lambda^\ast-\delta, \lambda^\ast)$.
Then Theorem~\ref{th:Bi.2.1} leads to the desired  claims.
\end{proof}

Comparing Theorem~\ref{th:Bi.2.6} with \cite[Theorem~4.2]{CorH} we have no  the (PS) condition,
but require higher smoothness of functionals.

 Corresponding to Corollaries~\ref{cor:Bi.2.4.1},~\ref{cor:Bi.2.4.2} (or
Corollaries~\ref{cor:Bif.2.2.1},~\ref{cor:Bif.2.2.2}) we have

\begin{corollary}\label{cor:Bi.2.6}
Let $U$ be an open neighborhood of $0$ in a real Hilbert space $H$,
and let $\mathcal{L}, \widehat{\mathcal{L}}\in C^2(U,\mathbb{R})$ satisfy
$\mathcal{L}'(0)=0$ and  $\widehat{\mathcal{L}}'(0)=0$. Suppose that
$\lambda^\ast\in\mathbb{R}$ is an  eigenvalue of
$\mathcal{L}''(0)u=\lambda\widehat{\mathcal{L}}''(0)u$ in $H$  of finite multiplicity.
% \begin{eqnarray}\label{e:Bi.1.5}
%  \mathcal{L}''(0)-\lambda\widehat{\mathcal{L}}''(0)=0
%  \end{eqnarray}
 Then the conclusions of Theorem~\ref{th:Bi.2.5} hold true with $\mathcal{L}_\lambda=\mathcal{L}-\lambda\widehat{\mathcal{L}}$
if  one of the following conditions is satisfied:
\begin{description}
\item[(a)] The eigenvalue $\lambda^\ast$ is isolated,  $\mathcal{L}_\lambda$  has finite Morse index  at $0$ for each $\lambda\in\mathbb{R}$ near $\lambda^\ast$, and  $\widehat{\mathcal{L}}''(0)$ is either semi-positive or semi-negative.
\item[(b)] For some small $\delta>0$, $\mathcal{L}_\lambda$ has an isolated local minimum (maximum) at zero for every
$\lambda\in (\lambda^\ast,\lambda^\ast+\delta]$ and an isolated local maximum (minimum) at
zero for every $\lambda\in [\lambda^\ast-\delta, \lambda^\ast)$.
\end{description}
\end{corollary}

When  $\widehat{\mathcal{L}}(u)=\frac{1}{2}(u,u)_H$, the case (a) shows that
 the  Rabinowitz bifurcation theorem stated at the beginning of this section
 can be obtained under the additional condition that $\mathcal{L}_\lambda$  has finite Morse index  at $0$ for
each $\lambda\in\mathbb{R}$ near $\lambda^\ast$. The latter  condition cannot imply that
$\mathcal{L}$  has finite Morse index  at $0$ if $\lambda^\ast\ne 0$.

\begin{proof}[\it Proof of Corollary~\ref{cor:Bi.2.6}]
The case that (b) holds follows from Theorem~\ref{th:Bi.2.6} directly.

When the case (a) holds, since $\lambda^\ast\in\mathbb{R}$ is an isolated eigenvalue of
$\mathcal{L}''(0)u=\lambda\widehat{\mathcal{L}}''(0)u$ in $H$  of finite multiplicity,
for each $\lambda\in\mathbb{R}$ near $\lambda^\ast$,
 $\mathcal{L}''(0)-\lambda\widehat{\mathcal{L}}''(0)$ is a bounded linear Fredholm operator.
(Hence $\mathcal{L}_\lambda$ has  finite  nullity $\nu_\lambda$  at $0$.)
Repeating the proof of Corollary~\ref{cor:Bi.2.4.1}
we can get (\ref{e:Bi.2.19.1}) and hence the condition (b) of Theorem~\ref{th:Bi.2.5} is satisfied.
%\noindent{\bf Case (b)}.\quad
%As in the proof of Theorem~\ref{th:Bi.2.4},  the original bifurcation problem  is reduced to that of
% \begin{eqnarray*}
% \nabla\mathcal{L}^\circ_{{\lambda}}(u)=0,\;u\in B_{H^0_{{\lambda}^\ast}}(0,\epsilon),\;
% \lambda\in [\lambda^\ast-\delta,\lambda^\ast+\delta]\setminus\{\lambda^\ast\}.
% \end{eqnarray*}
%Here each $\mathcal{L}^\circ_\lambda: B_H(0, \epsilon)\cap H^0_{\lambda^\ast}\to\mathbb{R}$ is
%   a $C^2$ functional given by (\ref{e:Bi.2.15}), and
% has differential at $z\in B_H(0, \epsilon)\cap H^0_{\lambda^\ast}$ given by (\ref{e:Bi.2.15+}).
%Clearly, the assumption (b) implies that $\mathcal{L}_\lambda^\circ$ has an isolated local minimum (maximum) at zero for every
%$\lambda\in (\lambda^\ast,\lambda^\ast+\delta]$ and an isolated local maximum (minimum) at
%zero for every $\lambda\in [\lambda^\ast-\delta, \lambda^\ast)$.
%Then Theorem~\ref{th:Bi.2.1} leads to the desired  claims.
\end{proof}

\begin{corollary}\label{cor:Bi.2.7}
Let $U$ be an open neighborhood of $0$ in a real Hilbert space $H$,
and let $\mathcal{L}, \widehat{\mathcal{L}}\in C^2(U,\mathbb{R})$ satisfy
$\mathcal{L}'(0)=0$ and  $\widehat{\mathcal{L}}'(0)=0$. %Suppose that
%$\lambda^\ast\in\mathbb{R}$ is an  eigenvalue of
%$\mathcal{L}''(0)u=\lambda\widehat{\mathcal{L}}''(0)u$ in $H$  of finite multiplicity.
%Let $\mathcal{L}\in C^1(U,\mathbb{R})$ (resp.  $\widehat{\mathcal{L}}\in C^1(U,\mathbb{R})$) satisfy
% Hypothesis~\ref{hyp:1.1} with $X=H$ (resp.  Hypothesis~\ref{hyp:1.2}).
 Suppose that the following two conditions are satisfied:
 \begin{description}
\item[(a)] $\mathcal{L}''(0)$ is invertible, $\widehat{\mathcal{L}}''(0)$ is compact
and $\lambda^\ast$ is an eigenvalue of $\mathcal{L}''(0)u=\lambda\widehat{\mathcal{L}}''(0)u$ in $H$.
 \item[(b)] For each $\lambda\in\mathbb{R}$ near $\lambda^\ast$,  $\mathcal{L}_\lambda=\mathcal{L}-\lambda\widehat{\mathcal{L}}$ has finite Morse index at $0\in H$, $\mathcal{L}''(0)\widehat{\mathcal{L}}''(0)=\widehat{\mathcal{L}}''(0)\mathcal{L}''(0)$,
    and  $\mathcal{L}''(0)$  is either positive  or negative on $H^0_{\lambda^\ast}:={\rm Ker}(\mathcal{L}''(0)-\lambda^\ast\widehat{\mathcal{L}}''(0))$.
 \end{description}
 Then the conclusions of Theorem~\ref{th:Bi.2.5} hold true.
 Moreover, if ${\mathcal{L}}''(0)$ is positive definite, the condition (b) is unnecessary.
\end{corollary}

Clearly, the assumption (a) implies that $\mathcal{L}_\lambda$  has finite nullity at $0$ for each $\lambda\in\mathbb{R}$.
Repeating the proof of Corollary~\ref{cor:Bi.2.4.1} gives the desired claim.
The ``Moreover" part strengthens Theorem~\ref{th:Bi.7.1}.

Compare Corollaries~\ref{cor:Bi.2.6},\ref{cor:Bi.2.7} with
%Corollary~\ref{cor:Bi.7} and
 \cite{Rab74,Rab}  by Rabinowitz.
Similarly, we can also write a corresponding result with Theorem~\ref{th:Bi.2.3}.

\section{Bifurcation for equivariant problems}\label{sec:B.3}
\setcounter{equation}{0}

There are abundant studies of bifurcations for equivariant maps,
see \cite{Ba1, FaRa1, FaRa2, Fe1, SmoWa} and the references therein.
In this section we shall generalize  main results of \cite{FaRa1, FaRa2}
in the setting of Section~\ref{sec:B.2}.
 %Others will be given in
%Sections~\ref{sec:BBH},~\ref{sec:CR}.\\

Let $G$ be a topological group. A (left) {\bf action} of it  over a Banach space $X$ is a continuous
map $G\times X\to X,\;(g,x)\mapsto gx$ satisfying the usual rules: $(g_1g_2)x=g_1(g_2x)$ and $ex=x$ for all $x\in X$
and $g_1,g_2\in G$, where $e\in G$ denotes the unite element. This action is called {\bf linear} (resp. {\bf
isometric}) if for each $g\in G$, $X\ni x\mapsto gx\in X$ is linear (and so a bounded linear operator
from $X$ to itself) (resp. $\|gx\|=\|x\|$ for all $x\in X$). When $X$ is a Hilbert space a linear isometric action
is also called  an orthogonal action (since a linear isometry must be an orthogonal linear transformation).

%\subsection{Generalizations of Fadell--Rabinowitz bifurcation theorems }\label{sec:B.3.1}

\subsection{Bifurcation theorems of Fadell--Rabinowitz type}\label{sec:B.3.1}

%\subsection{Finite dimensional case}\label{sec:B.3.0}

We first study finite dimensional situations.
The following is a refinement of \cite[Theorem~5.1]{Can}.

\begin{theorem}\label{th:Bi.2.1E}
Under the assumptions of Theorem~\ref{th:Bi.2.1},
let $X$ be equipped with a linear isometric action of a compact Lie group $G$ so that
each $f_\lambda$ is invariant under the $G$-action. Suppose also that
 the local minimums (maximums) in assumption c) of Theorem~\ref{th:Bi.2.1} are strict,
 and that $u=0$ is an isolated critical point of $f_{\lambda^\ast}$.
Then when the Lie group  $G$ is equal to $\mathbb{Z}_2=\{{\rm id}, -{\rm id}\}$ (resp. $S^1$
without fixed points except $0$,
    % with the fixed point set  $X^{S^1}=\{0\}$,
    which implies $\dim X$ to be an even more than one),
               then   there exist left and right  neighborhoods $\Lambda^-$ and $\Lambda^+$ of $\lambda^\ast$ in $\mathbb{R}$
and integers $n^+, n^-\ge 0$, such that $n^++n^-\ge\dim X$ (resp. $\frac{1}{2}\dim X$),
and that for $\lambda\in\Lambda^-\setminus\{\lambda^\ast\}$ (resp. $\lambda\in\Lambda^+\setminus\{\lambda^\ast\}$)
$f_\lambda$ has at least $n^-$ (resp. $n^+$) distinct critical
$G$-orbits different from $0$, which converge to
 $0$ as $\lambda\to\lambda^\ast$.
\end{theorem}

This is essentially contained in the proofs of  \cite{Rab, FaRa1, FaRa2}.
For completeness we will give most of its proof detail.

\begin{proof}[\it Proof of Theorem~\ref{th:Bi.2.1E}]
By assumptions we have either
\begin{eqnarray}\label{e:fBi.1}
0\in X\;\hbox{is a strict local}\left\{
\begin{array}{ll}
\hbox{minimizer of}\;f_{\lambda},&\quad
\forall \lambda\in [\lambda^\ast-\delta, \lambda^\ast),\\
\hbox{maximizer of}\;f_{\lambda},&\quad
\forall \lambda\in (\lambda^\ast, \lambda^\ast+\delta].
\end{array}\right.
\end{eqnarray}
or
\begin{eqnarray}\label{e:fBi.2}
0\in X\;\hbox{is a strict local}\left\{
\begin{array}{ll}
\hbox{maximizer of}\;f_{\lambda},&\quad
\forall \lambda\in [\lambda^\ast-\delta, \lambda^\ast),\\
\hbox{minimizer of}\;f_{\lambda},&\quad
\forall \lambda\in (\lambda^\ast, \lambda^\ast+\delta].
\end{array}\right.
\end{eqnarray}
By the assumption a) of Theorem~\ref{th:Bi.2.1}, replacing
$f_\lambda$ by $f_\lambda-f_\lambda(0)$ if necessary, we now and henceforth assume $f_\lambda(0)=0\;\forall\lambda$.
Since $u=0$ is an isolated critical point of $f_{\lambda^\ast}$, by \cite[p.136]{Ja} there exist mutually disjoint:\\
{\bf Case 1}. $0\in X$ is a local minimizer of $f_{\lambda^\ast}$;\\
{\bf Case 2}. $0\in X$ is a proper local maximizer of $f_{\lambda^\ast}$, i.e., it is a  local maximizer of $f_{\lambda^\ast}$
and $0$ belongs to the closure of $\{f_{\lambda^\ast}<0\}$;\\
{\bf Case 3}. $0\in X$ is a saddle point of $f_{\lambda^\ast}$, i.e., $f_{\lambda^\ast}$  takes both positive and negative values in every neighborhood of $0$.

Though our proof ideas are following   \cite[Theorem~2.2]{Rab} and \cite[Theorem~1.2]{FaRa1, FaRa2},
some technical improvements are necessary since each $f_{\lambda}$ is only of class $C^1$.

By the assumption b) of Theorem~\ref{th:Bi.2.1}, the function $(\lambda,z)\mapsto
Df_\lambda(z)$ is continuous on
$[\lambda^\ast-\delta, \lambda^\ast+\delta]\times B_X(0,\epsilon)$.
It follows that
\begin{eqnarray*}
R_{\delta,\epsilon}:%&=&\{(\lambda, z)\in [\lambda^\ast-\delta, \lambda^\ast+\delta]\times B_{H^0}(0, \epsilon)\,|\,D\mathcal{L}^\circ_\lambda(z)\ne 0\}\\
&=&\{(\lambda, z)\in (\lambda^\ast-\delta, \lambda^\ast+\delta)\times B_X(0,\epsilon)\,|\,z\in B_X(0,\epsilon)\setminus
K(f_\lambda)\}
\end{eqnarray*}
is an open subset in $[\lambda^\ast-\delta, \lambda^\ast+\delta]\times B_X(0,\epsilon)$, where
$K(f_\lambda)$ denotes the critical set of $f_\lambda$.

Though the local flow
 of $\nabla f_{\lambda}$ must not exist,
fortunately, by suitably modifying the standard constructions in \cite[\S4]{Pa1} (see also
 \cite[Propsitions~5.29, 5.57]{MoMoPa}) we have:

\begin{lemma}\label{lem:Fpseudogradient}
There exists a smooth map $R_{\delta,\epsilon}\to X,\;(\lambda,z)\mapsto\mathscr{V}_\lambda(z)$,
such that for each $\lambda\in(\lambda^\ast-\delta, \lambda^\ast+ \delta)$ the map
$\mathscr{V}_\lambda: B_X(0,\epsilon)\setminus
K(f_\lambda)\to X$
is  $G$-equivariant and satisfies
$$
\|\mathscr{V}_\lambda(z)\|\le 2\|Df_{\lambda}(z)\|\quad\hbox{and}\quad
\langle Df_{\lambda}(z), \mathscr{V}_\lambda(z)\rangle\ge
\|Df_{\lambda}(z)\|^2
$$
for all $z\in B_X(0,\epsilon)\setminus K(f_\lambda)$, i.e., it is pseudo-gradient vector field of
$f_\lambda$ in the sense of \cite{Pa1}.
\end{lemma}

Notice that this is actually true for any compact Lie group $G$.
There exist some variants and generalizations for the notion of the pseudo-gradient vector field in \cite{Pa1},
which are more easily constructed and are still effective in applications. For example,
according to \cite{Guo} $\mathscr{V}_\lambda$ is also called a pseudo-gradient vector field of
$f_\lambda$ even if two inequalities in Lemma~\ref{lem:Fpseudogradient} are replaced by
$$
\|\mathscr{V}_\lambda(z)\|\le \alpha\|Df_{\lambda}(z)\|\quad\hbox{and}\quad
\langle Df_{\lambda}(z), \mathscr{V}_\lambda(z)\rangle\ge
\beta\|Df_{\lambda}(z)\|^2,
$$
where $\alpha$ and $\beta$ are two positive constants.
%For completeness we also give the proof Lemma~\ref{lem:Fpseudogradient}.
% which is postponed after the proof of this theorem.

\begin{proof}[\it Proof of Lemma~\ref{lem:Fpseudogradient}]
Since $(\lambda^\ast-\delta, \lambda^\ast+\delta)\times B_{X}(0, \epsilon)\ni (\lambda,z)\mapsto
Df_\lambda(z)\in X^\ast$ is continuous, for any given
$(\lambda,z)\in R_{\delta,\epsilon}$ we have an open neighborhood
$O_{(z,\lambda)}$ of $z$ in $B_{X}(0, \epsilon)\setminus K(f_\lambda)$,
a positive number $r_{(z,\lambda)}$
with $(\lambda-r_{(z,\lambda)}, \lambda+r_{(z,\lambda)})\subset
(\lambda^\ast-\delta, \lambda^\ast+ \delta)$ and
a  vector $v_{(z,\lambda)}\in X$ such that
for all $(z',\lambda')\in O_{(z,\lambda)}\times (\lambda-r_{(z,\lambda)}, \lambda+r_{(z,\lambda)})$,
$$
\|v_{(z,\lambda)}\|< 2\|Df_{\lambda'}(z')\|\quad\hbox{and}\quad
\langle Df_{\lambda'}(z'), v_{(z,\lambda)}\rangle>
\|Df_{\lambda'}(z')\|^2.
$$
Now all above $O_{(z,\lambda)}\times (\lambda-r_{(z,\lambda)}, \lambda+r_{(z,\lambda)})$
form an open cover $\mathscr{C}$ of $R_{\delta,\epsilon}$,
and the latter admits a $C^\infty$-unit decomposition $\{\eta_\alpha\}_{\alpha\in\Xi}$
subordinate to a locally finite
refinement $\{W_\alpha\}_{\alpha\in\Xi}$ of $\mathscr{C}$.
Since each $W_\alpha$ can be contained in some open subset of form
$O_{(z,\lambda)}\times (\lambda-r_{(z,\lambda)}, \lambda+r_{(z,\lambda)})$, we have
a  vector $v_\alpha\in X$ such that for all $(z',\lambda')\in W_\alpha$,
$$
\|v_\alpha\|< 2\|Df_{\lambda'}(z')\|\quad\hbox{and}\quad
\langle Df_{\lambda'}(z'), v_\alpha\rangle>
\|Df_{\lambda'}(z')\|^2.
$$
Set $\chi(\lambda,z)=\sum_{\alpha\in\Xi}\eta_\alpha(\lambda,z) v_\alpha$. Then $\chi$ is a smooth map
from $R_{\delta,\epsilon}$ to $X$,  and satisfies
$$
\|\chi(z,\lambda)\|<2\|Df_{\lambda}(z)\|\quad\hbox{and}\quad
\langle Df_{\lambda}(z), \chi(z,\lambda)\rangle>
\|Df_{\lambda}(z)\|^2
$$
for all $(z,\lambda)\in R_{\delta,\epsilon}$. Let $d\mu$ denote the right invariant Haar measure on $G$. Define
$$
R_{\delta,\epsilon}\ni (z,\lambda)\mapsto\mathscr{V}_\lambda(z)=\int_G g^{-1}\chi(gz,\lambda)d\mu\in X.
$$
It is easily checked that $\mathscr{V}_\lambda$ satisfies requirements (see the proof \cite[Propsition~5.57]{MoMoPa}).
\end{proof}
%\hfill$\Box$\vspace{2mm}

Consider the ordinary differential equation
\begin{equation}\label{e:fBi.3}
\frac{d\varphi}{ds}=-\mathscr{V}_{\lambda^\ast}(\varphi),\quad\varphi(0,z)=z
\end{equation}
in $B_X(0,\epsilon)$. For each $z\in B_X(0,\epsilon)$ this equation possesses a unique solution $\varphi(t, z)$  whose
 the maximal existence interval is $(\omega_-(z),\omega_+(z))$ with
$-\infty\le\omega_-(z)<0<\omega_+(z)\le+\infty$.
Note that $0\in X$ is an isolated critical point of $f_{\lambda^\ast}$.
There exists  a small neighborhood $\mathcal{B}$ of $0$ in $B_X(0,\epsilon)$ such that
$u=0$ is a unique critical point of $f_{\lambda^\ast}$ siting in it. Put
\begin{eqnarray*}
&&S^+=\{z\in \mathcal{B}\,|\, \omega_+(z)=+\infty\;\&\;\varphi(t,z)\in \mathcal{B},\;\forall \ast t>0\},\\
&&S^-=\{z\in \mathcal{B}\,|\, \omega_-(z)=-\infty\;\&\;\varphi(t,z)\in \mathcal{B},\;\forall \ast t<0\}.
\end{eqnarray*}
Then $S^+\ne\emptyset$ (resp. $S^-\ne\emptyset$) if there are points near $0\in \mathcal{B}$ where $f_{\lambda^\ast}$
 is positive (resp. negative).
 %(See proofs of  \cite[Lemmas~1.15,1.16]{Rab}  and \cite[Lemma~8.28]{FaRa2}).
%  \cite[Lemmas~1.15,1.16]{Rab} (see also \cite[Lemma~2.5]{FaRa1}) and \cite[Lemma~8.28]{FaRa2}
Repeating the constructions in \cite[Lemmas~1.15,1.16]{Rab} and \cite[Lemma~8.28]{FaRa2}
(and replacing $\psi$ therein by $\varphi$) we can obtain:

\begin{lemma}\label{lem:FBi.3.3-}
There is a constant $\gamma>0$ and a $G$-invariant open neighborhood $\mathcal{Q}$ of $0$ in $X$
with compact closure  $\overline{\mathcal{Q}}$ contained in $\mathcal{B}$ such that
\begin{description}
\item[(i)] if $z\in \mathcal{Q}$ then
$|f_{\lambda^\ast}(z)|<\gamma$;
\item[(ii)] if $(z,t)\in \mathcal{Q}\times\mathbb{R}$ satisfies
$|f_{\lambda^\ast}(\varphi(t,z))|<\gamma$, then $\varphi(t,z)\in\mathcal{Q}$;
\item[(iii)] if $z\in \partial\mathcal{Q}$, then either
$|f_{\lambda^\ast}(z)|=\gamma$ or
$\varphi(t,z)\in \partial\mathcal{Q}$ for all $t$ satisfying
$|f_{\lambda^\ast}(\varphi(t,z))|\le\gamma$.
\end{description}
\end{lemma}

Actually, for sufficiently small $\rho>0$ with $\bar{B}_X(0,\rho)\subset\mathcal{B}$,
 $\mathcal{Q}$ may be chosen as
$$
\{z\in B_X(0,\epsilon)\,|\,-\gamma<f_{\lambda^\ast}(z)<\gamma\}\cap\{\varphi(t,z)\,|\,
z\in B_X(0,\rho),\;\omega_-(z)<t<\omega_+(z)\}.
$$
Clearly, when $0\in X$ is a strict local maximizer (resp. minimizer) of $f_{\lambda^\ast}$,
$\overline{\mathcal{Q}}\cap(f_{\lambda^\ast})^{-1}(\gamma)$ (resp. $\overline{\mathcal{Q}}\cap(f_{\lambda^\ast})^{-1}(-\gamma)$)
may be empty. Moreover, for any $z\in\overline{\mathcal{Q}}$, if $\varphi(t,z)\notin \overline{\mathcal{Q}}$
for some $t=t(z)>0$, then $f_{\lambda^\ast}(\varphi(t,z))<-\gamma$.

In the following, if $G=\mathbb{Z}_2==\{{\rm id}, -{\rm id}\}$  let $i_G$ denote the genus in \cite{Rab},
and if $G=S^1$ without fixed points except $0$
let $i_G$ denote the index ${\rm Index}^\ast_{\mathbb{C}}$ in \cite[\S7]{FaRa2}.
Let $T^\ast=S^\ast\cap\partial\overline{\mathcal{Q}}$, $\ast=+,-$.

\begin{lemma}\label{lem:FBi.3.4}
Both $T^+$ and $T^-$ are $G$-invariant compact subsets of $\partial\overline{\mathcal{Q}}$, and also satisfy
\begin{description}
\item[(i)] $\min\{f_\lambda(z)\,|\,z\in T^+\}>0$ and
$\max\{f_\lambda(z)\,|\,z\in T^-\}<0$;
\item[(ii)] $i_{\mathbb{Z}_2}(T^+)+ i_{\mathbb{Z}_2}(T^-)\ge\dim X$ and
$i_{S^1}(T^+)+ i_{S^1}(T^-)\ge\frac{1}{2}\dim X$.
\end{description}
\end{lemma}
By the definitions of $T^+$ and $T^-$, Lemma~\ref{lem:FBi.3.4}(i) is apparent.
Two inequalities in Lemma~\ref{lem:FBi.3.4}(ii) are \cite[Lemma~2.11]{FaRa1} and \cite[Theorem~8.30]{FaRa2}, respectively.

Define a piecewise-linear  function $\alpha:[0, \infty)\to [0, \infty)$
by $\alpha(t)=0$ for $t\le 1$, $\alpha(t)=-t+2$ for $t\in [1,2]$, and $\alpha(t)=0$ for $t\ge 2$.
For $\rho>0$ let $\alpha_\rho:[0, \infty)\to [0, \infty)$ be defined by
$\alpha_\rho(t)=\alpha(t/\rho)$.
Let $d(z, \partial\overline{\mathcal{Q}})=\inf\{\|z-u\|\,|\, u\in \partial\overline{\mathcal{Q}}\}$ for $z\in \overline{\mathcal{Q}}$.
For each $\lambda\in (\lambda^\ast-\delta, \lambda^\ast+ \delta)$, define
$$
\Xi_{\rho,\lambda}:\overline{\mathcal{Q}}\to H^0,\;z\mapsto\alpha_\rho(d(z,\partial\overline{\mathcal{Q}}))\mathscr{V}_{\lambda^\ast}(z)+(1-\alpha_\rho(d(z, \partial\overline{\mathcal{Q}})))
\mathscr{V}_{\lambda}(z),
$$
which is $G$-equivariant and Lipschitz continuous.
Since $\inf\{\|Df_{\lambda^\ast}(z)\|\,|\, z\in\partial\overline{\mathcal{Q}}\}>0$ we may choose
a small $\rho>0$ so that
\begin{eqnarray}\label{e:fBi.4}
&&\inf\{\|Df_{\lambda^\ast}(z)\|\,|\, z\in\overline{\mathcal{Q}}\;\&\; d(z, \partial\overline{\mathcal{Q}}))\le 2\rho\}>0,\\
&&\inf\{\langle Df_{\lambda^\ast}(z),\mathscr{V}_{\lambda^\ast}(z)\rangle\,|\, z\in\overline{\mathcal{Q}}\;\&\; d(z, \partial\overline{\mathcal{Q}}))\le 2\rho\}>0.\nonumber
\end{eqnarray}
Moreover, the function $(\lambda,z)\mapsto Df_\lambda(z)$ is uniform continuous on
$[\lambda^\ast-\delta, \lambda^\ast+\delta]\times \overline{\mathcal{Q}}$ by the assumption b) of Theorem~\ref{th:Bi.2.1}.
Then this and Lemma~\ref{lem:Fpseudogradient} imply that there exists a small $0<\delta_0<\delta$ such that
\begin{eqnarray}\label{e:fBi.5}
&&\frac{1}{2}\|Df_{\lambda^\ast}(z)\|\le \|Df_{\lambda}(z)\|\le 2\|Df_{\lambda^\ast}(z)\|,\\
&&\frac{1}{2}\langle Df_{\lambda^\ast}(z),\mathscr{V}_{\lambda^\ast}(z)\rangle\le \langle Df_{\lambda}(z),\mathscr{V}_{\lambda^\ast}(z)\rangle\le 2
\langle Df_{\lambda^\ast}(z),\mathscr{V}_{\lambda^\ast}(z)\rangle\label{e:fBi.6}
\end{eqnarray}
for all $z\in\overline{\mathcal{Q}}$ satisfying $d(z, \partial\overline{\mathcal{Q}}))\le 2\rho$ and for
all $\lambda\in [\lambda^\ast-\delta_0, \lambda^\ast+\delta_0]$.
(Clearly, (\ref{e:fBi.4}) and (\ref{e:fBi.5}) show that any critical point of $f_\lambda$ sitting in
$\overline{\mathcal{Q}}$ must be in
$\{ z\in\overline{\mathcal{Q}}\,|\, d(z, \partial\overline{\mathcal{Q}}))>2\rho\}$.)
By (\ref{e:fBi.5}) and Lemma~\ref{lem:Fpseudogradient} we deduce
\begin{eqnarray}\label{e:fBi.7}
\|\Xi_{\rho,\lambda}(z)\|&\le& \alpha_\rho(d(z,\partial\overline{\mathcal{Q}}))\|\mathscr{V}_{\lambda^\ast}(z)\|+
(1-\alpha_\rho(d(z, \partial\overline{\mathcal{Q}})))
\|\mathscr{V}_{\lambda}(z)\|\nonumber\\
&\le&2\alpha_\rho(d(z,\partial\overline{\mathcal{Q}}))\|Df_{\lambda^\ast}(z)\|+
2(1-\alpha_\rho(d(z, \partial\overline{\mathcal{Q}})))
\|Df_{\lambda}(z)\|\nonumber\\
&\le&  4\|Df_{\lambda}(z)\|,\quad\forall z\in \overline{\mathcal{Q}}\setminus K(f_\lambda),\quad
\forall \lambda\in [\lambda^\ast-\delta_0, \lambda^\ast+\delta_0].
\end{eqnarray}
Similarity, (\ref{e:fBi.5})-(\ref{e:fBi.6}) and Lemma~\ref{lem:Fpseudogradient} produce
\begin{eqnarray*}\label{e:fBi.8}
\langle Df_{\lambda}(z), \Xi_{\rho,\lambda}(z)\rangle
&=& \alpha_\rho(d(z,\partial\overline{\mathcal{Q}}))
\langle Df_{\lambda}(z),\mathscr{V}_{\lambda^\ast}(z)\rangle
+(1-\alpha_\rho(d(z, \partial\overline{\mathcal{Q}})))
\langle Df_{\lambda}(z), \mathscr{V}_{\lambda}(z)\rangle\\
&\ge&\frac{1}{2}\alpha_\rho(d(z,\partial\overline{\mathcal{Q}}))\|Df_{\lambda^\ast}(z)\|^2+
(1-\alpha_\rho(d(z, \partial\overline{\mathcal{Q}})))
\|Df_{\lambda}(z)\|^2\\
&\ge&\frac{1}{8}\alpha_\rho(d(z,\partial\overline{\mathcal{Q}}))\|Df_{\lambda}(z)\|^2+
(1-\alpha_\rho(d(z, \partial\overline{\mathcal{Q}})))
\|Df_{\lambda}(z)\|^2\\
&\ge&\frac{1}{8}\|Df_{\lambda}(z)\|^2,\quad\forall z\in \overline{\mathcal{Q}}\setminus K(f_\lambda),\quad
\forall \lambda\in [\lambda^\ast-\delta_0, \lambda^\ast+\delta_0].
\end{eqnarray*}
Hence  $\Xi_{\rho,\lambda}$ is a locally Lipschitz continuous
pseudo-gradient vector field of $f_{\lambda}(z)$ in the sense of
\cite{Guo}.

Replacing $\hat{V}$ in the proofs of \cite[Lemma~1.19]{Rab}  and \cite[Lemma~8.65]{FaRa2} by
$\Xi_{\rho,\lambda}$ and suitably modifying the proof of  \cite[Theorem~4]{Cl74} or \cite[Theorem~1.9]{Rab74}
we may obtain:

\begin{lemma}\label{lem:FBi.3.3}
Let $\gamma$ and $\mathcal{Q}$ be as in Lemma~\ref{lem:FBi.3.3-}.
For each $\lambda\in [\lambda^\ast-\delta_0, \lambda^\ast+ \delta_0]$
the restriction of $f_{\lambda}$ to $\overline{\mathcal{Q}}$
has the $G$-deformation property, i.e.,
for every $c\in\mathbb{R}$
and every $\hat{\tau}>0$, every $G$-neighborhood $U$ of
$K_{\lambda,c}:=K(f_\lambda)\cap\{z\in \overline{\mathcal{Q}}\,|\,
f_\lambda(z)\le c\}$ there exists an $\tau\in (0,\hat{\tau})$ and
a $G$ equivariant homotopy $\eta:[0, 1]\times \overline{\mathcal{Q}}\to \overline{\mathcal{Q}}$
with the following properties:
\begin{description}
\item[(i)] if $A\subset \overline{\mathcal{Q}}$ is closed and $G$-invariant, so is $\eta(s,A)$ for any $s\in [0,1]$;
\item[(ii)]  $\eta(s,z)=z$ if $z\in \overline{\mathcal{Q}}\setminus (f_\lambda)^{-1}[c-\hat{\tau}, c+\hat{\tau}]$;
\item[(iii)] $\eta(s,\cdot)$ is homeomorphism of $\overline{\mathcal{Q}}$
to $\overline{\mathcal{Q}}$ for each $s\in [0,1]$;
\item[(iv)] $\eta(1, A_{\lambda,c+\tau}\setminus U)\subset A_{\lambda,c-\tau}$,
 where $A_{\lambda, d}:=\{z\in \overline{\mathcal{Q}}\,|\, f_\lambda(z)\le d\}$ for $d\in\mathbb{R}$;
\item[(v)] if $K_{\lambda,c}=\emptyset$, $\eta(1, A_{\lambda,c+\tau})\subset A_{\lambda,c-\tau}$.
\end{description}
\end{lemma}

Let
\begin{eqnarray*}
&&\mathscr{E}(\overline{\mathcal{Q}})=\{A\subset \overline{\mathcal{Q}}\,|\,A\;\hbox{is closed and $G$-invariant}\},\\
&&\mathscr{E}(\overline{\mathcal{Q}})_j=\{A\in \mathcal{E}(\overline{\mathcal{Q}})\,|\, i_G(A)\le j\},\;j=1,2,\cdots,\\
&&\mathscr{F}=\{\chi\in C(\overline{\mathcal{Q}}, \overline{\mathcal{Q}})\,|\,\hbox{$\chi$
is $G$-equivariant, one to one, and $\chi(x)=x\,\forall x\in T^-$}\},\\
&&\mathbb{R}^-\cdot K=\{\varphi(s,x)\,|\, -\infty<s<0,\;x\in K\}\quad\hbox{for}\;K\subset T^-.
\end{eqnarray*}

Suppose $i_G(T^-)=k>0$.
For $1\le j\le k$ define
\begin{eqnarray*}
&&G_{j}=\{\chi(\mathbb{R}^-\cdot K )\,|\,\chi\in\mathscr{F}, K\subset T^-
\;\hbox{belongs to $\mathscr{E}(\overline{\mathcal{Q}})$ and $i_G(K)\ge j$}\},\\
&&\Gamma_{j}=\{\overline{A\setminus B}\,|\,\hbox{for some integer $q\in [j, i_G(T^-)]$},\;A\in G_{q}, B\in \mathscr{E}(\overline{\mathcal{Q}})_{q-j}\}.
\end{eqnarray*}
As in the proofs of \cite[Lemma~2.12]{FaRa1} and \cite[Lemma~8.55]{FaRa2} there hold:
\begin{description}
\item[1)] $\Gamma_{j+1}\subset\Gamma_{j},\;1\le j\le i_G(T^-)-1$.
\item[2)] If $\chi\in\mathscr{F}$ and $A\in\Gamma_{j}$, then $\chi(A)\in\Gamma_{j}$.
\item[3)] If $A\in\Gamma_{j}$ and $B\in \mathscr{E}(\overline{\mathcal{Q}})_s$ with $s<j$, then $\overline{A\setminus B}\in \Gamma_{j-s}$.
\end{description}

Let us define %(cf.  \cite[(2.13)]{FaRa1} and \cite[(8.56)]{FaRa2})
$$
c_{\lambda,j}=\inf_{A\in\Gamma_{j}}\max_{z\in A}f_{\lambda}(z),\quad j=1,\cdots,k.
$$
Then the above property 1) implies $c_{\lambda,1}\le c_{\lambda,2}\le\cdots\le c_{\lambda,k}$.
%Let $c_j$ be defined by , but $\bar{Q}$ and $g(\lambda,v)$
%are replaced by $\overline{\mathscr{Q}}$ and $\mathcal{L}^\circ_\lambda(z)$,
%respectively. We can modify the proof of (i) on the page 54 of \cite{FaRa1} as follows:

\textsf{Firstly, let us assume that (\ref{e:fBi.1}) is true.}
 For each $\lambda\in [\lambda^\ast-\delta, \lambda^\ast)$, $0\in X$ is always a local (strict) minimizer  of $f_{\lambda}$. Therefore for  sufficiently small $\rho_\lambda>0$, we have
$f_{\lambda}(z)>0$ for any $0<\|z\|\le\rho_\lambda$
and so
$$
c_{\lambda,1}\ge \min_{\|z\|=\rho_\lambda}f_{\lambda}(z)>0.
$$
Using Lemma~\ref{lem:FBi.3.3} and
repeating other arguments in \cite[Theorem~2.9]{FaRa1} and \cite[Lemmas~8.67, ~8.72]{FaRa2} we obtain:
\begin{description}
\item[(a)] each $c_{\lambda,j}$
is a critical value of $f_{\lambda}$ with a corresponding critical point in $\mathscr{Q}$;
\item[(b)] if $c_{\lambda,j+1}=\cdots=c_{\lambda,j+p}=c$ then $i_G(K_{\lambda,c})\ge p$,
where $K_{\lambda,c}$ is as in Lemma~\ref{lem:FBi.3.3};
\item[(c)] all critical points corresponding to $c_{\lambda,j}$, $1\le j\le i_G(T^-)$, converge to $z=0$ as $\lambda\to\lambda^\ast-$.
\end{description}
$G$ is equal to $\mathbb{Z}_2=\{{\rm id}, -{\rm id}\}$ (resp. $S^1$

Moreover, when $p>1$,  $K_{\lambda,c}$ contains infinitely many distinct $G$-orbits. (In fact, if $G=\mathbb{Z}_2=\{{\rm id}, -{\rm id}\}$
 this follows from  \cite[Lemma~2.8]{FaRa1}.
If $G=S^1$ without fixed points except $0$,
since  $0\notin K_{\lambda,c}$ and  the isotropy group $S^1_x$ of the $S^1$-action at any $x\in X$ is either finite or $S^1$,
we deduce that the induced $S^1$-action on $K_{\lambda,c}$ has only finite isotropy groups and thus that
 $K_{\lambda,c}/S^1$ is an infinite set by \cite[Remark~6.16]{FaRa2}.)
These lead to

%\noindent{\bf Case $G=\mathbb{Z}_2$}.\quad Suppose $i_{\mathbb{Z}_2}(T^-_\lambda)=k>0$.

\begin{claim}\label{cl:F4.5}
If $\lambda\in [\lambda^\ast-\delta, \lambda^\ast)$ is close to $\lambda^\ast$,
$f_{\lambda}$ has at least $k$ distinct
 nontrivial critical $G$-orbits, which also converge to $0$ as $\lambda\to\lambda^\ast$.
\end{claim}

%Next, assume that (\ref{e:Bi.2.19}) is true.
 For every $\lambda\in (\lambda^\ast, \lambda^\ast+\delta]$,
$0\in X$ is a local maximizer of $f_{\lambda}$, by considering
$-f_{\lambda}$ the same reason yields:

\begin{claim}\label{cl:F4.6}
If $i_{G}(T^+)=l>0$, for every $\lambda\in (\lambda^\ast, \lambda^\ast+\delta]$  close to $\lambda^\ast$,
$f_{\lambda}$ has at least $l$ distinct  nontrivial
critical $G$-orbits converging to $0$ as $\lambda\to\lambda^\ast$.
\end{claim}
These two claims together yield the  result of Theorem~\ref{th:Bi.2.1E}.

\textsf{Next, if (\ref{e:fBi.2}) holds true}, the similar arguments lead to:

\begin{claim}\label{cl:F4.7}
if $\lambda\in (\lambda^\ast, \lambda^\ast+\delta]$ (resp.
$\lambda\in [\lambda^\ast-\delta, \lambda^\ast)$) is close to $\lambda^\ast$,
$f_{\lambda}$ has at least $l$ (resp. $k$) distinct nontrivial
critical $G$-orbits, which also converge to $0$ as $\lambda\to\lambda^\ast$.
\end{claim}
So the expected result is still obtained.
\end{proof}

%\subsection{Bifurcations starting at a trivial critical orbit}\label{sec:B.3.1}

%\noindent{\bf 4.2}.\quad {\bf Bifurcations starting at a trivial critical orbit}.
%The following is a direct generalization of Fadell--Rabinowitz theorems
% \cite{FaRa1, FaRa2}.

%Our first result is a generalization of  \cite[Theorem~2.2]{Rab} and \cite[Theorem~1.2]{FaRa1}.

The following  may be viewed as a generalization of
Fadell--Rabinowitz theorems in \cite{FaRa1, FaRa2}.

\begin{theorem}\label{th:Bi.3.2}
Under the assumptions of Theorem~\ref{th:Bi.2.4}, let $H^0_{\lambda^\ast}$ be the eigenspace of
(\ref{e:Bi.2.7.4}) associated with $\lambda^\ast$. Let $H$ be equipped with
an orthogonal action of  a compact Lie group $G$ under which
$U$, $\mathcal{L}$ and $\widehat{\mathcal{L}}$  are $G$-invariant.
  If the Lie group  $G$ is equal to $\mathbb{Z}_2=\{{\rm id}, -{\rm id}\}$ (resp. $S^1$
     and the fixed point set of the induced $S^1$-action on $H^0_{\lambda^\ast}$,
    $(H^0_{\lambda^\ast})^{S^1}=\{x\in H^0\,|\, g\cdot x=x\;\forall g\in S^1\}$, is $\{0\}$,
    which implies $\dim H^0_{\lambda^\ast}$ to be an even more than one),
       % the unit sphere in $H^0$ is not a $G$-orbit),
        then  one of the following alternatives holds:
\begin{description}
\item[(i)] $(\lambda^\ast, 0)$ is not an isolated solution of (\ref{e:Bi.2.7.3}) in
 $\{\lambda^\ast\}\times U$;

\item[(ii)] there exist left and right  neighborhoods $\Lambda^-$ and $\Lambda^+$ of $\lambda^\ast$ in $\mathbb{R}$
and integers $n^+, n^-\ge 0$, such that $n^++n^-\ge\dim H^0_{\lambda^\ast}$ (resp. $\frac{1}{2}\dim H^0_{\lambda^\ast}$),
and that for $\lambda\in\Lambda^-\setminus\{\lambda^\ast\}$ (resp. $\lambda\in\Lambda^+\setminus\{\lambda^\ast\}$)
(\ref{e:Bi.2.7.3}) has at least $n^-$ (resp. $n^+$) distinct critical
$G$-orbits different from $0$, which converge to
 $0$ as $\lambda\to\lambda^\ast$.
\end{description}
In particular,  $(\lambda^\ast, 0)\in\mathbb{R}\times U$ is a bifurcation point  for the equation
(\ref{e:Bi.2.7.3}). %Moreover,  if ${\lambda}^\ast=0$,
%it suffices to require that $\widehat{\mathcal{L}}\in C^1(U,\mathbb{R})$ satisfies
% Hypothesis~\ref{hyp:1.1} with $X=H$.
\end{theorem}
\begin{proof}
The first claim follows from Theorem~\ref{th:Bi.2.4}.
For others, by the arguments above Step 1 in the proof of Theorem~\ref{th:Bi.2.4}
the problem is reduced to finding the $G$-critical orbits of $\mathcal{L}^\circ_\lambda$
near $0\in H^0_{\lambda^\ast}$ for $\lambda$ near $\lambda^\ast$,
where $\mathcal{L}^\circ_\lambda: B_{H}(0, \epsilon)\cap H^0_{\lambda^\ast}\to\mathbb{R}$ is given by (\ref{e:Bi.2.15}).
Suppose that (i) does not hold.  By shrinking $\delta>0$, we obtain
that  for each  $\lambda\in [\lambda^\ast-\delta, \lambda^\ast+\delta]$, $0\in H$ is an isolated critical point of
$\mathcal{L}_\lambda$ and thus $0\in H^0_{\lambda^\ast}$ is an isolated critical point of
$\mathcal{L}^\circ_\lambda$. Then we have
(\ref{e:Bi.2.18}) and (\ref{e:Bi.2.19}), and so (ii) follows from
Theorem~\ref{th:Bi.2.1E}.
 \end{proof}

By Corollaries~\ref{cor:Bi.2.4.1} and \ref{cor:Bi.2.4.2} we get

\begin{corollary}\label{cor:Bi.3.2.1}
If ``Under the assumptions of Theorem~\ref{th:Bi.2.4}" in Theorem~\ref{th:Bi.3.2}
is replaced by ``Under the assumptions of one of Corollaries~\ref{cor:Bi.2.4.1} and \ref{cor:Bi.2.4.2}",
then the conclusions of Theorem~\ref{th:Bi.3.2} hold true.
\end{corollary}

%\begin{remark}\label{rm:Bi.3.2.2}
%{\rm As in Remark~\ref{rm:Bi.2.4.3}
%we may have a more general versions of Theorem~\ref{th:Bi.3.2}
% under the assumptions of Theorem~\ref{th:Bi.1.1} and some additional conditions.
%% if $\mathcal{L}$, $\widehat{\mathcal{L}}$, $H^0_{\lambda^\ast}$ and $\lambda^\ast$ in Theorem~\ref{th:Bi.3.2}
%% are replaced by $\mathcal{F}_\lambda$, $0$, $H^0:={\rm Ker}(B_0(0))$ and $0$, respectively.
% }
%\end{remark}

%Similarly, we may get analogue  of the above two results in the setting of \cite{Lu1,Lu2}.

Now we consider generalizations of
Fadell--Rabinowitz theorems in \cite{FaRa1, FaRa2} in the setting of \cite{Lu1,Lu2}.
Combing Theorem~\ref{th:Bi.2.1E} with Theorem~\ref{th:Bif.2.2.0} can naturally lead to one.
Instead of this we adopt another way.
Recall that Fadell--Rabinowitz bifurcation theorems \cite{FaRa1, FaRa2}
 were generalized  to the case of arbitrary compact Lie groups by
Bartsch and Clapp \cite{BaCl}, and Bartsch \cite{Ba1}.
%However, their methods require the finite dimension reduction functional
%as in (\ref{e:Bi.2.15}) to be of class $C^2$.
Bartsch and Clapp \cite[\S4]{BaCl} proved the following theorem
(according to our notations).

\begin{theorem}\label{th:Bif.2.2.2+}
 Let $Z$ be a finite dimensional Hilbert space equipped with a linear isometric action of a compact Lie group $G$
 with $Z^G=\{0\}$, let $\delta>0$, $\epsilon>0$, $\lambda^\ast\in\mathbb{R}$ and
for every $\lambda\in [\lambda^\ast-\delta, \lambda^\ast+\delta]$ let
$f_\lambda:B_Z(0,\epsilon)\to\mathbb{R}$ be a function of class $C^2$.
Assume that
\begin{description}
\item[(a)] the functions $\{(\lambda,u)\to f_\lambda(u)\}$ and
$\{(\lambda,u)\to f'_\lambda(u)\}$  are continuous on
$[\lambda^\ast-\delta, \lambda^\ast+\delta]\times B_Z(0,\epsilon)$;
\item[(b)] $u=0$ is a critical point of each $f_\lambda$, $L_{\lambda}:=D(\nabla f_{\lambda})(0)\in\mathscr{L}_s(Z)$ is an isomorphism
for each $\lambda\in (\lambda^\ast-\delta, \lambda^\ast+\delta)\setminus\{\lambda^\ast\}$, and  $L_{\lambda^\ast}=0$;
\item[(c)] the eigenspaces $Z_\lambda$ and $Z_\mu$ belonging to
  $\sigma(L_{\lambda})\cap\mathbb{R}^-$ and $\sigma(L_{\mu})\cap\mathbb{R}^-$,
  respectively, are isomorphic if $(\lambda-\lambda^\ast)(\mu-\lambda^\ast)>0$ for any
  $\lambda,\mu\in (\lambda^\ast-\delta, \lambda^\ast+\delta)\setminus\{\lambda^\ast\}$, and thus
  the number
\begin{equation}\label{e:Bi.3.9.2*}
d:=\ell(SZ)-\min\{\ell(SZ_{\lambda})+\ell(SZ_{\mu}^\bot),
 \ell(SZ_{\lambda}^\bot)+\ell(SZ_{\mu})\}
\end{equation}
  is independent of $\lambda\in (\lambda^\ast-\delta, \lambda^\ast)$ and $\mu\in(\lambda^\ast, \lambda^\ast+\delta)$,
  where $Z_\lambda^\bot$ is the orthogonal complementary of $Z_\lambda$ in $Z$ and
  $\ell$ denote the $(\mathscr{G}, h^\ast)$-length  used in \cite{BaCl}.
\end{description}
Then one at least of the following assertions holds if $d>0$:
\begin{description}
\item[(i)] $u=0$ is not an isolated critical point of $f_{\lambda^\ast}$;
\item[(ii)] there exist left and right  neighborhoods $\Lambda^-$ and $\Lambda^+$ of $\lambda^\ast$ in $\mathbb{R}$
and integers $n^+, n^-\ge 0$, such that $n^++n^-\ge d$
and for $\lambda\in\Lambda^-\setminus\{\lambda^\ast\}$ (resp. $\lambda\in\Lambda^+\setminus\{\lambda^\ast\}$),
$f_\lambda$ has at least $n^-$ (resp. $n^+$) distinct critical
$G$-orbits different from $0$, which converge to
 $0$ as $\lambda\to\lambda^\ast$.
\end{description}
In particular,  $(\lambda^\ast, 0)\in [\lambda^\ast-\delta, \lambda^\ast+\delta]\times B_Z(0,\epsilon)$
is a bifurcation point of $f'_\lambda(u)=0$.
\end{theorem}

 \begin{theorem}\label{th:Bif.2.2.4-}
Under the assumptions of Theorem~\ref{th:Bif.2.2.0}
suppose that  a compact Lie group $G$  acts on $H$ orthogonally, which induces  $C^1$ isometric actions on $X$,
 and that both $U$ and $\mathcal{L}_\lambda$ are $G$-invariant (and hence $H^0_\lambda$, $H^\pm_\lambda$
are $G$-invariant subspaces).
If the fixed point set of the induced $G$-action on $H^0_{\lambda^\ast}$ is $\{0\}$ then
 one of the following alternatives occurs:
\begin{description}
\item[(i)] $(\lambda^\ast,0)$ is not an isolated solution  in  $\{\lambda^\ast\}\times U^X$ of the equation (\ref{e:Bif.2.2.1*});
\item[(ii)] there exist left and right  neighborhoods $\Lambda^-$ and $\Lambda^+$ of $\lambda^\ast$ in $\mathbb{R}$
and integers $n^+, n^-\ge 0$, such that $n^++n^-\ge \ell(SH^0_{\lambda^\ast})$
and for $\lambda\in\Lambda^-\setminus\{\lambda^\ast\}$ (resp. $\lambda\in\Lambda^+\setminus\{\lambda^\ast\}$),
$\mathcal{L}_\lambda$ has at least $n^-$ (resp. $n^+$) distinct critical
$G$-orbits different from $0$, which converge to
 $0$ as $\lambda\to\lambda^\ast$.
\end{description}
In particular,  $(\lambda^\ast, 0)\in [\lambda^\ast-\delta, \lambda^\ast+\delta]\times U^X$
is a bifurcation point of (\ref{e:Bif.2.2.1*}).
\end{theorem}

\begin{proof}
Follow the notations in the proof of Theorem~\ref{th:Bif.2.2.0}. In the present situation,
for each $\lambda\in [\lambda^\ast-\delta, \lambda^\ast+\delta]$,
the maps $\psi(\lambda, \cdot)$  and $\Phi_{\lambda}(\cdot,\cdot)$  are
  $G$-equivariant, and $C^2$ functional $\mathcal{L}^\circ_{\lambda}$ given by (\ref{e:NSpl.2.3}) is $G$-invariant.
As in the proof of Theorem~\ref{th:Bi.2.4} we obtain either
\begin{eqnarray}\label{e:Bif.2.2.3}
0\in H^0_{\lambda^\ast}\;\hbox{is a strict local}\left\{
\begin{array}{ll}
\hbox{minimizer of}\;\mathcal{L}^\circ_{\lambda},&\quad
\forall \lambda\in [\lambda^\ast-\delta, \lambda^\ast),\\
\hbox{maximizer of}\;\mathcal{L}^\circ_{\lambda},&\quad
\forall \lambda\in (\lambda^\ast, \lambda^\ast+\delta]
\end{array}\right.
\end{eqnarray}
or
\begin{eqnarray}\label{e:Bif.2.2.4}
0\in H^0_{\lambda^\ast}\;\hbox{is a strict local}\left\{
\begin{array}{ll}
\hbox{maximizer of}\;\mathcal{L}^\circ_{\lambda},&\quad
\forall \lambda\in [\lambda^\ast-\delta, \lambda^\ast),\\
\hbox{minimizer of}\;\mathcal{L}^\circ_{\lambda},&\quad
\forall \lambda\in (\lambda^\ast, \lambda^\ast+\delta].
\end{array}\right.
\end{eqnarray}
Note that for each $\lambda\in [\lambda^\ast-\delta, \lambda^\ast+\delta]\setminus\{\lambda^\ast\}$,
 $0\in H^0_{\lambda^\ast}$ is a nondegenerate critical point of
the functional $\mathcal{L}^\circ_\lambda$ by Remark~\ref{rm:Spl.2.4}(ii).
Let $L_\lambda:=D(\nabla\mathcal{L}_{\lambda}^\circ)(0)$. Then $L_{\lambda^\ast}=0$ by
(\ref{e:Spli.2.5}). Hence this and (\ref{e:Spli.2.3})-(\ref{e:Spli.2.4}) show that functionals
$f_\lambda:=\mathcal{L}_{\lambda}^\circ$ on $B_Z(0,\epsilon)$ with $Z:=H^0_{\lambda^\ast}$,
 $\lambda\in [\lambda^\ast-\delta, \lambda^\ast+\delta]$,
satisfy the conditions (a)-(b) in Theorem~\ref{th:Bif.2.2.2+}.
(In fact, $\lambda\mapsto L_\lambda\in\mathscr{L}_s(Z)$ is also continuous by (\ref{e:Spli.2.5}).)

It remains to prove that they also satisfy the condition (c) in Theorem~\ref{th:Bif.2.2.2+}.

Firstly, we assume (\ref{e:Bif.2.2.3}) holding. Then it implies that
$L_\lambda$ is positive definite for any $\lambda\in (\lambda^\ast-\delta, \lambda^\ast)$,
and negative definite for any $\lambda\in (\lambda^\ast, \lambda^\ast+\delta)$.
It follows that $\sigma(L_{\lambda})\cap\mathbb{R}^-=\emptyset$ and so $Z_\lambda=\emptyset$
for any $\lambda\in (\lambda^\ast-\delta, \lambda^\ast)$, and
$\sigma(L_{\lambda})\cap\mathbb{R}^-=\sigma(L_{\lambda})$ and so $Z_\lambda=Z$
for any $\lambda\in (\lambda^\ast, \lambda^\ast+\delta)$.
Hence the condition (c) in Theorem~\ref{th:Bif.2.2.2+} is satisfied. Note also that in
the present situation we have
\begin{eqnarray*}
d&=&\ell(SZ)-\min\{\ell(SZ_{\lambda})+\ell(SZ_{\mu}^\bot),
 \ell(SZ_{\lambda}^\bot)+\ell(SZ_{\mu})\}\\
 &=&\ell(SZ)-\min\{0,  \ell(SZ)+\ell(SZ)\}=\ell(SZ)
\end{eqnarray*}
 for any of $\lambda\in (\lambda^\ast-\delta, \lambda^\ast)$ and $\mu\in(\lambda^\ast, \lambda^\ast+\delta)$
 because the $(\mathscr{G}, h^\ast)$-length of a $G$-space,  $\ell(X)$, is equal to zero if and only if $X=\emptyset$ by
  \cite[Proposition~1.5]{BaCl}.

Next, let (\ref{e:Bif.2.2.4}) be satisfied. Then we have
$\sigma(L_{\lambda})\cap\mathbb{R}^-=\sigma(L_{\lambda})$ and so $Z_\lambda=Z$
for any $\lambda\in (\lambda^\ast-\delta, \lambda^\ast)$, and
 $\sigma(L_{\lambda})\cap\mathbb{R}^-=\emptyset$ and so $Z_\lambda=\emptyset$
for any $\lambda\in (\lambda^\ast, \lambda^\ast+\delta)$.
These  lead to
\begin{eqnarray*}
d=\ell(SZ)-\min\{\ell(SZ_{\lambda})+\ell(SZ_{\mu}^\bot),
 \ell(SZ_{\lambda}^\bot)+\ell(SZ_{\mu})\}=\ell(SZ)
\end{eqnarray*}
 for any of $\lambda\in (\lambda^\ast-\delta, \lambda^\ast)$ and $\mu\in(\lambda^\ast, \lambda^\ast+\delta)$.

 Note that we have always $\ell(SZ)>0$ by
 \cite[Proposition~1.5]{BaCl}. Hence Theorem~\ref{th:Bif.2.2.2+} gives the desired conclusions.
%Note that when $\ell(SZ)=0$ the claims are contained in
%a more general form  of Theorems~\ref{th:Bif.2.2.1},~\ref{th:Bif.2.2.2} as Theorem~\ref{th:Bi.2.4}.
 \end{proof}

When $G=\mathbb{Z}_2=\{-id_H, id_H\}$  (resp. $G=S^1$)
we have $d=\dim H^0_{\lambda^\ast}$ (resp. $d=\frac{1}{2}\dim H^0_{\lambda^\ast}$) as showed in
Remark~\ref{rem:Bif.3.4} below.

 \begin{corollary}\label{cor:Bif.3.3}
Under the assumptions of one of Corollaries~\ref{cor:Bif.2.2.1},~\ref{cor:Bif.2.2.2}
suppose that $H$ is equipped with an orthogonal action of  a compact Lie group $G$
  which induces a $C^1$ isometric action on $X$, and that $U$, $\mathcal{L}$ and $\widehat{\mathcal{L}}$ are $G$-invariant.
 Then the conclusions of Theorem~\ref{th:Bif.2.2.4-} hold true.
\end{corollary}

 \begin{remark}\label{rem:Bif.3.4}
{\rm For a finite-dimensional real $G$-module $M$ with $M^G=\{0\}$ and $h^\ast=H^\ast_G$ Borel cohomology it was proved in \cite[Propositions~2.4,2.6]{BaCl}:
\begin{description}
\item[(i)] if $G=(\mathbb{Z}/p)^r$, where $r>0$ and $p$ is a prime, then $(\mathscr{G}, H^\ast_G)$-{\rm length}(SM) is equal
to $\dim M$ for $p=2$, and to $\frac{1}{2}\dim M$ for $p>2$;
\item[(ii)] if $G=(S^1)^r$,  $r>0$, then $(\mathscr{G}, H^\ast_G)$-{\rm length}(SM) is equal
to  $\frac{1}{2}\dim M$;
\item[(iii)] if $G=S^1\times\Gamma$, $\Gamma$ is finite, and such that the fixed point set of $S^1\equiv S^1\times\{e\}$
is trivial,  then $(\mathscr{G}, H^\ast_G)$-{\rm length}(SM) is equal
to  $\frac{1}{2}\dim M$.
\end{description}
Here the coefficients in the Borel cohomology $H^\ast_G$ are $G=(\mathbb{Z}/p)^r$ in (i), and $\mathbb{Q}$ in (ii) and (iii).
}
\end{remark}

Let us point out that it is possible to generalize the above results with methods in  \cite[\S7.5]{Ba1}.
%The detailed arguments are omitted.

As remarked below Corollary~\ref{cor:Bif.2.2.2}, Theorem~\ref{th:Bif.2.2.4-}
and Corollary~\ref{cor:Bif.3.3} are powerful to applications in Lagrange systems and geodesics on Finsler manifolds.

 \begin{remark}\label{rem:Bif.3.5}
{\rm Under the assumptions of any one of Theorems~\ref{th:Bi.2.6},\ref{th:Bi.2.5}
and Corollaries~\ref{cor:Bi.2.6},\ref{cor:Bi.2.7},
if $H$ is equipped with an orthogonal action of  a compact Lie group $G$
 such that $U$ and $\mathcal{L}_\lambda$  are $G$-invariant and that
 the fixed point set of the induced $G$-action on $H^0_{\lambda^\ast}$ is $\{0\}$,
  then the corresponding conclusions with Theorem~\ref{th:Bi.3.2}  (for $G=\mathbb{Z}_2=\{{\rm id}, -{\rm id}\}$ or $S^1$)
 and those with Theorem~\ref{th:Bif.2.2.4-} with $X$ replaced by $H$ hold true.
Compare the latter with \cite[Therrem~3.1]{BaCl} by Bartsch and Clapp.

%In fact, in the above proofs we only need to replace the parameterized  splitting lemmas and Morse-Palais lemmas
%used with the parameterized version of the classical splitting lemma for $C^2$ functionals stated in
%Remark~\ref{rm:Spl.2.5} and the classical Morse-Palais lemma  for $C^2$ functionals.
 }
\end{remark}

\subsection{Bifurcations starting a nontrivial critical orbit}\label{sec:B.3.2}

%\noindent{\bf 4.4}.\quad {\bf Bifurcations starting a nontrivial critical orbit}.

\begin{hypothesis}[\hbox{\cite[Hypothesis~2.20]{Lu7}}]\label{hyp:S.6.2}
{\rm {\bf (i)} Let $G$ be a compact
Lie group, and  $\mathcal{H}$  a $C^3$ Hilbert-Riemannian $G$-space
(that is, ${\cal H}$ is a $C^3$ $G$-Hilbert manifold with a Riemannian
metric $(\!(\cdot,\cdot)\!)$ such that $T{\cal H}$ is a $C^2$ Riemannian $G$-vector bundle, see \cite{Was}).\\
 {\bf (ii)} The $C^1$ functional $\mathcal{ L}:\mathcal{H}\to\mathbb{R}$ is $G$-invariant, the gradient
 $\nabla\mathcal{L}:\mathcal{H}\to T\mathcal{H}$ is G\^ateaux differentiable
 (i.e., under any $C^3$ local chart the functional $\mathcal{L}$
 has a G\^ateaux differentiable gradient map), and $\mathcal{ O}$ is an isolated
 critical orbit which is a $C^3$ critical submanifold with  Morse index $\mu_\mathcal{O}$.}
\end{hypothesis}

Under the above assumptions
%Let $\mathcal{ O}\subset{\cal H}$
%be a compact $C^3$ submanifold without boundary, and
let $\pi:N\mathcal{ O}\to \mathcal{ O}$ denote the normal bundle of it. The
bundle is a $C^2$-Hilbert vector bundle over $\mathcal{ O}$, and can
be considered as a subbundle of $T_\mathcal{ O}{\cal H}$ via the
Riemannian metric $(\!(\cdot, \cdot)\!)$. The metric $(\!(\cdot, \cdot)\!)$
induces a natural $C^2$ orthogonal bundle
projection ${\bf \Pi}:T_{\mathcal{O}}\mathcal{H}\to N\mathcal{O}$. For $\varepsilon>0$,
the so-called normal disk bundle of radius $\varepsilon$ is denoted by
$N\mathcal{ O}(\varepsilon):=\{(x,v)\in N\mathcal{O}\,|\,\|v\|_{x}<\varepsilon\}$.
 If $\varepsilon>0$ is small enough  the exponential map $\exp$ gives a $C^2$-diffeomorphism
 $\digamma$ from  $N\mathcal{ O}(\varepsilon)$ onto an open
neighborhood of $\mathcal{ O}$ in ${\cal H}$, $\mathcal{
N}(\mathcal{ O},\varepsilon)$.
For $x\in\mathcal{ O}$, let  $\mathscr{L}_s(N\mathcal{O}_x)$ denote the space
of those operators $S\in \mathscr{L}(N\mathcal{ O}_x)$ which are self-adjoint
with respect to the inner product $(\!(\cdot, \cdot)\!)_x$, i.e.
$(\!(S_xu, v)\!)_x=(\!(u, S_xv)\!)_x$ for all $u, v\in N\mathcal{
O}_x$. Then we have a $C^1$ vector bundle $\mathscr{L}_s(N\mathcal{ O})\to
\mathcal{O}$ whose fiber at $x\in\mathcal{ O}$ is given by
$\mathscr{L}_s(N\mathcal{ O}_x)$.

\begin{hypothesis}\label{hyp:Bi.3.19}
{\rm Under Hypothesis~\ref{hyp:S.6.2}, let for some $x_0\in\mathcal{ O}$ the pair
$(\mathcal{L}\circ\exp|_{N\mathcal{O}(\varepsilon)_{x_0}},  N\mathcal{O}(\varepsilon)_{x_0})$
satisfy the corresponding conditions with Hypothesis~\ref{hyp:1.1} with $X=H=N\mathcal{O}(\varepsilon)_{x_0}$.
(For this goal we only need require that the pair $(\mathcal{L}\circ\exp_{x_0},  B_{T_{x_0}\mathcal{H}}(0,\varepsilon))$
satisfy the corresponding conditions with Hypothesis~\ref{hyp:1.1} with $X=H=T_{x_0}\mathcal{H}$
by \cite[Lemma~2.8]{Lu7}.) Let $\widehat{\mathcal{L}}\in C^1(\mathcal{H},\mathbb{R})$  be  $G$-invariant, have
a critical orbit $\mathcal{O}$, and also satisfy:
\begin{description}
\item[(i)] the gradient $\nabla(\widehat{\mathcal{L}}\circ\exp|_{B_{T_{x_0}\mathcal{H}}(0,\varepsilon)})$ is G\^ateaux differentiable, and its derivative
at any $u\in B_{T_{x_0}\mathcal{H}}(0,\varepsilon)$,
$$
d^2(\widehat{\mathcal{L}}\circ\exp|_{B_{T_{x_0}\mathcal{H}}(0,\varepsilon)})(u)\in\mathscr{L}_s(T_{x_0}\mathcal{H}),
$$
    is also a compact linear operator.
\item[(ii)]  $B_{T_{x_0}\mathcal{H}}(0,\varepsilon)\to \mathscr{L}_s(T_{x_0}\mathcal{H}),\;u\mapsto d^2(\widehat{\mathcal{L}}\circ\exp|_{B_{T_{x_0}\mathcal{H}}(0,\varepsilon)})(u)$
is continuous at $0\in T_{x_0}\mathcal{H}$.
%
%$\mathcal{G}''$ are continuous at each point $u\in\mathcal{O}$.
(Thus the assumptions on $\mathcal{G}$ assure that the functionals
$\mathcal{L}_{\lambda}:=\mathcal{L}-\lambda\widehat{\mathcal{L}}$,
$\lambda\in\mathbb{R}$, also satisfy the conditions of
\cite[Theorems~2.21 and 2.22]{Lu7}.)
\end{description}}
\end{hypothesis}

Under Hypothesis~\ref{hyp:Bi.3.19}, we say $\mathcal{O}$ to be a {\it bifurcation $G$-orbit
with parameter $\lambda^\ast$} of  the equation
\begin{equation}\label{e:Bi.3.11}
\mathcal{L}'(u)=\lambda\widehat{\mathcal{L}}'(u),\quad u\in \mathcal{H}
\end{equation}
if for any $\varepsilon>0$ and for any neighborhood $\mathscr{U}$ of $\mathcal{O}$ in $\mathcal{H}$
there exists a solution $G$-orbit $\mathcal{O}'\ne \mathcal{O}$ in $\mathscr{U}$ of
(\ref{e:Bi.3.11}) with some $\lambda\in (\lambda^\ast-\varepsilon, \lambda^\ast+\varepsilon)$.
%\cite{BetPS1}
Equivalently, for some (and so any) fixed $x_0\in\mathcal{O}$ there exists a sequence $(\lambda_n, u_n)\subset (\lambda^\ast-\varepsilon, \lambda^\ast+\varepsilon)\times \mathcal{H}$ such that
\begin{equation}\label{e:Bi.3.11.1}
(\lambda_n, u_n)\to(\lambda^\ast, x_0),\quad \mathcal{L}'(u_n)=\lambda_n\mathcal{G}'(u_n)
\quad\hbox{and}\quad u_n\notin \mathcal{O}\quad\forall n.
\end{equation}

For any $x_0\in\mathcal{O}$, $\mathcal{S}_{x_0}:=\exp_{x_0}({N\mathcal{O}(\varepsilon)_{x_0}})$
is a $C^2$ {\bf slice} for the action of $G$ on $\mathcal{H}$
%(cf.\cite{BetPS1}).
in the following sense:
\begin{description}
\item[(i)]  $\mathcal{S}_{x_0}$ is a $C^2$ submanifold of $\mathcal{H}$ containing $x_0$, and the tangent space $T_{x_0}\mathcal{S}_{x_0}\subset T_{x_0}\mathcal{H}$ is a closed complement to $T_{x_0}\mathcal{O}$;
\item[(ii)] $G\cdot\mathcal{S}_{x_0}$  is a neighborhood of the orbit $\mathcal{O}=G\cdot x_0$, i.e., the orbit of every
$x\in\mathcal{H}$ sufficiently close to $x_0$ must intersect $\mathcal{S}_{x_0}$ and $(G\cdot x)\pitchfork\mathcal{S}_{x_0}$;
\item[(iii)] $\mathcal{S}_{x_0}$ is invariant under the isotropy group $G_{x_0}$.
\end{description}
(The $C^1$ $G$-action is sufficient for (ii).)
Clearly, a point $u\in {N\mathcal{O}(\varepsilon)_{x_0}}$ near $0_{x_0}\in {N\mathcal{O}(\varepsilon)_{x_0}}$
is a critical point of $\mathcal{L}_\lambda\circ\exp|_{N\mathcal{O}(\varepsilon)_{x_0}}$
if and only if $x:=\exp_{x_0}(u)$ is a critical point of  $\mathcal{L}_\lambda|_{\mathcal{S}_{x_0}}$.
Since $d\mathcal{L}_\lambda(x)[\xi]=0\;\forall \xi\in T_{x}(G\cdot x)$ and
$T_x\mathcal{H}=T_{x}(G\cdot x)\oplus T_x\mathcal{S}_{x_0}$, we deduce
\begin{equation}\label{e:Bi.3.11.2}
d\mathcal{L}_\lambda(x)=0\quad\Longleftrightarrow\quad d(\mathcal{L}_\lambda\circ\exp|_{N\mathcal{O}(\varepsilon)_{x_0}})(u)=0.
\end{equation}

Note that for any $x_0\in\mathcal{O}$ the orthogonal complementary of $T_{x_0}\mathcal{O}$ in $T_{x_0}\mathcal{H}$, $N\mathcal{O}_{x_0}$,
 is an invariant subspace of
  $$
 \mathcal{L}''_\lambda(x_0):=d^2(\mathcal{L}\circ\exp|_{B_{T_{x_0}\mathcal{H}}(0,\varepsilon)})(0)\quad\forall\lambda\in\mathbb{R}
 $$
 (in particular
 $\mathcal{L}''(x_0):=d^2(\mathcal{L}\circ\exp|_{B_{T_{x_0}\mathcal{H}}(0,\varepsilon)})(0)$ and $\widehat{\mathcal{L}}''(x_0):=d^2(\widehat{\mathcal{L}}\circ\exp|_{B_{T_{x_0}\mathcal{H}}(0,\varepsilon)})(0)$).
 %
% $$
%\hbox{both}\; \mathcal{L}''(x_0):=d^2(\mathcal{L}\circ\exp|_{B_{T_{x_0}\mathcal{H}}(0,\varepsilon)})(0)\quad\hbox{and}\quad \mathcal{G}''(x_0):=d^2(\mathcal{G}\circ\exp|_{B_{T_{x_0}\mathcal{H}}(0,\varepsilon)})(0).
% $$
%
Let $\mathcal{L}''_\lambda(x_0)^\bot$  (resp.  $\mathcal{L}''(x_0)^\bot$,  $\widehat{\mathcal{L}}''(x_0)^\bot$) denote
the restriction self-adjoint operator of $\mathcal{L}''_\lambda(x_0)$ (resp. $\mathcal{L}''(x_0)$, $\widehat{\mathcal{L}}''(x_0)$)
from  $N\mathcal{O}_{x_0}$ to itself. Then $\mathcal{L}''_\lambda(x_0)^\bot=d^2(\mathcal{L}_\lambda\circ\exp|_{N\mathcal{O}(\varepsilon)_{x_0}})(0)$ and
$$
\mathcal{L}''(x_0)^\bot=d^2(\mathcal{L}\circ\exp|_{N\mathcal{O}(\varepsilon)_{x_0}})(0),\quad
%\hbox{and}\quad
\widehat{\mathcal{L}}''(x_0)^\bot=d^2(\widehat{\mathcal{L}}\circ\exp|_{N\mathcal{O}(\varepsilon)_{x_0}})(0).
$$
Applying Corollary~\ref{cor:Bi.2.4.1} to
$%(\mathcal{F}, \widehat{\mathcal{L}}, U)=
(\mathcal{L}\circ\exp|_{N\mathcal{O}(\varepsilon)_{x_0}}, \widehat{\mathcal{L}}\circ\exp|_{N\mathcal{O}(\varepsilon)_{x_0}},
N\mathcal{O}(\varepsilon)_{x_0})$ and using (\ref{e:Bi.3.11.2}) we obtain:

\begin{theorem}\label{th:Bi.3.20.1}
Under Hypothesis~\ref{hyp:Bi.3.19},
suppose that  $\lambda^\ast\in\mathbb{R}$  an isolated eigenvalue of
\begin{equation}\label{e:Bi.3.12}
\mathcal{L}''(x_0)^\bot v-\lambda\widehat{\mathcal{L}}''(x_0)^\bot v=0,\quad v\in N\mathcal{O}_{x_0},
\end{equation}
and  that $\widehat{\mathcal{L}}''(x_0)^\bot$ is either semi-positive or semi-negative.
 Then $\mathcal{O}$ is a  bifurcation $G$-orbit
with parameter $\lambda^\ast$ of  the equation
(\ref{e:Bi.3.11})
 and  one of the following alternatives occurs:
\begin{description}
\item[(i)] $\mathcal{O}$ is not an isolated critical orbit of $\mathcal{L}_{\lambda^\ast}$;

\item[(ii)]  for every $\lambda\in\mathbb{R}$ near $\lambda^\ast$ there is a critical point
$u_\lambda\notin\mathcal{O}$ of $\mathcal{L}_{\lambda}$ converging to $x_0$ as $\lambda\to\lambda^\ast$;

\item[(iii)] there is an one-sided  neighborhood $\Lambda$ of $\lambda^\ast$ such that
for any $\lambda\in\Lambda\setminus\{\lambda^\ast\}$,
$\mathcal{L}_{\lambda}$  has at least two critical points which sit in $\mathcal{S}_{x_0}\setminus\mathcal{O}$ converging to
$x_0$ as $\lambda\to\lambda^\ast$.
\end{description}
%Moreover, if ${\lambda}^\ast=0$, the assumption ``is also a compact linear operator" in
%Hypothesis~\ref{hyp:Bi.3.19}(ii) is unnecessary.
\end{theorem}

Suppose that $\mathcal{L}''(x_0)^\bot$ is invertible, or equivalently
 ${\rm Ker}(\mathcal{L}''(x_0))=T_{x_0}\mathcal{O}$. Then
$0$ is not an eigenvalue  of (\ref{e:Bi.3.12})
and $\lambda\in\mathbb{R}\setminus\{0\}$
is an eigenvalue  of (\ref{e:Bi.3.12}) if and only if $1/\lambda$
is an eigenvalue  of compact linear self-adjoint operator
$$
L_{x_0}:=[\mathcal{L}''(x_0)^\bot ]^{-1}\widehat{\mathcal{L}}''(x_0)^\bot\in\mathscr{L}_s(N\mathcal{O}_{x_0}).
$$
Hence $\sigma(L_{x_0})\setminus\{0\}=\{1/\lambda_k\}_{k=1}^\infty\subset\mathbb{R}$
with $|\lambda_k|\to \infty$, and each $1/\lambda_k$ has finite multiplicity.
Let $N\mathcal{O}_{x_0}^k$ be the eigensubspace corresponding to $1/\lambda_k$ for $k\in\mathbb{N}$.
Then
\begin{eqnarray}\label{e:Bi.3.13}
&&N\mathcal{O}_{x_0}^0:={\rm Ker}(L_{x_0})={\rm Ker}(\widehat{\mathcal{L}}''(x_0)^\bot),\nonumber\\
&&N\mathcal{O}_{x_0}^k={\rm Ker}(I/\lambda_k-L_{x_0})={\rm Ker}(\mathcal{L}''(x_0)^\bot-\lambda_k\widehat{\mathcal{L}}''(x_0)^\bot),\quad
\forall k\in\mathbb{N},
\end{eqnarray}
and $N\mathcal{O}_{x_0}=\oplus^\infty_{k=0}N\mathcal{O}_{x_0}^k$.
Since $T_{x_0}\mathcal{O}\subset{\rm Ker}(\mathcal{L}''(x_0))\cap{\rm Ker}(\widehat{\mathcal{L}}''(x_0))$,
for $\lambda\in\mathbb{R}$, $\mathcal{O}$ is a nondegenerate critical orbit
of $\mathcal{L}_\lambda$ if and only if $\lambda$ is not an eigenvalue  of (\ref{e:Bi.3.12}).
Applying Corollary~\ref{cor:Bi.2.4.2} to
$(\mathcal{L}\circ\exp|_{N\mathcal{O}(\varepsilon)_{x_0}}, \widehat{\mathcal{L}}\circ\exp|_{N\mathcal{O}(\varepsilon)_{x_0}},
N\mathcal{O}(\varepsilon)_{x_0})$ and using (\ref{e:Bi.3.11.2}) we obtain:

%%%%%%%%%%%%%%%%%%%%%%%%%%%%%%%%%%%%%%%%%%%%%%%%%%%%%%%%%%%%%%%%%%%%%%%%%%%%%%%%%%%%%%%%%%%%%%%
%%Moreover, there exists $\delta>0$ such that $\mathcal{O}$ is a nondegenerate critical orbit
%%of $\mathcal{L}_\lambda$ for each $\lambda\in [\lambda^\ast-\delta,\lambda^\ast+\delta]\setminus\{\lambda^\ast\}$.
%%let $\lambda^\ast$ be an eigenvalue  of
%%\begin{equation}\label{e:Bi.3.10}
%%\mathcal{L}''(x_0)v-\lambda\mathcal{G}''(x_0)v=0,\quad v\in T_{x_0}\mathcal{H},
%%\end{equation}
%%(thus $T_{x_0}\mathcal{O}\subset{\rm Ker}(\mathcal{L}''(x_0)-\lambda^\ast\mathcal{G}''(x_0))$),
%% such that $\dim{\rm Ker}(\mathcal{L}''(x_0)-\lambda^\ast\mathcal{G}''(x_0))>\dim \mathcal{O}$.
%%%%%%%%%%%%%%%%%%%%%%%%%%%%%%%%%%%%%%%%%%%%%%%%%%%%%%%%%%%%%%%%%%%%%%%%%%%%%%%%%%%%%%%%%%%%%%%%%%%%%

%By Corollary~\ref{cor:Bi.2.4.2} and (\ref{e:Bi.3.11.2}) we obtain

\begin{theorem}\label{th:Bi.3.20.2}
Under Hypothesis~\ref{hyp:Bi.3.19},
suppose that  $\mathcal{L}''(x_0)^\bot$ is invertible
and that $\lambda^\ast=\lambda_{k_0}$ is an eigenvalue of (\ref{e:Bi.3.12}) as above.
Then the conclusions of Theorem~\ref{th:Bi.3.20.1} hold true if
  one of the following two conditions is also satisfied:
 \begin{description}
\item[(a)] $\mathcal{L}''(x_0)^\bot$ is positive;
 \item[(b)] each $N\mathcal{O}_{x_0}^k$ in (\ref{e:Bi.3.13}) with $L=[\mathcal{L}''(0)]^{-1}\mathcal{G}''(0)$
 is an invariant subspace of $\mathcal{L}''(x_0)^\bot$  (e.g. these are true if $\mathcal{L}''(x_0)^\bot$
 commutes with $\widehat{\mathcal{L}}''(x_0)^\bot$), and $\mathcal{L}''(x_0)^\bot$
 is either positive definite or negative one on $N\mathcal{O}_{x_0}^{k_0}$.
 \end{description}
\end{theorem}

Of course, corresponding to Theorem~\ref{th:Bi.2.4} there is a more general result.
Moreover, if either $G_{x_0}=\mathbb{Z}_2$ acts on $N\mathcal{O}(\varepsilon)_{x_0}$ via the antipodal map
or $G_{x_0}=S^1$ acts orthogonally on $N\mathcal{O}_{x_0}$ and
the fixed point set of the induced $S^1$-action on $N\mathcal{O}_{x_0}^0$ is $\{0\}$,
we may apply Theorem~\ref{th:Bi.3.2} to
$(\mathcal{L}\circ\exp|_{N\mathcal{O}(\varepsilon)_{x_0}}, \widehat{\mathcal{L}}\circ\exp|_{N\mathcal{O}(\varepsilon)_{x_0}}, N\mathcal{O}(\varepsilon)_{x_0})$
near $0=0_{x_0}\in N\mathcal{O}(\varepsilon)_{x_0}$ to get a stronger result.

We can also apply to Corollaries~\ref{cor:Bif.2.2.1},~\ref{cor:Bif.2.2.2},
 Theorem~\ref{th:Bif.2.2.4-} and Corollary~\ref{cor:Bif.3.3}
to suitably chosen slices in the setting of \cite{Lu4} to get some  corresponding results.
It is better for us using this idea in concrete applications.
For Theorems~\ref{th:Bi.3.20.1},\ref{th:Bi.3.20.2} there also exists a remark similar to Remark~\ref{rem:Bif.3.5}.

\section{Bifurcations for potential operators of Banach-Hilbert regular functionals}\label{sec:BBH}
\setcounter{equation}{0}

The aim of this section is to generalize partial results in last three sections to
potential operators of Banach-Hilbert regular functionals on Banach spaces.
Different from bifurcation results in previous sections, these bifurcation theorems
are on Banach (rather than Hilbert) spaces.
There exist a few splitting theorems for Banach-Hilbert regular functionals
in the literature (though they were expressed with different terminologies),
for example, \cite[Theorem~1.2]{BoBu}, \cite[Theorem~2.5]{JM} and \cite[p. 436, Theorem~1]{DiHiTr}
(and \cite[Theorem~1.4]{Tr1}
for the special nondegenerate case). The latter required functionals to be of class $C^3$.
Though \cite[Theorem~2.5]{JM} seems to need less conditions,
 we choose the setting of \cite{BoBu} since proving parameterized versions of
\cite[Theorem~1.2]{BoBu} can be more directly obtained from the original one,
see Appendix~\ref{app:B}. Actually, it is not hard to deduce a parameterized version
of \cite[Theorem~2.5]{JM} after  the corresponding form of nondegenerate case of it
is proved by slightly more troublesome arguments. On the other hand,
in view of applications to bifurcations for quasi-linear elliptic systems
in Part~\ref{par:BifE}, there is no big differences for results
obtained by bifurcation theories on the basis of  these  splitting theorems.

%%%%%%%%%%%%%%%%%%%%%%%%%%%%%%%%%%%%%%%%%%%%%%%%%%%%%%%%%%%%%%%%%%%%%%%%%%%%%%
%%giving a proof of the parametric version of it is more troublesome than
%%writing such a form of \cite[Theorem~1.2]{BoBu}. Hence we choose to give the latter
%%in Appendix~\ref{app:B} and adopt the corresponding setting in this section.
%%%%%%%%%%%%%%%%%%%%%%%%%%%%%%%%%%%%%%%%%%%%%%%%%%%%%%%%%%%%%%%%%%%%%%%%%%%%%%%

In this section we always assume that  $H$ is a Hilbert space with inner product $(\cdot,\cdot)_H$
and the induced norm $\|\cdot\|$,  $X$ is a Banach space with
norm $\|\cdot\|_X$, and that they satisfy the condition (S) in Appendix~\ref{app:B}, i.e.,
$X\subset H$ is dense in $H$ and  the inclusion $X\hookrightarrow H$ is continuous.

\begin{theorem}\label{th:BBHH.1}
Let $H$ and $X$ be as above, and let $\{\mathscr{L}_\lambda\,|\,\lambda\in\Lambda\}$
 a family of $(B_X(0, \delta), H)$-regular functionals (cf. Appendix~\ref{app:B})
parameterized by an open interval $\Lambda\subset\mathbb{R}$ containing  $\lambda^\ast$
such that
$\mathscr{L}_\lambda$ and the corresponding $A_\lambda$, $B_\lambda$ depend on $\lambda$ continuously.
 Suppose:
 \begin{description}
\item[(a)] $d\mathscr{L}_\lambda(0)=0\;\forall\lambda$.
\item[(b)] $\sigma(B_{\lambda^\ast}(0)|_X)\setminus\{0\}$
is bounded away from the imaginary axis, and $H^0_{\lambda^\ast}:={\rm Ker}({B}_{\lambda^\ast}(0))$
is a nontrivial subspace of finite dimension contained in $X$.
\item[(c)] For each $\lambda\in\Lambda\setminus\{\lambda^\ast\}$,
$\sigma(B_{\lambda}(0)|_X)$ is bounded away from the imaginary axis,
and $H^0_{\lambda}:={\rm Ker}({B}_\lambda(0))=\{0\}$.
\item[(d)] The negative definite space  $H^-_{\lambda}$
of each ${B}_\lambda(0)$ is of finite dimension (and hence is contained in $X$ by  Lemma~\ref{lem:BB.8}).
\item[(e)] The Morse indexes of $\mathscr{L}_\lambda$
at $0\in H$ take values $\mu_{\lambda^\ast}$ and $\mu_{\lambda^\ast}+\nu_{\lambda^\ast}$
 as $\lambda\in\mathbb{R}$ varies in both sides of $\lambda^\ast$ and is close to $\lambda^\ast$,
where $\mu_{\lambda}$ and $\nu_{\lambda}$ are the Morse index and the nullity of  $\mathscr{L}_{\lambda}$
at $0$, respectively.
 \end{description}
   Then  one of the following alternatives occurs:
\begin{description}
\item[(i)] $0\in X$ is not an isolated critical point of $\mathscr{L}_{\lambda^\ast}$;

\item[(ii)]  for every $\lambda\in\Lambda$ near $\lambda^\ast$ there is a nonzero critical point
$x_\lambda$ of $\mathscr{L}_{\lambda}$ converging to $0$ as $\lambda\to\lambda^\ast$;

\item[(iii)] there is an one-sided  neighborhood $\Lambda$ of $\lambda^\ast$ such that
for any $\lambda\in\Lambda\setminus\{\lambda^\ast\}$,
$\mathscr{L}_{\lambda}$ has at least two nonzero critical points converging to
zero as $\lambda\to\lambda^\ast$.
\end{description}
In particular,  $(\lambda^\ast, 0)\in\mathbb{R}\times X$ is a bifurcation point  for the equation
\begin{equation}\label{e:BBHH.1}
d\mathscr{L}_\lambda(0)=0,\; x\in B_X(0, \delta).
\end{equation}
\end{theorem}

\begin{proof}
Since either $\sigma(B_{\lambda}(0)|_X)$ or $\sigma(B_{\lambda}(0)|_X)\setminus\{0\}$ is bounded away from the imaginary axis,
there exists a direct sum decomposition of Banach spaces  $X=X_0^{\lambda}\oplus X_+^{\lambda}\oplus X_-^{\lambda}$,
 which corresponds to the spectral sets $\{0\}$, $\sigma_+(B_{\lambda}(0)|_X)$
 and $\sigma_-(B_{\lambda}(0)|_X)$.  Denote by $P_0^{\lambda}$, $P_+^{\lambda}$ and $P_-^{\lambda}$ the corresponding
  projections to this decomposition.

Let $H^+_\lambda$, $H^-_\lambda$ and $H^0_\lambda$ be the positive definite, negative definite and zero spaces of
${B}_\lambda(0)$.  Denote by $P^\ast_\lambda$  the orthogonal projections onto $H^\ast_\lambda$,
 $\ast=+,-,0$. %and $H^\pm_\lambda=H^+_\lambda\oplus H^-_\lambda$,

By assumptions (b)-(d) and Lemma~\ref{lem:BB.8}, we have
\begin{description}
\item[1)] $X_0^{\lambda^\ast}=H^0_{\lambda^\ast}$, $X_-^{\lambda^\ast}=H^-_{\lambda^\ast}$
and $X_+^{\lambda^\ast}$ is a dense subspace in $H^+_{\lambda^\ast}$;

\item[2)] for each $\lambda\in\Lambda\setminus\{\lambda^\ast\}$,
$X_0^{\lambda}=H^0_{\lambda}=\{0\}$, $X_-^{\lambda}=H^-_{\lambda}$
and $X_+^{\lambda}$ is a dense subspace in $H^+_{\lambda}$.
\end{description}
Hence the nullity $\nu_\lambda:=\dim X^\lambda_0$ and  Morse index $\mu_\lambda:=\dim X^\lambda_-$
of the functional $\mathscr{L}_{\lambda}$ at $0$ are equal to $\dim H_\lambda^0$  and
$\dim H^-_\lambda$ (that is, the nullity and Morse index of the quadratic form $({B}_\lambda(0) u,u)$ on $H$),
respectively, and all of them are finite numbers.

%Let $H^+_\lambda$ be the positive definite space of ${B}_\lambda(0)$,
%and let $X_{\lambda,+}$ be the Banach subspace $X\cap H^+_\lambda$ of $X$. Denote by
%$P_\lambda^0$, $P_\lambda^+$ and $P_\lambda^-$  the orthogonal  projections to $H^0_\lambda$,
%$H^+_\lambda$ and $H^-_\lambda$, respectively. By assumptions (b)-(d) and Lemma~\ref{lem:BB.8},
%when $0\in\sigma({B}_\lambda(0)|_X)$ (resp. $0\notin\sigma({B}_\lambda(0)|_X)$)
%then $X=H_\lambda^0\oplus X_{\lambda,+}\oplus H^-_\lambda$
%(resp. $X=X_{\lambda,+}\oplus H^-_\lambda$) is the direct sum decomposition of Banach spaces,
% which corresponds to the spectral sets $\{0\}$, $\sigma_+({B}_{\lambda}(0)|_X)$
% and $\sigma_-({B}_{\lambda}(0)|_X)$ (resp. $\sigma_+({B}_{\lambda}(0)|_X)$
% and $\sigma_-({B}_{\lambda}(0)|_X)$).

By Theorem~\ref{th:BB.5}  there exist small numbers $0<\epsilon<\delta$ and $\rho>0$ with
$[\lambda^\ast-\rho,\lambda^\ast+\rho]\subset\Lambda$,
 a (unique) $C^1$ map
%\begin{equation}\label{e:BBHH.1}
$
\mathfrak{h}:[\lambda^\ast-\rho,\lambda^\ast+\rho]\times B_{X}(0,\epsilon)\cap H_{\lambda^\ast}^0\to
X_0^{\lambda^\ast}\oplus H^-_{\lambda^\ast}
$
%\end{equation}
satisfying
\begin{eqnarray}\label{e:BBHH.2}
%&&\mathfrak{h}(\lambda, 0)=0,\quad\forall\lambda\in [\lambda_0-\rho,\lambda_0+\rho],\nonumber\\
&&\mathfrak{h}(\lambda, 0)=0\quad\hbox{and}\quad (id_X-P_0^{\lambda^\ast})A_{\lambda}(z+ \mathfrak{h}(\lambda,z))=0
%&&\quad\forall (\lambda,z)\in [\lambda_0-\rho,\lambda_0+\rho]\times B_{X}(0,\epsilon)\cap H_{\lambda_0}^0,\nonumber
\end{eqnarray}
for all $(\lambda,z)\in [\lambda^\ast-\rho,\lambda^\ast+\rho]\times B_{X}(0,\epsilon)\cap H_{\lambda^\ast}^0$,
and another $C^1$ map
$$
[\lambda^\ast- \rho,\lambda^\ast+\rho]\times B_X(0,\epsilon)
\to  X,\;({\lambda}, x)\mapsto \Phi_{{\lambda}}(x)
$$
such that for each $\lambda\in [\lambda^\ast- \rho,\lambda^\ast+\rho]$ the map
 $\Phi_{\lambda}$ is a $C^1$  origin-preserving diffeomorphism from
$B_X(0,\epsilon)$ onto an open neighborhood $W_\lambda$ of $0$ in $X$ and that
the functional $\mathscr{L}_{\lambda}$ satisfies
\begin{eqnarray}\label{e:BBHH.3}
\mathscr{L}_{\lambda}\circ\Phi_{\lambda}(x)=\|P^{\lambda^\ast}_+x\|_H^2-\|P^{\lambda^\ast}_-x\|_H^2+ \mathscr{
L}_{{\lambda}}(P^{\lambda^\ast}_0x+ \mathfrak{h}({\lambda}, P^{\lambda^\ast}_0x)),\quad\forall x\in B_X(0,\epsilon).
\end{eqnarray}
 Moreover, % with {\bf notations}
% $P_\pm^{\lambda^\ast}:=P_-^{\lambda^\ast}+(P_+^{\lambda^\ast}|_{X_+^{\lambda^\ast}})$ and
% $X^{\lambda^\ast}_\pm:=H^-_{\lambda^\ast}+X_+^{\lambda^\ast}$
%we have
%$$
%d_z\mathfrak{h}(\lambda,0)=-[P_\pm^{\lambda^\ast}{B}_{\lambda^\ast}|_{X^{\lambda^\ast}_\pm}]^{-1}
%(P_\pm^{\lambda^\ast}{B}_\lambda)|_{H_{\lambda^\ast}^0},
%$$
%and a $C^1$
% functional
%\begin{equation}\label{e:BBHH.4}
%[\lambda_0- \rho,\lambda_0+\rho]\times B_{X}(0,\epsilon)\cap H_{\lambda_0}^0\to \mathbb{R}\ni (\lambda,u)\mapsto \mathscr{L}_{\lambda}^\circ(u)\in\mathbb{R},
%\end{equation}
$[\lambda^\ast- \rho,\lambda^\ast+\rho]\times B_{X}(0,\epsilon)\cap H_{\lambda^\ast}^0\to \mathbb{R}\ni (\lambda,u)\mapsto \mathscr{L}_{\lambda}^\circ(u)\in\mathbb{R}$ is $C^1$,  and
%\begin{equation}\label{e:BBHH.5}
$$\mathscr{L}_{\lambda}^\circ: B_{X}(0,\epsilon)\cap H_{\lambda^\ast}^0\to \mathbb{R},\;
z\mapsto\mathscr{L}_{\lambda}(z+ \mathfrak{h}({\lambda}, z))$$
%\end{equation}
 is of class $C^{2}$ and  has the first-order derivative at $z_0\in
B_{X}(0,\epsilon)\cap H_{\lambda^\ast}^0$,
% \begin{eqnarray}\label{e:BBHH.6}
$$d\mathscr{L}^\circ_\lambda(z_0)[z]=\bigl(A_\lambda(z_0+ \mathfrak{h}(\lambda, z_0)), z\bigr)_H,\quad\forall z\in H_{\lambda^\ast}^0.$$
%\end{eqnarray}
%and the second-order derivative at $0\in
%B_{X}(0,\epsilon)\cap H_{\lambda^\ast}^0$,
%\begin{eqnarray}\label{e:BBHH.7}
%  &&d^2\mathscr{L}^\circ_\lambda(0)[z,z']=\left(P_0^{\lambda^\ast}\bigr[{B}_\lambda(0)-
%{B}_\lambda(0)(P_\pm^{\lambda^\ast}{B}_{\lambda^\ast}(0)|_{X^{\lambda^\ast}_\pm})^{-1}
%(P_\pm^{\lambda^\ast}{B}_\lambda(0))\bigr]z, z'\right)_H\nonumber\\
%&& \hspace{40mm} \forall z,z'\in H_{\lambda^\ast}^0.
% \end{eqnarray}
%%By (\ref{e:BBHH.2}) and (\ref{e:BBHH.6}) it is easy to see that
Since the map $z\mapsto z+ \mathfrak{h}({\lambda}, z)$
induces an one-to-one correspondence
 between the critical points of  $\mathscr{L}_{\lambda}^\circ$ near $0\in H_{\lambda^\ast}^0$
and those of $\mathscr{L}_{\lambda}$ near $0\in X$, $0\in H_{\lambda^\ast}^0$
is an isolated critical point of  $\mathscr{L}_{\lambda}^\circ$ if and only if
$0\in X$ is such a critical point of $\mathscr{L}_{\lambda}$
(after shrinking $\rho>0$ if necessary).
By the assumption (c) and Claim~\ref{cl:BB.6+}  we get

\begin{claim}\label{cl:BBHH.5.2}
For each $\lambda\in[\lambda^\ast- \rho,\lambda^\ast+\rho]\setminus\{\lambda^\ast\}$,
$0\in B_{X}(0,\epsilon)\cap H_{\lambda^\ast}^0$ is a nondegenerate critical point of $\mathscr{L}_{\lambda}^\circ$.
\end{claim}

(So far the condition (e) is not used yet!)
By assumptions (b)-(e)  there hold
\begin{eqnarray}%\label{e:BBHH.8}
&&\nu_\lambda=0,\quad\forall\lambda\in \Lambda\setminus\{\lambda^\ast\},\nonumber\\
&&\hbox{either}\; \mu_\lambda=\mu_{\lambda^\ast}\;\forall\lambda<\lambda^\ast,\;
\hbox{and}\; \mu_\lambda=\mu_{\lambda^\ast}+\nu_{\lambda^\ast}\;\forall\lambda>\lambda^\ast,\label{e:BBHH.9}\\
&&\hbox{or}\; \mu_\lambda=\mu_{\lambda^\ast}+\nu_{\lambda^\ast}\;\forall\lambda<\lambda^\ast,\;
\hbox{and}\; \mu_\lambda=\mu_{\lambda^\ast}\;\forall\lambda>\lambda^\ast. \label{e:BBHH.10}
\end{eqnarray}

For each $\lambda\in \Lambda\setminus\{\lambda^\ast\}$, (\ref{e:BBHH.8}) implies that
$0\in X$ is a  a nondegenerate (and hence isolated) critical point of $\mathscr{L}_\lambda$.
It follows from this and Theorem~\ref{th:BB.3} that
\begin{equation}\label{e:BBHH.11}
C_q(\mathscr{L}_\lambda, 0;{\bf K})=\delta_{q\mu_\lambda}{\bf K}\quad\forall
q=0, 1,\cdots.
\end{equation}
By Claim~\ref{cl:BBHH.5.2}, (\ref{e:BBHH.3}) and Corollary~\ref{cor:BB.6} we obtain that for any Abel group ${\bf K}$,
\begin{equation}\label{e:BBHH.12}
C_q(\mathscr{L}_\lambda, 0;{\bf K})\cong
C_{q-\mu_{\lambda^\ast}}(\mathscr{L}^{\circ}_\lambda, 0;{\bf K})\quad\forall
q=0, 1,\cdots.
\end{equation}
Then (\ref{e:BBHH.11})-(\ref{e:BBHH.12}) lead to
\begin{eqnarray}\label{e:BBHH.13}
C_{j}(\mathscr{L}^\circ_{\lambda},0;{\bf K})=
\left\{\begin{array}{ll}
\delta_{j0}{\bf K},\quad\forall \lambda<\lambda^\ast,\\
\delta_{j\nu_{\lambda^\ast}}{\bf K},\quad\forall \lambda>\lambda^\ast
\end{array}\right.
\end{eqnarray}
for any $j\in\mathbb{N}_0$ if (\ref{e:BBHH.9}) holds, and
\begin{eqnarray}\label{e:BBHH.14}
C_{j}(\mathscr{L}^\circ_{\lambda},0;{\bf K})=
\left\{\begin{array}{ll}
\delta_{j0}{\bf K},\quad\forall \lambda>\lambda^\ast,\\
\delta_{j\nu_{\lambda^\ast}}{\bf K},\quad\forall \lambda<\lambda^\ast
\end{array}\right.
\end{eqnarray}
for any $j\in\mathbb{N}_0$ if (\ref{e:BBHH.10}) holds. These two groups of equalities and (\ref{e:Bi.2.17}) lead, respectively, to
\begin{eqnarray}\label{e:BBHH.14.1}
0\in H^0_{\lambda^\ast}\;\hbox{is a strict local}\left\{
\begin{array}{ll}
\hbox{minimizer of}\;\mathscr{L}^\circ_{\lambda},&\quad
\forall \lambda<\lambda^\ast,\\
\hbox{maximizer of}\;\mathscr{L}^\circ_{\lambda},&\quad
\forall \lambda>\lambda^\ast
\end{array}\right.
\end{eqnarray}
if (\ref{e:BBHH.9}) holds, and
\begin{eqnarray}\label{e:BBHH.15}
0\in H^0_{\lambda^\ast}\;\hbox{is a strict local}\left\{
\begin{array}{ll}
\hbox{maximizer of}\;\mathscr{L}^\circ_{\lambda},&\quad
\forall \lambda<\lambda^\ast,\\
\hbox{minimizer of}\;\mathscr{L}^\circ_{\lambda},&\quad
\forall \lambda>\lambda^\ast
\end{array}\right.
\end{eqnarray}
if (\ref{e:BBHH.10}) holds. The desired conclusions follow from
 Theorem~\ref{th:Bi.2.1} directly.
\end{proof}

Now, we give analogues of Corollaries~\ref{cor:Bi.2.4.1} and \ref{cor:Bi.2.4.2} and \ref{cor:Bi.3.2.1} under the following hypothesis.
(Related notations may be found in Appendix~\ref{app:B}. )

\begin{hypothesis}\label{hyp:BBH.1}
{\rm
Let $H$ and $X$ be as in Theorem~\ref{th:BBHH.1}, let $\mathscr{L}, \widehat{\mathscr{L}}:B_X(0, \delta)\to\mathbb{R}$ be two $(B_X(0, \delta), H)$-regular functionals   with critical point $0\in X$ and with corresponding operators  $A, B$ and $\widehat{A}, \widehat{B}$, respectively.
Assume that  $\lambda^\ast\in\mathbb{R}$ is an  eigenvalue of finite multiplicity of
\begin{equation}\label{e:BBH.0}
B(0)v-\lambda \widehat{B}(0)v=0,\quad v\in H,
\end{equation}
and that for each $\lambda$ near $\lambda^\ast$ the operator $\mathfrak{B}_{\lambda}:=B(0)-\lambda \widehat{B}(0)$
satisfies the following properties:
\begin{description}
\item[(i)] either $\sigma(\mathfrak{B}_\lambda|_X)$ or $\sigma(\mathfrak{B}_\lambda|_X)\setminus\{0\}$ is bounded away from the imaginary axis;
\item[(ii)] $H^0_\lambda:={\rm Ker}(\mathfrak{B}_\lambda)\subset X$, and the negative definite space
 $H^-_{\lambda}$
of $\mathfrak{B}_\lambda$ is of finite dimension.
\end{description}
}
\end{hypothesis}

Suppose  $\widehat{B}(0)$ is compact. Since $B(0)-\lambda^\ast\widehat{B}(0)$ is Fredholm, so is each
$B(0)-\lambda\widehat{B}(0)$. In this situation, if $\dim H^-_{\lambda_0}<\infty$ for some $\lambda_0$,
by the proof of \cite[Proposition~2.3.2]{Ab} we see that  $P_{H^-_{\lambda}}-P_{H^-_{\lambda_0}}$
is a compact operator for each $\lambda$, where $P_{H^-_{\lambda}}$ is the orthogonal projection
onto $H^-_{\lambda}$. It follows that all $\dim H^-_{\lambda}$ are finite.
Since $\mathfrak{B}_{\lambda}+P_{H^-_{\lambda}}$ is invertible, and 
$P_{H^-_{\lambda}}$ is compact, by \cite[Lemma~2.2]{BoBu}
(see Lemma~\ref{lem:BB.7}),  $0$ is an isolated point of $\sigma(\mathfrak{B}_\lambda)$ and 
an eigenvalue of $\mathfrak{B}_\lambda$ of the finite multiplicity, which implies (i).

Under Hypothesis~\ref{hyp:BBH.1} consider the bifurcation problem
near $(\lambda^\ast, 0)\in\mathbb{R}\times X$ for the equation
\begin{equation}\label{e:BBHH.8}
d\mathscr{L}(u)=\lambda d\widehat{\mathscr{L}}(u),
\quad u\in B_X(0, \delta).
\end{equation}

%
%For simplicity, we only give analogues of two corollaries of Theorem~\ref{th:Bi.2.4} below.

\begin{corollary}\label{cor:BBH.2}
Under Hypothesis~\ref{hyp:BBH.1}, suppose:
\begin{description}
\item[(a)]   the eigenvalue $\lambda^\ast\in\mathbb{R}$ of (\ref{e:BBH.0})   is  isolated;
\item[(b)] $\widehat{B}(0)\in \mathscr{L}_s(H)$ is compact, and either semi-positive or semi-negative.
%\item[(b)] $B(0)$ is  a Fredholm operator;
%$B(0)=P(0)+Q(0)$, where $P(0)\in \mathscr{L}_s(H)$ is positive definite
%and $Q(0)\in \mathscr{L}_s(H)$ is compact;
\end{description}
  Then  one of the following alternatives occurs:
 \begin{description}
\item[(i)] $(\lambda^\ast,0)$ is not an isolated solution  in  $\{\lambda^\ast\}\times B_X(0, \delta)$ of the equation (\ref{e:BBHH.8});
\item[(ii)]  for every $\lambda\in\mathbb{R}$ near $\lambda^\ast$ there is a nontrivial solution $u_\lambda$ of (\ref{e:BBHH.8}) in $B_X(0, \delta)$, which  converges to $0$  as $\lambda\to\lambda^\ast$;
\item[(iii)] there is an one-sided  neighborhood $\Lambda$ of $\lambda^\ast$ such that
for any $\lambda\in\Lambda\setminus\{\lambda^\ast\}$, (\ref{e:BBHH.8}) has at least two nontrivial solutions in
$B_X(0, \delta)$, which  converge to $0$  as $\lambda\to\lambda^\ast$.
\end{description}
In particular, $(\lambda^\ast, 0)\in\mathbb{R}\times X$ is a bifurcation point  for (\ref{e:BBHH.8}).
 \end{corollary}

\begin{proof}
By (a),   $\lambda^\ast\in\mathbb{R}$  is an isolated eigenvalue  of (\ref{e:BBH.0}).
Thus we may take a number $\rho>0$ so small that $[\lambda^\ast- \rho,\lambda^\ast+\rho]$ contains
a unique eigenvalue $\lambda^\ast$ of (\ref{e:BBH.0}), and that (i) and (ii) in Hypothesis~\ref{hyp:BBH.1}
hold for each $\lambda\in[\lambda^\ast- \rho,\lambda^\ast+\rho]$.
Then conditions of Theorem~\ref{th:BBHH.1}, except (e), are satisfied for
 family $\{\mathscr{L}_\lambda=\mathscr{L}_1-\lambda \mathscr{L}_2\,|\,\lambda\in [\lambda^\ast- \rho,\lambda^\ast+\rho]\}$.

It remains to prove that (e) of Theorem~\ref{th:BBHH.1} is true for this family.

From  (a) and Hypothesis~\ref{hyp:BBH.1} we deduce that $0\in X$ is a nondegenerate
(and thus isolated) critical point of $\mathscr{L}_{\lambda}$ for each
$\lambda\in [\lambda^\ast- \rho,\lambda^\ast+\rho]\setminus\{\lambda^\ast\}$.
This is equivalent to
 the fact that  $0\in H$ is a nondegenerate critical point of the $C^\infty$ functional
 $H\ni u\mapsto\mathfrak{L}_\lambda(u):=([B(0)-\lambda \widehat{B}(0)]u,u)_H$
 for each $\lambda\in [\lambda^\ast-\rho,\lambda^\ast+\rho]\setminus\{\lambda^\ast\}$.
 Since $B(0)$ is  Fredholm by the arguments below Hypothesis~\ref{hyp:BBH.1}, and $\widehat{B}(0)$ is compact,
  it is easily seen that $\mathfrak{L}_\lambda$ satisfies the (PS) condition on any closed ball.
 As in the proof of Corollary~\ref{cor:Bi.2.4.1}, we can replace
 $\mathcal{L}''(0)$ and  $\widehat{\mathcal{L}}''(0)$ by
 $B(0)$ and  $\widehat{B}(0)$, respectively,  to obtain
\begin{equation}\label{e:BBH.9.1}
\mu_\lambda=\left\{\begin{array}{ll}
\mu_{\lambda^\ast},&\quad\forall \lambda\in [\lambda^\ast-\rho,\lambda^\ast),\\
\mu_{\lambda^\ast}+\nu_{\lambda^\ast},&\quad\forall\lambda\in (\lambda^\ast,\lambda^\ast+\rho]
\end{array}\right.
\end{equation}
if $\widehat{B}(0)\ge 0$, and
\begin{equation}\label{e:BBH.9.2}
\mu_\lambda=\left\{\begin{array}{ll}
\mu_{\lambda^\ast},&\quad\forall \lambda\in (\lambda^\ast,\lambda^\ast+\rho],\\
\mu_{\lambda^\ast}+\nu_{\lambda^\ast},&\quad\forall\lambda\in [\lambda^\ast-\rho,\lambda^\ast)
\end{array}\right.
\end{equation}
if $\widehat{B}(0)\le 0$. Hence (e) of Theorem~\ref{th:BBHH.1} holds.
\end{proof}

\begin{corollary}\label{cor:BBH.3}
Under Hypothesis~\ref{hyp:BBH.1}, suppose that the following two conditions are satisfied:
 \begin{description}
\item[(a)] $B(0)$ is invertible, and $\widehat{B}(0)$ is compact.
 \item[(b)] $B(0)\widehat{B}(0)=\widehat{B}(0)B(0)$,   and $B(0)$
 is either positive  or negative on $H^0_{\lambda^\ast}={\rm Ker}(B(0)-\lambda^\ast\widehat{B}(0))$.
 \end{description}
Then the conclusions of Corollary~\ref{cor:BBH.2} hold true.
Moreover, if $B(0)$ is positive definite, the condition (b) is unnecessary.
\end{corollary}

Indeed, in the proof of Corollary~\ref{cor:Bi.2.4.2} we only replace
$\mathcal{L}''(0)$ and  $\widehat{\mathcal{L}}''(0)$ by  $B(0)$ and  $\widehat{B}(0)$, respectively,
and the sentence ``Since $\mathcal{L}_\lambda:=\mathcal{L}-\lambda\widehat{\mathcal{L}}$
 satisfies Hypothesis~\ref{hyp:1.1} with $X=H$, it has finite Morse index $\mu_\lambda$ and nullity $\nu_\lambda$ at $0$."
 by  ``By Hypothesis~\ref{hyp:BBH.1}, $\mathscr{L}_\lambda:=\mathscr{L}-\lambda\widehat{\mathscr{L}}$
has finite Morse index $\mu_\lambda$ and nullity $\nu_\lambda$ at $0$."

 By Theorem~\ref{th:Bi.2.1E} and the proof of Theorem~\ref{th:BBHH.1}
we may directly get corresponding results with  Theorem~\ref{th:Bi.3.2} and Corollary~\ref{cor:Bi.3.2.1}.
Instead of these, as Theorem~\ref{th:Bif.2.2.4-} and Corollary~\ref{cor:Bif.3.3} we use  Theorem~\ref{th:Bif.2.2.2+} by Bartsch and Clapp \cite[\S4]{BaCl}
to obtain:
%generalize Theorem~\ref{th:BBHH.2} and Corollary~\ref{th:BBH.4} to the case of arbitrary compact Lie group.
%Let us sketch these and generalizations of some results in \cite{Ba1, BaCl}.

\begin{theorem}\label{th:BBH.6}
Under the assumptions of Theorem~\ref{th:BBHH.1},
let $G$ be a compact Lie group acting on $H$ orthogonally,
which induces a $C^1$ isometric action on $X$. Suppose that each $\mathscr{L}_\lambda$ is $G$-invariant and that
 $A_\lambda, B_\lambda$  are equivariant, and that
$H^0_{\lambda^\ast}:={\rm Ker}({B}_{\lambda^\ast}(0))$ only intersects at zero with the fixed point set $H^G$.
Then one of the following alternatives occurs:
\begin{description}
\item[(i)] $0\in X$ is not an isolated critical point of $\mathscr{L}_{\lambda^\ast}$;
\item[(ii)] there exist left and right  neighborhoods $\Lambda^-$ and $\Lambda^+$ of $\lambda^\ast$ in $\mathbb{R}$
and integers $n^+, n^-\ge 0$, such that $n^++n^-\ge \ell(SH^0_{\lambda^\ast})$
and for $\lambda\in\Lambda^-\setminus\{\lambda^\ast\}$ (resp. $\lambda\in\Lambda^+\setminus\{\lambda^\ast\}$),
$\mathscr{L}_\lambda$ has at least $n^-$ (resp. $n^+$) distinct critical
$G$-orbits different from $0$, which converge to
 $0$ as $\lambda\to\lambda^\ast$.
 \end{description}
In particular,  $(\lambda^\ast, 0)\in \Lambda\times X$
is a bifurcation point of (\ref{e:BBHH.1}).
\end{theorem}

By Remark~\ref{rem:Bif.3.4}, $\ell(SH^0_{\lambda^\ast})=\dim H^0_{\lambda^\ast}$ (resp. $\frac{1}{2}\dim H^0_{\lambda^\ast}$)
if the Lie group  $G$ is equal to $\mathbb{Z}_2=\{{\rm id}, -{\rm id}\}$ (resp. $S^1$).

 \begin{corollary}\label{th:BBH.7}
 Under the assumptions of one of Corollaries~\ref{cor:BBH.2} and \ref{cor:BBH.3}
let $G$ be a compact Lie group acting on $H$ orthogonally, which induces a $C^1$ isometric action on $X$.
Suppose that $\mathscr{L}, \widehat{\mathscr{L}}$ are $G$-invariant and that
 $A, B$ and $\widehat{A}, \widehat{B}$ are equivariant.
 %If the Lie group  $G$ is equal to $\mathbb{Z}_2=\{{\rm id}, -{\rm id}\}$ (resp. $S^1$
%     and the fixed point set of the induced $S^1$-action on $H^0_{\lambda_0}={\rm Ker}(\mathfrak{B}_\lambda)$
%    is $\{0\}$),
   Then the conclusions of Theorem~\ref{th:BBH.6} hold.
  \end{corollary}

With methods in  \cite[\S7.5]{Ba1}
it is possible to make further generalizations for the above results.

\part{Applications to bifurcations for quasi-linear elliptic systems}\label{par:BifE}

%
%\section{Bifurcations for quasi-linear elliptic systems}\label{sec:BifE}
%\setcounter{equation}{0}
This part studies applications of the abstract theorems in Part I to bifurcations for quasi-linear elliptic systems.
Without special statements, throughout this part we always assume that integers $N\ge 1$, $n>1$,
 and that $\Omega\subset\R^n$ is a bounded domain   with  boundary $\partial\Omega$.
(The case $n=1$ will be considered in \cite{Lu8} independently.)
 For showing our methods  we only consider the Dirichlet boundary conditions in many results.
 %Section~\ref{sec:BifE.1} contains some preliminaries.
%Then we give a few bifurcation results for quasi-linear elliptic systems
%with growth restrictions in Section~\ref{sec:BifE.2}.
%As applications of theorems in Section~\ref{sec:BBH}, some
%bifurcation theorems for quasi-linear elliptic systems
%without growth restrictions are obtained in Section~\ref{sec:BifE.3}.
%In Section~\ref{sec:BifE.4} we study bifurcations from deformations of domains
%and generalize previous results.

 It is possible that these results are generalized to the case of unbounded domains  $\Omega\subset\R^n$ with \cite{Vo}.
 We believe that theories in previous sections  can also be used to improve bifurcation results
for geometric variational problems such as \cite{Bor} and
\cite{BetPS1, BetPS2}, etc.

\section{Structural hypotheses and preliminaries}\label{sec:BifE.1}
\setcounter{equation}{0}

 In \cite{Lu7} we introduced the following (denoted by \textsf{Hypothesis} $\mathfrak{F}_{2,N}$ in \cite{Lu6}).

\noindent{\textsf{Hypothesis} $\mathfrak{F}_{2,N,m,n}$}.\quad
 For each multi-index $\gamma$ as above, let
 \begin{eqnarray*}
 &&2_\gamma\in (2,\infty)\;\hbox{if}\;
 |\gamma|=m-n/2,\qquad 2_\gamma=\frac{2n}{n-2(m-|\gamma|)}
\;\hbox{if}\; m-n/2<|\gamma|\le m,\\
 &&2'_\gamma=1\;\hbox{if}\;|\gamma|<m-n/2,\qquad
 2'_\gamma=\frac{2_\gamma}{2_\gamma-1} \;\hbox{if}\;m-n/2\le |\gamma|\le m;
 \end{eqnarray*}
 and for each two multi-indexes $\alpha, \beta$ as above, let
 $2_{\alpha\beta}=2_{\beta\alpha}$ be defined by the conditions
\begin{eqnarray*}
&&2_{\alpha\beta}= 1-\frac{1}{2_\alpha}-\frac{1}{2_\beta}\quad
 \hbox{if}\;|\alpha|=|\beta|=m,\\
&&2_{\alpha\beta}=  1-\frac{1}{2_\alpha}\quad \hbox{if}\;m-n/2\le |\alpha|\le
 m,\; |\beta|<m-n/2,\\
&&2_{\alpha\beta}= 1 \quad \hbox{if}\; |\alpha|, |\beta|<m-n/2, \\
&& 0<2_{\alpha\beta}<1-\frac{1}{2_\alpha}-\frac{1}{2_\beta}\quad\hbox{if}\;|\alpha|,\;|\beta|\ge
 m-n/2,\;|\alpha|+|\beta|<2m.
\end{eqnarray*}
Let $M(k)$ be the number of $n$-tuples
$\alpha=(\alpha_1,\cdots,\alpha_n)\in (\mathbb{N}_0)^n$  of length
 $|\alpha|:=\alpha_1+\cdots+\alpha_n\le k$, $M_0(k)=M(k)-M(k-1)$, $k=0,\cdots,m$,
where $M(-1)=\emptyset$ and  $M(0)=M_0(0)$ only consists of
${\bf 0}=(0,\cdots,0)\in (\mathbb{N}_0)^n$. Write $\xi\in \prod^m_{k=0}\mathbb{R}^{N\times M_0(k)}$ as
$\xi=(\xi^0,\cdots,\xi^m)$, where
 $ \xi^0=(\xi^1_{\bf 0},\cdots,\xi^N_{\bf 0})^T\in \mathbb{R}^{N}$ and
 $$
 \xi^k=\left(\xi^i_\alpha\right)\in \mathbb{R}^{N\times M_0(k)},
 \quad k=1,\cdots,m,\quad 1\le i\le N,\quad |\alpha|=k.
 $$
  Denote by  $\xi^k_\circ=\{\xi^k_\alpha\,:\,|\alpha|<m-n/2\}$ for $k=1,\cdots,N$. Let
\begin{eqnarray}\label{e:6.0}
\overline\Omega\times\prod^m_{k=0}\mathbb{R}^{N\times M_0(k)}\ni (x,
\xi)\mapsto F(x,\xi)\in\R
\end{eqnarray}
be twice continuously differentiable in $\xi$ for almost all $x$,
measurable in $x$ for all values of $\xi$, and $F(\cdot,\xi)\in L^1(\Omega)$ for $\xi=0$.
Suppose that derivatives of $F$ fulfill  the following properties:
\begin{description}
\item[(i)]  For $i=1,\cdots,N$ and $|\alpha|\le m$, functions
 $F^i_\alpha(x,\xi):= F_{\xi^i_\alpha}(x,\xi)$ for $\xi=0$
belong to $L^1(\Omega)$ if $|\alpha|<m-n/2$,
and to $L^{2'_\alpha}(\Omega)$ if $m-n/2\le |\alpha|\le m$.
\item[(ii)] There exists a continuous, positive, nondecreasing functions $\mathfrak{g}_1$ such that
for $i,j=1,\cdots,N$ and $|\alpha|, |\beta|\le m$  functions
$\overline\Omega\times\R^{M(m)}\to\R,\; (x, \xi)\mapsto
F^{ij}_{\alpha\beta}(x,\xi):=F_{\xi^i_\alpha\xi^j_\beta}(x,\xi)$
satisfy:
\begin{eqnarray}\label{e:6.1}
 |F^{ij}_{\alpha\beta}(x,\xi)|\le
\mathfrak{g}_1(\sum^N_{k=1}|\xi_\circ^k|)\left(1+
\sum^N_{k=1}\sum_{m-n/2\le |\gamma|\le
m}|\xi^k_\gamma|^{2_\gamma}\right)^{2_{\alpha\beta}}.
\end{eqnarray}
\item[(iii)] There exists a continuous, positive, nondecreasing functions $\mathfrak{g}_2$ such that
\begin{eqnarray}\label{e:6.2}
\sum^N_{i,j=1}\sum_{|\alpha|=|\beta|=m}F^{ij}_{\alpha\beta}(x,\xi)\eta^i_\alpha\eta^j_\beta\ge
\mathfrak{g}_2(\sum^N_{k=1}|\xi^k_\circ|)
\sum^N_{i=1}\sum_{|\alpha|= m}(\eta^i_\alpha)^2
\end{eqnarray}
for any $\eta=(\eta^{i}_{\alpha})\in\R^{N\times M_0(m)}$.

{\it Note}:  If $m\le n/2$ the functions  $\mathfrak{g}_1$ and $\mathfrak{g}_2$ should be understand as positive constants.
\end{description}

In \cite[Proposition~A.1]{Lu7} it was proved that
\textsf{Hypothesis} $\mathfrak{F}_{2,N,1,n}$ is weaker than the
 {\bf controllable growth conditions} (abbreviated to CGC below) \cite[page 40]{Gi}.
(CGC was called  `common condition of Morrey' or `the natural assumptions of Ladyzhenskaya and Ural'tseva' \cite[page 38,(I)]{Gi}.)\\

\noindent{\bf CGC}: $\overline\Omega\times\mathbb{R}^N\times\mathbb{R}^{N\times n}\ni (x,
z,p)\mapsto F(x, z,p)\in\R$ is of class $C^2$, and
there exist positive constants $\nu, \mu, \lambda, M_1, M_2$,  such that with
$|z|^2:=\sum^N_{l=1}|z_l|^2$ and $|p|^2:=\sum_{|\alpha|=1}\sum^N_{k=1}|p^k_\alpha|^2$,
\begin{eqnarray*}
&\nu\left(1+|z|^2+|p|^2\right)-\lambda\le F(x, z,p)
\le\mu\left(1+|z|^2+|p|^2\right),\\
&|F_{p^i_\alpha}(x,z,p)|, |F_{p^i_\alpha x_l}(x,z,p)|, |F_{z_j}(x,z,p)|, |F_{z_jx_l}(x,z,p)|\le \mu\left(1+|z|^2+|p|^2\right)^{1/2},\\
&|F_{p^i_\alpha z_j}(x,z,p)|,\quad |F_{z_iz_j}(x,z,p)|\le \mu,\\
&M_1\sum^N_{i=1}\sum_{|\alpha|= 1}(\eta^i_\alpha)^2\le\sum^N_{i,j=1}\sum_{|\alpha|=|\beta|=1}F_{p^i_\alpha p^j_\beta}(x,z,p)\eta^i_\alpha\eta^j_\beta\le
M_2\sum^N_{i=1}\sum_{|\alpha|= 1}(\eta^i_\alpha)^2\\
&\forall\eta=(\eta^{i}_{\alpha})\in\R^{N\times n}.
\end{eqnarray*}
Moreover, if $F=F(x,p)$ does not depend explicitly on $z$, the first three lines are replaced by
\begin{eqnarray*}
&\nu\left(1+|p|^2\right)-\lambda\le F(x,p)
\le\mu\left(1+|p|^2\right)\quad\hbox{and}\\
&|F_{p^i_\alpha}(x,p)|,\quad |F_{p^i_\alpha x_l}(x,p)|\le \mu\left(1+|p|^2\right)^{1/2}.
\end{eqnarray*}

 A bounded domain $\Omega$ in $\R^n$ is said to be a {\bf Sobolev domain} for $(2,m,n)$
 if  the Sobolev embeddings theorems for the spaces $W^{m, 2}(\Omega)$ hold.
Let $W^{m,2}_0(\Omega,\mathbb{R}^N)$ be equipped with the usual inner product
 \begin{eqnarray}\label{e:6.2.2}
 (\vec{u},\vec{v})_H=\sum^N_{i=1}\sum_{|\alpha|=m}\int_\Omega D^\alpha u^i D^\alpha v^i dx.
\end{eqnarray}

The following two theorems are contained in \cite[Theorem~4.1]{Lu6} (or \cite[Theorems~4.1,4.2]{Lu7}).

\begin{theorem}\label{th:6.1}
 Given  integers $m, N\ge 1$, $n\ge 2$, let $\Omega\subset\R^n$ be a Sobolev domain
 for $(2,m,n)$, and let $V_0$ be a closed subspace of $W^{m,2}(\Omega, \mathbb{R}^N)$ and $V=\vec{w}+V_0$
 for some $\vec{w}\in W^{m,2}(\Omega, \mathbb{R}^N)$.
Suppose that (i)-(ii) in \textsf{Hypothesis} $\mathfrak{F}_{2,N,m,n}$ hold.
 Then we have the following:
 \begin{description}
  \item[(A)] The restriction $\mathfrak{F}_V$ of  the functional
\begin{equation}\label{e:6.3}
W^{m,2}(\Omega, \mathbb{R}^N)\ni \vec{u}\mapsto\mathfrak{F}(\vec{u})=\int_\Omega F(x, \vec{u},\cdots, D^m\vec{u})dx
\end{equation}
 to $V$ is bounded on any bounded subset, of class $C^1$, and has  derivative
$\mathfrak{F}'_V(\vec{u})$  at $\vec{u}\in V$ given by
\begin{equation}\label{e:6.4}
\langle \mathfrak{F}'_V(\vec{u}), \vec{v}\rangle=\sum^N_{i=1}\sum_{|\alpha|\le m}\int_\Omega F^i_\alpha(x,
\vec{u}(x),\cdots, D^m \vec{u}(x))D^\alpha v^i dx,\quad\forall \vec{v}\in V_0.
\end{equation}
Moreover, the map $V\ni\vec{u}\to \mathfrak{F}'_V(\vec{u})\in V_0^\ast$ also maps bounded subset into bounded one.\\
\item[(B)]  At each $\vec{u}\in V$, the map $\mathfrak{F}'_V$  has the G\^ateaux derivative
$D\mathfrak{F}'_V(\vec{u})\in \mathscr{L}(V_0, V^\ast_0)$ given by
  \begin{equation}\label{e:6.5}
   \langle D\mathfrak{F}'_V(\vec{u})[\vec{v}],\vec{\varphi}\rangle=\sum^N_{i,j=1}\sum_{|\alpha|,|\beta|\le m}\int_\Omega
  F^{ij}_{\alpha\beta}(x, \vec{u}(x),\cdots, D^m \vec{u}(x))D^\beta v^j\cdot D^\alpha\varphi^i dx.
    \end{equation}
(Equivalently,   the gradient map of $\mathfrak{F}_V$,
$V\ni \vec{u}\mapsto\nabla \mathfrak{F}_V(\vec{u})\in V_0$,  given by
$(\nabla\mathfrak{F}_V(\vec{u}), \vec{v})_{m,2}=\langle \mathfrak{F}'_V(\vec{u}), \vec{v}\rangle\;
 \forall \vec{v}\in V_0$,
has the G\^ateaux derivative $D(\nabla \mathfrak{F}_V)(\vec{u})\in\mathscr{L}_s(V_0)$ at every $\vec{u}\in V$.)
Moreover,   $D\mathfrak{F}'_V$ also satisfies the following properties:
\begin{description}
\item[(i)] For every given $R>0$, $\{D\mathfrak{F}'_V(\vec{u})\,|\, \|\vec{u}\|_{m,p}\le R\}$
is bounded in $\mathscr{L}_s(V_0)$.
Consequently,  $\mathfrak{F}_V$ is  of class $C^{2-0}$.
\item[(ii)] For any $\vec{v}\in V_0$, $\vec{u}_k\to
\vec{u}_0$ implies
$D\mathfrak{F}'_V(\vec{u}_k)[\vec{v}]\to D\mathfrak{F}'_V(\vec{u}_0)[\vec{v}]$ in $V^\ast_0$.
\item[(iii)] If $F(x,\xi)$ is independent of all variables $\xi^k_\alpha$, $|\alpha|=m$,
$k=1,\cdots,N$, then
$$V\to \mathscr{L}(V_0, V^\ast_0),\;\vec{u}\mapsto D\mathfrak{F}'_V(\vec{u})
$$
is  continuous, (namely $\mathfrak{F}_V$ is of class $C^2$),  and  $D(\nabla\mathfrak{F}_V)(\vec{u}): V_0\to V_0$ is a completely continuous  linear operator for each $\vec{u}\in V$.
\end{description}
\end{description}
\end{theorem}

\begin{theorem}\label{th:6.2}
Under assumptions of Theorem~\ref{th:6.1},  suppose that
\textsf{Hypothesis} $\mathfrak{F}_{2,N,m,n}$(iii) is also satisfied. Then
\begin{description}
\item[(C)]  $\mathfrak{F}': W^{m,2}(\Omega, \mathbb{R}^N)\to (W^{m,2}(\Omega, \mathbb{R}^N))^\ast$ is of class  $(S)_+$.
\item[(D)]   For $u\in V$, let $D(\nabla\mathfrak{F}_V)(\vec{u})$,
$P(\vec{u})$ and $Q(\vec{u})$ be
  operators in $\mathscr{L}(V_0)$ defined by
  \begin{eqnarray*}
  (D(\nabla\mathfrak{F}_V)(\vec{u})[\vec{v}],\vec{\varphi})_{m,2}&=&\sum^N_{i,j=1}\sum_{|\alpha|,|\beta|\le m}\int_\Omega
  F^{ij}_{\alpha\beta}(x, \vec{u}(x),\cdots, D^m \vec{u}(x))D^\beta v^j\cdot D^\alpha\varphi^i dx,\\
   (P(\vec{u})\vec{v}, \vec{\varphi})_{m,2}&=&\sum^N_{i,j=1}\sum_{|\alpha|=|\beta|=m}\int_\Omega
  F^{ij}_{\alpha\beta}(x, \vec{u}(x),\cdots, D^m \vec{u}(x))D^\beta v^j\cdot D^\alpha\varphi^i dx\\
  &&+ \sum^N_{i=1}\sum_{|\alpha|\le m-1}\int_\Omega  D^\alpha v^i\cdot D^\alpha\varphi^i dx,\\
   (Q(\vec{u})\vec{v},\vec{\varphi})_{m,2}&=&\sum^N_{i,j=1}\sum_{|\alpha|+|\beta|<2m}\int_\Omega
  F^{ij}_{\alpha\beta}(x, \vec{u}(x),\cdots, D^m \vec{u}(x))D^\beta v^j\cdot D^\alpha\varphi^i dx\\
  &&-\sum^N_{i=1}\sum_{|\alpha|\le m-1}\int_\Omega  D^\alpha v^i\cdot D^\alpha\varphi^i dx,
    \end{eqnarray*}
  respectively. (If $V\subset W^{m,2}_0(\Omega, \mathbb{R}^N)$,
  the final terms in the definitions of $P$ and $Q$ can be deleted.)
    Then $D(\nabla\mathfrak{F}_V)=P+ Q$,  and
 \begin{description}
\item[(i)]  for any $\vec{v}\in V_0$, the map $V\ni \vec{u}\mapsto P(\vec{u})\vec{v}\in V_0$ is continuous;
\item[(ii)] for every given $R>0$ there exist positive constants $C(R, n, m, \Omega)$ such that
$$
(P(\vec{u})\vec{v},\vec{v})_{m,2}\ge C\|\vec{v}\|^2_{m,2},\quad\forall \vec{v}\in V_0,\;
\forall\vec{u}\in V\;\hbox{with}\;\|\vec{u}\|_{m,2}\le R;
$$
\item[(iii)] $V\ni \vec{u}\mapsto Q(\vec{u})\in\mathscr{L}(V_0)$ is continuous,
and  $Q(\vec{u})$ is completely continuous
for each $\vec{u}$;
\item[(iv)] for every given $R>0$ there exist positive constants $C_j(R, n, m, \Omega), j=1,2$ such that
\begin{eqnarray*}
&&(D(\nabla\mathfrak{F}_V)(\vec{u})[\vec{v}],\vec{v})_{m,2}\ge C_1\|\vec{v}\|^2_{m,2}-C_2\|\vec{v}\|^2_{m-1,2},\\
&&\qquad\forall \vec{v}\in V_0,\;\forall\vec{u}\in V\;\hbox{with}\;\|\vec{u}\|_{m,2}\le R.
\end{eqnarray*}
\end{description}
\end{description}
\end{theorem}

\begin{remark}\label{rem:BifE.3}
{\rm As noted in \cite[Remark~4.4]{Lu7}, Theorems~\ref{th:6.1} and \ref{th:6.2} and thus all results above
 have also more general versions in the setting of
\cite{Sma, Pa2}.  In particular, $\Omega\subset\mathbb{R}^n$ may be replaced by the torus $\mathbb{T}^n=\mathbb{R}^n/\mathbb{Z}^n$. In
 this situation, $F$ in \textsf{Hypothesis} $\mathfrak{F}_{2,N,m,n}$ is understood as a function
 on $\mathbb{R}^n\times\prod^m_{k=0}\mathbb{R}^{N\times M_0(k)}$, which is not only $1$-periodic in each variable
$x_i$, $i=1,\cdots,n$, but also satisfies \textsf{Hypothesis} $\mathfrak{F}_{2,N,m,n}$  with $\overline\Omega=[0,1]^n$.
Then all previous results in this section  also hold if $W^{m,2}(\Omega, \mathbb{R}^N)$ is replaced
by $W^{m,2}(\mathbb{T}^n, \mathbb{R}^N)$.}
\end{remark}

The tori $\mathbb{T}^n$ and $\mathbb{T}^1$ act, respectively, on $W^{m,2}(\mathbb{T}^n,\mathbb{R}^N)$  by the
isometric linear representations
\begin{eqnarray}
&&([t_1,\cdots,t_n]\cdot\vec{u})(x_1,\cdots,x_n)=\vec{u}(x_1+t_1,\cdots, x_n+t_n),\quad
 [t_1,\cdots,t_n]\in \mathbb{T}^n,\label{e:T^n-action1}\\
&& ([t]\cdot\vec{u})(x_1,\cdots,x_n)=\vec{u}(x_1+t,\cdots, x_n+t),\quad
 [t]\in \mathbb{T}^1.\label{e:T^n-action2}
\end{eqnarray}
The set of fixed points of the action in (\ref{e:T^n-action1}),
${\rm Fix}(\mathbb{T}^n)$, consists of all constant vector functions from $\mathbb{T}^n$ to $\mathbb{R}^N$.
Under Hypothesis~\ref{hyp:BifE.1} with $\Omega$ replaced by $\mathbb{T}^n$,
 every critical orbit different from points in ${\rm Fix}(\mathbb{T}^n)$ must be homeomorphic to some $T^s$, $1\le s\le n$.

 If $\mathbb{S}^1=\mathbb{T}^1$ acts on $\mathbb{R}^n$ by the orthogonal representation, and $\Omega$
is symmetric under the action, we get a $\mathbb{S}^1$ action on $W^{m,2}(\Omega,\mathbb{R}^N)$
and $W^{m,2}_0(\Omega,\mathbb{R}^N)$ by
\begin{eqnarray}\label{e:S^1-action}
 ([t]\cdot\vec{u})(x)=\vec{u}([t]\cdot x),\quad
 [t]\in \mathbb{S}^1.
\end{eqnarray}

There exists a natural $\mathbb{Z}_2$-action on $W^{m,2}(\Omega,\mathbb{R}^N)$ given by
\begin{eqnarray}\label{e:Z2-action.1}
 [0]\cdot\vec{u}=\vec{u},\quad [1]\cdot\vec{u}=-\vec{u},\quad\forall\vec{u}\in W^{m,2}(\Omega,\mathbb{R}^N).
\end{eqnarray}
If $\Omega$ is symmetric with respect to the origin, there is also another obvious $\mathbb{Z}_2$-action on $W^{m,2}(\Omega,\mathbb{R}^N)$,
\begin{eqnarray}\label{e:Z2-action.2}
 [0]\cdot\vec{u}=\vec{u},\quad ([1]\cdot\vec{u})(x)=\vec{u}(-x),\quad\forall\vec{u}\in W^{m,2}(\Omega,\mathbb{R}^N).
\end{eqnarray}

\begin{hypothesis}\label{hyp:BifE.1}
{\rm Let $\Omega\subset\R^n$ be a bounded  Sobolev domain, $N\in\mathbb{N}$, and let
  functions
\begin{eqnarray}\label{e:BifE.0}
F:\overline\Omega\times\prod^m_{k=0}\mathbb{R}^{N\times M_0(k)}\to \R\quad\hbox{and}\quad
K:\overline\Omega\times\prod^{m-1}_{k=0}\mathbb{R}^{N\times M_0(k)}\to \R
\end{eqnarray}
satisfy \textsf{Hypothesis} $\mathfrak{F}_{2,N,m,n}$ and (i)-(ii) in \textsf{Hypothesis} $\mathfrak{F}_{2,N,m,n}$,
respectively.  Let $V_0$ be a closed subspace of $W^{m,2}(\Omega,\mathbb{R}^N)$
containing $W^{m,2}_0(\Omega,\mathbb{R}^N)$, and $V=\vec{w}+V_0$
 for some $\vec{w}\in W^{m,2}(\Omega, \mathbb{R}^N)$.}
\end{hypothesis}

Consider (generalized) bifurcation solutions of the boundary value problem
corresponding to  $V$:
\begin{eqnarray}\label{e:BifE.1}
&&\sum_{|\alpha|\le m}(-1)^{|\alpha|}D^\alpha F^i_\alpha(x, \vec{u},\cdots, D^m\vec{u})=
\lambda\sum_{|\alpha|\le m-1}(-1)^{|\alpha|}D^\alpha K^i_\alpha(x, \vec{u},\cdots, D^{m-1}\vec{u}),\nonumber\\
&&\hspace{20mm}i=1,\cdots,N.
\end{eqnarray}
Call $\vec{u}\in V$ a {\it generalized solution} of (\ref{e:BifE.1}) if it is a
 critical point of the functional $\mathfrak{F}_V-\lambda \mathfrak{K}_V$, where
 $\mathfrak{F}_V$ is as in Theorem~\ref{th:6.1}, and $\mathfrak{K}_V$ is the restrictions of
 $\mathfrak{K}(\vec{u})=\int_\Omega K(x, \vec{u},\cdots, D^{m-1}\vec{u})dx$ to $V$.\\

%%%%%%%%%%%%%%%%%%%%%%%%%%%%%%%%%%%%%%%%%%%%%%%%%%%%%%%%%%%%
%%\begin{equation}\label{e:BifE.2}
%%\mathfrak{F}(\vec{u})-\lambda \mathfrak{K}(\vec{u})=0,
%%\end{equation}
%%%%%%%%%%%%%%%%%%%%%%%%%%%%%%%%%%%%%%%%%%%%%%%%%%%%%%%%%%%%

\section{Bifurcations for quasi-linear elliptic systems
with growth restrictions}\label{sec:BifE.2}
\setcounter{equation}{0}

By the above Theorem~\ref{th:6.2} and  \cite[Theorem~7.1, Chapter~4]{Skr1}
we immediately obtain the following result, which in $N=1$ is a special case of \cite[Theorem~7.2, Chapter~4]{Skr1}.

\begin{theorem}\label{th:BifE.2}
Under Hypothesis~\ref{hyp:BifE.1}, assume $V=V_0$ and
\begin{description}
\item[(i)] the functionals $\mathfrak{F}_V$ and $\mathfrak{K}_V$ are even, $\mathfrak{F}_V({0})=\mathfrak{K}_V({0})=0$,
$\mathfrak{K}_V(\vec{u})\ne 0$ and $\mathfrak{K}'_V(\vec{u})\ne 0$ for any $\vec{u}\in V\setminus\{{0}\}$;
\item[(ii)] $\langle\mathfrak{F}_V'(\vec{u}), \vec{u}\rangle\ge\nu(\|\vec{u}\|_{m,2})$, where $\nu(t)$ is a continuous function and positive for $t>0$;
\item[(iii)] $\mathfrak{F}_V(\vec{u})\to +\infty$ as $\|\vec{u}\|_{m,2}\to\infty$.
\end{description}
Then for any $c>0$ there exists at least a sequence $(\lambda_j,\vec{u}_j)\subset\mathbb{R}\times
\mathfrak{F}_V^{-1}(c)$ satisfying (\ref{e:BifE.1}).
\end{theorem}

  Corollary~\ref{cor:Bi.2.2} and Theorem~\ref{th:6.2} directly lead to:

\begin{theorem}\label{th:BifE.3}
Under Hypothesis~\ref{hyp:BifE.1}, let $\vec{u}_0\in V$ satisfy $\mathfrak{F}'_V(\vec{u}_0)=0$ and $\mathfrak{K}'_V(\vec{u}_0)=0$.
Suppose that $(\lambda^\ast, \vec{u}_0)$
is a bifurcation point for (\ref{e:BifE.1}). Then the linear problem
 \begin{eqnarray}\label{e:BifE.3}
 &&\sum^N_{j=1}\sum_{|\alpha|,|\beta|\le m}
(-1)^{|\alpha|}D^\alpha\bigl[F^{ij}_{\alpha\beta}(x, \vec{u}_0(x),\cdots, D^m \vec{u}_0(x))D^\beta v^j\bigr]
\nonumber\\
&&=\lambda\sum^N_{j=1}\sum_{|\alpha|,|\beta|\le m-1}
(-1)^{|\alpha|}D^\alpha\bigl[K^{ij}_{\alpha\beta}(x, \vec{u}_0(x),\cdots, D^{m-1}\vec{u}_0(x))D^\beta v^j\bigr]\nonumber\\
&&\hspace{30mm}i=1,\cdots,N
 \end{eqnarray}
 with $\lambda=\lambda^\ast$
 has a nontrivial solution in $V_0$, namely $\vec{u}_0$ is a degenerate
critical point of  $\mathfrak{F}_V-\lambda^\ast\mathfrak{K}_V$.
\end{theorem}

The following are some sufficient criteria.
By  Corollary~\ref{cor:Bi.2.4.1}
and Theorem~\ref{th:6.2} we get:

\begin{theorem}\label{th:BifE.7}
Under Hypothesis~\ref{hyp:BifE.1}, let $\vec{u}_0\in V$ satisfy $\mathfrak{F}'_V(\vec{u}_0)=0$ and
$\mathfrak{K}'_V(\vec{u}_0)=0$. Suppose that  $\lambda^\ast$ is an isolated eigenvalue
of (\ref{e:BifE.3}) (this is true if the following Hypothesis~\ref{hyp:BifE.4} is satisfied), and that
$\mathfrak{K}''_V(\vec{u}_0)$ is either semi-positive or semi-negative.
Then $(\lambda^\ast, \vec{u}_0)\in\mathbb{R}\times V$ is a bifurcation point of
(\ref{e:BifE.1}),  and  one of the following alternatives occurs:
\begin{description}
\item[(i)] $(\lambda^\ast, \vec{u}_0)$ is not an isolated solution of (\ref{e:BifE.1}) in
 $\{\lambda^\ast\}\times V$;

\item[(ii)]  for every $\lambda\in\mathbb{R}$ near $\lambda^\ast$ there is a nontrivial solution $\vec{u}_\lambda$ of (\ref{e:BifE.1}) converging to $\vec{u}_0$ as $\lambda\to\lambda^\ast$;

\item[(iii)] there is an one-sided  neighborhood $\Lambda$ of $\lambda^\ast$ such that
for any $\lambda\in\Lambda\setminus\{\lambda^\ast\}$,
(\ref{e:BifE.1}) has at least two nontrivial solutions converging to
$\vec{u}_0$ as $\lambda\to\lambda^\ast$.
\end{description}
\end{theorem}

%Conversely, if  $\vec{u}_0$ is a degenerate critical point of the functional $\mathfrak{F}_V-\lambda^\ast\mathfrak{K}_V$,
%in order to guarantee that $(\lambda^\ast, \vec{u}_0)$
%is a bifurcation point for (\ref{e:BifE.1}) we also need to make the following hypothesis.

The following hypothesis can guarantee that that every eigenvalue
of (\ref{e:BifE.3}) is isolated.

\begin{hypothesis}\label{hyp:BifE.4}
{\rm Under Hypothesis~\ref{hyp:BifE.1},  assume that $\vec{u}_0\in V$ satisfy $\mathfrak{F}'_V(\vec{u}_0)=0$ and $\mathfrak{K}'_V(\vec{u}_0)=0$, and that the linear problem
 \begin{eqnarray}\label{e:BifE.4}
 \sum^N_{j=1}\sum_{|\alpha|,|\beta|\le m}
(-1)^{|\alpha|}D^\alpha\bigl[F^{ij}_{\alpha\beta}(x, \vec{u}_0(x),\cdots, D^m \vec{u}_0(x))D^\beta v^j\bigr]
=0,\quad i=1,\cdots,N
 \end{eqnarray}
has no nontrivial solutions in $V_0$, that is, $\mathfrak{F}''_V(\vec{u}_0)\in\mathscr{L}(V_0)$ has
a bounded linear inverse operator.}
\end{hypothesis}

Under Hypothesis~\ref{hyp:BifE.4}, by the arguments above Theorem~\ref{th:Bi.2.4},
 all eigenvalues of (\ref{e:BifE.3}) form a discrete subset of $\mathbb{R}$,
$\{\lambda_j\}^\infty_{j=1}$, which contains no zero and satisfies $|\lambda_j|\to\infty$
as $j\to\infty$; moreover, each $\lambda_j$ has finite multiplicity. Let $E_j\subset V_0$ be the eigensubspace of  (\ref{e:BifE.3}) associated with the eigenvalue $\lambda_j$, $j=1,2,\cdots$.
 From Corollary~\ref{cor:Bi.3} (resp. Corollary~\ref{cor:Bi.2.4.2}) and Theorem~\ref{th:6.2} we derive
 the first (resp. second) conclusion of the following theorem.

 %${\mathcal{L}}''(0)$ is invertible, ${\mathcal{L}}''(0)\widehat{\mathcal{L}}''(0)=\widehat{\mathcal{L}}''(0){\mathcal{L}}''(0)$

\begin{theorem}\label{th:BifE.5}
Under Hypothesis~\ref{hyp:BifE.4}, for an eigenvalue $\lambda_j$
of (\ref{e:BifE.3}) as above, there holds:
\begin{description}
\item[(1)]  $(\lambda_j, \vec{u}_0)\in\mathbb{R}\times V$ is a bifurcation point of
(\ref{e:BifE.1}) provided  that $\mathfrak{F}''_V(\vec{u}_0)$
commutes with $\mathfrak{K}''_V(\vec{u}_0)$ and that the positive and negative indexes of inertia of the restriction of $\mathfrak{F}''_V(\vec{u}_0)$ to $E_j$
  are different.
\item[(2)] The conclusions in Theorem~\ref{th:BifE.7} hold if  one of the following conditions is satisfied:
  \begin{description}
\item[(i)] $\mathfrak{F}''_V(\vec{u}_0)$ commutes with $\mathfrak{K}''_V(\vec{u}_0)$, and
$\mathfrak{F}''_V(\vec{u}_0)$ is either positive or negative on $E_j$;
\item[(ii)] $\mathfrak{F}''_V(\vec{u}_0)$ is  positive definite, i.e., for some $c>0$,
\begin{eqnarray*}
 \sum^N_{i,j=1}\sum_{|\alpha|,|\beta|\le m}\int_\Omega
  F^{ij}_{\alpha\beta}(x, \vec{u}_0(x),\cdots, D^m \vec{u}_0(x))D^\beta v^j\cdot D^\alpha v^i dx\ge c\|\vec{v}\|_{m,2},\quad\forall \vec{v}\in V.
 \end{eqnarray*}
\end{description}
  \end{description}
 \end{theorem}

\begin{remark}\label{rmk:BifE.6}
{\rm  When $N=1$, $V=W_0^{m,2}(\Omega)$, and for some $c>0$ it holds  that
\begin{eqnarray}\label{e:BifE.7}
(\mathfrak{F}'_V({u}),{u})_{m,2}\ge c\|u\|^2_{m,2}
\end{eqnarray}
near $0\in V$, it was proved in \cite[Chap.1, Theorem~3.5]{Skr2} that
$(\lambda^\ast, 0)$ is a bifurcation point of (\ref{e:BifE.1}) if and only if
 $\lambda^\ast$ is  an eigenvalue of (\ref{e:BifE.3}) with ${u}=0$.
Since $\mathfrak{F}'_V(0)=0$, it is clear that (\ref{e:BifE.7})
implies $\mathfrak{F}''_V(0)$ to be positive definite, that is,
 the condition (ii) in Theorem~\ref{th:BifE.5}
is satisfied at $u_0=0$. Hence Theorem~\ref{th:BifE.5}
is a significant generalization of \cite[Chap.1, Theorem~3.5]{Skr2}.
}
\end{remark}

Corollary~\ref{cor:Bi.3.2.1}
and Theorem~\ref{th:6.2} yield the following two results.

\begin{theorem}\label{th:BifE.7.1}
Let the assumptions of either Theorem~\ref{th:BifE.5}(2) or Theorem~\ref{th:BifE.7}
hold
%Under Hypothesis~\ref{hyp:BifE.4}
with $V=V_0$ and $\vec{u}_0={0}$.
% let $\lambda^\ast$ be an  eigenvalue
%of (\ref{e:BifE.3}), and $\mathfrak{K}''_V({0})$ is either semi-positive  or semi-negative.
Suppose also that both ${F}(x,\xi)$ and ${K}(x,\xi)$ are even with respect to $\xi$,
and that $E_{\lambda^\ast}$ denotes the solution space  of  the linear problem (\ref{e:BifE.3}) with $\lambda=\lambda^\ast$
 in $V_0$.
%  has dimension at least two.
  Then  one of the following alternatives holds:
 \begin{description}
\item[(i)] $(\lambda^\ast, {0})$ is not an isolated solution of (\ref{e:BifE.1}) in
 $\{\lambda^\ast\}\times V$;
 \item[(ii)] there exist left and right  neighborhoods $\Lambda^-$ and $\Lambda^+$ of $\lambda_\ast$ in $\mathbb{R}$
and integers $n^+, n^-\ge 0$, such that $n^++n^-\ge\dim E_{\lambda^\ast}$
and for $\lambda\in\Lambda^-\setminus\{\lambda^\ast\}$ (resp. $\lambda\in\Lambda^+\setminus\{\lambda^\ast\}$),
(\ref{e:BifE.1}) has at least $n^-$ (resp. $n^+$) distinct pairs of solutions of form $\{\vec{u},-\vec{u}\}$
 different from ${0}$, which converge to
 ${0}$ as $\lambda\to\lambda^\ast$.
\end{description}
%In particular,
% (\ref{e:BifE.1}) has at least $\dim E_{\lambda^\ast}$ distinct pairs of solutions of form $\{\vec{u},-\vec{u}\}$
%  different from $(\lambda^\ast,{0})$ in any neighborhood of $(\lambda^\ast,{0})\in\mathbb{R}\times V$.
  \end{theorem}

\begin{theorem}\label{th:BifE.7.2}
Let the assumptions of either Theorem~\ref{th:BifE.5} or Theorem~\ref{th:BifE.7}
hold
%Under Hypothesis~\ref{hyp:BifE.4}
with $V=V_0$ and $\vec{u}_0={0}$,
% let $\lambda^\ast$ be an  eigenvalue
%of (\ref{e:BifE.3}), and $\mathfrak{K}''_V({0})$ is either semi-positive  or semi-negative.
and let $E_{\lambda^\ast}$ be the solution space  of  the linear problem (\ref{e:BifE.3}) with $\lambda=\lambda^\ast$
 in $V_0$.
%Under Hypothesis~\ref{hyp:BifE.4} with $V=V_0$ and $\vec{u}_0=\vec{0}$, let $\lambda^\ast$ be an  eigenvalue
%of (\ref{e:BifE.3}), and $\mathfrak{K}''_V(\vec{0})$ is either semi-positive or semi-negative.
%Suppose also that the solution space $E_{\lambda^\ast}$ of  the linear problem  (\ref{e:BifE.3}) with $\lambda=\lambda^\ast$  in $V$ has dimension at least two.
 We have:
\begin{description}
\item[(I)]  If $\Omega$ is symmetric with respect to the origin,
 both $\mathfrak{F}_{V_0}$ and $\mathfrak{K}_{V_0}$ are invariant for the $\mathbb{Z}_2$-action
 in (\ref{e:Z2-action.2}), and (\ref{e:BifE.3}) with $\lambda=\lambda^\ast$ has no nontrivial solutions
 $\vec{u}\in V_0$ satisfying $\vec{u}(-x)=\vec{u}(x)\;\forall x\in\Omega$,
 then the conclusions in  Theorem~\ref{th:BifE.7.1},
 after ``pairs of solutions of form $\{\vec{u},-\vec{u}\}$" being changed into
 ``pairs of solutions of form $\{\vec{u}(\cdot), \vec{u}(-\cdot)\}$", still holds.
\item[(II)] Let $\mathbb{S}^1$ act on $\mathbb{R}^n$ by the orthogonal representation, and $\Omega$
be symmetric under the action, let both $\mathfrak{F}_{V_0}$ and $\mathfrak{K}_{V_0}$ are invariant for the $\mathbb{S}^1$ action on $V_0$
 in (\ref{e:S^1-action}). If the fixed point set of the induced $S^1$-action in $E_{\lambda^\ast}$ is $\{0\}$,
 %not an orbit for the $\mathbb{S}^1$  action,
 then   the conclusions of Theorem~\ref{th:BifE.7.1},
 after ``$n^++n^-\ge\dim E_{\lambda^\ast}$" and ``pairs of solutions of form $\{\vec{u},-\vec{u}\}$" being changed into ``critical $\mathbb{S}^1$-orbits" and ``$n^++n^-\ge\frac{1}{2}\dim E_{\lambda^\ast}$", hold.
 %may be moved to this place.
 \end{description}
 \end{theorem}

\begin{remark}\label{rm:BifE.8}
{\rm By Remark~\ref{rem:BifE.3}, some of the above results also hold if
 $\Omega\subset\mathbb{R}^n$ is replaced by the torus $\mathbb{T}^n=\mathbb{R}^n/\mathbb{Z}^n$.}
\end{remark}

The following is a result associated with Theorems~5.4.2 and 5.7.4 in \cite{EKBB}.

\begin{theorem}\label{th:BifE.9}
 Let $\Omega\subset\R^n$ be a bounded  Sobolev domain, $N\in\mathbb{N}$. Suppose that
 $$
\overline\Omega\times\prod^m_{k=0}\mathbb{R}^{N\times M_0(k)}\times [0, 1]\ni (x,
\xi,\lambda)\mapsto F(x,\xi;\lambda)\in\R
$$
is differentiable with respect to $\lambda$, and satisfies the following conditions:
 \begin{description}
 \item[(i)] All $F(\cdot;\lambda)$ satisfy Hypothesis~$\mathfrak{F}_{2,N,m,n}$ uniformly with respect to $\lambda\in [0,1]$, i.e., the inequalities (\ref{e:6.1}) and (\ref{e:6.2}) are uniformly satisfied for all $\lambda\in [0,1]$.
 \item[(ii)] For $2'_\alpha=1$ in Hypothesis~$\mathfrak{F}_{2,N,m,n}$, it holds that
 $$
\sup_{|\alpha|\le m}\sup_{1\le i\le N}\sup_\lambda\int_\Omega \left[|D_\lambda F(x,0;\lambda)|+ |D_\lambda F^i_{\alpha}(x, 0;\lambda)|^{2'_\alpha}\right]dx<\infty.
$$
 \item[(iii)] For all $i=1,\cdots,N$ and $|\alpha|\le m$,
 \begin{eqnarray*}
&&|D_\lambda F^i_\alpha(x,\xi;\lambda)|\le|D_\lambda F^i_\alpha(x,0;\lambda)|\\
&&+ \mathfrak{g}(\sum^N_{k=1}|\xi_0^k|)\sum_{|\beta|<m-n/2}\bigg(1+
\sum^N_{k=1}\sum_{m-n/2\le |\gamma|\le
m}|\xi^k_\gamma|^{2_\gamma }\bigg)^{2_{\alpha\beta}}\nonumber\\
&&+\mathfrak{g}(\sum^N_{k=1}|\xi^k_0|)\sum^N_{l=1}\sum_{m-n/2\le |\beta|\le m} \bigg(1+
\sum^N_{k=1}\sum_{m-n/2\le |\gamma|\le m}|\xi^k_\gamma|^{2_\gamma }\bigg)^{2_{\alpha\beta}}|\xi^l_\beta|;
\end{eqnarray*}
 where $\mathfrak{g}:[0,\infty)\to\mathbb{R}$ is a continuous, positive, nondecreasing function,
 and is constant if $m<n/2$.
 \end{description}
 Let $V_0$ be a closed subspace of $W^{m,2}(\Omega, \mathbb{R}^N)$ and $V=\vec{w}+V_0$
 for some $\vec{w}\in W^{m,2}(\Omega, \mathbb{R}^N)$. For each $\lambda\in [0,1]$,
  suppose that the functional
  $$
V\ni\vec{u}\mapsto\mathfrak{F}_\lambda(\vec{u})=\int_\Omega F(x, \vec{u},\cdots, D^m\vec{u};\lambda)dx
$$
has a critical point $\vec{u}_\lambda$ such that
 $[0, 1]\ni\lambda\mapsto \vec{u}_\lambda\in V$ is continuous. Then
one of the following alternatives occurs:
\begin{description}
\item[(I)] There exists certain $\lambda_0\in [0,1]$ such that $(\lambda_0, \vec{u}_{\lambda_0})$
           is a bifurcation point of $\nabla\mathfrak{F}_\lambda(\vec{u})=0$.
\item[(II)] Each $\vec{u}_\lambda$ is an isolated critical point of $\mathfrak{F}_\lambda$ and
$C_\ast(\mathfrak{F}_\lambda, \vec{u}_\lambda;{\bf K})=C_\ast(\mathfrak{F}_0, \vec{u}_0;{\bf K})$
for all $\lambda\in [0,1]$; moreover, $\vec{u}_\lambda$ is a local minimizer of $\mathfrak{F}_\lambda$ if and only if $\vec{u}_0$ is a local minimizer of $\mathfrak{F}_0$.
 \end{description}
\end{theorem}

If $D_\lambda F(\cdot;\lambda)$ uniformly satisfy  the inequalities (\ref{e:6.1}) and (\ref{e:6.2})  for all $\lambda\in [0,1]$, (ii) can yield (iii). When $N=1$ and $V=W^{m,2}_0(\Omega, \mathbb{R}^N)$,  (II) and (I), in some sense,
generalize Theorems~5.4.2 and 5.7.4 in \cite{EKBB}, respectively.

\begin{proof}[Proof of Theorem~\ref{th:BifE.9}]
 Suppose that (I) does not hold. Then for each $\mu\in [0,1]$ we have neighborhoods of it and $u_\mu$ in $[0,1]$ and $V$
 respectively,
 $\mathscr{N}_\mu$ and $\mathscr{O}_\mu$, such that $u_\lambda$ is
 a unique critical point of $\mathfrak{F}_\lambda$  in
the closure $\overline{\mathscr{O}_\mu}$ of  $\mathscr{O}_\mu$ for each $\lambda\in\mathscr{N}_\mu$.
 Moreover, we can assume that each  $\mathscr{N}_\mu$ is an interval.
 % each $\vec{u}_\lambda$ is an isolated critical point of $\mathfrak{F}_\lambda$.
   Since $[0, 1]\ni\lambda\mapsto \vec{u}_\lambda\in V$ is continuous,
we can take $R>0$ such that all $u_\lambda$ and $\overline{\mathscr{O}_\mu}$
are contained in $B_V(0, R)$.
%
%find a bounded open subset $\mathscr{O}$ in $V$ such that
%$\vec{u}_\lambda$ is a unique  critical point of $\mathfrak{F}_\lambda$
%contained in the closure $\overline{\mathscr{O}}$ of $\mathscr{O}$.
% $\overline{\mathscr{O}}\subset B_V(0, R)$.

As in the proof of \cite[(4.8)]{Lu6} we can derive from (iii) that with $2'_\alpha$ in (ii),
\begin{eqnarray*}
&&|D_\lambda F(x,\xi;\lambda)|\le |D_\lambda F(x,0;\lambda)|+
\Big(\sum^N_{k=1}|\xi^k_0|\Big)\sum^N_{i=1}\sum_{|\alpha|<m-n/2}|D_\lambda F^i_{\alpha}(x, 0;\lambda)|\\
&&+\sum^N_{i=1}\sum_{m-n/2\le|\alpha|\le m}|D_\lambda F^i_{\alpha}(x, 0;\lambda)|^{q_\alpha}
+ \widehat{\mathfrak{g}}(\sum^N_{k=1}|\xi^k_0|)\bigg(1+\sum^N_{l=1}\sum_{m-n/2\le|\alpha|\le m}|\xi^l_\alpha|^{2_\alpha}\bigg)
\end{eqnarray*}
for all $(x,\xi,\lambda)$ and some  continuous, positive, nondecreasing function
$\widehat{\mathfrak{g}}:[0,\infty)\to\mathbb{R}$. Thus for every given $R>0$, as before we have a constant
$C=C(m,n,N, R)>0$ such that
$$
\sup\bigg\{\sum_{|\alpha|<m-n/2}|D^\alpha\vec{u}(x)|\,\bigg|\, x\in\Omega\bigg\}<C,
\quad\forall \vec{u}\in W^{m,2}(\Omega,\mathbb{R}^N)\;\hbox{with}\;\|\vec{u}\|_{m,2}\le 2R.
$$
It follows that for any $\lambda_i\in [0,1]$, $i=1,2$,
\begin{eqnarray*}
&&|\mathfrak{F}_{\lambda_1}(\vec{u})- \mathfrak{F}_{\lambda_2}(\vec{u})|\le |\lambda_2-\lambda_1|\int_\Omega\sup_\lambda
|D_\lambda F(x, \vec{u},\cdots, D^m\vec{u};\lambda)|dx\\
&&\le |\lambda_2-\lambda_1|\biggl[\sup_\lambda\int_\Omega |D_\lambda F(x,0;\lambda)|dx +
C\sum^N_{i=1}\sum_{|\alpha|<m-n/2}\sup_\lambda\int_\Omega |D_\lambda F^i_{\alpha}(x, 0;\lambda)|dx\\
&&+\sum^N_{i=1}\sum_{m-n/2\le|\alpha|\le m}\sup_\lambda\int_\Omega
|D_\lambda F^i_{\alpha}(x, 0;\lambda)|^{q_\alpha}dx\\
&&\hspace{20mm}+ \widehat{\mathfrak{g}}(C)\int_\Omega\bigg(1+\sum^N_{l=1}\sum_{m-n/2\le|\alpha|\le m}|D^\alpha u^l|^{2_\alpha}\bigg)dx
\biggr].
\end{eqnarray*}
This implies that $[0,1]\ni\lambda\mapsto \mathfrak{F}_\lambda$
is continuous in $C^0(\bar{B}_V(0, 2R))$. Similarly,
 (ii)--(iii) yield the continuity of the map $[0,1]\ni\lambda\mapsto \nabla\mathfrak{F}_\lambda$
 in $C^0(\bar{B}_V(0, 2R), V)$. Hence the map
 $[0,1]\ni\lambda\mapsto \mathfrak{F}_\lambda$
is continuous in $C^1(\bar{B}_V(0, 2R))$.

%Note that the above $R>0$ can be chosen arbitrarily large.
%We may assume that $\vec{u}_\lambda\in B_V(0,R/2)$ for all $\lambda\in [0, 1]$.
Replacing $F^{ij}_{\alpha\beta}(x, \vec{u}(x),\cdots, D^m \vec{u}(x))$
with $F^{ij}_{\alpha\beta}(x, \vec{u}(x),\cdots, D^m \vec{u}(x);\lambda)$
in the definitions of $P$ and $Q$ in Theorem~\ref{th:6.2},
we obtain operators $P_\lambda$ and $Q_\lambda$. Then
$D(\nabla\mathfrak{F}_\lambda)=P_\lambda+Q_\lambda$.
Put
$$
\widehat{\mathfrak{F}}_\lambda(\vec{u}):=\mathfrak{F}_\lambda(\vec{u}-\vec{u}_\lambda),\quad
\widehat{P}_\lambda(\vec{u}):=P_\lambda(\vec{u}-\vec{u}_\lambda),\quad
\widehat{Q}_\lambda(\vec{u}):=Q_\lambda(\vec{u}-\vec{u}_\lambda).
$$
Then the map  $[0,1]\ni\lambda\mapsto\widehat{\mathfrak{F}}_\lambda$
is continuous in $C^1(\bar{B}_V(0, R))$, and  $D(\nabla\widehat{\mathfrak{F}}_\lambda)=\widehat{P}_\lambda+ \widehat{Q}_\lambda$.
By Theorems~\ref{th:6.1}, \ref{th:6.2} and the conditions (i)-(iii) it is direct to prove that
$\widehat{\mathfrak{F}}_\lambda$, $\widehat{B}_\lambda=D(\nabla\widehat{\mathfrak{F}}_\lambda)$,
and $\widehat{P}_\lambda$, $\widehat{Q}_\lambda$ satisfy the conditions of
Proposition~\ref{prop:stablity1}. Hence each functional $\widehat{\mathfrak{F}}_\lambda$ satisfies
the (PS) condition on $\bar{B}_V(0,R)$.

Let $\mathscr{N}_\mu$ and $\mathscr{O}_\mu$ be as in the beginning of the proof.
Shrinking $\mathscr{N}_\mu$ we have $\epsilon>0$ such that
$\vec{u}_\lambda+ B_V(0,\epsilon)\subset \mathscr{O}_\mu$ for all $\lambda\in \mathscr{N}_\mu$.
Then $\widehat{\mathfrak{F}}_\lambda$ has a unique critical point $0$ in $B_V(0,\epsilon)$
 for each $\lambda\in \mathscr{N}_\mu$.
Applying Theorem~\ref{th:stablity1} to the family
$\{\widehat{\mathfrak{F}}_\lambda|_{B_V(0,\epsilon)}\,|\,\lambda\in \mathscr{N}_\mu\}$
we deduce that critical groups  $C_\ast(\mathfrak{F}_\lambda, \vec{u}_\lambda;{\bf K})=
C_\ast(\widehat{\mathfrak{F}}_\lambda, 0;{\bf K})$ are independent of choices of
$\lambda$ in $\mathscr{N}_\mu$. Note that $[0,1]$ may be covered by finitely many
 $\mathscr{N}_\mu$. We arrive at $C_\ast(\mathfrak{F}_\lambda, \vec{u}_\lambda;{\bf K})=C_\ast(\mathfrak{F}_0, \vec{u}_0;{\bf K})$
for all $\lambda\in [0,1]$.

%As in Step~1 of the proof of Theorem~\ref{th:Bi.1.1},
%the stability of critical groups (cf.  \cite[Theorem~III.4]{ChGh} and \cite[Theorem~5.1]{CorH})
%leads to the first claim in (II).

 For the second claim in (II), it suffices to prove that
 $0\in V$ is a local minimizer of  $\widehat{\mathfrak{F}}_0$
 provided $0\in V$ is a local minimizer of $\widehat{\mathfrak{F}}_\lambda$.
 Since  $0\in V$ is an isolated critical point of $\widehat{\mathfrak{F}}_\lambda$,
 by Example~1 in \cite[page 33]{Ch} we have $C_q(\widehat{\mathfrak{F}}_\lambda, 0;{\bf K})=\delta_{q0}{\bf K}$
 for each $q\in\mathbb{N}_0$.   It follows that
 $C_q(\widehat{\mathfrak{F}}_0, 0;{\bf K})=\delta_{q0}{\bf K}$ for each $q\in\mathbb{N}_0$.
 By \cite[Theorem~2.3]{Lu7} (or Theorem~\ref{th:A.3})   this means that the Morse
  index of $\widehat{\mathfrak{F}}_0$ at $0\in V$
 must be zero, and
 $C_q(\widehat{\mathfrak{F}}^\circ_0, 0;{\bf K})=\delta_{q0}{\bf K}$
 for each $q\in\mathbb{N}_0$, where $\widehat{\mathfrak{F}}^\circ_0$ is the finite dimensional reduction
 of $\widehat{\mathfrak{F}}_0$ near $0\in V^0:={\rm Ker}(D(\nabla\mathfrak{F}_\lambda)(\vec{u}_0))$.
   By Example~4 in \cite[page 43]{Ch},  $0\in V^0$ is a local minimizer of
 $\widehat{\mathfrak{F}}^\circ_0$.
  From this and \cite[Theorem~2.2]{Lu7} (or Theorem~\ref{th:A.2} with $\lambda=0$) it follows that
 $0\in V$ must be  a local minimizer of  $\widehat{\mathfrak{F}}_0$.
 \end{proof}
 %\hfill$\Box$\vspace{2mm}

\section{Bifurcations for quasi-linear elliptic systems
without growth restrictions}\label{sec:BifE.3}
\setcounter{equation}{0}

Now let us begin with some applications of results in Section~\ref{sec:BBH}.
For simplicity we only consider the Dirichlet boundary conditions. So
till the end of this subsection, we take the Hilbert space $H:=W^{m,2}_0(\Omega,\mathbb{R}^N)$ with the usual inner product (\ref{e:6.2.2}). The following special case of \cite[Theorem 6.4.8]{Mor} is key for our arguments
in this subsection.

\begin{proposition}\label{prop:BifE.9.1}
  For a real $p\ge 2$ and an integer $k\ge m+\frac{n}{p}$,  let $\Omega\subset\R^n$ be a bounded
   domain with boundary of class $C^{k-1,1}$, $N\in\mathbb{N}$, and let bounded and measurable functions on $\overline{\Omega}$,
    $A^{ij}_{\alpha\beta}$, $i,j=1,\cdots,N$, $|\alpha|,|\beta|\le m$,  fulfill the following conditions:
    \begin{description}
\item[(i)]  $A^{ij}_{\alpha\beta}\in C^{k+|\alpha|-2m-1,1}(\overline{\Omega})$ if $2m-k<|\alpha|\le m$;
\item[(ii)] there exists $c_0>0$ such that
$$
\sum^N_{i,j=1}\sum_{|\alpha|=|\beta|=m}\int_\Omega
A^{ij}_{\alpha\beta}\eta^i_\alpha\eta^j_\beta\ge c_0\sum^N_{i=1}\sum_{|\alpha|=m}|\eta^i_\alpha|^2,\quad\forall \eta\in\mathbb{R}^{N\times M_0(m)}.
$$
\end{description}
Suppose that $\vec{u}=(u^1,\cdots,u^N)\in W^{m,2}_0(\Omega,\mathbb{R}^N)$ and $\lambda\in (-\infty, 0]$ satisfy
$$
\sum^N_{i,j=1}\sum_{|\alpha|,|\beta|\le m}\int_\Omega
(A^{ij}_{\alpha\beta}-\lambda\delta_{ij}\delta_{\alpha\beta})D^\beta u^i\cdot D^\alpha v^j dx=0,\quad\forall v\in W^{m,2}_0(\Omega,\mathbb{R}^N).
$$
Then $\vec{u}\in W^{k,p}(\Omega,\mathbb{R}^N)$. Moreover, for $f_j=\sum_{|\alpha|\le m}(-1)^{|\alpha|}D^\alpha f^j_\alpha$,
where
$$
f^j_\alpha\in\left\{\begin{array}{ll}
 W^{k-2m+|\alpha|,p}(\Omega),&\quad\hbox{if}\;|\alpha|>2m-k,\\
  L^p(\Omega),&\quad\hbox{if}\;|\alpha|\le 2m-k,
  \end{array}\right.
  $$
   if $\vec{u}\in W^{m,2}_0(\Omega,\mathbb{R}^N)$ satisfy
$$
\int_\Omega\sum^N_{j=1}\sum_{|\alpha|\le m}\left[\sum^N_{i=1}\sum_{|\beta|\le m}A^{ij}_{\alpha\beta}D^\beta u^i-f^j_\alpha      \right]D^\alpha v^jdx=0,\quad\forall v\in W^{m,2}_0(\Omega,\mathbb{R}^N),
$$
then there also holds $\vec{u}\in W^{k,p}(\Omega,\mathbb{R}^N)$.
\end{proposition}

\begin{theorem}\label{th:BifE.9.2}
 Let a real $p\ge 2$, integers $k$ and $m$ satisfy $k> m+\frac{n}{p}$,  and let $\Omega\subset\R^n$  be a bounded
   domain with boundary of class $C^{k-1,1}$, $N\in\mathbb{N}$.
   Consider the Banach subspace  of $W^{k,p}(\Omega,\mathbb{R}^N)$,
  $$
  X_{k,p}:=W^{k,p}(\Omega,\mathbb{R}^N)\cap W^{m,2}_0(\Omega,\mathbb{R}^N).
  $$
   (It can be continuously embedded to the space $C^m(\overline{\Omega}, \mathbb{R}^N)$,
   and is dense in $H$).  Let  the functions
   \begin{eqnarray*}%\label{e:BifE.0}
F:\overline\Omega\times\prod^m_{k=0}\mathbb{R}^{N\times M_0(k)}\to \R\quad\hbox{and}\quad
K:\overline\Omega\times\prod^{m-1}_{k=0}\mathbb{R}^{N\times M_0(k)}\to \R
\end{eqnarray*}
%    $F$ and $K$ as in (\ref{e:BifE.0})
    be of class $C^{k-m+3}$,  and let $\vec{u}_0\in C^k(\overline\Omega,\mathbb{R}^N)\cap W^{2,m}_0(\Omega,\mathbb{R}^N)$
  be a common critical point of functionals on $X_{k,p}$  given by
 \begin{eqnarray}\label{e:BifE.8.1}
&&\mathscr{L}_1(\vec{u})=\int_\Omega F(x,\vec{u},D\vec{u},\cdots,D^m\vec{u})dx,\quad\vec{u}\in X_{k,p},\\
&&\mathscr{L}_2(\vec{u})=\int_\Omega K(x,\vec{u},D\vec{u},\cdots,D^{m-1}\vec{u})dx,\quad\vec{u}\in X_{k,p}.\label{e:BifE.8.2}
\end{eqnarray}
Suppose also:
\begin{description}
\item[(a)]  there exists some $c>0$ such that for all $x\in\overline{\Omega}$
and for all $\eta=(\eta^{i}_{\alpha})\in\R^{N\times M_0(m)}$,
  \begin{eqnarray}\label{e:BifE.8.3}
\sum^N_{i,j=1}\sum_{|\alpha|=|\beta|=m}F^{ij}_{\alpha\beta}(x, \vec{u}_0(x),\cdots, D^m \vec{u}_0(x))\eta^i_\alpha\eta^j_\beta\ge
c\sum^N_{i=1}\sum_{|\alpha|= m}(\eta^i_\alpha)^2;
\end{eqnarray}
\item[(b)]  either
\begin{eqnarray*}
 \sum^N_{i,j=1}\sum_{|\alpha|,|\beta|\le m-1}
\int_\Omega K^{ij}_{\alpha\beta}(x, \vec{u}_0(x),\cdots, D^{m-1} \vec{u}_0(x))D^\beta v^j(x)D^\alpha v^i(x)dx\ge 0,\quad\forall v\in H,
 \end{eqnarray*}
or
\begin{eqnarray*}
 \sum^N_{i,j=1}\sum_{|\alpha|,|\beta|\le m-1}
\int_\Omega K^{ij}_{\alpha\beta}(x, \vec{u}_0(x),\cdots, D^{m-1} \vec{u}_0(x))D^\beta v^j(x)D^\alpha v^i(x)dx\le 0,\quad\forall v\in H;
 \end{eqnarray*}
\item[(c)]   $\lambda^\ast\in\mathbb{R}$ is an isolated eigenvalue of
the linear eigenvalue problem  in $H$,
 \begin{eqnarray}\label{e:BifE.3*}
 &&\sum^N_{j=1}\sum_{|\alpha|,|\beta|\le m}
(-1)^{|\alpha|}D^\alpha\bigl[F^{ij}_{\alpha\beta}(x, \vec{u}_0(x),\cdots, D^m \vec{u}_0(x))D^\beta v^j\bigr]
\nonumber\\
&&=\lambda\sum^N_{j=1}\sum_{|\alpha|,|\beta|\le m-1}
(-1)^{|\alpha|}D^\alpha\bigl[K^{ij}_{\alpha\beta}(x, \vec{u}_0(x),\cdots, D^{m-1}\vec{u}_0(x))D^\beta v^j\bigr]\nonumber\\
&&\hspace{30mm}i=1,\cdots,N.
 \end{eqnarray}
\end{description}
 Then $(\lambda^\ast, \vec{u}_0)\in\mathbb{R}\times X_{k,p}$ is a bifurcation point of
 \begin{eqnarray}\label{e:BifE.8.4}
d\mathscr{L}_1(\vec{u})=\lambda d\mathscr{L}_2(\vec{u}),\quad(\lambda,\vec{u})\in\mathbb{R}\times X_{k,p},
 \end{eqnarray}
  and  one of the following alternatives occurs:
\begin{description}
\item[(i)] $(\lambda^\ast, \vec{u}_0)$ is not an isolated solution of (\ref{e:BifE.8.4}) in
 $\{\lambda^\ast\}\times X_{k,p}$;

\item[(ii)]  for every $\lambda\in\mathbb{R}$ near $\lambda^\ast$ there is a nontrivial solution $\vec{u}_\lambda$ of (\ref{e:BifE.8.4}) converging to $\vec{u}_0$ as $\lambda\to\lambda^\ast$;

\item[(iii)] there is an one-sided  neighborhood $\Lambda$ of $\lambda^\ast$ such that
for any $\lambda\in\Lambda\setminus\{\lambda^\ast\}$,
(\ref{e:BifE.8.4}) has at least two nontrivial solutions converging to
$\vec{u}_0$ as $\lambda\to\lambda^\ast$.
\end{description}
\end{theorem}
\begin{proof}
For $s<0$ let $W^{s,p}(\Omega, \mathbb{R}^N)=[W_0^{-s,p'}(\Omega, \mathbb{R}^N)]^\ast$ as usual, where $p'=p/(p-1)$. Note that the $m$th power of the Laplace operator, $\triangle^m: X_{k,p}\to W^{k-2m,p}(\Omega, \mathbb{R}^N)$,  is an isomorphism,
and that its inverse, denoted by $\triangle^{-m}$, is from $W^{k-2m,p}(\Omega, \mathbb{R}^N)$ to $X_{k,p}$.

Since the functions $F$ and $K$ are of class $C^{k-m+3}$, and
  $X_{k,p}\hookrightarrow C^m(\overline{\Omega}, \mathbb{R}^N)$ is continuous,  using $\omega$-lemma (cf. \cite[Lemma~2.96]{Wend}) we had proved below \cite[Theorem~4.20]{Lu7}:\\
   {\bf I)} The functionals
  $\mathscr{L}_1$ and $\mathscr{L}_2$ are $C^{k-m+3}$.\\
   {\bf II)} The operators $A_1, A_2: X_{k,p}\to X_{k,p}$ given by
 \begin{eqnarray}\label{e:BifE.8.5}
 &&(A_1(\vec{u}))^i=\triangle^{-m}\sum_{|\alpha|\le m}(-1)^{m+|\alpha|}D^\alpha(F^i_\alpha(\cdot,
\vec{u}(\cdot),\cdots, D^m \vec{u}(\cdot))),\\
&&(A_2(\vec{u}))^i=\triangle^{-m}\sum_{|\alpha|\le m-1}(-1)^{m+|\alpha|}D^\alpha(K^i_\alpha(\cdot,
\vec{u}(\cdot),\cdots, D^{m-1} \vec{u}(\cdot)))\label{e:BifE.8.6}
 \end{eqnarray}
 for $i=1,\cdots,N$,
are of class $C^2$,  thus locally uniformly continuously differentiable, and
  maps $B_1, B_2: X_{k,p}\to\mathscr{L}_s(X_{k,p})$ given by
\begin{eqnarray}\label{e:BifE.8.7}
&&(B_1(\vec{u})\vec{v})^i= \triangle^{-m}\sum^N_{j=1}\sum_{\scriptsize\begin{array}{ll}
   &|\alpha|\le m,\\
   &|\beta|\le m
   \end{array}}(-1)^{m+|\alpha|}
D^\alpha(F^{ij}_{\alpha\beta}(\cdot, \vec{u}(\cdot),\cdots, D^m \vec{u}(\cdot))D^\beta v^j),\\
&&(B_2(\vec{u})\vec{v})^i= \triangle^{-m}\sum^N_{j=1}\sum_{\scriptsize\begin{array}{ll}
   &|\alpha|\le m-1,\\
   &|\beta|\le m-1
   \end{array}}(-1)^{m+|\alpha|}
D^\alpha(K^{ij}_{\alpha\beta}(\cdot, \vec{u}(\cdot),\cdots, D^{m-1} \vec{u}(\cdot))D^\beta v^j)\nonumber\\
\label{e:BifE.8.8}
\end{eqnarray}
 for $i=1,\cdots,N$,
are of class $C^1$.\\
{\bf III)} $B_1$ and $B_2$ as maps from $X_{k,p}$ to $\mathscr{L}_s(H)$
are uniformly continuous on any bounded subsets of $X_{k,p}$ because the above two equalities
can be used to extend $B_1(\vec{u})$ and $B_2(\vec{u})$ into operators in $\mathscr{L}(H)$, also denoted by
$B_1(\vec{u})$ and $B_2(\vec{u})$, see Claim~4.17 in \cite{Lu7}.

It is easily checked that
 \begin{eqnarray*}
 &&d\mathscr{L}_1(\vec{u})[\vec{v}]=(A_1(\vec{u}), \vec{v})_H,\quad dA_1(\vec{u})[\vec{v}]=B_1(\vec{u})\vec{v},\quad\forall
 \vec{u},\vec{v}\in X_{k,p},\\
 &&d\mathscr{L}_2(\vec{u})[\vec{v}]=(A_2(\vec{u}), \vec{v})_H,\quad dA_2(\vec{u})[\vec{v}]=B_2(\vec{u})\vec{v},\quad\forall
 \vec{u},\vec{v}\in X_{k,p}.
 \end{eqnarray*}
These show that $\mathscr{L}_1, \mathscr{L}_2:X_{k,p}\to\mathbb{R}$ are two $(B(\vec{u},r), H)$-regular functionals
 for some ball $B(\vec{u},r)\subset X_{k,p}$ centred at any given $\vec{u}\in X_{k,p}$.
(See Appendix~\ref{app:B} for the definition).

For each $\vec{u}\in X_{k,p}$, we may write $B_1(\vec{u})=P_1(\vec{u})+ Q_1(\vec{u})$, where
\begin{eqnarray}\label{e:BifE.8.9}
(P_1(\vec{u})\vec{v},\vec{w})_{m,2}&=& \sum^N_{i,j=1}\sum_{|\alpha|=|\beta|=m}
\int_\Omega F^{ij}_{\alpha\beta}(x, \vec{u}(x),\cdots, D^m \vec{u}(x))D^\beta v^j(x) D^\alpha w^i(x)dx\nonumber\\
&&+\sum^N_{i=1}\int_\Omega v^i(x) w^i(x)dx,\\
(Q_1(\vec{u})\vec{v},\vec{w})_{m,2}&=&\sum^N_{i,j=1}\sum_{\scriptsize\begin{array}{ll}
   &|\alpha|\le m, |\beta|\le m,\\
   &|\alpha|+|\beta|<2m
   \end{array}}   F^{ij}_{\alpha\beta}(x, \vec{u}(x),\cdots, D^m \vec{u}(x))D^\beta v^j(x)D^\alpha w^i(x)dx\nonumber\\
 &&-\sum^N_{i=1}\int_\Omega v^i(x) w^i(x)dx.\label{e:BifE.8.10}
\end{eqnarray}
Clearly, both $P_1(\vec{u})$ and $Q_1(\vec{u})$ are in $\mathscr{L}_s(H)$,
and $P_1(\vec{u})|_{X_{k,p}}, Q_1(\vec{u})|_{X_{k,p}}\in \mathscr{L}(X_{k,p})$.
Moreover, $Q_1(\vec{u})$ and $B_2(\vec{u})$ as operators in $\mathscr{L}_s(H)$ are compact.
By (\ref{e:BifE.8.3}), $P_1(\vec{u}_0)\in \mathscr{L}_s(H)$ is positive definite.
The assumption {\bf (b)} shows that $B_2(\vec{u}_0)$ as operators in $\mathscr{L}_s(H)$
is either semi-positive or semi-negative.
Note that the linear eigenvalue problem (\ref{e:BifE.3*}) with $v\in H$ is equivalent to
\begin{eqnarray}\label{e:BifE.8.11}
B_1(\vec{u}_0)\vec{v}=\lambda B_2(\vec{u}_0)\vec{v},\quad(\lambda,\vec{v})\in\mathbb{R}\times H.
\end{eqnarray}
By {\bf (c)}, $0\in\sigma(B_1(\vec{u}_0)-\lambda^\ast B_2(\vec{u}_0))$. Since
$B_1(\vec{u}_0)-\lambda^\ast B_2(\vec{u}_0)=P_1(\vec{u}_0)+ (Q_1(\vec{u}_0)-\lambda^\ast B_2(\vec{u}_0))$
is a sum of a positive operator and a compact self-adjoint one, and so Fredholm, we obtain
%Then Lemma~\ref{lem:BB.7} leads to:
\begin{claim}\label{cl:BifE.9.2+}
 $\lambda^\ast\in\mathbb{R}$ is an isolated eigenvalue of
  (\ref{e:BifE.8.11}) of the finite multiplicity.
\end{claim}

For each fixed $\lambda\in\mathbb{R}$, it is not hard to prove that
\begin{eqnarray}\label{e:BifE.8.12}
A^{ij}_{\alpha\beta}:=F^{ij}_{\alpha\beta}(\cdot, \vec{u}_0(\cdot),\cdots, D^m \vec{u}_0(\cdot))-\lambda
K^{ij}_{\alpha\beta}(\cdot, \vec{u}_0(\cdot),\cdots, D^{m-1} \vec{u}_0(\cdot))
\end{eqnarray}
 satisfy the conditions of Proposition~\ref{prop:BifE.9.1}.
Hence the solution spaces of (\ref{e:BifE.8.11}) are contained in $X_{k,p}$.
\textsf{These imply that  Hypothesis~\ref{hyp:BBH.1}(ii) is satisfied.}

It remains to check that  Hypothesis~\ref{hyp:BBH.1}(i) holds.
Consider the complexification of  $H$ and $X_{k,p}$,
$H^{\mathbb{C}}=H+iH$ and $X^{\mathbb{C}}_{k,p}=X_{k,p}+iX_{k,p}$ (cf., \cite[p.14]{DaKr}).
 The inner product $(\cdot,\cdot)_H$ in (\ref{e:6.2.2}) is extended
 into a Hermite inner product $\langle\cdot,\cdot\rangle_H$ on $H^{\mathbb{C}}$
 in natural ways. Let $\mathbb{B}^C_\lambda$ be the natural complex linear extension on $H^{\mathbb{C}}$ of
 $\mathbb{B}_\lambda:=B_1(\vec{u}_0)-\lambda B_2(\vec{u}_0)$. Then
 $\mathbb{B}^C_\lambda|_{X^{\mathbb{C}}_{k,p}}$ is such an extension on
 $X^{\mathbb{C}}_{k,p}$ of $\mathbb{B}_\lambda|_{X^{\mathbb{C}}_{k,p}}$.
 Both are symmetric with respect to $\langle\cdot,\cdot\rangle_H$.

Let $\vec{g},\vec{h}\in X_{k,p}$. Since $\mathbb{B}^C_\lambda$ is a self-adjoint operator on $H^{\mathbb{C}}$,
for every $\tau\in\mathbb{R}\setminus\{0\}$ there exist unique $\vec{u},\vec{v}\in H$ such that
$\mathbb{B}^C_\lambda(\vec{u}+i\vec{v})-i\tau(\vec{u}+i\vec{v})=\vec{g}+i\vec{h}$,
%%%%%%%%%%%%%%%%%%%%%%%%%%%%%%%%%%%%%%%%%%%%%%%%%%%%%%%%%%%%%%%%%%%%%%%%%%%%%%%%%%%%%
%%$$
%%(B_1(\vec{u}_0)\vec{v}-\lambda B_2(\vec{u}_0))\vec{u}+ i(B_1(\vec{u}_0)\vec{v}-\lambda B_2(\vec{u}_0))\vec{v}
%%-i\tau(\vec{u}+i\vec{v})=\vec{g}+i\vec{h},
%%$$
%%%%%%%%%%%%%%%%%%%%%%%%%%%%%%%%%%%%%%%%%%%%%%%%%%%%%%%%%%%%%%%%%%%%%%%%%
%%$(B_1(\vec{u}_0)\vec{v}-\lambda B_2(\vec{u}_0))\vec{u}+\tau\vec{v}=\vec{g}$
%%and $(B_1(\vec{u}_0)\vec{v}-\lambda B_2(\vec{u}_0))\vec{v}-\tau\vec{u}=\vec{h}$,
%%%%%%%%%%%%%%%%%%%%%%%%%%%%%%%%%%%%%%%%%%%%%%%%%%%%%%%%%%%%%%%%%%%%%%%%%%%%%%%%
which is equivalent to
\begin{eqnarray}\label{e:BifE.8.13}
\left.\begin{array}{ll}
&\sum^N_{k=1}\sum_{|\alpha|\le m}\sum_{|\beta|\le m}(-1)^{|\alpha|}D^\alpha (A^{jk}_{\alpha\beta}D^\beta u^k)+ \tau\triangle^mv^j=g^j,\\
&-\tau\triangle^mu^j+\sum^N_{k=1}\sum_{|\alpha|\le m}\sum_{|\beta|\le m}(-1)^{|\alpha|}D^\alpha (A^{jk}_{\alpha\beta}D^\beta v^k)=h^j,\\
&j=1,\cdots,N,
\end{array}\right\}
\end{eqnarray}
where $A^{ij}_{\alpha\beta}$ are as in (\ref{e:BifE.8.12}), $\triangle^mv^j=(-1)^m\sum_{|\alpha|=m}D^\alpha(D^\alpha v^j)$ and
$$
g^j=\sum_{|\alpha|\le m}(-1)^{|\alpha|}D^\alpha g^j_\alpha,\quad
h^j=\sum_{|\alpha|\le m}(-1)^{|\alpha|}D^\alpha h^j_\alpha.
$$
Define
\begin{eqnarray*}
&&\hat{A}^{jk}_{\alpha\beta}=A^{jk}_{\alpha\beta},\quad j,k=1,\cdots,N,\\
&&\hat{A}^{jk}_{\alpha\beta}=\tau\delta_{\alpha\beta}\delta_{j(k-N)},\quad j=1,\cdots,N,\;k=N+1,\cdots,2N,\\
&&\hat{A}^{jk}_{\alpha\beta}=-\tau\delta_{\alpha\beta}\delta_{(j-N)k},\quad j=N+1,\cdots,2N,\;k=1,\cdots,N,\\
&&\hat{A}^{jk}_{\alpha\beta}=A^{(j-N)(k-N)}_{\alpha\beta},\quad j,k=N+1,\cdots,2N,\\
&&\hat f^j_\alpha=g^j_\alpha,\;j=1,\cdots,N,\quad
\hat f^j_\alpha=h^{j-N}_\alpha,\;j=N+1,\cdots,2N,\\
&&\hat w^j_\alpha=u^j_\alpha,\;j=1,\cdots,N,\quad
\hat w^j_\alpha=v^{j-N}_\alpha,\;j=N+1,\cdots,2N.
\end{eqnarray*}
Then (\ref{e:BifE.8.13}) is equivalent to:
\begin{eqnarray}\label{e:BifE.8.14}
\sum^{2N}_{k=1}\sum_{|\alpha|\le m}\sum_{|\beta|\le m}(-1)^{|\alpha|}D^\alpha (\hat{A}^{jk}_{\alpha\beta}D^\beta \hat{w}^k)= \hat{f}^j,\quad
j=1,\cdots,2N.
\end{eqnarray}
 Note that (\ref{e:BifE.8.3}) implies
\begin{eqnarray*}
\sum^{2N}_{j,k=1}\sum_{|\alpha|=m=|\beta|}\hat{A}^{jk}_{\alpha\beta}\xi^j_\alpha\xi^k_\beta
&=&\sum^{N}_{j,k=1}\sum_{|\alpha|=m=|\beta|}{A}^{jk}_{\alpha\beta}\xi^j_\alpha\xi^k_\beta
+\sum^{N}_{j,k=1}\sum_{|\alpha|=m=|\beta|}{A}^{jk}_{\alpha\beta}\xi^{j+N}_\alpha\xi^{k+N}_\beta\\
&\ge& c\sum^N_{j=1}\sum_{|\alpha|=m}|\xi^j_\alpha|^2+ c\sum^N_{j=1}\sum_{|\alpha|=m}|\xi^{N+j}_\alpha|^2.
\end{eqnarray*}
Applying Proposition~\ref{prop:BifE.9.1} to (\ref{e:BifE.8.14}) we get
$\vec{\hat{w}}\in W^{k,p}(\Omega,\mathbb{R}^{2N})$, and thus
$\vec{u},\vec{v}\in X_{k,p}$.
Hence $i\tau$ is a regular value of $\mathbb{B}^C_\lambda|_{X^{\mathbb{C}}_{k,p}}$.

By Claim~\ref{cl:BifE.9.2+}, for each $\lambda$ near $\lambda^\ast$,  $\lambda\ne\lambda^\ast$,
 $\mathbb{B}^C_\lambda:H^{\mathbb{C}}\to H^{\mathbb{C}}$ is an isomorphism.
 Thus we may take $\tau=0$ in the arguments above. This shows that
 $0$ is also a regular value of $\mathbb{B}^C_\lambda|_{X^{\mathbb{C}}_{k,p}}$.
 Hence the  spectrum of $\mathbb{B}^C_\lambda|_{X^{\mathbb{C}}_{k,p}}$
 is bounded away from the imaginary axis.

Claim~\ref{cl:BifE.9.2+} implies that $N:={\rm Ker}(\mathbb{B}^C_{\lambda^\ast})={\rm Ker}(\mathbb{B}^C_{\lambda^\ast}|_{X^{\mathbb{C}}_{k,p}})$
 is of finite dimension.   Denote by $Y$ the intersection
of $X_{k,p}^{\mathbb{C}}$ and the orthogonal complement of $N$ in $(H^{\mathbb{C}},\langle\cdot,\cdot\rangle_H)$.
It is an invariant subspace of $\mathbb{B}^C_{\lambda^\ast}|_{X^{\mathbb{C}}_{k,p}}$,
and there exists a direct sum decomposition of Banach spaces, $X_{k,p}^{\mathbb{C}}=N\oplus Y$.
 The above arguments imply that
 the  spectrum of the restriction of $\mathbb{B}^C_{\lambda^\ast}|_{X^{\mathbb{C}}_{k,p}}$ to $Y$
  is bounded away from the imaginary axis. Hence
   $\sigma\left(\mathbb{B}^C_{\lambda^\ast}|_{X^{\mathbb{C}}_{k,p}}\right)\setminus\{0\}$
   has a positive distance to the imaginary axis.

 To summarize what we have proved, \textsf{Hypothesis~\ref{hyp:BBH.1}(i) holds.}
 Hence the desired conclusions follow from Corollary~\ref{cor:BBH.2}.
\end{proof}

% An alternate method is to prove
%(A) and (B) in Remark~\ref{rem:BBH.12} for $P_1=P_1(\vec{u}_0)$ and $Q_1=Q_1(\vec{u}_0)$.
%It is left to the interesting reader.

Similarly, by Corollary~\ref{cor:BBH.3} we may obtain

\begin{theorem}\label{th:BifE.9.3}
The conclusions  of  Theorem~\ref{th:BifE.9.2} also hold true if the conditions {\bf (a)}--{\bf (c)} in Theorem~\ref{th:BifE.9.2}
%are replaced by
%\begin{description}
%\item[(A)] the operator $B_1(\vec{u}_0), B_2(\vec{u}_0)\in\mathscr{L}_s(H)$ defined by (\ref{e:BifE.8.7}) and (\ref{e:BifE.8.8})
%commute and $B_1(\vec{u}_0)$ is invertiable;
%\item[(B)] $\lambda^\ast\in\mathbb{R}$ is an  eigenvalue of
%the linear eigenvalue problem (\ref{e:BifE.3*}) in $H$,
%and on the corresponding eigenspace  $B_1(\vec{u}_0)$ is either positive or negative.
%\end{description}
%Moreover, the conditions {\bf (a)}--{\bf (c)} in Theorem~\ref{th:BifE.9.2}
are replaced by the following
\begin{description}
\item[(a')]  there exists some $c>0$ such that for all $x\in\overline{\Omega}$
and for all $\eta=(\eta^{i}_{\alpha})\in\R^{N\times M_0(m)}$,
  \begin{eqnarray}\label{e:BifE.8.3}
\sum^N_{i,j=1}\sum_{|\alpha|\le m, |\beta|\le m}F^{ij}_{\alpha\beta}(x, \vec{u}_0(x),\cdots, D^m \vec{u}_0(x))\eta^i_\alpha\eta^j_\beta\ge
c\sum^N_{i=1}\sum_{|\alpha|\le m}(\eta^i_\alpha)^2.
\end{eqnarray}
\end{description}
\end{theorem}

%Clearly, if $B_1(\vec{u}_0)$ is  positive definite and commutes with $B_2(\vec{u}_0)$,
%then (A)-(B) are satisfied.

%\begin{theorem}\label{th:BifE.9.4}
%Theorem~\ref{th:BifE.9.3} still holds if ({\bf ab}') therein is replaced by
% \begin{description}
%\item[(b'')]
% the operator $B_1(\vec{u}_0)\in\mathscr{L}_s(H)$ defined by (\ref{e:BifE.8.7}) is invertible,
%each eigensubspace  of  (\ref{e:BifE.8.11}) is invariant for $B_1(\vec{u}_0)$,
%and  $B_1(\vec{u}_0)$ is either positive or negative on such an eigensubspace.
%\end{description}
%%; \\
%%{\bf (c')}  $\lambda^\ast\in\mathbb{R}$ is an eigenvalue of
%%the problem (\ref{e:BifE.8.11}).
%\end{theorem}

By Corollary~\ref{th:BBH.7} we immediately obtain the following two results.

\begin{theorem}\label{th:BifE.9.5}
Let the assumptions of one of Theorems~\ref{th:BifE.9.2} and \ref{th:BifE.9.3}
with $\vec{u}_0=0$ be satisfied.
Suppose also that both ${F}(x,\xi)$ and ${K}(x,\xi)$ are even with respect to $\xi$,
and that  $E_{\lambda^\ast}$ denotes the solution space of (\ref{e:BifE.3*})  with $\lambda=\lambda^\ast$.
Then  one of the following alternatives occurs:
 \begin{description}
\item[(i)] $(\lambda^\ast, {0})$ is not an isolated solution of (\ref{e:BifE.8.4}) in
 $\{\lambda^\ast\}\times X_{k,p}$.
 \item[(ii)] There exist left and right  neighborhoods $\Lambda^-$ and $\Lambda^+$ of $\lambda_\ast$ in $\mathbb{R}$
and integers $n^+, n^-\ge 0$, such that $n^++n^-\ge\dim E_{\lambda^\ast}$
and for $\lambda\in\Lambda^-\setminus\{\lambda^\ast\}$ (resp. $\lambda\in\Lambda^+\setminus\{\lambda^\ast\}$),
(\ref{e:BifE.8.4}) has at least $n^-$ (resp. $n^+$) distinct pairs of solutions of form $\{\vec{u},-\vec{u}\}$
 different from ${0}$, which converge to
 ${0}$ as $\lambda\to\lambda^\ast$.
\end{description}
\end{theorem}

\begin{theorem}\label{th:BifE.9.6}
Let the assumptions of one of Theorems~\ref{th:BifE.9.2} and \ref{th:BifE.9.3}
with $\vec{u}_0=0$ be satisfied, and let $E_{\lambda^\ast}$ be  the solution space  of  (\ref{e:BifE.8.11})
 with $\lambda=\lambda^\ast$. Then we have:
\begin{description}
\item[(I)]  If $\Omega$ is symmetric with respect to the origin,
 both $\mathscr{L}_1$ and $\mathscr{L}_2$ are invariant for the $\mathbb{Z}_2$-action
 in (\ref{e:Z2-action.2}), and (\ref{e:BifE.3*}) with $\lambda=\lambda^\ast$ has no nontrivial solutions
 $\vec{u}\in H=W^{m,2}_0(\Omega,\mathbb{R}^N)$ satisfying $\vec{u}(-x)=\vec{u}(x)\;\forall x\in\Omega$,
 then the conclusions in  Theorem~\ref{th:BifE.9.5},
 after ``pairs of solutions of form $\{\vec{u},-\vec{u}\}$" being changed into
 ``pairs of solutions of form $\{\vec{u}(\cdot), \vec{u}(-\cdot)\}$", still holds.
\item[(II)] Let $\mathbb{S}^1$ act on $\mathbb{R}^n$ by the orthogonal representation, let $\Omega$
be symmetric under the action,  both $\mathscr{L}_1$ and $\mathscr{L}_2$  be invariant for the $\mathbb{S}^1$ action on $H=W^{m,2}_0(\Omega,\mathbb{R}^N)$  in (\ref{e:S^1-action}).
If the fixed point set of the induced $S^1$-action in $E_{\lambda^\ast}$ is $\{0\}$,
 %not an orbit for the $\mathbb{S}^1$  action,
 then   the conclusions of  Theorem~\ref{th:BifE.9.5},
 after ``$n^++n^-\ge\dim E_{\lambda^\ast}$" and ``pairs of solutions of form $\{\vec{u},-\vec{u}\}$" being changed into ``critical $\mathbb{S}^1$-orbits" and ``$n^++n^-\ge\frac{1}{2}\dim E_{\lambda^\ast}$", hold.
 \end{description}
\end{theorem}

\begin{remark}\label{rem:BifE.9.6+}
{\rm {\bf (i)}. For Theorem~\ref{th:BifE.9.6}(I), if both $\mathscr{L}_1$ and $\mathscr{L}_2$ are invariant for the $\mathbb{Z}_2$-action
 in (\ref{e:Z2-action.2}), and both ${F}(x,\xi)$ and ${K}(x,\xi)$ are even with respect to $\xi$,
then $\mathscr{L}_1$ and $\mathscr{L}_2$ are invariant for the product $\mathbb{Z}_2\times \mathbb{Z}_2$-action given by
(\ref{e:Z2-action.1}) and (\ref{e:Z2-action.2}),
\begin{eqnarray*}
 ([0],[0])\cdot\vec{u}=\vec{u},\quad ([1],[0])\cdot\vec{u}=-\vec{u},\quad ([0],[1])\cdot\vec{u}=\vec{u}(-\cdot),\quad
 ([1],[1])\cdot\vec{u}=-\vec{u}(-\cdot)
\end{eqnarray*}
for all $\vec{u}\in W^{m,2}(\Omega,\mathbb{R}^N)$. Since the $\mathbb{Z}_2$-action given by
(\ref{e:Z2-action.1}) has no nonzero fixed point, so is this product action.
By Remark~\ref{rem:Bif.3.4} and Corollary~\ref{th:BBH.7}
 the conclusions in  Theorem~\ref{th:BifE.9.5},
 after ``pairs of solutions of form $\{\vec{u},-\vec{u}\}$" being changed into
 `` critical $(\mathbb{Z}_2\times \mathbb{Z}_2)$-orbits", still holds.

\noindent{\bf (ii)}. For Theorem~\ref{th:BifE.9.6}(II), if both $\mathscr{L}_1$ and $\mathscr{L}_2$ are invariant
for the $\mathbb{S}^1$ action on $H=W^{m,2}_0(\Omega,\mathbb{R}^N)$  in (\ref{e:S^1-action}),
and both ${F}(x,\xi)$ and ${K}(x,\xi)$ are even with respect to $\xi$,
then $\mathscr{L}_1$ and $\mathscr{L}_2$ are invariant for the product $(\mathbb{S}^1\times \mathbb{Z}_2)$-action given by
(\ref{e:S^1-action}) and (\ref{e:Z2-action.1}),
\begin{eqnarray*}
 ([t],[0])\cdot\vec{u}(x)=\vec{u}([t]\cdot x),\quad
 ([t],[1])\cdot\vec{u}=-\vec{u}([t]\cdot x), \quad\forall [t]\in \mathbb{S}^1,\quad x\in\overline{\Omega}.
\end{eqnarray*}
This action has no nonzero fixed point yet.
By Remark~\ref{rem:Bif.3.4} and Corollary~\ref{th:BBH.7}
 the conclusions of  Theorem~\ref{th:BifE.9.5},
 after ``$n^++n^-\ge\dim E_{\lambda^\ast}$" and ``pairs of solutions of form $\{\vec{u},-\vec{u}\}$" being changed into ``critical $(\mathbb{S}^1\times \mathbb{Z}_2)$-orbits" and ``$n^++n^-\ge\frac{1}{2}\dim E_{\lambda^\ast}$", hold.
}
\end{remark}

\begin{corollary}\label{cor:BifE.9.7}
 For positive integers  $n>1$ and $m,N$, let real $p\ge 2$ be such that $m>1+\frac{n}{p}$,
 and let $\Omega\subset\R^n$  be a bounded
   domain with boundary of class $C^{m-1,1}$. Suppose that $C^{m+2}$ functions
$G:\overline{\Omega}\times\mathbb{R}^N\to\mathbb{R}$ and
$A^{ij}_{kl}=A^{ji}_{lk}:\overline{\Omega}\times\mathbb{R}^N\to\mathbb{R}$, $i,j=1,\cdots,N$
and $k,l=1,\cdots,n$,  satisfy  the following conditions:
 \begin{description}
\item[A.1)] % the functions $A_{kl}^{ij}$ and $G$    are of class $C^{m+2}$, and
   $\nabla_{\xi}G(x,0)=0$;
 \item[A.2)]  there exists $c>0$ such that
for $a.e.\;x\in\Omega$ and all $(\eta_k^i)\in\mathbb{R}^{N\times n}$,
\begin{eqnarray*}
\sum^N_{i,j=1}\sum^n_{k,l=1}A_{kl}^{ij}(x, 0)\eta^i_k\eta^j_l\ge
c\sum^N_{i=1}\sum^n_{k=1}(\eta^i_k)^2;
\end{eqnarray*}
    \item[A.3)] $\lambda^\ast$ is an eigenvalue of
    \begin{eqnarray}\label{e:BifE.8.14.1}
&&\int_\Omega\sum^N_{i,j=1}\sum^n_{k,l=1}A_{kl}^{ij}(x,0)D_ku^i D_lv^j dx-\int_\Omega \sum^N_{i,j=1}
\frac{\partial^2 G}{\partial\xi_i\partial\xi_j}(x,0)u^iv^j dx\nonumber\\
&&=\lambda\int_\Omega \sum^N_{i=1}u^iv^idx,\quad\forall \vec{v}\in W^{1,2}_0(\Omega,\mathbb{R}^N).
\end{eqnarray}
   \end{description}
Then with  the Banach subspace of  $W^{m,p}(\Omega,\mathbb{R}^N)$,  $X_{m,p}:=W^{m,p}(\Omega,\mathbb{R}^N)\cap W^{1,2}_0(\Omega,\mathbb{R}^N)$,
    $(\lambda^\ast, 0)\in\mathbb{R}\times X_{m,p}$ is a bifurcation point of the quasilinear eigenvalue problem
  \begin{eqnarray}\label{e:BifE.8.15}
&&-\sum^N_{i=1}\sum^n_{k,l=1}D_l(A_{kl}^{ij}(x,\vec{u})D_ku^i)
+\frac{1}{2}\sum^N_{i,r=1}\sum^n_{k,l=1}\nabla_{\xi_j}A^{ir}_{kl}(x,\vec{u}) D_ku^iD_lu^r-
\nabla_{\xi_j}G(x,\vec{u})\nonumber\\
&&=\lambda u^j\hspace{5mm}\quad\hbox{in}\;\Omega,\quad j=1,\cdots,N,\\
&&\vec{u}=0\quad\hbox{on}\;\partial\Omega,\label{e:BifE.8.16}
\end{eqnarray}
 and  one of the following alternatives occurs:
\begin{description}
\item[(i)] $(\lambda^\ast, 0)$ is not an isolated solution of (\ref{e:BifE.8.15})--(\ref{e:BifE.8.16}) in
 $\{\lambda^\ast\}\times X_{m,p}$;

\item[(ii)]  for every $\lambda\in\mathbb{R}$ near $\lambda^\ast$ there is a nontrivial solution $\vec{u}_\lambda$ of (\ref{e:BifE.8.15})--(\ref{e:BifE.8.16}) converging to $0$ as $\lambda\to\lambda^\ast$;

\item[(iii)] there is an one-sided  neighborhood $\Lambda$ of $\lambda^\ast$ such that
for any $\lambda\in\Lambda\setminus\{\lambda^\ast\}$,
(\ref{e:BifE.8.15})--(\ref{e:BifE.8.16}) has at least two nontrivial solutions converging to
$0$ as $\lambda\to\lambda^\ast$.
\end{description}
Moreover, if both $A^{ij}_{kl}(x,\xi)$ and $G(x,\xi)$ are even with respect to $\xi$,
and  $E_{\lambda^\ast}$ is the solution space of (\ref{e:BifE.8.14.1})  with $\lambda=\lambda^\ast$,
 then  one of the following alternatives holds:
 \begin{description}
\item[(iv)] $(\lambda^\ast, {0})$ is not an isolated solution of (\ref{e:BifE.8.15})--(\ref{e:BifE.8.16}) in
 $\{\lambda^\ast\}\times X_{m,p}$;
 \item[(v)] there exist left and right  neighborhoods $\Lambda^-$ and $\Lambda^+$ of $\lambda^\ast$ in $\mathbb{R}$
and integers $n^+, n^-\ge 0$, such that $n^++n^-\ge\dim E_{\lambda^\ast}$
and for $\lambda\in\Lambda^-\setminus\{\lambda^\ast\}$ (resp. $\lambda\in\Lambda^+\setminus\{\lambda^\ast\}$),
(\ref{e:BifE.8.15})--(\ref{e:BifE.8.16}) has at least $n^-$ (resp. $n^+$) distinct pairs of solutions of form $\{\vec{u},-\vec{u}\}$  different from ${0}$, which converge to
 ${0}$ as $\lambda\to\lambda^\ast$.
\end{description}
\end{corollary}

\begin{proof}
Note that  A.2) implies  that each eigenvalue of (\ref{e:BifE.8.14.1}) is isolated (cf. \cite[\S6.5]{Mor}).
The claims above ``Moreover" can be obtained by
 applying Theorem~\ref{th:BifE.9.2} to the Hilbert space $H=W^{1,2}_0(\Omega,\mathbb{R}^N)$,  the Banach space
   $X_{m,p}:=W^{m,p}(\Omega,\mathbb{R}^N)\cap W^{1,2}_0(\Omega,\mathbb{R}^N)$ and $\vec{u}_0=0$,
  and functions
   \begin{eqnarray*}
  && F(x,\xi_1,\cdots,\xi_N;\eta^i_k)=\frac{1}{2}\sum^N_{i,j=1}\sum^n_{k,l=1}
A_{kl}^{ij}(x,\xi_1,\cdots,\xi_N)\eta_k^i\eta_l^j-G(x,\xi_1,\cdots,\xi_N),\\
&&K(x,\xi_1,\cdots,\xi_N;\eta^i_k)=\frac{1}{2}\sum^N_{k=1}\xi_k^2.
\end{eqnarray*}
Others follow from Theorem~\ref{th:BifE.9.5} directly.
\end{proof}

Similarly, under suitable assumptions as in Theorem~\ref{th:BifE.9.6} the corresponding conclusions hold.

\begin{remark}\label{rm:BifE.9.8}
{\rm {\bf (i)} In Corollary~\ref{cor:BifE.9.7}, if $\partial\Omega$ is of class $C^{2,1}$, and functions
$G, A^{ij}_{kl}=A^{ji}_{lk}:\overline{\Omega}\times\mathbb{R}^N\to\mathbb{R}$, $i,j=1,\cdots,N$ and $k,l=1,\cdots,n$,
   are all $C^5$,  taking $m=3$ and $p>n$ we see that
solutions in $X_{3,p}$  are classical solutions
(since $\vec{u}\in X_{3,p}\subset C^2(\overline{\Omega},\mathbb{R}^N)\cap W^{1,2}_0(\Omega,\mathbb{R}^N)$
satisfies $\vec{u}|_{\partial\Omega}=0$ by \cite[Theorem~9.17]{Bre}).
Similar conclusions hold true for Theorem~\ref{th:BifE.9.2}--Theorem~\ref{th:BifE.9.6}.\\
{\bf (ii)} By contrast, when $N=1$, $\Omega\subset\mathbb{R}^n$ is a bounded open subset,
$G, A^{ij}_{kl}=A^{ji}_{lk}:\overline{\Omega}\times\mathbb{R}\to\mathbb{R}$
is measurable in $x$ for all $\xi\in \mathbb{R}$, and of class $C^1$ in $\xi$ for $a.e.\;x\in\Omega$,
some authors obtained corresponding results in space $W^{1,2}_0(\Omega)\cap L^\infty(\Omega)$
under some additional growth conditions on $A^{ij}$ and $G$,
see \cite{Can,Can1} and references therein.}
\end{remark}

\begin{example}\label{ex:BifE.9.9}
{\rm  Let $p, k, m$ and $\Omega$ be as in Theorem~\ref{th:BifE.9.2}.
Assume
\begin{eqnarray*}
 &&F(x,\xi,\eta_1,\cdots,\eta_n)=(1+\sum^n_{i=1}|\eta_i+D_iu_0(x)|^2)^{1/2}-\mu\xi,\\
&&K(x,\xi,\eta_1,\cdots,\eta_n)=\frac{1}{2}(\xi+u_0(x))^2,
\end{eqnarray*}
where $\mu\in\mathbb{R}$ is a constant and
${u}_0\in C^{k-m+3}(\overline\Omega)$ satisfies the following equation
\begin{eqnarray}\label{e:BifE.8.17}
\sum^n_{i=1}D_i\left(\frac{D_iu_0}{\sqrt{1+|Du_0|^2}}\right)-\mu=\lambda^\ast u_0
\end{eqnarray}
for some constant $\lambda^\ast\in\mathbb{R}$. Consider  the bifurcation problem
\begin{eqnarray}\label{e:BifE.8.18}
\sum^n_{i=1}D_i\left(\frac{D_i(u+u_0)}{\sqrt{1+|D(u+u_0)|^2}}\right)-\mu=\lambda (u+u_0),\quad
u|_{\partial\Omega}=0,
\end{eqnarray}
where the left side is the Euler-Lagrange operator for the correponding functional
$\mathscr{L}_1$ as in (\ref{e:BifE.8.1}), that is, the area functional if $u_0=0$. The corresponding  linearized problem at the trivial
solution $u=0$ is
\begin{eqnarray}\label{e:BifE.8.19}
\sum^n_{i,j=1}D_i\left(\frac{\delta_{ij}(1+|Du_0|^2)-D_iu_0D_ju_0}{(1+|Du_0|^2)^{3/2}}D_ju\right)=\lambda u,\quad
u|_{\partial\Omega}=0,
\end{eqnarray}
which is a linear elliptic problem since it is easy to check that for all $\zeta\in\mathbb{R}^n$,
\begin{eqnarray*}
\sum^n_{i,j=1}\frac{\delta_{ij}(1+|Du_0|^2)-D_iu_0D_ju_0}{(1+|Du_0|^2)^{3/2}}\zeta_i\zeta_j&\ge&
\frac{1}{(1+|Du_0|^2)^{3/2}}|\zeta|^2\\
&\ge& \frac{1}{(1+\max|Du_0|^2)^{3/2}}|\zeta|^2.
\end{eqnarray*}
Hence each eigenvalue of (\ref{e:BifE.8.19}) is isolated.
Suppose that $\lambda^\ast$ is an eigenvalue of (\ref{e:BifE.8.19}). Since
${u}_0\in C^{k-m+3}(\overline\Omega)$ implies that $F$ is of class $C^{k-m+3}$,  we may
obtain some bifurcation results for (\ref{e:BifE.8.18}) near $(\lambda^\ast,0)\in\mathbb{R}\times X_{k,p}$ with Theorem~\ref{th:BifE.9.2}--Theorem~\ref{th:BifE.9.6},
which are also classical solutions for $k> 2+\frac{n}{p}$.

When $n=2$ (i.e., $\Omega\subset\mathbb{R}^2$),  (\ref{e:BifE.8.18})
occurred as a mathematical model for many problems of  hydrodynamics and theory of spring
 membrane (cf. \cite[\S5]{Bor}). In particular, if $\mu=0$ and $u_0=0$,
(\ref{e:BifE.8.18}) becomes
\begin{eqnarray}\label{e:BifE.8.20}
-(1+|\nabla u|^2)^{-3/2}(\triangle u+ u^2_yu_{xx}-2u_xu_yu_{xy}+ u_x^2u_{yy})=\lambda u,\quad
u|_{\partial\Omega}=0,
\end{eqnarray}
which  was studied in \cite{Bor, BorMa} with functional-topological
 properties of the Plateau operator. When  $\partial\Omega$ is of class $C^{2,1}$, $n=2$, $k=3$ and $p>2$
 our results  are supplement for \cite{Bor, BorMa}. }
\end{example}

%Theorem~\ref{th:BifE.9.2}

\section{Bifurcations from deformations of domains}\label{sec:BifE.4}
\setcounter{equation}{0}

Let  $\Omega\subset\R^n$ be as in Theorem~\ref{th:6.1}. We also assume that
 it has $C^{m'}$ boundary $\partial\Omega$ for some integer $m'\ge m$. Let $H=W^{m,2}_0(\Omega,\mathbb{R}^N)$.
 By a {\it deformation} $\{\overline{\Omega}_t\}_{0\le t\le 1}$
of $\overline{\Omega}$ we mean a continuous curve of $C^{m'}$ embedding
$\varphi_t:\overline{\Omega}\to\overline{\Omega}$ such that $\overline{\Omega}_t=
\varphi_t(\overline{\Omega})$, $\varphi_0=id_{\overline{\Omega}}$ and
$\partial\Omega_t=\varphi_t(\partial\Omega)$ for $0\le t\le 1$.
Call the deformation $C^{m'}$ {\it smooth} (resp. {\it contracting}) if
$\partial\Omega_t$ depends $C^{m'}$-smoothly on $t$ in the sense that
$[0,1]\times\overline{\Omega}\to\overline{\Omega},\;(t,x)\mapsto \varphi_t(x)$
is $C^{m'}$ (resp. $\Omega_{t_2}\subset\Omega_{t_1}$ for $t_1<t_2$).
In the following we always assume that $\{\overline{\Omega}_t\}_{0\le t\le 1}$
is a $C^m$-smooth contracting of $\overline{\Omega}$.
(Such a deformation can always be obtained by the negative flow of a suitable Morse function on $\overline{\Omega}$.)
Then we have a Banach space isomorphism
\begin{eqnarray}\label{e:BifE.10-2}
\varphi_t^\ast:W^{m,2}_0(\Omega,\mathbb{R}^N)\to
W^{m,2}_0(\Omega_t,\mathbb{R}^N),\;\vec{u}\mapsto \vec{u}\circ\varphi_t^{-1}.
\end{eqnarray}
Let $F$  satisfy \textsf{Hypothesis} $\mathfrak{F}_{2,N,m,n}$.
It gives a family of functionals
\begin{eqnarray}\label{e:BifE.10-1}
\mathfrak{F}_t(\vec{u})=\int_{\Omega_t} F(x, \vec{u},\cdots, D^m\vec{u})dx,\quad\forall\vec{u}\in W^{m,2}_0(\Omega_t,\mathbb{R}^N),
\;0\le t\le 1.
\end{eqnarray}
The critical points of $\mathfrak{F}_t$ correspond to weak solutions of
\begin{eqnarray}\label{e:BifE.10}
\left.\begin{array}{ll}
&\sum_{|\alpha|\le m}(-1)^{|\alpha|}D^\alpha F^i_\alpha(x, \vec{u},\cdots, D^m\vec{u})=0\quad
\hbox{on}\;\Omega_{t},\quad i=1,\cdots,N,\\
&\hspace{20mm}D^i\vec{u}|_{\partial\Omega_{t}}=0,\quad i=0,1,\cdots,m-1.
\end{array}\right\}
\end{eqnarray}
Let $\vec{u}=0$ be a solution of (\ref{e:BifE.10}) with $t=0$. We say $t^\ast\in [0,1]$
to be a {\it bifurcation point} for the system (\ref{e:BifE.10}) if
there exist  sequences $t_k\to t^\ast$ and $(\vec{u}_k)\subset W^{m,2}_0(\Omega_{t_k}, \mathbb{R}^N)$
such that each $u_k$ is a nontrivial weak solution of (\ref{e:BifE.10}) with $t=t_k$ and
$\|u_k\|_{m,2}\to 0$. Recently, such a problem was studied for semilinear elliptic Dirichlet problems on a ball
 in \cite{PoWa}. We shall here  generalize their result. To this goal,
define $\mathcal{F}:[0,1]\times H\to\mathbb{R}$  by
\begin{eqnarray}\label{e:BifE.11}
\mathcal{F}(t,\vec{u})=\int_{\Omega_t} F(x, \vec{u}\circ\varphi_t^{-1},\cdots, D^m(\vec{u}\circ\varphi_t^{-1}))dx,
\quad\forall\vec{u}\in H.
 \end{eqnarray}
This is continuous, and $\mathcal{F}_t=\mathcal{F}(t,\cdot)=\mathfrak{F}_t\circ\varphi_t^\ast$,
that is, $\mathcal{F}_t$ is the pull-back of $\mathfrak{F}_t$ via $\varphi_t^\ast$.
Clearly, $\vec{u}\in H$ is a critical point of $\mathcal{F}_t$ if and only if $\varphi_t^\ast(\vec{u})$
is that of $\mathfrak{F}_t$ and both have the same Morse indexes and nullities. Note that $\vec{u}=0\in H$ is
the critical point of each $\mathcal{F}_t$ (and so $\mathfrak{F}_t$).
Denote by $\mu_t$ and $\nu_t$ the common Morse index and nullity of $\mathfrak{F}_t$ and
$\mathcal{F}_t$ at zeros. According to Smale's Morse index theorem \cite{Sma1} (precisely see  Uhlenbeck \cite[Th.3.5]{Uh} and Swanson \cite[Th.5.7]{Swa}), we have:

\begin{proposition}\label{prop:BifE.10}
Let $F$ be $C^4$ (for $m=1$) and $C^\infty$ (for $m>1$).  Suppose either that
 \begin{eqnarray}\label{e:BifE.12}
\left.\begin{array}{ll}
&\sum^N_{i,j=1}\sum_{|\alpha|,|\beta|\le m}
(-1)^{\alpha}D^\alpha\bigl[F^{ij}_{\alpha\beta}(x, 0,\cdots, 0)D^\beta v^j\bigr]
=0,\\
&\hspace{20mm}D^i\vec{u}|_{\partial\Omega}=0,\quad i=0,1,\cdots,m-1
\end{array}\right\}
 \end{eqnarray}
have no nontrivial solutions in $W^{m,2}_0(\Omega,\mathbb{R}^N)$ with compact
support in any of the manifolds $\Omega_t$, or that
(\ref{e:BifE.12}) has no solutions $\vec{u}\ne 0$ such that $\vec{u}$
vanishes on some open set in $\Omega$. Then
$$
\mu_0-\mu_1=\sum_{0\le t\le 1}\nu_t.
$$
 Moreover, if $\Omega_1$ has sufficiently small volume, then $\mu_1=0$.
\end{proposition}
 The time $t\in [0, 1]$ with $\nu_t\ne 0$ is
called a {\it conjugate point}.

For the sake of simplicity, we restrict to the case that \textsf{$\Omega$ is star-shaped},
and take
\begin{eqnarray}\label{e:BifE.star}
\Omega_t=\{tx\,|\,x\in\Omega\}\quad\hbox{and}\quad
\varphi_t(x)=tx\quad\hbox{for}\quad t\in (0, 1].
\end{eqnarray}
 Then since
$D^\alpha(\vec{u}\circ\varphi_t^{-1})(x)=\frac{1}{t^{|\alpha|}}(D^\alpha\vec{u})(x/t)$ we have
\begin{eqnarray}\label{e:BifE.13}
\mathcal{F}(t,\vec{u})&=&\int_{\Omega_t} F(x, \vec{u}(x/t),\cdots,
\frac{1}{t^m}(D^m\vec{u})(x/t))dx\nonumber\\
%&=&t^n\int_{\Omega} F(tx, \vec{u}(x),\cdots,
%\frac{1}{t^m}(D^m\vec{u})(x))dx\nonumber\\
&=&\int_{\Omega} F(x, \vec{u}(x),\cdots,
(D^m\vec{u})(x);t)dx
 \end{eqnarray}
 where $F(x,\xi;t)=t^nF(tx, \xi^0, \frac{1}{t}\xi^1,\cdots, \frac{1}{t^m}\xi^m)$. Note that
 \begin{eqnarray}\label{e:BifE.14}
 D_tF(x,\xi;t)&=&nt^{n-1}F(tx, \xi^0, \frac{1}{t}\xi^1,\cdots, \frac{1}{t^m}\xi^m)+
 t^{n+1}\sum^n_{l=1}F_{x_l}(tx, \xi^0, \frac{1}{t}\xi^1,\cdots, \frac{1}{t^m}\xi^m)\nonumber\\
 &+& \sum^N_{i=1}\sum^m_{|\alpha|=1}t^{n-|\alpha|}F^i_{\alpha}(tx, \xi^0, \frac{1}{t}\xi^1,\cdots, \frac{1}{t^m}\xi^m),\\
  F^i_\alpha(x,\xi;t)&=&t^{n-|\alpha|}F^i_\alpha(tx, \xi^0, \frac{1}{t}\xi^1,\cdots, \frac{1}{t^m}\xi^m)
  \end{eqnarray}
 and thus
 \begin{eqnarray}\label{e:BifE.15}
 D_tF^i_\alpha(x,\xi;t)&=&(n-|\alpha|)t^{n-|\alpha|-1}F^i_\alpha(tx, \xi^0, \frac{1}{t}\xi^1,\cdots, \frac{1}{t^m}\xi^m)\nonumber\\
 &+&t^{n-|\alpha|+1}\sum^n_{l=1}F^i_{\alpha x_l}(tx, \xi^0, \frac{1}{t}\xi^1,\cdots, \frac{1}{t^m}\xi^m)\nonumber\\
 &+& \sum^N_{j=1}\sum^m_{|\beta|=1} t^{n-|\alpha|-|\beta|}F^{ij}_{\alpha\beta}(tx, \xi^0, \frac{1}{t}\xi^1,\cdots, \frac{1}{t^m}\xi^m).
   \end{eqnarray}
%%%%%%%%%%%%%%%%%%%%%%%%%%%%%%%%%%%%%%%%%%%%%%%%%%%%%%%%%%%%%%%%%%%%%%%%%%%%%%%
%%  and, in particular,
%% \begin{eqnarray*}
%%  D_tF(x,0;t)&=&nt^{n-1}F(tx, 0)+ t^{n+1}\sum^N_{i=1}F_{x_i}(tx, 0)
%% + \sum^N_{i=1}\sum^m_{|\alpha|=1}t^{n-|\alpha|}F^i_{\alpha}(tx, 0),\\
%% D_tF^i_\alpha(x,\xi;t)&=&(n-|\alpha|)t^{n-|\alpha|-1}F^i_\alpha(tx, 0)
%% +t^{n-|\alpha|+1}\sum^n_{l=1}F^i_{\alpha x_l}(tx, 0)\\
%% &+& \sum^N_{j=1}\sum^m_{|\beta|=1} t^{n-|\alpha|-|\beta|}F^{ij}_{\alpha\beta}(tx, 0),
%% \end{eqnarray*}
%%%%%%%%%%%%%%%%%%%%%%%%%%%%%%%%%%%%%%%%%%%%%%%%%%%%%%%%%%%

\begin{theorem}\label{th:BifE.11}
Let $\Omega\subset\mathbb{R}^n$ be a star-shaped bounded  domain
with $C^{m+1}$-smooth boundary, and let $\Omega_t$ and
$\varphi_t$ be as in (\ref{e:BifE.star}) for $t\in (0, 1]$. Assume that $F$  satisfy \textsf{Hypothesis} $\mathfrak{F}_{2,N,m,n}$,
$\vec{u}=0$ is a solution of (\ref{e:BifE.10}) and that
the conditions of Proposition~\ref{prop:BifE.10} are fulfilled.
Suppose also that for all $\alpha, i, l$,
\begin{eqnarray}\label{e:BifE.16}
&&|F^i_{\alpha x_l}(x,\xi)|\le  \mathfrak{g}(\sum^N_{k=1}|\xi_0^k|)\sum_{|\beta|<m-n/2}\bigg(1+
\sum^N_{k=1}\sum_{m-n/2\le |\gamma|\le
m}|\xi^k_\gamma|^{2_\gamma }\bigg)^{2_{\alpha\beta}}\nonumber\\
&&+\mathfrak{g}(\sum^N_{k=1}|\xi^k_0|)\sum^N_{l=1}\sum_{m-n/2\le |\beta|\le m} \bigg(1+
\sum^N_{k=1}\sum_{m-n/2\le |\gamma|\le m}|\xi^k_\gamma|^{2_\gamma }\bigg)^{2_{\alpha\beta}}|\xi^l_\beta|,
\end{eqnarray}
 where $\mathfrak{g}:[0,\infty)\to\mathbb{R}$ is a continuous, positive, nondecreasing function,  and is constant if $m<n/2$.
Then for each $\epsilon\in (0, 1)$
there is only a finite number of $t\in [\epsilon, 1]$ such that
$\mathcal{F}_t=\mathcal{F}(t,\cdot)$  has nonzero nullity $\nu_t$ at $\vec{u}=0$. Moreover,
$(t,0)\in (0, 1]\times W^{m,2}_0(\Omega,\mathbb{R}^N)$ is a bifurcation point  for the equation
\begin{eqnarray}\label{e:BifE.16.0}
\mathcal{F}_{\vec{u}}'(t, \vec{u})=0
\end{eqnarray}
 if and only if $\nu_t\ne 0$.
\end{theorem}

 This generalizes a recent result for semilinear elliptic Dirichlet problems on a ball
 in \cite{PoWa}.   If $n=\dim\Omega=1$ and $m\ge 1$ the similar result
 can be proved with \cite[Theorem~2.4]{Uh}, see \cite{Lu8}.

 \begin{proof}[Proof of Theorem~\ref{th:BifE.11}]
 The first claim follows from Proposition~\ref{prop:BifE.10} directly.

 Let $\nu_s\ne 0$ for some $s\in (0, 1)$.
For each fixed $\epsilon\in (0, 1)$, it follows from (\ref{e:BifE.14})--(\ref{e:BifE.16}) and \cite[Proposition~4.3]{Lu6}
that
$$
\overline\Omega\times\prod^m_{k=0}\mathbb{R}^{N\times M_0(k)}\times [\epsilon, 1]\ni (x,
\xi, t)\mapsto F(x,\xi;t)=t^nF(tx, \xi^0, \frac{1}{t}\xi^1,\cdots, \frac{1}{t^m}\xi^m)\in\R
$$
satisfies the conditions of  Theorem~\ref{th:BifE.9}.
 By the first claim we may assume that $\epsilon<s$, $\nu_\epsilon=0$
and $s_1<\cdots<s_k$ are all points in $[\epsilon,1]$ in which $\nu_t\ne 0$.
Clearly, we can assume $s=s_1$.  Proposition~\ref{prop:BifE.10} we deduce that
$\mu_t=\mu_s$ for all $t\in [\epsilon, s]$, and $\mu_t=\mu_s+\nu_s$ for all
$t\in (s,s_2)$. Thus $\vec{u}=0\in V$ is a nondegenerate critical point of
$\mathcal{F}_{t}$ for each $t\in [\epsilon, s]\cup (s,s_2)$.
It follows from \cite[Theorem~2.1]{Lu7} (or Theorem~\ref{th:A.1}) that
$$
C_q(\mathcal{F}_{t}, 0;{\bf K})=\delta_{q\mu_{s}}\;\forall t\in [\epsilon, s]\quad\hbox{and}\quad
C_q(\mathcal{F}_{t}, 0;{\bf K})=\delta_{q(\mu_{s}+\nu_s)}\;\forall t\in (s,s_2).
$$
For any $t_1<t_2$ satisfying $\epsilon\le t_1<s<t<t_2<s_2$,
applying  Theorem~\ref{th:BifE.9} to the family $[t_1,t_2]\ni t\mapsto\mathcal{F}_{t}$
we see that for some $\bar{t}\in[t_1, t_2]$, $(\bar{t},0)\in (0, 1]\times W^{m,2}_0(\Omega,\mathbb{R}^N)$ is a bifurcation point  for
(\ref{e:BifE.16.0}). Since $t_1$ and $t_2$ may be  arbitrarily close to $s$,
$(s,0)\in (0, 1]\times W^{m,2}_0(\Omega,\mathbb{R}^N)$ must be a bifurcation point for (\ref{e:BifE.16.0}).

Conversely, suppose that $(s,0)\in (0, 1]\times W^{m,2}_0(\Omega,\mathbb{R}^N)$ is a bifurcation point  for the equation
(\ref{e:BifE.16.0}). Then for sufficiently small $\rho\in (0,s)$ it is not hard to check that
$\mathcal{F}_t$ with $t\in\Lambda:=[s-\rho, s+\rho]\cap (0,1]$ satisfies the conditions of
Theorem~\ref{th:Ka1}.

%({\it Note}: As in Remark~\ref{rm:Bi.2.4.3},
%it is very possible to prove a more general parameterized splitting theorem
% under the assumptions  of Theorem~\ref{th:Ka1}. Then    )

\end{proof}
%\hfill$\Box$\vspace{2mm}

If $m=1$, as in the first paragraph of \cite[Appendix A]{Lu6}
we can write $F$ as
$$
\overline\Omega\times\mathbb{R}^N\times\mathbb{R}^{N\times n}\ni (x,
z,p)\mapsto F(x, z,p)\in\R.
$$
 Let $\kappa_n=2n/(n-2)$ for $n>2$, and $\kappa_n\in (2,\infty)$ for $n=2$.
Then (\ref{e:BifE.16}) means that there exist positive constants  $\mathfrak{g}'_1$,  $\mathfrak{g}'_2$ and
$s\in (0, \frac{\tau_n-2}{\tau_n})$, $r_\alpha\in (0,\frac{\tau_n-2}{2\tau_n})$ for each $\alpha\in \mathbb{N}_0^n$ with $|\alpha|=1$,  such that for $i=1,\cdots,N$, $l=1,\cdots,n$ and $|\alpha|=1$,
\begin{eqnarray*}
&|F_{z_jx_l}(x,z,p)|\le \mathfrak{g}'_1\left(1+\sum^N_{l=1}|z_l|^{\tau_n}+
\sum^N_{k=1}|p^k_\alpha|^2\right)^{s+\frac{1}{2}},\\
&|F_{p^i_\alpha x_l}(x,z,p)|\le \mathfrak{g}'_2\left(1+\sum^N_{l=1}|z_l|^{\tau_n}+
\sum^N_{k=1}|p^k_\alpha|^2\right)^{r_\alpha+\frac{1}{2}}.
\end{eqnarray*}
Clearly, the {\bf CGC} above Theorem~\ref{th:6.1} implies these.
Combing \cite[Proposition~A.1]{Lu6} we see that Theorem~\ref{th:BifE.11} leads to

\begin{theorem}\label{th:BifE.12}
Let $\Omega\subset\mathbb{R}^n$ be a star-shaped bounded Sobolev domain
 for $(2,1,n)$,   $\Omega_t=\{tx\,|\,x\in\Omega\}$ and
$\varphi_t(x)=tx$ for $t\in (0, 1]$. For a $C^4$ function
$$
\overline\Omega\times\mathbb{R}^N\times\mathbb{R}^{N\times n}\ni (x,
z,p)\mapsto F(x, z,p)\in\R
$$
satisfying {\bf CGC}.
Suppose that $\vec{u}=0$ is a solution of (\ref{e:BifE.10}) with $m=1$.
Then for
$$
\mathcal{F}:[0,1]\times W^{1,2}_0(\Omega,\mathbb{R}^N)\to\mathbb{R},\;(t,\vec{u})\mapsto
\mathcal{F}_t(\vec{u})=
\int_{\Omega_t} F\left(x, \vec{u}(x/t), \frac{1}{t}(D\vec{u})(x/t)\right)dx,
$$
and each $\epsilon\in (0, 1)$, there is only a finite number of $t\in [\epsilon, 1]$ such that
$\mathcal{F}_t$  has nonzero nullity $\nu_t$ at $\vec{u}=0$. Moreover,
$(t,0)\in (0, 1]\times W^{1,2}_0(\Omega,\mathbb{R}^N)$ is a bifurcation point  for the equation
$\mathcal{F}_{\vec{u}}'(t, \vec{u})=0$ if and only if $\nu_t\ne 0$.
\end{theorem}

\begin{remark}\label{rem:BifE.12}
{\rm It is possible to remove the assumption
  that $\Omega$ is star--shaped in Theorems~\ref{th:BifE.11} and \ref{th:BifE.12}
if $F(x,0)=0\;\forall x$. For $\vec{u}\in W^{m,2}_0(\Omega,\mathbb{R}^N)$,
we extend $\varphi_t^\ast(\vec{u})\in W^{m,2}_0(\Omega_t,\mathbb{R}^N)$ to a function
on $\Omega$ by defining $\varphi_t^\ast(\vec{u})(x)=0$ outside $\Omega_t$.
Then for $0\le t_1<t_2\le 1$ and $\vec{u}\in V$ we have
\begin{eqnarray*}\label{e:BifE.20}
\mathcal{F}_{t_2}(\vec{u})- \mathcal{F}_{t_1}(\vec{u}) &=&\int_{\Omega_{t_1}} [F(x, \varphi_{t_2}^\ast(\vec{u}),\cdots,
D^m(\varphi_{t_2}^\ast(\vec{u})))-  F(x, \varphi_{t_1}^\ast(\vec{u}),\cdots,
D^m(\varphi_{t_1}^\ast(\vec{u})))]dx\\
&=&\mathfrak{F}_0(\varphi_{t_2}^\ast(\vec{u}))-\mathfrak{F}_0(\varphi_{t_1}^\ast(\vec{u}))\\
&=&\langle D\mathfrak{F}_0(s\varphi_{t_2}^\ast(\vec{u})+ (1-s)\varphi_{t_1}^\ast(\vec{u})), \varphi_{t_2}^\ast(\vec{u})-\varphi_{t_1}^\ast(\vec{u})\rangle
\end{eqnarray*}
for some $s\in [0, 1]$. Let $\iota_t:W^{m,2}_0(\Omega_t,\mathbb{R}^N)\to W^{m,2}_0(\Omega,\mathbb{R}^N)$
be the inclusion. Then $t\mapsto\iota_t\circ\varphi_t^\ast\in \mathscr{L}(W^{m,2}_0(\Omega,\mathbb{R}^N))$
is continuous. Fix $R>0$ we have  $C(R)>0$ such that
$$
\sup\{\|s\varphi_{t_2}^\ast(\vec{u})+ (1-s)\varphi_{t_1}^\ast(\vec{u})\|\;|\;s\in [0,1],\;\|\vec{u}\|_{m,2}\le R,\;
t_1, t_2\in [0, 1]\}\le C(R).
$$
It follows from  Theorem~\ref{th:6.1}(A) that for some $C_1(R)>0$,
\begin{eqnarray*}
&&|\mathcal{F}_{t_2}(\vec{u})- \mathcal{F}_{t_1}(\vec{u})|\le C_1(R)\|\varphi_{t_2}^\ast(\vec{u})-\varphi_{t_1}^\ast(\vec{u})\|_{m,2},\\
&&\quad\forall \|\vec{u}\|_{m,2}\le R,\;0\le t_1<t_2\le 1.
\end{eqnarray*}
Using this it should be able to prove that $t\mapsto \mathcal{F}_{t}|_{B(0, R)}\in C^0(B(0, R))$
is continuous. Similarly, by (i) of Theorem~\ref{th:6.1}(B) we may show that
$t\mapsto \nabla\mathcal{F}_{t}|_{B(0, R)}\in C^0(B(0, R), V)$
is continuous. Hence Theorem~\ref{th:Bi.1.1} is applicable. }
\end{remark}

 By increasing smoothness of $F$ and $\partial\Omega$, we can even use
Theorem~\ref{th:BB.5} and Proposition~\ref{prop:BifE.10} to
obtain a similar result to Rabinowitz bifurcation theorem \cite{Rab}
with neither the assumption  \textsf{Hypothesis} $\mathfrak{F}_{2,N,m,n}$
for $F$ nor the requirement that $\Omega\subset\mathbb{R}^n$ is  star-shaped.

 \begin{theorem}\label{th:BifE.13}
  Let a real $p\ge 2$, integers $k$ and $m$ satisfy $k> m+\frac{n}{p}$,  and let
   $\Omega\subset\R^n$ be a bounded
   domain with boundary of class $C^{k}$, $N\in\mathbb{N}$. Let
   $\{\overline{\Omega}_t\}_{0\le t\le 1}$ be a $C^k$-smooth contracting of $\overline{\Omega}$.
(Thus the Banach space isomorphism in (\ref{e:BifE.10-2}) is also a Banach space isomorphism
from $C^k(\overline\Omega,\mathbb{R}^N)$ (resp. $X_{k,p}$)
 to $C^k(\overline\Omega_t,\mathbb{R}^N)$ (resp. $X^t_{k,p}$), still denoted by
 $\varphi_t^\ast$, where
$X^t_{k,p}=C^k(\overline\Omega_t,\mathbb{R}^N)\cap W^{2,m}_0(\Omega_t,\mathbb{R}^N)$.)
Let
$$
F:\overline\Omega\times\prod^m_{k=0}\mathbb{R}^{N\times M_0(k)}\to\R
$$
be $C^{k-m+3}$ (resp. $C^\infty$) for $m=1$ (resp. $m>1$),
and satisfy Proposition~\ref{prop:BifE.10}.
Let $\mathcal{F}:[0,1]\times X_{k,p}\to\mathbb{R}$ be still defined by
the right side of (\ref{e:BifE.11}). Assume that  $\vec{u}=0$ is a solution of (\ref{e:BifE.10}) with $t=0$,
and that (\ref{e:BifE.8.3}) with $\vec{u}=0$ is satisfied.
(In this case Proposition~\ref{prop:BifE.10} and all arguments before it are also valid if $V$ is replaced
by $X_{k,p}$.) Fix a conjugate point $t_0\in (0, 1)$, i.e., $\nu_{t_0}\ne 0$.  Then
$(t_0, 0)\in (0, 1]\times X_{k,p}$ is a bifurcation point  for the equation
(\ref{e:BifE.10})  and  one of the following alternatives occurs:
\begin{description}
\item[(i)] $(t_0, 0)$ is not an isolated solution of (\ref{e:BifE.10}) in
 $\{t_0\}\times X_{k,p}$;

\item[(ii)]  for every $t\in (0,1)$ near $t_0$ there is a nontrivial solution
$\vec{u}_t$ of (\ref{e:BifE.10}) converging to $0$ as $t\to t_0$;

\item[(iii)] there is an one-sided  neighborhood $\mathfrak{T}$ of $t_0$ such that
for any $t\in \mathfrak{T}\setminus\{t_0\}$,
(\ref{e:BifE.10}) has at least two nontrivial solutions converging to
zero as $t\to t_0$.
\end{description}
 \end{theorem}
\begin{proof}
If $t$ is sufficiently close to $t_0$, then  $\vec{u}=0\in X_{k,p}$ is nondegenerate as a critical point of
$\mathcal{F}_{t}$. Applying Theorem~\ref{th:BB.5} to
the family $\{\mathcal{F}_{t}\,|\,0\le t\le 1\}$ and $\lambda_0=t_0$ we get
 $0<\delta<\min\{t_0,1-t_0\}$, $0<\epsilon\ll 1$ and a family of   functionals
of class $C^{2}$ (with $X=X_{k,p}$)
$$
\mathcal{F}_{t}^\circ: B_{X}(0, 2\epsilon)\cap X_0\to \mathbb{R},\;
z\mapsto\mathcal{F}_{t}(z+ \mathfrak{h}({\lambda}, z)),\quad t\in T:=[t_0-\delta, t_0+\delta],
$$
which depends on $t$ continuously. By (\ref{e:BB.5}) we also deduce that
$\{(z,t)\to d\mathcal{F}_{t}^\circ(z)\}$ is continuous on
$T\times (B_{X}(0, 2\epsilon)\cap X_0)$.

 Suppose (by shrinking $\epsilon>0$ if necessary) that $0\in B_{X}(0, 2\epsilon)\cap X_0$
is a unique critical point of $\mathcal{F}_{t}^\circ$ for each $t\in T$.
By shrinking $\delta>0$ we assume that each $t\in T\setminus\{t_0\}$ is not
conjugate point.
As in the proof of (\ref{e:Bi.2.16})
we derive from Corollary~\ref{cor:BB.6} that  for any $j\in\mathbb{N}_0$,
\begin{eqnarray}\label{e:BifE.16.1}
C_{j}(\mathcal{F}_{t}^\circ, 0;{\bf K})=\delta_{(j+\mu_{t_0})\mu_t}{\bf K},\quad
\forall t\in T\setminus\{t_0\}.
\end{eqnarray}
By Proposition~\ref{prop:BifE.10} we have
$$
\mu_t=\left\{\begin{array}{ll}
\mu_{t_0-\delta},&\quad\forall t\in [t_0-\delta,t_0),\\
\mu_{t_0-\delta}+\nu_{t_0},&\quad\forall t\in (t_0, t_0+\delta].
\end{array}\right.
$$
 Hence (\ref{e:BifE.16.1}) leads to
$$
C_{j}(\mathcal{F}_{t}^\circ, 0;{\bf K})=
\left\{\begin{array}{ll}
\delta_{j0}{\bf K},&\quad
\forall t\in [t_0-\delta, t_0),\\
\delta_{j\nu_{t_0}}{\bf K},&\quad
\forall t\in (t_0, t_0+\delta].
\end{array}\right.
$$
As in the proof of Theorem~\ref{th:Bi.2.4} the conclusions follow from Theorem~\ref{th:Bi.2.1}.
\end{proof}

\part{Appendix}\label{par:app}

\appendix

\section{Appendix:\quad
 Parameterized splitting theorems by the author}\label{app:A}\setcounter{equation}{0}

Under Hypothesis~\ref{hyp:1.1} or Hypothesis~\ref{hyp:1.3}
let $H=H^+\oplus H^0\oplus H^-$ be the orthogonal decomposition
according to the positive definite, null and negative definite spaces of $B(0)$.
Denote by $P^\ast$ the orthogonal projections onto $H^\ast$, $\ast=+,0,-$.
$\nu:=\dim H^0$ and $\mu:=\dim H^-$ are called the {\it Morse index} and
{\it nullity} of the critical point $0$. In particular, if $\nu=0$ the critical point $0$ is
said to be {\it nondegenerate}.
Such a critical point is isolated by \cite[Theorem~2.13]{Lu7}] (resp.
by Claim~2(i) and the arguments in Step 3 in the proof of \cite[Theorem~1.1]{Lu1}])
for the case of Hypothesis~\ref{hyp:1.1} (resp. Hypothesis~\ref{hyp:1.3}).
We have the following parameterized Morse-Palais lemma.

\begin{theorem}[\hbox{\cite[Theorem~2.9]{Lu7}}]\label{th:A.1}
Let  $\mathcal{L}\in C^1(U,\mathbb{R})$ satisfy Hypothesis~\ref{hyp:1.1}, and
let $\widehat{\mathcal{L}}\in C^1(U,\mathbb{R})$ satisfy Hypothesis~\ref{hyp:1.2}
without requirement that each $\widehat{\mathcal{L}}''(u)\in\mathscr{L}_s(H)$ is compact.
%: {\rm i)} $\mathcal{G}'(0)=0$,
%{\rm ii)} the gradient $\nabla\mathcal{G}$ has
%G\^ateaux derivative $\mathcal{G}''(u)\in \mathscr{L}_s(H)$ at any $u\in U$, and $\mathcal{G}'': U\to \mathscr{L}_s(H)$
% are continuous at $0$.
Suppose that the critical point  $0$  of $\mathcal{L}$ is a nondegenerate.
 Then there exist $\rho>0$, $\epsilon>0$, a family of open neighborhoods of $0$ in
$H$, $\{W_\lambda\,|\, |\lambda|\le\rho\}$
and a family of origin-preserving homeomorphisms, $\phi_\lambda: B_{H^+}(0,\epsilon) +
B_{H^-}(0,\epsilon)\to W_\lambda$, $|\lambda|\le\rho$,
 such that
$$
(\mathcal{L}+\lambda\widehat{\mathcal{L}})\circ\phi_\lambda(u^++ u^-)=\|u^+\|^2-\|u^-\|^2,
\quad\forall (u^+, u^-)\in B_{H^+}(0,\epsilon)\times
B_{H^-}(0,\epsilon).
$$
Moreover, $[-\rho,\rho]\times (B_{H^+}(0,\epsilon) +
B_{H^-}(0,\epsilon))\ni (\lambda, u)\mapsto \phi_\lambda(u)\in H$
is continuous, and $0$ is an isolated critical point of each $\mathcal{L}+\lambda\widehat{\mathcal{L}}$.
%Finally, if $\hat{H}$ is a closed subspace containing $H^-$, and $\hat{H}^+$ is the orthogonal
%complement of $H^-$ in $\hat{H}$, i.e., $\hat{H}^+=\hat{H}\cap H^+$, then each
%$\phi_\lambda$ restricts to a homeomorphism $\hat{\phi}_\lambda:(B_{\hat{H}^+}(0,\epsilon) + B_{H^-}(0,\epsilon))
%\to\hat{W}_\lambda:=W_\lambda\cap\hat{H}$, and
%$(\mathcal{L}+\lambda\mathcal{G})\circ\hat{\phi}_\lambda(u^++ u^-)=\|u^+\|^2-\|u^-\|^2$
%for all $(u^+, u^-)\in B_{\hat{H}^+}(0,\epsilon)\times B_{H^-}(0,\epsilon)$.
\end{theorem}

If $\nu=\dim H^0\ne 0$ there exists the following parameterized splitting theorem.

\begin{theorem}[\hbox{\cite[Theorems~2.12,2.16]{Lu7}}]\label{th:A.2}
Let  $\mathcal{L}\in C^1(U,\mathbb{R})$ satisfy Hypothesis~\ref{hyp:1.1}, and
let $\widehat{\mathcal{L}}_j\in C^1(U,\mathbb{R})$,  $j=1,\cdots,n$, satisfy Hypothesis~\ref{hyp:1.2}
without requirement that each $\widehat{\mathcal{L}}''_j(u)\in\mathscr{L}_s(H)$ is compact.
 Suppose further that the nullity $\mathcal{L}$ at $0$, $\nu\ne 0$.
%\begin{description}
%\item[(i)] $\mathcal{G}'_j(0)=0$, $j=1,\cdots,n$;
%\item[(ii)] for each $j=1,\cdots,n$, the gradient $\nabla\mathcal{G}_j$ has  G\^ateaux derivative
%$\mathcal{G}''_j(u)\in \mathscr{L}_s(H)$ at any $u\in V$, and $\mathcal{G}''_j: U\to \mathscr{L}_s(H)$ is continuous at $0$.
%\end{description}
Then for sufficiently small  $\delta>0$, $\epsilon>0$ and $r>0, s>0$, there exist:
\begin{description}
\item[(a)]  a unique continuous map
 \begin{eqnarray*}
\psi:[-\delta, \delta]^n\times B_H(0,\epsilon)\cap H^0\to \mathcal{Q}_{r,s}:=B_{H^+}(0,r)\oplus B_{H^-}(0,s)\subset
(H^0)^\bot:=H^+\oplus H^-
 \end{eqnarray*}
 such that for all $(\vec{\lambda}, z)\in [-\delta, \delta]^n\times B_{H}(0,\epsilon)\cap H^0$
 with $\vec{\lambda}=(\lambda_1,\cdots,\lambda_n)$,
$\psi(\vec{\lambda},0)=0$,  and
\begin{equation}\label{e:S.4.22}
 P^\bot\nabla\mathcal{L}(z+ \psi(\vec{\lambda}, z))+
 \sum^n_{j=1}\lambda_j P^\bot\nabla\widehat{\mathcal{L}}_j(z+ \psi(\vec{\lambda}, z))=0,
 \end{equation}
 where  $P^\bot$ is the orthogonal projection onto $(H^0)^\bot$,
\item[(b)] an open neighborhood $W$ of $0$ in $H$ and an origin-preserving homeomorphism
\begin{eqnarray*}
&&[-\delta, \delta]^n\times B_{H^0}(0,\epsilon)\times
\left(B_{H^+}(0, r) + B_{H^-}(0, s)\right)\to [-\delta, \delta]^n\times W,\nonumber\\
&&(\vec{\lambda}, z, u^++u^-)\mapsto (\vec{\lambda},\Phi_{\vec{\lambda}}(z, u^++u^-))
\end{eqnarray*}
such that
for all $(\vec{\lambda}, z, u^+ + u^-)\in [-\delta,\delta]^n\times B_{H^0}(0,\epsilon)\times\left(B_{H^+}(0, r) + B_{H^-}(0, s)\right)$,
\begin{eqnarray*}
\mathcal{L}_{\vec{\lambda}}\circ\Phi_{\vec{\lambda}}(z, u^++ u^-)=\|u^+\|^2-\|u^-\|^2+ \mathcal{
L}_{\vec{\lambda}}(z+ \psi(\vec{\lambda}, z)).
\end{eqnarray*}
\end{description}
Moreover, there also hold:
\begin{description}
\item[(i)] the functional  $B_H(0, \epsilon)\cap H^0\ni z\mapsto \mathcal{L}^\circ_{\vec{\lambda}}$ defined by
 \begin{eqnarray}\label{e:S.4.23}
 \mathcal{L}^\circ_{\vec{\lambda}}(z):=\mathcal{L}_{\vec{\lambda}}(z+\psi_{\vec{\lambda}}(z))
 =\mathcal{L}(z+\psi(\vec{\lambda}, z))+ \sum^n_{j=1}\lambda_j\widehat{\mathcal{L}}_j(z+\psi(\vec{\lambda}, z))
  \end{eqnarray}
  is of class $C^1$, and its differential is given by
\begin{eqnarray}\label{e:S.4.24}
D\mathcal{L}_{\vec{\lambda}}^\circ(z)[h]=D\mathcal{L}(z+\psi(\vec{\lambda},z))[h]+
\sum^n_{j=1}\lambda_j D\widehat{\mathcal{L}}_j(z+\psi(\vec{\lambda},z))[h],\quad\forall h\in H^0;
\end{eqnarray}
 %The functional $\mathcal{L}_{\vec{\lambda}}^\circ: B_H(0, \epsilon)\cap H^0\to \mathbb{R}$
%   given by (\ref{e:S.5.3}) is of class $C^1$, and its differential is given by (\ref{e:S.5.4}).
\item[(ii)] if $\mathcal{L}$ and $\widehat{\mathcal{L}}_j$, $j=1,\cdots,n$, are of class $C^{2-0}$,
then so is $\mathcal{L}_{\vec{\lambda}}^\circ$ for each $\vec{\lambda}\in [-\delta, \delta]^n$;
\item[(iii)] if a compact Lie group $G$  acts on $H$ orthogonally, and
$V$, $\mathcal{L}$ and $\widehat{\mathcal{L}}_j$ are $G$-invariant (and hence $H^0$, $(H^0)^\bot$
are $G$-invariant subspaces), then for each $\vec{\lambda}\in [-\delta, \delta]^n$, $\psi(\vec{\lambda}, \cdot)$
 and $\Phi_{\vec{\lambda}}(\cdot,\cdot)$  are $G$-equivariant, and $\mathcal{L}^\circ_{\vec{\lambda}}(z)=\mathcal{
L}_{\vec{\lambda}}(z+ \psi(\vec{\lambda}, z))$ is $G$-invariant.
\end{description}
\end{theorem}

For each $\vec{\lambda}\in [-\delta, \delta]^n$, it is easily checked that
the map $z\mapsto z+ \psi(\vec{\lambda}, z))$ induces an one-to-one correspondence
 between the critical points of  $\mathcal{L}_{\vec{\lambda}}^\circ$ near $0\in H^0$
and those of $\mathcal{L}_{\vec{\lambda}}$ near $0\in H$.
Thus $0\in H^0$ is an isolated critical point of $\mathcal{L}_{\vec{\lambda}}^\circ$ if and only if
$0\in H$ an isolated critical point of $\mathcal{L}_{\vec{\lambda}}$.

\begin{theorem}[Parameterized Shifting Theorem (\hbox{\cite[Theorem~2.18]{Lu7}})]\label{th:A.3}
Under the assumptions of Theorem~\ref{th:A.2}, suppose
 for some $\vec{\lambda}\in [-\delta, \delta]^n$ that $0\in H$ is an isolated critical point of $\mathcal{L}_{\vec{\lambda}}$ (thus $0\in H^0$ is that of $\mathcal{L}_{\vec{\lambda}}^\circ$). Then
\begin{eqnarray*}
C_q(\mathcal{L}_{\vec{\lambda}}, 0;{\bf K})=C_{q-\mu}(\mathcal{L}^\circ_{\vec{\lambda}}, 0;{\bf K}),\quad\forall
q\in\mathbb{N}\cup\{0\}.
\end{eqnarray*}
%where $\mathcal{L}^\circ_{\vec{\lambda}}(z)=\mathcal{
%L}_{\vec{\lambda}}(z+ \psi(\vec{\lambda}, z))= \mathcal{L}(z+\psi(\vec{\lambda}, z))+
%\sum^n_{j=1}\lambda_j\mathcal{G}_j(z+\psi(\vec{\lambda}, z))$ is
% as in (\ref{e:S.5.3}).
\end{theorem}

There also exist the parameterized Morse-Palais lemma around critical orbits (\cite[Theorem~2.21]{Lu7}),
 the parameterized splitting theorem around critical orbits (\cite[Theorem~2.22]{Lu7}) and the corresponding
parameterized shifting theorem (\cite[Corollary~2.27]{Lu7}).

Under Hypothesis~\ref{hyp:1.3}, it easily follows from the proof of \cite[Theorem~2.9]{Lu7}
and \cite[Remark~2.2]{Lu2} that Theorem~\ref{th:A.1} has the following corresponding version.

\begin{theorem}\label{th:A.4}
Let  $\mathcal{L}\in C^1(U,\mathbb{R})$ satisfy Hypothesis~\ref{hyp:1.3}, and
let $\widehat{\mathcal{L}}\in C^1(U,\mathbb{R})$ satisfy Hypothesis~\ref{hyp:1.4}
without requirement that each $\widehat{B}(x)\in\mathscr{L}_s(H)$ is compact.
Suppose that the critical point  $0$  of $\mathcal{L}$ is a nondegenerate.
Then Theorem~\ref{th:A.1} still holds.
\end{theorem}

We can also prove a more general version of this result if the condition
``${\rm Ker}(B_{\lambda^\ast}(0))\ne\{0\}$" in the assumptions of the following
theorem is changed into ``${\rm Ker}(B_{\lambda^\ast}(0))=\{0\}$".

\begin{theorem}\label{th:A.5-}
Let $H$, $X$ and $U$ be as in Hypothesis~\ref{hyp:1.3},
 and $\Lambda$ a topological space.
Let $\mathcal{L}_\lambda\in C^1(U, \mathbb{R})$, $\lambda\in\Lambda$, be a continuous family of functionals
    satisfying $\mathcal{L}'_\lambda(0)=0\;\forall\lambda$.
 For each $\lambda\in\Lambda$, assume that
 there exist maps $A_\lambda\in C^1(U^X, X)$ and $B_\lambda: U\cap X\to \mathscr{L}_s(H)$,
which  depend on $\lambda$ continuously,  such that for all $x\in U\cap X$  and $u, v\in X$,
 $$
 D\mathcal{L}_\lambda(x)[u]=(A_\lambda(x), u)_H\quad\hbox{and}\quad
(DA_\lambda(x)[u], v)_H=(B_\lambda(x)u, v)_H,
$$
 and  that  $B_\lambda$ has a decomposition
$B_\lambda=P_\lambda+Q_\lambda$, where for each $x\in U\cap X$,
 $P_\lambda(x)\in\mathscr{L}_s(H)$ is  positive definitive and
$Q_\lambda(x)\in\mathscr{L}_s(H)$ is compact.
   Let $0\in H$ be a degenerate critical point of
  some $\mathcal{L}_{\lambda^\ast}$, i.e.,  ${\rm Ker}(B_{\lambda^\ast}(0))\ne\{0\}$.
Suppose also that $P_\lambda$ and $Q_\lambda$ satisfy the following conditions:
    \begin{description}
\item[(i)]  For each $h\in H$, it holds that $\|P_{\lambda}(x)h-P_{\lambda^\ast}(0)h\|\to 0$
as $x\in U\cap X$ approaches to $0$ in $H$ and $\lambda\in\Lambda$ converges to $\lambda^\ast$.

 \item[(ii)]  For some small $\delta>0$, there exist positive constants $c_0>0$ such that
$$
(P_\lambda(x)u, u)\ge c_0\|u\|^2\quad\forall u\in H,\;\forall x\in
\bar{B}_H(0,\delta)\cap X,\quad\forall\lambda\in \Lambda.
$$
 \item[(iii)]  $Q_\lambda: U\cap X\to \mathscr{L}_s(H)$ is uniformly continuous at $0$  with respect to $\lambda\in \Lambda$.
  \item[(iv)]  If $\lambda\in \Lambda$ converges to $\lambda^\ast$ then
  $\|Q_{\lambda}(0)-Q_{\lambda^\ast}(0)\|\to 0$.
   \item[(v)] $(\mathcal{L}_{\lambda^\ast}, H, X, U, A_{\lambda^\ast}, B_{\lambda^\ast}=P_{\lambda^\ast}+ Q_{\lambda^\ast})$ satisfies Hypothesis~\ref{hyp:1.3}.
    \end{description}
%[Clearly, each $(\mathcal{L}_\lambda, H, X, U)$ satisfies Hypothesis~\ref{hyp:1.3}.]
 Let $H^+_\lambda$, $H^-_\lambda$ and $H^0_\lambda$ be the positive definite, negative definite and zero spaces of
${B}_\lambda(0)$.  Denote by $P^0_\lambda$ and $P^\pm_\lambda$ the orthogonal projections onto $H^0_\lambda$ and $H^\pm_\lambda=H^+_\lambda\oplus H^-_\lambda$,
and by $X^\star_\lambda=X\cap H^\star_\lambda$ for $\star=+,-$, and by  $X^\pm_\lambda=P^\pm_\lambda(X)$.
 Then there exists a neighborhood $\Lambda_0$ of $\lambda^\ast$ in $\Lambda$,
$\epsilon>0$, a (unique) $C^0$ map
\begin{equation}\label{e:Spli.2.1.1}
\psi:\Lambda_0\times B_{H^0_{\lambda^\ast}}(0,\epsilon)\to X^\pm_{\lambda^\ast}
\end{equation}
which is $C^1$ in the second variable and
satisfies $\psi(\lambda, 0)=0\;\forall\lambda\in \Lambda_0$ and
\begin{equation}\label{e:Spli.2.1.2}
 P^\pm_{\lambda^\ast}A_\lambda(z+ \psi(\lambda,z))=0\quad\forall (\lambda,z)\in \Lambda_0
 \times B_{H^0_{\lambda^\ast}}(0,\epsilon),
 \end{equation}
an open neighborhood $W$ of $0$ in $H$ and a homeomorphism
\begin{eqnarray}\label{e:Spli.2.1.3}
&&\Lambda_0\times B_{H^0_{\lambda^\ast}}(0,\epsilon)\times
\left(B_{H^+_{\lambda^\ast}}(0, \epsilon) + B_{H^-_{\lambda^\ast}}(0, \epsilon)\right)\to \Lambda_0\times W,\nonumber\\
&&\hspace{20mm}({\lambda}, z, u^++u^-)\mapsto ({\lambda},\Phi_{{\lambda}}(z, u^++u^-))
\end{eqnarray}
satisfying $\Phi_{{\lambda}}(0)=0$, such that for each $\lambda\in \Lambda_0$,
%the functional $\mathcal{L}_{\lambda}$ satisfies
\begin{eqnarray}\label{e:Spli.2.2}
&&\mathcal{L}_{\lambda}\circ\Phi_{\lambda}(z, u^++ u^-)=\|u^+\|^2-\|u^-\|^2+ \mathcal{
L}_{{\lambda}}(z+ \psi({\lambda}, z))\\
&& \quad\quad \forall (z, u^+ + u^-)\in  B_{H^0_{\lambda^\ast}}(0,\epsilon)\times
\left(B_{H^+_{\lambda^\ast}}(0, \epsilon) + B_{H^-_{\lambda^\ast}}(0, \epsilon)\right).\nonumber
\end{eqnarray}
 Moreover, there also hold: {\bf (i)} %with $\mathscr{B}_\lambda(0)=  B_1(0)-\lambda B_2(0)$,
 $$
d_z\psi(\lambda,z)=-[P^\pm_{\lambda^\ast}\circ({B}_{\lambda}(z+\psi(\lambda,z))|_{X^\pm_{\lambda^\ast}})]^{-1}
\circ(P^\pm_{\lambda^\ast}\circ({B}_\lambda(z+\psi(\lambda,z))|_{H^0_{\lambda^\ast}})).
$$

\noindent{\bf (ii)} The functional
\begin{equation}\label{e:Spli.2.3}
\mathcal{L}_{\lambda}^\circ: B_{H^0_{\lambda^\ast}}(0,\epsilon)\to \mathbb{R},\;
z\mapsto\mathcal{L}_{\lambda}(z+ \psi({\lambda}, z))
\end{equation}
 is of class $C^{2}$, its first-order and second-order differentials  at $z\in
B_{H^0}(0, \epsilon)$ are given by
 \begin{eqnarray}\label{e:Spli.2.4}
&& d\mathcal{L}^\circ_\lambda(z)[\zeta]=\bigl(A_\lambda(z+ \psi(\lambda, z)), \zeta\bigr)_H\quad\forall \zeta\in H^0,\\
  &&d^2\mathcal{L}^\circ_\lambda(0)[z,z']=\left(P^0_{\lambda^\ast}\bigr[{B}_\lambda(0)-
 {B}_\lambda(0)(P^\pm_{\lambda^\ast}{B}_{\lambda}(0)|_{X^\pm_{\lambda^\ast}})^{-1}
 (P^\pm_{\lambda^\ast}{B}_\lambda(0))\bigr]z, z'\right)_H,\nonumber\\
&& \hspace{40mm} \forall z,z'\in H^0.\label{e:Spli.2.5}
 \end{eqnarray}
%\noindent{\bf (iii)}  For $p\in\mathbb{N}$, if $\Lambda$ is a $C^p$ manifold and
%$\Lambda\times U^X\ni (\lambda,x)\mapsto A(\lambda, x):=A_\lambda(x)\in X$ is $C^p$, then  $\psi$ is $C^p$ and
% \begin{eqnarray}\label{e:Spli.2.5.0}
% d_\lambda\psi(\lambda,z)[\lambda']=-[P^\pm_{\lambda^\ast}\circ({B}_{\lambda}(z+\psi(\lambda,z))|_{X^\pm_{\lambda^\ast}})]^{-1}
%(P^\pm_{\lambda^\ast}(D_\lambda{A}(\lambda, z+\psi(\lambda,z))[\lambda']))
%  \end{eqnarray}
% for $\lambda'\in T_\lambda\Lambda$, and
% \begin{eqnarray}\label{e:Spli.2.5.0+}
% &&D_\lambda(\nabla\mathcal{L}^\circ_\lambda(z))[\lambda']=
% P^0_{\lambda^\ast}(D_\lambda{A}(\lambda, z+\psi(\lambda,z))[\lambda'])-\nonumber\\
% &&P^0_{\lambda^\ast}(B_\lambda(\lambda, z+\psi(\lambda,z))
% \bigl([P^\pm_{\lambda^\ast}({B}_{\lambda}(z+\psi(\lambda,z))|_{X^\pm_{\lambda^\ast}})]^{-1}
%(P^\pm_{\lambda^\ast}(D_\lambda{A}(\lambda, z+\psi(\lambda,z))[\lambda']))\bigr)\nonumber\\
%   \end{eqnarray}
%  for $\lambda'\in T_\lambda\Lambda$. In particular,
%    \begin{eqnarray}\label{e:Spli.2.5.0++}
% D_\lambda(\nabla\mathcal{L}^\circ_\lambda(z))|_{(\lambda, z)=(\lambda^\ast,0)}[\lambda']=
% P^0_{\lambda^\ast}(D_\lambda{A}(\lambda^\ast, 0)[\lambda'])\quad\forall\lambda'\in T_{\lambda^\ast}\Lambda
%  \end{eqnarray}
% since the second term in the right side of (\ref{e:Spli.2.5.0+}) becomes zero at $(\lambda, z)=(\lambda^\ast,0)$.
%

\noindent{\bf (iii)}  If a compact Lie group $G$  acts on $H$ orthogonally, which induces  $C^1$ isometric actions on $X$,
 both $U$ and $\mathcal{L}_\lambda$ are $G$-invariant (and hence $H^0_\lambda$, $H^\pm_\lambda$
are $G$-invariant subspaces), then for each $\lambda\in \Lambda$,
the above maps $\psi(\lambda, \cdot)$  and $\Phi_{\lambda}(\cdot,\cdot)$  are
  $G$-equivariant, and $\mathcal{L}^\circ_{\lambda}$ is $G$-invariant.
  If for some $p\in\mathbb{N}$,  $\Lambda$ is a $C^p$ manifold and
$\Lambda\times U^X\ni (\lambda,x)\mapsto A(\lambda, x):=A_\lambda(x)\in X$ is $C^p$, then so is $\psi$.
\end{theorem}

%\begin{proof}[\it Proof of Theorem~\ref{th:Bi.2.1E}]
%\begin{proof}[{\it Proof}\;{\rm (sketch)}]
\begin{proof}
%Let us sketch its proof as follows.
Take $\eta>0$ so small that $B_{H^0_{\lambda^\ast}}(0,\eta)\oplus B_{X^\pm_{\lambda^\ast}}(0,\eta)\subset U^X$.
Since $P^\pm_{\lambda^\ast}\circ(B(0)|_{X^\pm_{\lambda^\ast}})$ is a Banach space isomorphism from $X^\pm_{\lambda^\ast}$ onto itself, applying
the implicit function theorem to the map
$$
\Lambda\times B_{H^0_{\lambda^\ast}}(0,\eta)\oplus B_{X^\pm_{\lambda^\ast}}(0,\eta)\to X^\pm_{\lambda^\ast},\;(\lambda,
z, x)\mapsto P^\pm_{\lambda^\ast}(A_\lambda(z+ x))
$$
near $(\lambda^\ast, 0)$ we can get (\ref{e:Spli.2.1.1})-(\ref{e:Spli.2.1.2}) and (i)-(ii).

By (v), $X^\star_{\lambda^\ast}=H^\star_{\lambda^\ast}$, $\star=0,-$, have finite dimensions.
Let $e_1,\cdots,e_m$ be an unit orthogonal basis of $H^0_{\lambda^\ast}\oplus H^-_{\lambda^\ast}$.
By the proof of \cite[Theorem~1.1]{Lu1} (or
the proof of \cite[Lemma~3.3]{Lu2}), we have
\begin{eqnarray}\label{e:Spli.2.5.1}
|(B_\lambda(x)u,v)_H-(B_{\lambda^\ast}(0)u,v)_H|\le \omega_\lambda(x)\|u\|\|v\|
\end{eqnarray}
for any $x\in U\cap X$, $u\in H^0_{\lambda^\ast}\oplus H^-_{\lambda^\ast}$ and $v\in H$, where
$$
\omega_\lambda(x)=\left(\sum^m_{i=1}\|P_\lambda(x)e_i-P_{\lambda^\ast}(0)e_i\|^2\right)^{1/2}+ \sqrt{m}\|Q_\lambda(x)-Q_{\lambda^\ast}(0)\|.
$$
Clearly,  (i), (iii) and (iv) imply that $\omega_\lambda(x)\to 0$
as $x\in U\cap X$ approaches to $0$ in $H$ and $\lambda\in\Lambda$ converges to $\lambda^\ast$.
It is  clear that (\ref{e:Spli.2.5.1}) leads to
\begin{eqnarray}\label{e:Spli.2.5.2}
|(B_\lambda(x)u,v)_H|\le \omega_\lambda(x)\|u\|\|v\|,\quad\forall u\in H^+_{\lambda^\ast}, \;\forall
v\in H^0_{\lambda^\ast}\oplus H^-_{\lambda^\ast}
\end{eqnarray}
for any $x\in U\cap X$.
Moreover,  there exists $a_0>0$ such that $(B_{\lambda^\ast}(0)u, u)_H\ge 2a_0\|u\|^2\;\forall u\in H^+_{\lambda^\ast}$.
Take a neighborhood of $0\in H$, $V\subset U$,  and shrink $\Lambda_0\subset\Lambda$, such that
 $\omega_\lambda(x)<a_0$ for all $x\in V\cap X$ and $\lambda\in\Lambda_0$.
  %(or the proof of \cite[Theorem~1.1]{Lu1}),
  It follows from this and (\ref{e:Spli.2.5.1}) that
\begin{eqnarray}\label{e:Spli.2.5.3}
(B_\lambda(x)v,v)_H\le -2a_0\|v\|^2+ \omega_\lambda(x)\|v\|^2\le -a_0\|v\|^2\quad\forall v\in H^-_{\lambda^\ast}
\end{eqnarray}
for all $x\in V\cap X$ and $\lambda\in\Lambda_0$. As in the proof of \cite[Lemma~3.4]{Lu2}(i), (see below),
by shrinking $V$ and $\Lambda_0$ we can get $a_1>0$ such that
for all $x\in V\cap X$ and $\lambda\in\Lambda_0$,
\begin{eqnarray}\label{e:Spli.2.5.4}
(B_\lambda(x)u,u)_H\ge a_1\|u\|^2\quad\forall u\in H^+_{\lambda^\ast}.
\end{eqnarray}
%We postpone its proof until the end.
%(Indeed,  by  contradiction suppose  that  there exist
%sequences $(x_n)\subset V\cap X$ with $\|x_n\|\to 0$, $(\lambda_n)\subset\Lambda$ with $\lambda_n\to\lambda^\ast$
%and $(u_n)\subset SH^+_{\lambda^\ast}$, such that
%$(B_{\lambda_n}(x_n)u_n, u_n)_H<1/n$ for all $n=1,2,\cdots$.
%Then in the proof of \cite[Lemma~3.4]{Lu2}(i), it suffices to replace
%$B(x_n)$ and $B(\theta)$, $P(x_n)$ and $P(\theta)$, $Q(x_n)$ and $Q(\theta)$,
%$H^+$ by $B_{\lambda_n}(x_n)$ and $B_{\lambda^\ast}(0)$, $P_{\lambda_n}(x_n)$ and $P_{\lambda^\ast}(0)$, $Q_{\lambda_n}(x_n)$ and $Q_{\lambda^\ast}(0)$, $H^+_{\lambda^\ast}$, respectively.)
Since $\psi(\lambda,0)=0$, we can choose $\epsilon\in (0,\eta)$ such that
$z+\psi(\lambda,z)+ u^++u^-\in V$ for all $\lambda\in\Lambda_0$,
$z\in\bar{B}_{H^0_{\lambda^\ast}}(0,\epsilon)$ and $u^\star\in \bar{B}_{H^\star_{\lambda^\ast}}(0,\epsilon)$, $\star=+,-$.
For each $\lambda\in \Lambda_0$, we define a functional
\begin{eqnarray*}
{\bf F}_\lambda: B_{H^0_{\lambda^\ast}}(0,\epsilon)\oplus B_{X^\pm_{\lambda^\ast}}(0,\epsilon)\to \mathbb{R},
(z,u)\mapsto \mathcal{L}_{\lambda}(z+\psi({\lambda}, z)+u)- \mathcal{L}_{{\lambda}}(z+ \psi({\lambda}, z)).
\end{eqnarray*}
Then following the proof ideas of \cite[Lemma~3.5]{Lu2} and shrinking $\Lambda_0$ and $\epsilon>0$
(if necessary) we can obtain positive constants $\mathfrak{a}_1$ and $\mathfrak{a}_2$ such that
\begin{eqnarray*}
 &&(D_2{\bf F}_{{\lambda}}(z, u^+ + u^-_2)-D_2{\bf F}_{{\lambda}}(z, u^++ u^-_1))[u^-_2-u^-_1]\le
-\mathfrak{a}_1\|u^-_2-u^-_1\|^2,\\
&&D_2{\bf F}_{{\lambda}}(z, u^++u^-)[u^+-u^-]\ge  \mathfrak{a}_2(\|u^+\|^2+ \|u^-\|^2)
\end{eqnarray*}
for all $\lambda\in \Lambda_0$,
$z\in B_{H^0_{\lambda^\ast}}(0,\epsilon)$ and $u^+\in B_{X^+_{\lambda^\ast}}(0,\epsilon)$, $u^-\in B_{X^-_{\lambda^\ast}}(0,\epsilon)$.
Using \cite[Theorem~A.2]{Lu2} leads to the desired results.
%\end{proof}
%As in the proof of \cite[Theorems~2.12,2.16]{Lu7} the other conclusions may follow.

%\noindent{\bf Proof of (\ref{e:Spli.2.5.4})}.\quad
\textsf{Finally, for completeness let us prove (\ref{e:Spli.2.5.4})  by  contradiction.}  %that (\ref{e:Spli.2.5.4}) does not hold.}
 Then there exist
sequences $(x_n)\subset V\cap X$ with $\|x_n\|\to 0$, $(\lambda_n)\subset\Lambda$ with $\lambda_n\to\lambda^\ast$
and
$(u_n)\subset SH^+_{\lambda^\ast}$, such that
$(B_{\lambda_n}(x_n)u_n, u_n)_H<1/n$ for any $n=1,2,\cdots$.
Passing a subsequence, we can assume that
\begin{equation}\label{e:2.24}
(B_{\lambda_n}(x_n)u_n, u_n)_H\to\beta\le 0\;\hbox{as}\;n\to\infty,
\end{equation}
and that $u_n\rightharpoonup u_0$ in $H$. We claim: $u_0\ne
\theta$. In fact, by (ii) we have
\begin{eqnarray}\label{e:2.25}
(B_{\lambda_n}(x_n)u_n, u_n)_H&=&(P_{\lambda_n}(x_n)u_n, u_n)_H + (Q_{\lambda_n}(x_n)u_n, u_n)_H\nonumber\\
&\ge & c_0+ (Q_{\lambda_n}(x_n)u_n, u_n)_H\quad\forall n>n_0.
\end{eqnarray}
 Moreover, a direct computation gives
\begin{eqnarray}\label{e:2.26}
&&\!\!\!\!\!\quad |(Q_{\lambda_n}(x_n)u_n, u_n)_H-(Q_{\lambda^\ast}(0)u_0, u_0)_H|\\
&&\!\!\!\!\!=|((Q_{\lambda_n}(x_n)-Q_{\lambda^\ast}(0))u_n, u_n)_H+ (Q_{\lambda^\ast}(0)u_n, u_n)_H-(Q_{\lambda^\ast}(0)u_0, u_n)_H\nonumber\\
&&\hspace{70mm}+ (Q_{\lambda^\ast}(0)u_0, u_n-u_0)_H|\nonumber\\
&&\!\!\!\!\!\le \|Q_{\lambda_n}(x_n)-Q_{\lambda^\ast}(0)\|\cdot\|u_n\|^2+
\|Q_{\lambda^\ast}(0)u_n-Q_{\lambda^\ast}(0)u_0\|\cdot\|u_n\|\nonumber\\
&&\hspace{40mm}+
|(Q_{\lambda^\ast}(0)u_0, u_n-u_0)_H|\nonumber\\
&&\!\!\!\!\!\le \|Q_{\lambda_n}(x_n)-Q_{\lambda^\ast}(0)\|+ \|Q_{\lambda^\ast}(0)u_n-Q_{\lambda^\ast}(0)u_0\|+
|(Q_{\lambda^\ast}(0)u_0, u_n-u_0)_H|.\nonumber
\end{eqnarray}
Since $u_n\rightharpoonup u_0$ in $H$,
$\lim_{n\to\infty}|(Q_{\lambda^\ast}(0)u_0, u_n-u_0)_H|=0$. We have also
\begin{equation}\label{e:2.27}
\lim_{n\to\infty}\|Q_{\lambda^\ast}(0)u_n-Q_{\lambda^\ast}(0)u_0\|=0
\end{equation}
by the  compactness  of $Q_{\lambda^\ast}(0)$, and
\begin{equation}\label{e:2.28}
\lim_{n\to\infty}\|Q_{\lambda_n}(x_n)-Q_{\lambda^\ast}(0)\|\le
\lim_{n\to\infty}\|Q_{\lambda_n}(x_n)-Q_{\lambda_n}(0)\|+\lim_{n\to\infty}\|Q_{\lambda_n}(0)-Q_{\lambda^\ast}(0)\|=0
\end{equation}
 by the conditions (iii)-(iv).
Hence  (\ref{e:2.26})-(\ref{e:2.28}) give
\begin{equation}\label{e:2.29}
\lim_{n\to\infty}(Q_{\lambda_n}(x_n)u_n, u_n)_H=(Q_{\lambda^\ast}(0)u_0, u_0)_H.
\end{equation}
This and (\ref{e:2.24})-(\ref{e:2.25}) yield
$$
0\ge \beta=\lim_{n\to\infty}(B_{\lambda_n}(x_n)u_n, u_n)_H\ge c_0 +
(Q_{\lambda^\ast}(0)u_0,u_0)_H,
$$
which implies $u_0\ne\theta$. Note that $u_0$ also sits in $H^+_{\lambda^\ast}$.

As above, using (\ref{e:2.28}) we derive
\begin{eqnarray}\label{e:2.30}
&&|(Q_{\lambda_n}(x_n)u_0, u_n)_H- (Q_{\lambda^\ast}(0)u_0, u_0)_H|\\
&\le& |(Q_{\lambda_n}(x_n)u_0, u_n)_H- (Q_{\lambda^\ast}(0)u_0, u_n)_H|+ |(Q_{\lambda^\ast}(0)u_0,
u_n)_H- (Q_{\lambda^\ast}(0)u_0,
u_0)_H|\nonumber\\
&\le&  \|Q_{\lambda_n}(x_n)-Q_{\lambda^\ast}(0)\|\cdot\|u_0\|+ |(Q_{\lambda^\ast}(0)u_0, u_n-u_0)_H
|\to 0.\nonumber
\end{eqnarray}
 Note that
\begin{eqnarray*}
&&(B_{\lambda_n}(x_n)(u_n-u_0), u_n-u_0)_H\\
&=&(P_{\lambda_n}(x_n)(u_n-u_0), u_n-u_0)_H + (Q_{\lambda_n}(x_n)(u_n-u_0), u_n-u_0)_H\\
&\ge& c_0\|u_n-u_0\|^2+ (Q_{\lambda_n}(x_n)(u_n-u_0), u_n-u_0)_H\\
&\ge& (Q_{\lambda_n}(x_n)u_n, u_n)_H-2(Q_{\lambda_n}(x_n)u_0, u_n)_H+ (Q_{\lambda^\ast}(0)u_0, u_0)_H.
\end{eqnarray*}
It follows from this and (\ref{e:2.29})-(\ref{e:2.30}) that
\begin{eqnarray}\label{e:2.31}
\liminf_{n\to\infty}(B_{\lambda_n}(x_n)(u_n-u_0),
u_n-u_0)_H\ge\lim_{n\to\infty}(Q_{\lambda_n}(x_n)(u_n-u_0), u_n-u_0)_H = 0.
\end{eqnarray}
Note that $u_n\rightharpoonup u_0$ implies that $(P_{\lambda^\ast}(0)u_0,
u_n-u_0)_H\to 0$. By (D2) and (\ref{e:2.30}) we get
\begin{eqnarray*}
&&|(B_{\lambda_n}(x_n)u_0, u_n)_H-(B_{\lambda^\ast}(0)u_0, u_0)_H|\\
&=&|(P_{\lambda_n}(x_n)u_0, u_n)_H+ (Q_{\lambda_n}(x_n)u_0, u_n)_H- (P_{\lambda^\ast}(0)u_0, u_0)_H-
(Q_{\lambda^\ast}(0)u_0, u_0)_H|\\
&\le & |(P_{\lambda_n}(x_n)u_0, u_n)_H-(P_{\lambda^\ast}(0)u_0, u_0)_H|+ |(Q_{\lambda^\ast}(x_n)u_0,
u_n)_H-(Q_{\lambda^\ast}(0)u_0, u_0)_H| \\
&\le & |(P_{\lambda_n}(x_n)u_0, u_n)_H-(P_{\lambda^\ast}(0)u_0, u_n)_H|+|(P_{\lambda^\ast}(0)u_0,
u_n)_H-(P_{\lambda^\ast}(0)u_0, u_0)_H| \\
&&\quad + |(Q_{\lambda_n}(x_n)u_0, u_n)_H-(Q_{\lambda^\ast}(0)u_0, u_0)_H|
\\
&\le & \|P_{\lambda_n}(x_n)u_0- P_{\lambda^\ast}(0)u_0\| + |(P_{\lambda^\ast}(0)u_0, u_n-u_0)_H| \\
&&\quad + |(Q_{\lambda_n}(x_n)u_0, u_n)_H-(Q_{\lambda^\ast}(0)u_0, u_0)_H| \to 0.
\end{eqnarray*}
 Similarly, we have
$\lim_{n\to\infty}(B_{\lambda_n}(x_n)u_0, u_0)_H=(B_{\lambda^\ast}(0)u_0, u_0)_H$.
From these, (\ref{e:2.24}) and (\ref{e:2.31}) it follows that
\begin{eqnarray*}
0&\le& \liminf_{n\to\infty}(B_{\lambda_n}(x_n)(u_n-u_0), u_n-u_0)_H\\
&=& \liminf_{n\to\infty}[(B_{\lambda_n}(x_n)u_n,u_n)_H-2(B_{\lambda_n}(x_n)u_0, u_n)_H+
(B_{\lambda_n}(x_n)u_0,
u_0)_H]\\
&=&\lim_{n\to\infty}(B_{\lambda_n}(x_n)u_n,u_n)_H- (B_{\lambda^\ast}(0)u_0, u_0)_H
=\beta- (B_{\lambda^\ast}(0)u_0, u_0)_H.
\end{eqnarray*}
Namely, $(B_{\lambda^\ast}(0)u_0,u_0)_H\le\beta\le 0$. It contradicts to
the fact that $(B_{\lambda^\ast}(0)u, u)_H\ge 2a_0\|u\|^2\;\forall u\in H^+_{\lambda^\ast}$
 because $u_0\in H^+_{\lambda^\ast}\setminus\{0\}$.
 \end{proof}
%\hfill$\Box$\vspace{2mm}

\begin{remark}\label{rm:Spl.2.4}
{\rm
{\bf (i)} As in the proof of \cite[Claim~2.17]{Lu7} we may show:
if $0\in H$ is a nondegenerate critical point of $\mathcal{L}_\lambda$,
i.e., ${\rm Ker}(B_\lambda(0))=\{0\}$,
then $0\in H^0_{\lambda^\ast}$ is such a critical point of $\mathcal{L}^\circ_\lambda$ too.\\
{\bf (ii)}   Every critical point $z$ of $\mathcal{L}_{\lambda}^\circ$ in
$B_{H^0_{\lambda^\ast}}(0,\epsilon)$ yields a critical point $z+ \psi({\lambda}, z)$
of $\mathcal{L}_{\lambda}$, and  $z+ \psi({\lambda}, z)$ sits in $X$.
Moreover, if $z\to 0$ in $H^0_{\lambda^\ast}$, then $z+ \psi({\lambda}, z)\to 0$ in $X$.
Conversely, every critical point of $\mathcal{L}_{\lambda}|_X$ near $0\in X$
has the form $z+\psi({\lambda}, z)$, where $z\in B_{H^0_{\lambda^\ast}}(0,\epsilon)$ is a
critical point of $\mathcal{L}_{\lambda}^\circ$.
Clearly, if $0\in H$ is an isolated critical point of $\mathcal{L}_{\lambda}$,
then $0\in H^0_{\lambda^\ast}$ is also an isolated critical point of $\mathcal{L}_{\lambda}^\circ$.
Conversely, it might not be true, which is  different from \cite[Theorem~2.2]{Lu7}.
%$\mathcal{F}_{\lambda}$ may have a critical point near $0\in H$
%even if $0\in H^0$ is an isolated critical point of $\mathcal{F}_{\lambda}^\circ$.
%These are important for the  bifurcation theorems in Section~\ref{sec:CR}.
}
\end{remark}

Let $\nu_\lambda$ and $\mu_\lambda$ denote the nullity and  Morse index
of the functional $\mathcal{L}_{\lambda}$ at $0$. They are finite numbers and are
equal to $\dim H_\lambda^0$  and $\dim H^-_\lambda$ (that is, the nullity and Morse index
of the quadratic form $({B}_\lambda u,u)$ on $H$), respectively.
 As in \cite[Corollary~2.6]{Lu2} we have

\begin{corollary}[Shifting]\label{cor:A.6}
Under the assumptions of Theorem~\ref{th:A.5-}, suppose
 for some $\lambda\in [\lambda^\ast-\delta, \lambda^\ast+\delta]$ that $0\in H$ is an isolated critical point of $\mathcal{L}_{{\lambda}}$ (thus $0\in H^0_{\lambda^\ast}$ is that of $\mathcal{L}_{{\lambda}}^\circ$). Then
\begin{eqnarray*}
C_q(\mathcal{L}_{{\lambda}}, 0;{\bf K})=C_{q-\mu_{\lambda^\ast}}(\mathcal{L}^\circ_{{\lambda}}, 0;{\bf K}),\quad\forall
q\in\mathbb{N}\cup\{0\}.
\end{eqnarray*}
\end{corollary}

%%If $\nu=\dim H^0\ne 0$,
% By the proof of \cite[Theorem~2.1]{Lu2} we may directly obtain its
%parameterized version as follows.

\begin{corollary}\label{cor:A.5}
Under Hypothesis~\ref{hyp:Bif.2.2.0} the conclusions of Theorem~\ref{th:A.5-}
hold with family $\{\mathcal{L}_\lambda:=\mathcal{L}-\lambda\widehat{\mathcal{L}}\,|\,
\lambda\in\Lambda=\mathbb{R}\}$.
\end{corollary}

\begin{remark}\label{rm:Spl.2.4}
{\rm If  Hypothesis~\ref{hyp:Bif.2.2.0} in Corollary~\ref{cor:A.5} is replaced by
Hypothesis~\ref{hyp:Bif.2.2.0+}, that is, we  only assume that $A: U^X\to X$  is
 G\^ateaux differentiable, and strictly Fr\'{e}chet differentiable at $0\in U^X$,
then we can still prove Corollary~\ref{cor:A.5} with weaker conclusions
that $\psi$ and $\mathcal{L}_{\lambda}^\circ$ are $C^{1-0}$ and $C^{2-0}$, respectively,
and that all $\psi(\lambda,\cdot)$ and $d\mathcal{L}_{\lambda}^\circ$ are strictly Fr\'echet differentiable
at $0\in H^0_{\lambda^\ast}$.}
\end{remark}

\begin{remark}\label{rm:Spl.2.5}
{\rm From the proofs of  \cite[Th.5.1.13]{Ch1} and \cite[Th.8.3]{MaWi} it is easily seen that
the classical splitting lemma for $C^2$ functionals has also the parameterized version as Theorem~\ref{th:A.5-}.
Let $U\subset H$ be as in Theorem~\ref{th:A.5-}, $\Lambda$ a  topological space,
and let $\mathcal{L}_\lambda\in C^2(U, \mathbb{R})$, $\lambda\in\Lambda$, be a  family of functionals
    satisfying $\mathcal{L}'_\lambda(0)=0\;\forall\lambda$, and be such that
  $\Lambda\times U\ni (\lambda,u)\mapsto\nabla\mathcal{L}_\lambda(u)\in H$ is  continuous.
 For some $\lambda^\ast\in\Lambda$, assume that $\mathcal{L}^{\prime\prime}_{\lambda^\ast}(0)$ is a Fredholm operator with nontrivial kernel.
 Then the conclusions of Theorem~\ref{th:A.5-} hold as long as we take $X=H$, $A_\lambda=\nabla \mathcal{L}_\lambda$ and
 $B_\lambda=\mathcal{L}^{\prime\prime}_{\lambda}$, and  replace the paragraph
 `` an open neighborhood $W$ of $0$ in $H$ and a homeomorphism ......
satisfying $\Phi_{{\lambda}}(0)=0$" by
 `` a $C^0$ map
\begin{eqnarray*}
\Lambda_0\times B_{H^0_{\lambda^\ast}}(0,\epsilon)\oplus
B_{H^+_{\lambda^\ast}}(0, \epsilon)\oplus B_{H^-_{\lambda^\ast}}(0, \epsilon)
\to  H,\quad ({\lambda}, u)\mapsto \Phi_{{\lambda}}(u)
\end{eqnarray*}
 such that for each $\lambda\in \Lambda_0$ the map
 $\Phi_{\lambda}$ is an  origin-preserving homeomorphism from
$B_{H^0_{\lambda^\ast}}(0,\epsilon)\oplus
B_{H^+_{\lambda^\ast}}(0, \epsilon)\oplus B_{H^-_{\lambda^\ast}}(0, \epsilon)$ onto
an open neighborhood $W_\lambda$ of $0$ in $H$".

If the kernel of $\mathcal{L}^{\prime\prime}_{\lambda^\ast}(0)$ is trivial,
we have a corresponding  parameterized Morse-Palais lemma.

Moreover, there also exist corresponding corollaries of Theorem~\ref{th:A.4},~\ref{th:A.5-} as above.

As noted in \cite[Remark~2.4]{Lu2}, these can directly follow from
 Theorem~\ref{th:A.4},~\ref{th:A.5-} and their corollaries if $\mathcal{L}^{\prime\prime}_{\lambda^\ast}(0)$
has also a finite dimensional negative definite space.
}
\end{remark}

\section{Appendix:\quad
 Parameterized Bobylev-Burman splitting  lemmas}\label{app:B}\setcounter{equation}{0}

%Parameterized versions for the splitting  lemmas in \cite{BoBu}}\label{app:B}\setcounter{equation}{0}

We here give parameterized versions of splitting  lemmas in \cite{BoBu}
in consistent notations with those of this paper.
%For being consistent with notations in this paper we  slight change  ones in \cite{BoBu}.
Let $H$ be a Hilbert space with inner product $(\cdot,\cdot)_H$
and the induced norm $\|\cdot\|$,  $X$  a Banach space with
norm $\|\cdot\|_X$, such that
\begin{description}
\item[(S)]   $X\subset H$ is dense in $H$ and
 the inclusion $X\hookrightarrow H$ is continuous.
\end{description}

Suppose that a $C^2$ functional $\mathscr{L}:B_X(0, \delta)\to\mathbb{R}$
satisfies  for some constant $M>0$:
\begin{description}
\item[(a)] $|d\mathscr{L}(x)[u]|\le M\|u\|,\;\forall x\in B_X(0, \delta)$, $\forall u\in X$.
\item[(b)] $|d^2\mathscr{L}(x)[u,v]|\le M\|u\|\cdot\|v\|,\;\forall x\in B_X(0, \delta)$, $\forall u,v\in X$.
\end{description}

%
%{\bf (a)} $|d\mathscr{L}(x)[u]|\le M\|u\|,\;\forall x\in B_X(0, \delta)$, $\forall u\in X$.\\
%{\bf (b)} $|d^2\mathscr{L}(x)[u,v]|\le M\|u\|\cdot\|v\|,\;\forall x\in B_X(0, \delta)$, $\forall u,v\in X$.

Then there are bounded maps $A:B_X(0, \delta)\to H$
 and $B:B_X(0, \delta)\to \mathscr{L}_s(H)$ such that
\begin{eqnarray}\label{e:BB.-2}
d\mathscr{L}(x)[u]=(A(x),u)_H\quad\hbox{and}\quad
d^2\mathscr{L}(x)[u,v]=(B(x)u,v)_H
\end{eqnarray}
for all $x\in B_X(0, \delta)$ and for all $u,v\in X$.

 Suppose also:
 \begin{description}
\item[(c)] $A(B_X(0, \delta))\subset X$ and $A:B_X(0, \delta)\to X$ is
           uniformly continuously differentiable.
\item[(d)] (For any $x\in B_X(0, \delta)$, (c) and (\ref{e:BB.-2}) imply that $B(x)(X)\subset X$
and $X\ni y\mapsto B(x)y\in X$ is an element in $\mathscr{L}(X)$, denoted by
$B(x)|_X$ without confusions.)
           $B(\cdot)|_X: B_X(0, \delta)\to \mathscr{L}(X)$ is continuously differentiable.
\item[(e)]  $B:B_X(0, \delta)\to \mathscr{L}_s(H)$ is uniformly continuous.
\end{description}

A functional $\mathscr{L}$ satisfying the conditions (a)-(e) is called {\bf $(B_X(0, \delta),
H)$-regular}. It was proved in {\cite[Lemma~1.1]{BoBu}} that
$B(\cdot)|_X: B_X(0, \delta)\to \mathscr{L}(X)$ is a Fr\'echet derivative of $A$ as a map
from $B_X(0, \delta)$ to $X$.

In this paper, \textsf{by the spectrum of a linear operator on a real Banach space
we always mean one of its natural complex linear extension on the complexification of
the real Banach space} (\cite[p.14]{DaKr}). Similarly, if  others cannot be clearly explained in the real world,
we consider their complexification and then take invariant parts under complex adjoint on the complexification
spaces.

Let $\mathscr{S}(X):=\{L\in\mathscr{L}(X)\,|\,  (Lu,v)_H=(u,Lv)_H\;\forall u,v\in X\}$.
Recall the following Riesz lemma (cf. \cite[Lemma~1.3]{BoBu} and \cite[p.19]{DaKr}, \cite[p.54, Lemma~1]{Tr} and
 \cite[p.59, Propositions~2,3]{Tr}).

\begin{lemma}\label{lem:BB.2}
 Under the assumption {\bf (S)}, let $D\in\mathscr{S}(X)$ and  its spectrum have a decomposition
 $$
 \sigma(D)=\{0\}\cup\sigma_+(D)\cup\sigma_-(D),
 $$
 where $\sigma_+(D)$ and $\sigma_-(D)$ are closed subsets of $\sigma(D)$ contained in
 the interiors of the left and right halfplanes, respectively.
  Then the space $X$ can be decomposed into a direct sum of
spaces $X_0$, $X_+$ and $X_-$, which are closed in $X$, invariant with respect to $D$, orthogonal in $H$, and
$\sigma(D_\ast)=\sigma_\ast(D)$ for $D_\ast=D|_{X_\ast}$, $\ast=+,-$.  The projections $P_0$, $P_+$ and $P_-$,
corresponding to this decomposition, are symmetric operators with respect to the inner product in $H$, and
satisfy
$$
(P_\ast)^2=P_\ast,\quad \ast=0,+,-,\quad P_0+P_++P_-=id_X,
$$
and any two of three projections have zero compositions,
and moreover $P_\ast$, $\ast=+,-,0$,  are expressible as a limit of power series in $D$.
In addition, square roots $S_-$ and $S_+$ to $¡ªD_-$ and $D_+$
may be defined with the functional calculus, and are expressible in power
series in $D_-$ and $D_+$, and also symmetric with respect to the inner product in $H$.
%%%%%%%%%%%%%%%%%%%%%%%%%%%%%%%%%%%%%%%%%%%%%%%%%%%%%%%%%%%%%%%%%%%%%%%%%%%%%%%%%%%%%%%%%%%%%%%%%%%%%%%%%%
%%If $\dim X_-<\infty$, $(Du,v)_H=0$ for any  $u$ and $v$ belonging to any two distinct spaces of $X_\ast$,
%% $\ast=0,+,-$.
%%%%%%%%%%%%%%%%
\end{lemma}

 For the above $(B_X(0, \delta), H)$-regular functional $\mathscr{L}$,
 suppose that  $d\mathscr{L}(0)=0$, and that the spectrum $\sigma(B(0)|_X)$ of
 $B(0)|_X\in\mathscr{L}(X)$ satisfies:
\textsl{$\sigma(B(0)|_X)\setminus\{0\}$ is bounded away from the imaginary axis.}
By Lemma~\ref{lem:BB.2}, corresponding to the spectral sets $\{0\}$,
 $\sigma_+(B(0)|_X)$ and $\sigma_-(B(0)|_X)$
 we have a direct sum decomposition of Banach spaces,
 $X=X_0\oplus X_+\oplus X_-$.   Denote by  $P_0$, $P_+$ and $P_-$
  the projections corresponding to the decomposition.
 When $\sigma(B(0)|_X)$ does not intersect the imaginary axis,
 namely the operator $B(0)|_X\in\mathscr{L}(X)$  is  {\bf hyperbolic} (\cite{Uh0}),
  (in particular, $\dim X_0=0$),
  we say that the critical point $0$ is  {\bf nondegenerate}.
  (Such nondegenerate critical points are isolated by Morse lemma \cite[Theorem~1.1]{BoBu}.)
  Call $\nu:=\dim X_0$ and $\mu:=\dim X_-$
  the {\bf nullity} and the {\bf Morse index} of $0$, respectively.
(Using Lemma~\ref{lem:BB.2} we may prove
that $\dim X_-$ is the supremum of the dimensions of the vector subspaces
of $X$ on which $(B(0)u,u)_H$ is negative definite).
The following is a parametric version of \cite[Theorem~1.1]{BoBu}.

 %Bobylev and  Burman  proved the following Morse and splitting lemmas  in \cite{BoBu}.
%%%%%%%%%%%%%%%%%%%%%%%%%%%%%%%%%%%%%%%%%%%%%%%%%%%%%%%%%%%%%%%%%%%%%%%%%%%%%%%%%%%%%%%%
%%\begin{theorem}[{\cite[Theorem~1.1]{BoBu}}]\label{th:BB.3}
%%Let $\mathscr{L}:B_X(0, \delta)\to\mathbb{R}$
%%be a $(B_X(0, \delta), H)$-regular functional. Suppose that $0$
%%is a nondegenerate critical point of $\mathscr{L}$ in the above sense, i.e.,
%% $B(0)|_X$ a hyperbolic operator.  Then there exists $\epsilon\in (0, \delta)$ and
%%a $C^1$ origin-preserving  diffeomorphism  $\varphi$ from $B_X(0,\epsilon)$ onto an
%% open neighborhood of $0$ in $X$  such that
%%$$
%%\mathscr{L}\circ\varphi(x)=\frac{1}{2}(B(0)x, x)_H+ \mathscr{L}(0),
%%\quad\forall x\in B_X(0,\epsilon).
%%$$
%%Moreover,  after suitably shrinking $\epsilon>0$ the diffeomorphism  $\varphi$ is chosen to satisfy
%%$$
%%\mathscr{L}\circ\varphi(x)=\|P_+x\|_H^2-\|P_-x\|_H^2+ \mathscr{L}(0),
%%\quad\forall x\in B_X(0,\epsilon).
%%$$
%5\end{theorem}
%%%%%%%%%%%%%%%%%%%%%%%%%%%%%%%%%%%%%%%%%%%%%%%%%%%%%%%%%%%%%%%%%%%%%%%%%%%%

\begin{theorem}\label{th:BB.3}
Let $\Lambda$ be a topological space, and  $\{\mathscr{L}_\lambda\,|\,\lambda\in\Lambda\}$
 a family of $(B_X(0, \delta), H)$-regular functionals such that
%$\mathscr{L}_\lambda$ and
the corresponding operators
$A_\lambda$, $B_\lambda$ and  $B_\lambda|_X$   depend on $\lambda$ continuously.
 Suppose that $d\mathscr{L}_\lambda(0)=0\;\forall\lambda$,
and that for some $\lambda^\ast\in\Lambda$  the operator $B_{\lambda^\ast}(0)|_X\in\mathscr{L}(X)$
is hyperbolic (i.e., $0$ is a nondegenerate critical point of $\mathscr{L}_{\lambda_0}$).
Then there exist a neighborhood $\Lambda_0$ of $\lambda^\ast$ in $\Lambda$,
  $\epsilon\in (0, \delta)$ and  a family of  origin-preserving  $C^1$ diffeomorphism $\varphi_\lambda$ from
$B_X(0,\epsilon)$ onto an open neighborhood of $0$ in $X$, which continuously depend on $\lambda\in\Lambda_0$
  (i.e., $(\lambda,u)\mapsto\varphi_\lambda(u)$ is of class $C^0$),  such that
\begin{eqnarray}\label{e:BB.-1}
\mathscr{L}_\lambda\circ\varphi_\lambda(x)=\frac{1}{2}(B_{\lambda^\ast}(0)x, x)_H+ \mathscr{L}_\lambda(0),\quad\forall x\in B_X(0,\epsilon),\quad\forall\lambda\in\Lambda_0.
\end{eqnarray}
Actually, after suitably shrinking $\Lambda_0$ and $\epsilon>0$ these diffeomorphisms  $\varphi_\lambda$ can be chosen to satisfy
\begin{eqnarray}\label{e:BB.-0.5}
\mathscr{L}_\lambda\circ\varphi_\lambda(x)=\|P_+^{\lambda^\ast}x\|_H^2-\|P_-^{\lambda^\ast}x\|_H^2+ \mathscr{L}_\lambda(0),\quad\forall x\in B_X(0,\epsilon),\quad\forall\lambda\in\Lambda_0,
\end{eqnarray}
where $X=X^{\lambda^\ast}_+\oplus X_-^{\lambda^\ast}$ is a direct sum decomposition of Banach spaces
corresponding to the spectral sets  $\sigma_+(B_{\lambda^\ast}(0)|_X)$
 and $\sigma_-(B_{\lambda^\ast}(0)|_X)$, and $P_+^{\lambda^\ast}$ and $P_-^{\lambda^\ast}$
 denote the projections corresponding to the decomposition.
Moreover, if $\Lambda$ is a $C^1$ manifold,
and %$\mathscr{L}_\lambda$,
$A_\lambda$, $B_\lambda$ and $B_\lambda|_X$ depend on $\lambda$ in the $C^1$ way,
then $(\lambda,u)\mapsto\varphi_\lambda(u)$ is of class $C^1$.
\end{theorem}

Clearly, (\ref{e:BB.-1}) implies that  the nondegenerate critical point $0$ of $\mathscr{L}_{\lambda_0}$
is isolated.

(\ref{e:BB.-0.5})  was not pointed out in \cite[Theorem~1.1]{BoBu} explicitly.
It is easily derived from Lemma~\ref{lem:BB.2}. In fact, let
$S^-$ and $S^+$ be  square roots  to $-B_{\lambda^\ast}(0)|_{X_-}$ and $B_{\lambda^\ast}(0)|_{X_+}$,
respectively. Then
$$
(B_{\lambda^\ast}(0)x, x)_H=\|S_+P_+^{\lambda^\ast}x\|_H^2-\|S_-P_-^{\lambda^\ast}x\|_H^2=
\|P_+^{\lambda^\ast}S_+P_+^{\lambda^\ast}x\|_H^2-\|P_-^{\lambda^\ast}S_-P_-^{\lambda^\ast}x\|_H^2.
$$
 Let $\psi:X\to X$ be the isomorphism
defined by $\psi(x)=S_+P_+^{\lambda^\ast}x+ S_-P_-^{\lambda^\ast}x$. Replacing $\varphi_\lambda$ with
$\varphi_\lambda\circ\psi^{-1}$ in (\ref{e:BB.-1}) yields (\ref{e:BB.-0.5}).

For the sake of clarity we also give the proof of Theorem~\ref{th:BB.3}
though it is a slight modification of that of \cite[Theorem~1.1]{BoBu}.

\begin{proof}[\it Proof of Theorem~\ref{th:BB.3}]
Let $\mathcal{B}_\lambda(u)=\int^1_0(1-t)B_\lambda(tu)|_X dt$. Then
$\mathcal{B}_\lambda(0)=B_\lambda(0)|_X$,
 $B_X(0, \delta)\ni x\mapsto\mathcal{B}_\lambda(0)\in\mathscr{L}(X)$ is $C^1$ by (d),
  and $\mathscr{L}_\lambda(u)=(\mathcal{B}_\lambda(u)u,u)_H$ for all $u\in B_X(\theta,\delta)$
  by \cite[Lemma~1.2]{BoBu}. Since $B_{\lambda}(0)|_X$ continuously depending on $\lambda\in\Lambda$
  and the spectrum $\sigma(B_{\lambda^\ast}(0)|_X)$ does not intersect the imaginary axis,
  by shrinking $\Lambda$ we can assume that for some $r>0$,
\begin{eqnarray}\label{e:S.3.2}
\sigma(B_{\lambda}(0)|_X)\cap\{z=x+iy\,|\, |x|\le r\}=\emptyset,\quad\forall \lambda\in\Lambda,
\end{eqnarray}
and that for any $\lambda\in\Lambda$, $\sigma(B_{\lambda}(0)|_X)$ is equal to the union
$\sigma(B_{\lambda}(0)|_X)_+\cup\sigma(B_{\lambda}(0)|_X)_-$,  where
$$
\sigma(B_{\lambda}(0)|_X)_+=\{\mu\in \sigma(B_{\lambda}(0)|_X)\,|\, {\rm Re}\mu>0\},\quad
\sigma(B_{\lambda}(0)|_X)_-=\{\mu\in \sigma(B_{\lambda}(0)|_X)\,|\, {\rm Re}\mu<0\}.
$$
By Lemma~\ref{lem:BB.2} with $D=B_{\lambda}(0)|_X$, corresponding to the
spectral sets $\sigma(B_{\lambda}(0)|_X)_+$ and $\sigma(B_{\lambda}(0)|_X)_-$ we have a
decomposition of spaces, $X=X_+^{\lambda}\oplus X_-^{\lambda}$,
 and  projections $P_\ast^\lambda:X\to X_\ast^\lambda$
belonging to $\mathscr{S}(X)$,  $\ast=+,-$.
%(Note that $X_{\lambda_0,\ast}=X_\ast$ and $P_{\lambda,\ast}=P_\ast$
%for $\ast=+,-$.)
By shrinking $\Lambda$ (if necessary) we can take $M>0$ such that
\begin{eqnarray}\label{e:S.3.2.1}
\sigma(B_{\lambda^\ast}(0)|_X)_+\cap(-M^2\sigma(B_{\lambda^\ast}(0)|_X)_-)=\emptyset.
\end{eqnarray}

\noindent{\bf Step 1}. {\it Special case with assumption $\sigma(B_{\lambda^\ast}(0)|_X)\cap(-\sigma(B_{\lambda^\ast}(0)|_X))=\emptyset$.}
Consider the  map
$$
\Phi:\Lambda\times B_X(0,\delta)\times\mathscr{S}(X)\to \mathscr{S}(X),\;(\lambda,u,Y)\mapsto
Y\mathcal{B}_\lambda(Yu)Y- \mathcal{B}_{\lambda^\ast}(0).
$$
It is continuous, and $C^1$ in $(u,Y)$, and satisfies $D_Y\Phi(\lambda^\ast,0,id_X)=q_{\mathcal{B}_{\lambda^\ast}(0)}$,
where the operator
$$
q_{\mathcal{B}_{\lambda^\ast}(0)}:\mathscr{S}(X)\to\mathscr{S}(X),\; Z\mapsto \mathcal{B}_{\lambda^\ast}(0)Z+ Z\mathcal{B}_{\lambda^\ast}(0).
$$
By \cite[Lemma~1.4]{BoBu}, the operator $q_{\mathcal{B}_{\lambda^\ast}(0)}$
has a linear bounded invertible one. With the classical implicit function theorem
we obtain a neighborhood $\Lambda_0$ of $\lambda^\ast$ in $\Lambda$,
  $\epsilon\in (0, \delta)$ and  a continuous map $\psi$ from
$\Lambda_0\times B_X(\theta,\epsilon)$ to $\mathscr{S}(X)$, which is also $C^1$
 with respect to the second variable, such that
 $$
 \psi(\lambda^\ast, 0)=id_X\quad\hbox{and}\quad
  \Phi(\lambda, u, \psi(\lambda,u))\equiv 0\quad\forall (\lambda,u)\in\Lambda_0\times B_X(\theta,\epsilon).
 $$
 (Clearly, $\Phi$ and so $\psi$ is $C^1$ if the conditions in the ``Moreover" part hold.)
 %Since $D_u\Phi(\lambda_0,0,id_X)=\mathcal{B}'_{\lambda_0}(0)$ we get
% $$
% D_u\psi(\lambda_0, 0)=- (D_Y\Phi(\lambda_0,0,id_X))^{-1}D_u\Phi(\lambda_0,0,id_X)=-(q_{\mathcal{B}_{\lambda_0}(0)})^{-1}\mathcal{B}'_{\lambda_0}(0).
% $$
Shrinking $\Lambda_0$ and $\epsilon>0$ such that
 $\|\psi(\lambda,u)-\psi(\lambda^\ast, 0)\|<1/2$ for all $(\lambda, u)\in \Lambda_0\times B_X(\theta,\epsilon)$,
we get that for each $\lambda\in\Lambda_0$,
$$
\varphi_\lambda(u): B_X(\theta,\epsilon)\to X,\;u\mapsto\psi(\lambda,u)u
$$
% given by $\varphi_\lambda(u)=\psi(\lambda,u)u$
 is  an origin-preserving  $C^1$-differeomorphism
 from  $B_X(\theta,\epsilon)$ onto an open neighborhood of $0\in X$ in $X$, and
 satisfies (\ref{e:BB.-1}) as checked in \cite{BoBu}.

\noindent{\bf Step 2}. {\it General case.}
Set $R_\lambda := P_+^{\lambda}+M P_-^{\lambda}$ and $\widehat{\mathscr{L}}_\lambda:=\mathscr{L}_\lambda\circ R_\lambda$.
By \cite[Lemma~1.2]{BoBu},
  $$
\widehat{\mathscr{L}}_\lambda(u)=(\widehat{\mathcal{B}}_\lambda(u)u,u)_H,\quad\forall u\in B_X(\theta,\delta),
$$
where $\widehat{\mathcal{B}}_\lambda=R_\lambda {\mathcal{B}}_\lambda R_\lambda$. Since
$\sigma(\widehat{\mathcal{B}}_\lambda)=\sigma(B_{\lambda}(0)|_X)_+\cup(M^2\sigma(B_{\lambda}(0)|_X)_-)$,
 (\ref{e:S.3.2.1}) implies that $\sigma(\widehat{\mathcal{B}}_{\lambda^\ast})$ is disjoint with
$-\sigma(\widehat{\mathcal{B}}_{\lambda^\ast})$. By the first step,
there exists a neighborhood $\Lambda_0$ of $\lambda^\ast$ in $\Lambda$,
  $\epsilon\in (0, \delta)$ and  a continuous map $\widehat\varphi$ from
$\Lambda_0\times B_X(\theta,\epsilon)$ to $X$, which is also $C^1$
 with respect to the second variable, such that
 for each $\lambda\in\Lambda_0$, $\widehat\varphi(\lambda,\cdot)$ is an origin-preserving  $C^1$-differeomorphism
 from  $B_X(\theta,\epsilon)$ onto an open neighborhood of $0\in X$ in $X$, and that
 $$
 \widehat{\mathscr{L}}_\lambda(\widehat\varphi(\lambda,u))=\frac{1}{2}(\widehat{\mathcal{B}}_{\lambda^\ast}(0)u,u)_H+
 \widehat{\mathscr{L}}_\lambda(0)=\frac{1}{2}(R_{\lambda^\ast}\circ(B_{\lambda^\ast}(0)|_X)\circ R_{\lambda^\ast} u,u)_H+
 {\mathscr{L}}_\lambda(0)
 $$
 Set $\tilde\varphi_\lambda(u):=R_\lambda\widehat\varphi(\lambda,(R_{\lambda^\ast})^{-1}u)$,
 where $(R_\lambda)^{-1}=P_+^{\lambda}+M^{-1} P_-^{\lambda}$ is the inverse  of $R_\lambda$. Then
 \begin{eqnarray*}
 {\mathscr{L}}_\lambda(\tilde\varphi_{\lambda}(u))&=&{\mathscr{L}}_\lambda(R_\lambda\widehat\varphi(\lambda,(R_{\lambda^\ast})^{-1}u)
 )= \widehat{\mathscr{L}}_\lambda(\widehat\varphi(\lambda,(R_{\lambda^\ast})^{-1}u) )\\
 &=&
 \frac{1}{2}(B_{\lambda^\ast}(0)|_X u,u)_H+ {\mathscr{L}}_\lambda(0),\quad\forall (\lambda,u)\in\Lambda_0\times B_X(0,\delta).
 \end{eqnarray*}
 The final claim is easily seen from the above proof.

\end{proof}

By the standard arguments
(see the following proof of Theorem~\ref{th:BB.5})
Theorem~\ref{th:BB.3}  can lead to:

\begin{theorem}[{\cite[Theorem~1.2]{BoBu}}]\label{th:BB.4}
Let $\mathscr{L}:B_X(0, \delta)\to\mathbb{R}$
be a $(B_X(0, \delta), H)$-regular functional. Suppose that $d\mathscr{L}(0)=0$, $0\in\sigma(B(0)|_X)$ and
that $\sigma(B(0)|_X)\setminus\{0\}$ is bounded away from the imaginary axis.
 Let $X=X_0\oplus X_+\oplus X_-$ be direct sum decomposition of Banach spaces,
 which corresponds to the spectral sets $\{0\}$,
  $\sigma_+(B(0)|_X)$ and $\sigma_-(B(0)|_X)$ as above.
 Then there exists $\epsilon\in (0, \delta)$,  a $C^1$ map $h:B_X(0,\epsilon)\cap X_0\to X_+\oplus X_-$ with $h(0)=0$ and
 a $C^1$ origin-preserving  diffeomorphism $\varphi$ from
$B_X(0,\epsilon)$ onto an open neighborhood of $0$ in $X$  such
that
\begin{eqnarray}\label{e:BB.0}
\mathscr{L}\circ\varphi(x)=\frac{1}{2}(B(0)x, x)_H+ \mathscr{L}^\circ(P_0x),\quad\forall x\in B_X(0,\epsilon).
\end{eqnarray}
where the functional
$\mathscr{L}^\circ: B_{X_0}(0, \epsilon)\to \mathbb{R}$, defined by
$\mathscr{L}^\circ(z)=\mathscr{L}(z+ h(z))$,
 is of class $C^{2}$, and has the first-order  derivative at $z_0\in
B_{X_0}(0, \epsilon)$  given by
 \begin{eqnarray}\label{e:BB.0+}
d\mathscr{L}^\circ(z_0)[z]=\bigl(A(z_0+ h(z_0)), z\bigr)_H,\quad\forall z\in X_0,
\end{eqnarray}
and the second-order derivative at $0\in B_{X_0}(0, \epsilon)$, $d^2\mathscr{L}^\circ(z_0)=0$.
Moreover, after suitably shrinking $\epsilon>0$ the diffeomorphism  $\varphi$ can be chosen to satisfy
$$
\mathscr{L}\circ\varphi(x)=\|P_+x\|_H^2-\|P_-x\|_H^2+ \mathscr{L}(h(P_0x)+P_0x),\quad\forall x\in B_X(0,\epsilon).
$$
\end{theorem}

(\ref{e:BB.0}) and (\ref{e:BB.0+}) imply  that $0\in X_0$ is an isolated critical point of $\mathscr{L}^\circ$
if and only if $0\in X$ is such a critical point of of $\mathscr{L}$.

\begin{remark}\label{rm:BB.4+}
{\rm (\ref{e:BB.0}) was proved in \cite[p.436, Theorem~1]{DiHiTr}
when $\mathscr{L}$ is of class $C^3$, and
the conditions (c)--(e) and the assumption that the operator $B(0)|_X\in\mathscr{L}(X)$
is hyperbolic are, respectively, replaced by the following:
 \begin{description}
\item[1)]  $A(B_X(0, \delta))\subset X$ and $A:B_X(0, \delta)\to X$ is $C^2$,
\item[2)] the operator $B(0)|_X\in\mathscr{L}(X)$ restricts an isomorphism on $X_++X_-$.
 \end{description}
Correspondingly, (\ref{e:BB.-1})  with $\Lambda=\{\lambda^\ast\}$ was proved in \cite[Theorem~1.4]{Tr1}
under similar conditions. When $\varphi$ is only required
to be a homeomorphism, Ming Jiang \cite[Theorem~2.5]{JM} proved (\ref{e:BB.0})
if the conditions {\bf (c)}--{\bf (e)} and the hyperbolicity assumption for $B(0)|_X\in\mathscr{L}(X)$ are replaced by the following:
 \begin{description}
\item[3)]  $A(B_X(0, \delta))\subset X$ and $A:B_X(0, \delta)\to X$ is $C^1$,
\item[4)]  $B:B_X(0, \delta)\to \mathscr{L}_s(H)$ is continuous,
\item[5)] C2) in Hypothesis~\ref{hyp:1.3}, and either $0\notin\sigma(B(0))$ or
$0$ is an isolated point of $\sigma(B(0))$.
 \end{description}
 The proof methods of the first two
are easily modified to give their parametric versions as below.
It seems troublesome to use the latter one. Later on, we shall discuss their relations
for the case used in this paper.
}
\end{remark}

%Carefully checking the proof of Theorem~\ref{th:BB.4} in \cite[Theorem~1.2]{BoBu}
%it is not hard to prove the following more general parametric version of it.
Here is a more general parametric version of Theorem~\ref{th:BB.4} (\cite[Theorem~1.2]{BoBu}).

\begin{theorem}\label{th:BB.5}
Let $\Lambda$ be a topological space, $\lambda^\ast\in\Lambda$, and $\{\mathscr{L}_\lambda\,|\,\lambda\in\Lambda\}$
 a family of $(B_X(0, \delta), H)$-regular functionals such that
 the corresponding operators $A_\lambda$, $B_\lambda$ and  $B_\lambda|_X$   depend on $\lambda$ continuously.
  Suppose that $d\mathscr{L}_\lambda(0)=0\;\forall\lambda$,
and  that  $\sigma(B_{\lambda^\ast}(0)|_X)\setminus\{0\}$
is bounded away from the imaginary axis.
 Let $X=X_0^{\lambda^\ast}\oplus X_+^{\lambda^\ast}\oplus X_-^{\lambda^\ast}$ be direct sum decomposition of Banach spaces,
 which corresponds to the spectral sets $\{0\}$, $\sigma_+(B_{\lambda^\ast}(0)|_X)$
 and $\sigma_-(B_{\lambda^\ast}(0)|_X)$.
 Denote by $P_0^{\lambda^\ast}$, $P_+^{\lambda^\ast}$ and $P_-^{\lambda^\ast}$ the corresponding  projections to this decomposition.
    Then  there exist  a neighborhood $\Lambda_0$ of $\lambda^\ast$ in $\Lambda$,
$\epsilon>0$, and
\begin{description}
\item[{\bf (i)}] a (unique) $C^0$ map
\begin{eqnarray}\label{e:BB.1-}
\mathfrak{h}:\Lambda_0\times B_{X_0^{\lambda^\ast}}(0,\epsilon)\to X_+^{\lambda^\ast}\oplus X_-^{\lambda^\ast}
\end{eqnarray}
which is $C^1$ in the second variable and
satisfies $\mathfrak{h}(\lambda, 0)=0\;\forall\lambda\in \Lambda_0$ and
\begin{eqnarray}\label{e:BB.1}
(id_X-P_0^{\lambda^\ast})A_\lambda(z+ \mathfrak{h}(\lambda,z))=0,\quad\forall (\lambda,z)\in \Lambda_0\times B_{X_0^{\lambda^\ast}}(0,\epsilon),
\end{eqnarray}
\item[{\bf (ii)}] a $C^0$ map
\begin{eqnarray}\label{e:BB.2}
\Lambda_0\times B_X(0,\epsilon)
\to  X,\quad
({\lambda}, x)\mapsto \Phi_{{\lambda}}(x)
\end{eqnarray}
 such that for each $\lambda\in \Lambda_0$ the map
 $\Phi_{\lambda}$ is a $C^1$  origin-preserving diffeomorphism from
$B_X(0,\epsilon)$ onto an open neighborhood $W_\lambda$ of $0$ in $X$ and
 satisfies  for any $x\in B_X(0,\epsilon)$,
\begin{eqnarray}\label{e:BB.3}
\mathscr{L}_{\lambda}\circ\Phi_{\lambda}(x)=\|P_+^{\lambda^\ast}x\|_H^2-\|P_-^{\lambda^\ast}x\|_H^2+ \mathscr{
L}_{{\lambda}}(P_0^{\lambda^\ast}x+ \mathfrak{h}({\lambda}, P_0^{\lambda^\ast}x)). %,\;\forall x\in B_X(0,\epsilon).
\end{eqnarray}
\end{description}
In addition, with {\bf notations} $X_\pm^{\lambda^\ast}:=X_+^{\lambda^\ast}\oplus X_-^{\lambda^\ast}$ and
$P_\pm^{\lambda^\ast}:=P_+^{\lambda^\ast}+P_-^{\lambda^\ast}$,
%for each $\lambda\in\Lambda_0$,
there also hold:
\begin{description}
\item[{\bf (iii)}]
%\begin{equation}\label{e:BB.3+}
%d_z\mathfrak{h}(\lambda, 0)=-[P_\pm{B}_{\lambda_0}(0)|_{X_\pm}]^{-1}
%(P_\pm{B}_\lambda(0))|_{X_0};
%\end{equation}
\begin{equation}\label{e:BB.3+}
d_z\mathfrak{h}(\lambda,z)=-[P_\pm^{\lambda^\ast}\circ({B}_{\lambda}(z+\mathfrak{h}(\lambda,z))|_{X_\pm^{\lambda^\ast}})]^{-1}
\circ(P_\pm^{\lambda^\ast}\circ({B}_\lambda(z+\mathfrak{h}(\lambda,z))|_{X_0^{\lambda^\ast}})).
\end{equation}
\item[{\bf (iv)}]  the functional
\begin{equation}\label{e:BB.4}
\mathscr{L}_{\lambda}^\circ: B_{X_0^{\lambda^\ast}}(0, \epsilon)\to \mathbb{R},\;
z\mapsto\mathscr{L}_{\lambda}(z+ \mathfrak{h}({\lambda}, z))
\end{equation}
 is of class $C^{2}$, and has the first-order  derivative at $z_0\in
B_{X_0^{\lambda^\ast}}(0, \epsilon)$  given by
 \begin{eqnarray}\label{e:BB.5}
d\mathscr{L}^\circ_\lambda(z_0)[z]=\bigl(A_\lambda(z_0+ \mathfrak{h}(\lambda, z_0)), z\bigr)_H,\quad\forall z\in X_0^{\lambda^\ast},
\end{eqnarray}
and the second-order derivative at $0\in
B_{X_0^{\lambda^\ast}}(0, \epsilon)$
 \begin{eqnarray}\label{e:BB.6}
  &&d^2\mathscr{L}^\circ_\lambda(0)[z,z']=\left(P_0^{\lambda^\ast}\bigr[{B}_\lambda(0)-
 {B}_\lambda(0)(P_\pm^{\lambda^\ast}{B}_{\lambda}
 (0)|_{X_\pm^{\lambda^\ast}})^{-1}(P_\pm^{\lambda^\ast}{B}_\lambda(0))\bigr]z, z'\right)_H,\nonumber\\
&& \hspace{40mm} \forall z,z'\in X_0^{\lambda^\ast};
 \end{eqnarray}
\item[{\bf (v)}] the map $z\mapsto z+ \mathfrak{h}({\lambda}, z))$ induces an one-to-one correspondence
 between the critical points of  $\mathscr{L}_{\lambda}^\circ$ near $0\in X_0^{\lambda^\ast}$
and those of $\mathscr{L}_{\lambda}$ near $0\in X$.
\end{description}

Let $G$ be a compact Lie group acting on $H$ orthogonally,
which induces a $C^1$ isometric action on $X$. Suppose that each $\mathscr{L}_\lambda$ is $G$-invariant and that
 $A_\lambda, B_\lambda$  are equivariant. Then for each $\lambda\in\Lambda_0$,
 both $\mathfrak{h}(\lambda,\cdot)$ and $\Phi_\lambda$ are  equivariant, and
 $\mathscr{L}_{\lambda}^\circ$ is $G$-invariant.

 Finally, if $\Lambda$ is a $C^1$ manifold, and $\mathscr{L}_\lambda$, $A_\lambda$, $B_\lambda$ depend
 on $\lambda$ in the $C^1$ way, then the maps in (\ref{e:BB.1-}) and (\ref{e:BB.2}) are of class $C^1$.
\end{theorem}

\begin{proof}
A standard implicit function theorem argument yields (i) and (\ref{e:BB.3+}).
As in Step 2 of the proof of \cite[Lemma~3.1]{Lu2} we can get (\ref{e:BB.5}).
Then it and (\ref{e:BB.3+}) lead to (\ref{e:BB.6}).

By shrinking $\Lambda_0$ and $\epsilon$ we can assume that
$z+ \mathfrak{h}(\lambda,z))\in B_X(0, \delta/2)$ for all $(\lambda,z)\in \Lambda_0\times B_{X_0^{\lambda^\ast}}(0,\epsilon)$.
Let  $H_\pm^{\lambda^\ast}$ be the closure of $X_\pm^{\lambda^\ast}$ in $H$.
For each $(\lambda,z)\in \Lambda_0\times B_{X_0^{\lambda^\ast}}(0,\epsilon)$, define
$$
{\bf L}_{(\lambda,z)}:B_{X_\pm^{\lambda^\ast}}(0,\delta/2)\to\mathbb{R},\;u\mapsto \mathscr{L}_{\lambda}(z+ \mathfrak{h}(\lambda,z))+u).
$$
It is easily checked that ${\bf L}_{(\lambda,z)}$ is
$(B_{X_\pm^{\lambda^\ast}}(0, \delta/2), H_\pm^{\lambda^\ast})$-regular with corresponding bounded maps
\begin{eqnarray*}
&&{\bf A}_{(\lambda,z)}: B_{X_\pm^{\lambda^\ast}}(0, \delta/2)\to H_\pm^{\lambda^\ast},\;u\mapsto (P_+^{\lambda^\ast}+P_-^{\lambda^\ast})A(z+\mathfrak{h}(\lambda,z))+u),\\
&&{\bf B}_{(\lambda,z)}: B_{X_\pm^{\lambda^\ast}}(0, \delta/2)\to \mathscr{L}_s(H_\pm^{\lambda^\ast}),\;u\mapsto (P_+^{\lambda^\ast}
+P_-^{\lambda^\ast})\circ (B(z+\mathfrak{h}(\lambda,z))+u)|_{H_\pm^{\lambda^\ast}}).
\end{eqnarray*}
Clearly, ${\bf B}_{(\lambda^\ast, 0)}|_{X_\pm^{\lambda^\ast}}=(P_+^{\lambda^\ast}+P_-^{\lambda^\ast})\circ
(B(0)|_{X_\pm^{\lambda^\ast}})\in\mathscr{L}(X_\pm^{\lambda^\ast})$  and is  hyperbolic.
By Theorem~\ref{th:BB.3}, by further shrinking $\Lambda_0$ and $\epsilon$ we have
 a family of  origin-preserving  $C^1$ diffeomorphism $\varphi_{(\lambda,z)}$ from
$B_{X_\pm^{\lambda^\ast}}(0,\epsilon)$ onto an open neighborhood of $0$ in $X_\pm^{\lambda^\ast}$, which continuously depends on
$(\lambda,z)\in \Lambda_0\times B_{X_0^{\lambda^\ast}}(0,\epsilon)$ (and  on $z$ in $C^1$ way),
  such that for all $(\lambda,z)\in \Lambda_0\times B_{X_0^{\lambda^\ast}}(0,\epsilon)$,
\begin{eqnarray}\label{e:BB.6.1}
{\bf L}_{(\lambda,z)}\circ\varphi_{(\lambda,z)}(x)=\|P_+^{\lambda^\ast}x\|_H^2-\|P_-^{\lambda^\ast}x\|_H^2+ {\bf L}_{(\lambda,z)}(0),\quad\forall x\in B_{X_\pm^{\lambda^\ast}}(0,\epsilon).
\end{eqnarray}
For each $\lambda\in\Lambda_0$, define $\Phi_\lambda: B_X(0,\epsilon/2)\to  X$ by
$\Phi_\lambda(x)=z+\mathfrak{h}(\lambda,z)+ \varphi_{(\lambda,z)}(y)$, where $z=P_0^{\lambda^\ast}x$ and $y=P_\pm^{\lambda^\ast} x$.
This and (\ref{e:BB.6.1}) show that (\ref{e:BB.3}) holds for all $x\in B_X(0,\epsilon/2)$.

 For $z'\in X_0^{\lambda^\ast}$ and $y'\in X_\pm^{\lambda^\ast}$ a direct computation
leads to
 \begin{eqnarray}\label{e:BB.6.2}
 D\Phi_\lambda(0)[z'+y']&=&z'+ D_z\mathfrak{h}(\lambda,0)[z']+ D\varphi_{(\lambda,0)}(0)[y']+ \frac{d}{dt}|_{t=0}\varphi_{(\lambda,tz')}(0)\nonumber\\
 &=&z'+ D_z\mathfrak{h}(\lambda,0)[z']+ D\varphi_{(\lambda,0)}(0)[y']
 \end{eqnarray}
because  $\varphi_{(\lambda,z)}(0)=0\;\forall (\lambda,z)\in \Lambda_0\times B_{X_0^{\lambda^\ast}}(0,\epsilon)$.
Note that the last two terms in (\ref{e:BB.6.2}) sit in $X_\pm^{\lambda^\ast}$.
It is easily checked that $D\Phi_\lambda(0)$ is an injection.
For $u+v\in X$, where $u\in X_0^{\lambda^\ast}$ and $v\in X_\pm^{\lambda^\ast}$,
since $D\varphi_{(\lambda,0)}(0)\in\mathscr{L}(X_\pm^{\lambda^\ast})$ is an isomorphism,
we can choose $y'\in X_\pm^{\lambda^\ast}$ such that
$D\varphi_{(\lambda,0)}(0)[y']=v-D_z\mathfrak{h}(\lambda,0)[u]$, and therefore
$D\Phi_\lambda(0)[u+y']=u+v$. This shows that  $D\Phi_\lambda(0)\in\mathscr{L}(X)$ is surjective, and so
a Banach space isomorphism by the Banach inverse operator theorem.
%of (\ref{e:BB.4}) and the fact that
%Since $D\varphi_{(\lambda,0)}(0)\in\mathscr{L}(X_\pm^{\lambda^\ast})$ is an isomorphism, so is
%$D\Phi_\lambda(0)\in\mathscr{L}(X)$.
Hence after shrinking $\Lambda_0$
we can get $0<\epsilon'\ll\epsilon/2$ such that $\Phi_{\lambda}$ is a $C^1$  origin-preserving diffeomorphism from
$B_X(0,\epsilon')$ onto an open neighborhood $W_\lambda$ of $0$ in $X$ by the inverse function theorem.
 Other claims easily follows by carefully checking the above proof.
\end{proof}

%By (\ref{e:BB.1}) and (\ref{e:BB.5}), it is easily proved that
Since $0\in X_0^{\lambda^\ast}$ is an isolated critical point of  $\mathscr{L}_{\lambda}^\circ$ if and only if
$0\in X$ is such a critical point of $\mathscr{L}_{\lambda}$,
as usual Theorem~\ref{th:BB.5} lead to:

\begin{corollary}[Shifting]\label{cor:BB.6}
Under the assumptions of Theorem~\ref{th:BB.5},
let $\nu_{\lambda^\ast}:=\dim X_0^{\lambda^\ast}<\infty$ and $\mu_{\lambda^\ast}:=\dim X_-^{\lambda^\ast}<\infty$.
For any Abel group ${\bf K}$, and for any given $\lambda\in\Lambda_0$, if $0\in X$ is an
isolated critical point of $\mathscr{L}$ then
$C_q(\mathscr{L}_\lambda, 0;{\bf K})\cong
C_{q-\mu_{\lambda^\ast}}(\mathscr{L}^{\circ}_\lambda, 0;{\bf K})$ for all
$q\in\mathbb{N}\cup\{0\}$.
\end{corollary}

\begin{claim}\label{cl:BB.6+}
In Theorem~\ref{th:BB.5}, if  for some $\lambda\in\Lambda_0\setminus\{\lambda^\ast\}$,
$0\in X$ is a nondegenerate critical point of $\mathscr{L}_\lambda$,
i.e.,  the operator $B_\lambda(0)|_X\in\mathscr{L}(X)$  is   hyperbolic (\cite{Uh0}),
%${\rm Ker}({B}_\lambda(0))=\{0\}$,
then $0\in X_0^{\lambda^\ast}$ is such a critical point of $\mathscr{L}^\circ_\lambda$ too.
\end{claim}

\begin{proof}
Let $z\in X_0^{\lambda^\ast}$ be such that  $d^2\mathscr{L}^\circ_\lambda(0)[z,z']=0\;\forall z'\in X_0^{\lambda^\ast}$.
By (\ref{e:BB.6}) we deduce
 \begin{eqnarray}\label{e:BB.7}
  P_0^{\lambda^\ast}\bigr[{B}_\lambda(0)-
 {B}_\lambda(0)(P_\pm^{\lambda^\ast}{B}_{\lambda^\ast}
 (0)|_{X_\pm^{\lambda^\ast}})^{-1}(P_\pm^{\lambda^\ast}{B}_\lambda(0))\bigr]z=0.
 \end{eqnarray}
Moreover, differentiating the equality in (\ref{e:BB.1}) at $z=0$ and using (\ref{e:BB.3+}) we get that
\begin{eqnarray*}
0&=&(id_X-P_0^{\lambda^\ast})B_\lambda(0)|_X(z+ d_z\mathfrak{h}(\lambda,0)z)\\
&=&(id_X-P_0^{\lambda^\ast})B_\lambda(0)|_X(z-[P_\pm^{\lambda^\ast}{B}_{\lambda^\ast}(0)|_{X_\pm^{\lambda^\ast}}]^{-1}
(P_\pm^{\lambda^\ast}{B}_\lambda(0))|_{X_0^{\lambda^\ast}}z)\\
&=&(id_X-P_0^{\lambda^\ast})[B_\lambda(0)-B_\lambda(0)(P_\pm^{\lambda^\ast}{B}_{\lambda^\ast}(0)|_{X_\pm^{\lambda^\ast}})^{-1}
(P_\pm^{\lambda^\ast}{B}_\lambda(0))|_{X_0^{\lambda^\ast}})]z\\
&=&[B_\lambda(0)-B_\lambda(0)(P_\pm^{\lambda^\ast}{B}_{\lambda^\ast}(0)|_{X_\pm^{\lambda^\ast}})^{-1}
(P_\pm^{\lambda^\ast}{B}_\lambda(0))|_{X_0^{\lambda^\ast}})]z\\
&-&P_0^{\lambda^\ast}[B_\lambda(0)-B_\lambda(0)(P_\pm^{\lambda^\ast}{B}_{\lambda^\ast}(0)|_{X_\pm^{\lambda^\ast}})^{-1}
(P_\pm^{\lambda^\ast}{B}_\lambda(0))|_{X_0^{\lambda^\ast}})]z\\
&=&[B_\lambda(0)-B_\lambda(0)(P_\pm^{\lambda^\ast}{B}_{\lambda^\ast}(0)|_{X_\pm^{\lambda^\ast}})^{-1}
(P_\pm^{\lambda^\ast}{B}_\lambda(0))|_{X_0^{\lambda^\ast}})]z
\end{eqnarray*}
because of  (\ref{e:BB.7}). Since $B_\lambda(0)$ is invertible, we deduce
$$
z=(P_\pm^{\lambda^\ast}{B}_{\lambda^\ast}(0)|_{X_\pm^{\lambda^\ast}})^{-1}
(P_\pm^{\lambda^\ast}{B}_\lambda(0))|_{X_0^{\lambda^\ast}})z.
$$
Note that $z\in X_0^{\lambda^\ast}$ and that the right side belongs to $X_\pm^{\lambda^\ast}$.
Hence $z=0$.
\end{proof}

%Let us give some convenient conditions under which
In general, it is not easy to judge whether $\sigma(B(0)|_X)\setminus\{0\}$ is bounded away from the imaginary axis.
However, it is not hard to prove that this condition is equivalent to the following:
\begin{description}
\item[(*)] $0$ is at most an isolated point of
$\sigma(B(0)|_X)$ and $B(0)$ induces a hyperbolic  operator on the quotient space $X/X_0$,
where $X_0={\rm Ker}(B(0)|_X)$.
\end{description}

For a class of operators used in this paper the following lemma determines when $0$
is at most an isolated point of $\sigma(B(0))$.

\begin{lemma}\label{lem:BB.7}
 Let a self-adjoined operator $\mathfrak{B}\in\mathscr{L}_s(H)$ be a sum $\mathfrak{B} = \mathfrak{P} +
 \mathfrak{Q}$, where
$\mathfrak{P}\in\mathscr{L}(H)$ is invertible, and $\mathfrak{Q}\in\mathscr{L}(H)$ is compact. Let
 $0\in\sigma(\mathfrak{B})$. Then $0$ is an
isolated point of $\sigma(\mathfrak{B})$ and an eigenvalue of $\mathfrak{B}$ of the finite multiplicity. ({\cite[Lemma~2.2]{BoBu}})

 In addition, if $\mathfrak{P}$ is also positive definite, then every $\lambda<\inf\{(\mathfrak{P}u,u)_H\,|\,\|u\|_H=1\}$
 is either a regular value of $\mathfrak{B}$ or an isolated point of
$\sigma(\mathfrak{B})$, which is also an eigenvalue of finite multiplicity. (\cite[Proposition~B.2]{Lu2})
\end{lemma}

 As before let $H^0$, $H^+$ and $H^-$ be null, positive and negative definite spaces of $B(0)$,
and let $P^\ast:H\to H^\ast$ for $\ast=0,+,-$,  be the orthogonal projections.
Clearly, $X_0\subset H^0$. The following lemma shows how the Morse index $\mu$
is computed.

\begin{lemma}\label{lem:BB.8}
 If $X_0=H^0$ and
either $\dim H^-<\infty$ or $\dim H^+<\infty$, then
$X_\ast$ are dense subspaces of $H^\ast$ for $\ast=+,-$, moreover
$X_-=H^-$ (resp. $X_+=H^+$) as
 $\dim H^-<\infty$ (resp. $\dim H^+<\infty$).
(These imply $X_\ast=H^\ast\cap X$ and $P_\ast=P^\ast|_{X}$
for $\ast=0,+,-$, and hence
 $X_\pm=X\cap H^\pm=(P^++P^-)(X)$, where $H^\pm:=H^++H^-$.)
\end{lemma}
\begin{proof}
Since $X$ is dense in $H$, it is easily proved that $X_\pm:=X_+\oplus X_-$ is dense in
$H^\pm$. Let $X_+^H$ and $X_-^H$ be the closure of $X_+$ and $X_-$ in $H$, respectively.
They are orthogonal in $H$, invariant with respect to $B(0)$ and $H^\pm=X_+^H+ X_-^H$.
Since $B(0)$ is semi-positive (resp. semi-negative) on $X_+^H$ (resp. $X_-^H$)
by Lemma~\ref{lem:BB.2}, and $B(0)$ has no nontrivial kernel on $H^\pm$, we deduce that
$B(0)$ must be positive (resp. negative) on $X_+^H$ (resp. $X_-^H$).
Let $\dim H^-<\infty$. If $X^H_-\cap H^-\ne\{0\}$, it is an invariant subspace of $B(0)$.
Denote by $X^H_{--}$ and $H^{--}$ the orthogonal complements of $X^H_-\cap H^-$ in $X^H_-$ and $H^-$,
respectively. Then $X^H_{--}+ H^{--}+X^H_-\cap H^-$ is a negative definite subspace of $B(0)$.
But $\dim H^-$ is the maximal dimension of negative definite subspaces of $B(0)$.
Hence $X^H_{--}=\{0\}$,  and thus $X^H_{-}=X^H_-\cap H^-$. The latter means $X^H_{-}\subset H^-$
and so $X^H_{+}\supseteq H^+$. For any $x\in H^-$, since
$X_\pm=X_+\oplus X_-$ is dense in $H^\pm=H^++H^-$, we have a sequence $(x_n^++x_n^-)\subset X_+\oplus X_-$
such that $x_n^++x_n^-\to x$. Note that
$\|x_n^++x_n^--x_m^+-x_m^-\|^2=\|x_n^+-x_m^+\|^2+\|x_n^--x_m^-\|^2$.
We get $x_n^+\to x^+\in X^H_+$ and $x_n^-\to x^-\in X^H_-$. Hence $x=x^++x^-$.
But $(x,x^+)_H=0$. So $x^+=0$ and $x=x^-$. This shows that $X_-$ is dense in $H^-$
and thus $X_-=X_-^H=H^-$ (since $\dim H^-<\infty$). The latter implies $X_+^H=H^+$.
 Similarly, we can deal with the case that $\dim H^+<\infty$.
\end{proof}

%Regrettably, the corresponding notion of Morse indexes with this setting was not mentioned in \cite{BoBu}.
%It is necessary for us to derive suitable shifting theorems from the above splitting theorems.

%\bibliographystyle{amsnumber}

\medskip

\begin{tabular}{l}
 School of Mathematical Sciences, Beijing Normal University\\
 Laboratory of Mathematics and Complex Systems, Ministry of Education\\
 Beijing 100875, The People's Republic of China\\
 E-mail address: gclu@bnu.edu.cn\\
\end{tabular}

%\medskip
%\quad% The data information below will be filled by AIMS editorial staff
%%Received August 2017; revised October 2017.
%\medskip
\end{document}